%% file: main.tex
\RequirePackage{amsmath, amssymb, bm}

\documentclass[reqno]{amsart}

\input{part-infra}

\input{rigorous_cgc/00-preamble_light}

\input{ideal_sampling/00-preamble}

\input{ideal_sampling/00-arakelov-preamble}

\input{ideal_sampling/00-IS-preamble}

\author[K. de Boer, A. Pellet-Mary, B. Wesolowski]{
\uppercase{Koen de Boer\textsuperscript{1},
 Alice Pellet-Mary\textsuperscript{2},
 and Benjamin Wesolowski\textsuperscript{3}}\\
 \\
 \textsuperscript{1}Leiden University, Mathematical Institute, Leiden, The Netherlands\\
 \textsuperscript{2}Univ. Bordeaux, CNRS, Inria, Bordeaux INP, IMB, UMR 5251,  F-33400, Talence, France\\
 \textsuperscript{3}ENS de Lyon, CNRS, UMPA, UMR 5669, Lyon, France
}

\title[Rigorous methods for computational number theory]{Rigorous methods for \\computational number theory}

\date{\today}

\begin{document}

\begingroup
	\let\MakeUppercase\relax %
	\maketitle
\endgroup

\input{abstract}

\setcounter{tocdepth}{1}
\tableofcontents

\input{introduction}

\input{acknowledgements}

\part{Sampling ideals in a class: \texorpdfstring{\except{toc}{\linebreak}}{}smooth, near-prime or otherwise}
\label{part1}

\input{ideal_sampling/A01-introduction}

\input{ideal_sampling/A02-prelim}

\input{ideal_sampling/A03-RWRAY}

\input{ideal_sampling/A04-correspondence}

\input{ideal_sampling/A05-samplinginabox}

\input{ideal_sampling/A06-ideal-sampling-algorithm}

\input{ideal_sampling/A07-properties-of-sampling-algorithm}

\input{ideal_sampling/A08-norm-modulus-bound}

\part{Rigorous computation of \texorpdfstring{\except{toc}{\linebreak}}{}class groups and unit groups}
\label{part2}

\input{rigorous_cgc/B01-introduction}

\input{rigorous_cgc/B02-prelim}

\input{rigorous_cgc/B03-density-smooth-ideals}

\input{rigorous_cgc/B04a-s-unit-background}

\input{rigorous_cgc/B04b-covering-radius-s-unit-lattice}

\input{rigorous_cgc/B04-generating-a-single-s-unit}

\input{rigorous_cgc/B05-generating-many-independent-s-units-A}

\input{rigorous_cgc/B06-postprocessing.tex}

\input{rigorous_cgc/B05b-generating-exceptional-unit}

\input{rigorous_cgc/B07-finaltheorem.tex}

\part{Provable lattice reduction techniques}

\label{part2b}
\input{discretization/introduction}

\input{discretization/C-A-buchmannkessler}

\input{discretization/C-A-approxdeterminant}

\input{discretization/C-A-postprocessing}

\input{discretization/C-A-BKZ}

\bibliographystyle{abbrv}
\bibliography{bib/ANT,bib/lattices,bib/mybib}

\part{Appendix}
\label{part3}

\appendix
\renewcommand{\thesection}{A.\arabic{section}}

\input{appendix/C-A2-RWproofs}

\input{appendix/C-A3-lambdan_upperbound}

\input{appendix/C-A4-f_upperbound}

\end{document}

%% file: rigorous_cgc/00-preamble_light.tex
\usepackage{thmtools}
\usepackage{thm-restate}
\usepackage{xspace}
\usepackage{mdframed}

\usepackage{resizegather}

\makeatletter
\newcommand\footnoteref[1]{\protected@xdef\@thefnmark{\ref{#1}}\@footnotemark}
\makeatother

\newcommand{\dell}{\ell^*}
\newcommand{\bc}{\mathbf{c}}

\newcommand{\rayelt}[1]{\dExp{#1}_\tau}

\newcommand*{\placeholder}{\makebox[2ex]{\textbf{${-}$}}}%

\newcommand{\normbase}[1]{\lVert #1 \rVert}

\newcommand{\divnorm}[1]{\normbase{#1}}

\newcommand{\normx}[1]{\normbase{#1}}
\newcommand{\normtwoexplicit}[1]{\normbase{#1}_2}

\newcommand{\normmattwo}[1]{\normbase{#1}_2}
\newcommand{\normfrob}[1]{\normbase{#1}_F} %
\newcommand{\normt}[1]{\normbase{#1}_{2,\infty}}

\newcommand{\bs}{\mathbf{s}}
\newcommand{\bt}{\mathbf{t}}

\newcommand{\taunfrm}{\tau\nfrm}

\newcommand{\emb}{\sigma}
\newcommand{\embn}{\sigma_\nu}
\newcommand{\pl}{\nu}
\newcommand{\npl}{n_\nu} %

\newcommand{\nplemb}{n_{\nu_\emb}}

\newcommand{\geps}{\eps_{\mathcal{G}}}
\newcommand{\gaussepsinv}{100}
\newcommand{\mN}{\mathbf{N}}
\newcommand{\prob}[1]{\underset{#1}{\mathbb{P}}}
\newcommand{\vx}{\mathbf{x}}

\newcommand{\av}{\mathbf{a}}

\newcommand{\units}{\OK^\times}

\DeclareMathOperator{\rank}{rank}
\newcommand{\dedres}{\rho_K}
\newcommand{\letterSunits}{S} %
\newcommand{\mT}{{\mathbb{\letterSunits}}} %
\newcommand{\mS}{\mT}
\newcommand{\sunits}{\mathcal{O}_{K,\mT}^\times}
\newcommand{\logs}{\Log_{\mT}}

\newcommand{\lsu}{\logsunits}
\newcommand{\logunits}{\Log(\OKstar)}
\newcommand{\lograyunits}{\logunitsmodu}
\newcommand{\logunitsmodu}{\Log(\OKmodu)}
\newcommand{\logsunits}[1][\mT]{\Log_{#1}(\mathcal{O}_{K,#1}^\times)}
\newcommand{\Xunits}[1]{\mathcal{O}_{K,#1}^\times}

\newcommand{\mG}{\mathbf{G}}
\newcommand{\mGa}{\tilde{\mathbf{G}}}

\newcommand{\mA}{\mathbf{A}}
\newcommand{\mAa}{\mathbf{\tilde{A}}}
\newcommand{\mAh}{\mathbf{\hat{A}}}

\newcommand{\mB}{\mathbf{B}}
\newcommand{\mBa}{\tilde{\mathbf{B}}}
\newcommand{\mC}{\mathbf{C}}

\newcommand{\mBz}{\mathbf{B}_0}
\newcommand{\mBza}{\tilde{\mathbf{B}}_0}
\newcommand{\mQ}{\mathbf{Q}}

\newcommand{\gammamodu}{\gamma_\modu}

\newcommand{\tmB}{\tilde{\mB}}
\newcommand{\mD}{\mathbf{D}}
\newcommand{\tmD}{\tilde{\mD}}
\newcommand{\tmH}{\tilde{\mathbf{H}}}

\newcommand{\dlat}{\Lambda^\vee} %
\newcommand{\dbb}{\bb^\vee} %

\newcommand{\gmB}{\mB^\star} %
\newcommand{\gbb}{\bb^\star} %

\newcommand{\lineref}[1]{\ref{#1}}
\newcommand{\itemref}[1]{\ref{#1}}
\newcommand{\partref}[1]{\ref{#1}}

\newcommand{\fp}{{\mathfrak{p}}}
\newcommand{\fq}{{\mathfrak{q}}}
\newcommand{\fa}{{\mathfrak{a}}}

\newcommand{\fs}{{\mathfrak{s}}}

\newcommand{\Gaussian}{\mathcal{G}}
\newcommand{\divsd}{\varsigma}
\newcommand{\latsd}{\varsigma}

\newcommand{\gaussemi}{\Gaussian}
\newcommand{\gaussemis}{\gaussemi_\divsd}

\newcommand{\gaussian}{g}
\newcommand{\KR}{K_{\R}}
    
\newcommand{\Walk}{\mathcal{W}}
\newcommand{\WalkT}{\mathcal{W}^T}

\newcommand{\logtunit}{Log-$\mT$-unit}

\DeclareMathOperator{\covol}{Vol}
\DeclareMathOperator{\cov}{cov}

\DeclareMathOperator{\rr}{\varrho}

\newcommand{\Sample}{\ensuremath{\textup{\texttt{Sample}}}}  %

\newcommand{\mI}{\mathbf{I}}
\newcommand{\mM}{\mathbf{M}}
\newcommand{\mR}{\mathbf{R}}
\newcommand{\mE}{\mathbf{E}}

\newcommand{\mv}{\mathbf{m}}
\newcommand{\bbv}{\mathbf{b}}
\newcommand{\gv}{\mathbf{g}}

\newcommand{\detlogsunits}{ \reg \cdot \classnumber \cdot \sqrt{\rem + \cem}}

\newcommand{\bkp}{Buchmann-Kessler-Pohst\xspace}

%% file: ideal_sampling/00-preamble.tex
\usepackage[english]{babel}

\usepackage{amsthm}
\usepackage{scrextend}

\usepackage{hyperref}
\usepackage{algorithm}
\usepackage{algorithmic}
\usepackage{todonotes}
\usepackage{booktabs}
\usepackage{multirow}
\usepackage{makecell}

\usepackage{cleveref}
\usepackage{stmaryrd}
\usepackage{enumerate}

\hypersetup{colorlinks=true, citecolor=blue}

\usepackage{tikz-cd}
\usepackage{bbm}

\usepackage{array}
\usepackage{subcaption}
\usepackage{attachfile}
\usepackage{hyperref}
\usepackage{tikz} 
\usepackage{pgfplots} 

\usepackage{mathtools}

\pgfplotsset{compat=1.16}

\usetikzlibrary{matrix}

\newcommand{\idealset}{\idset}

\newcommand{\eps}{\varepsilon}
\newcommand{\N}{\mathbb{N}}
\newcommand{\Z}{\mathbb{Z}}
\newcommand{\Q}{\mathbb{Q}}
\newcommand{\R}{\mathbb{R}}
\newcommand{\C}{\mathbb{C}}
\newcommand{\K}{K}
\newcommand{\OK}{\mathcal O_K}
\newcommand{\ma}{\mathfrak a}
\newcommand{\mbb}{\mathfrak b}
\newcommand{\mb}{\mathfrak b}
\newcommand{\mc}{\mathfrak c}

\renewcommand{\mp}{\mathfrak p}
\newcommand{\ba}{\mathbf a}
\newcommand{\bb}{\mathbf b}
\newcommand{\bd}{\mathbf d}
\newcommand{\be}{\mathbf e}

\newcommand{\bu}{\mathbf u}
\newcommand{\bv}{\mathbf v}
\newcommand{\bx}{\mathbf x}
\newcommand{\by}{\mathbf y}
\newcommand{\vnorm}[1]{\left\lVert#1\right\rVert}
\newcommand{\inner}[1]{\left \langle #1 \right \rangle}

\newcommand{\hyper}{H} %

\newcommand{\sprime}{\tilde{s}}

\newcommand{\hyb}{b}

\newcommand{\stupidfactor}{2800} %

\newcommand{\Ballinf}{{{\mathcal B}_\infty}}

\newcommand{\Ballinftau}{{{\mathcal B}_\infty^\tau}}

\newcommand{\covtwo}{{\operatorname{cov}}} %
\newcommand{\covinf}{{\operatorname{cov}_\infty}}
\newcommand{\Ballinfmodu}{{{\mathcal B}_\infty^{\modu}}}
\newcommand{\blocksize}{\mathsf{b}}
\newcommand{\voronoi}{\mathcal{V}_0}
\newcommand{\minksum}{\boxplus} %

\newcommand{\order}{R}
\newcommand{\condord}{\mathfrak{f}_{\order}}
\newcommand{\discord}{\Delta_{\order}}

\newcommand*{\ldb}{\llparenthesis}
\newcommand*{\rdb}{\rrparenthesis}

\newcommand{\lat}{\mathcal L}
\renewcommand{\vec}{}
\DeclareMathOperator{\Tr}{Tr}
\DeclareMathOperator{\Cl}{Cl}
\DeclareMathOperator{\Div}{{Div}}
\DeclareMathOperator{\ideals}{\mathcal{I}_\mathnormal{K}}

\DeclareMathOperator{\ord}{ord}

\DeclareMathOperator{\size}{size}

\newcommand{\reg}{R_K}
\newcommand{\disc}{\Delta_K}
\newcommand{\rem}{{n_\R}}
\newcommand{\cem}{{n_\C}}

\newcommand{\moduz}{\modu_0}
\newcommand{\moduinf}{\modu_\infty}
\newcommand{\modur}{\modu_\R}

\newcommand{\dimh}{\mathbbm{r}}

\newcommand{\infancond}{\mathfrak{q}_\infty} %

\newcommand{\torus}{T}

\newcommand\restr[2]{{%
  \left.\kern-\nulldelimiterspace %
  #1 %
  \vphantom{\big|} %
  \right|_{#2} %
  }}
\newcommand\period[2]{{%
  \left.\kern-\nulldelimiterspace %
  #1 %
  \right|^{#2} %
  }}

\newcommand{\mean}{\mathbb{E}}
\newcommand{\SD}{\text{SD}} 
 
\newcommand{\sd}{s}
\newcommand{\Klein}{\widehat{\mathcal{G}}} %

\DeclareMathOperator{\Pic}{{Pic}}
\DeclareMathOperator{\Log}{Log}
\DeclareMathOperator{\Exp}{Exp}
\DeclareMathOperator{\vol}{Vol}
\DeclareMathOperator{\poly}{\textsf{poly}}
\DeclareMathOperator{\Span}{span}
\DeclareMathOperator{\sign}{sign}

\DeclareMathOperator{\idlat}{IdLat}

\DeclareMathOperator{\norm}{\mathcal{N}}

\newcommand{\distr}{\mathcal{D}}
\newcommand{\bdistr}{\bar{\mathcal{D}}}
\newcommand{\from}{\leftarrow}

\newcommand{\primeset}{\mathcal{P}}
\newcommand{\boxnf}{\mathcal{B}} %

\newcommand{\banas}[1]{\beta_{#1}}

\newcommand{\softO}{\widetilde {O}}

\usepackage{thm-restate}

\newtheorem*{maintheorem*}{Main theorem}

\newtheorem{theorem}{Theorem}[section]
\newtheorem{lemma}[theorem]{Lemma}

\newtheorem{definition}[theorem]{Definition}
\newtheorem{remark}[theorem]{Remark}
\newtheorem{proposition}[theorem]{Proposition}
\newtheorem{corollary}[theorem]{Corollary}

\newtheorem{notation}[theorem]{Notation}

\renewcommand{\paragraph}[1]{\vspace{1mm} \textit{#1}}

\usepackage{comment}%

\includecomment{foundations}
\excludecomment{applications}

%% file: ideal_sampling/00-arakelov-preamble.tex
\newcommand{\nfr}{K_\R}
\newcommand{\volabs}[1]{\lvert #1 \rvert}
\newcommand{\dcrk}{\Delta_K}

\newcommand{\unittorus}{T}

\newcommand{\nfrstar}{\nfr^{\times}}
\newcommand{\nfrm}{\nfr^{\moduinf}}

\newcommand{\heckeop}{\mathcal{H}}

\newcommand{\modu}{\mathfrak{m}}

\newcommand{\varNformula}{\radpar^n \cdot \exp(11 + 16 \log^2 |\dcrk| + 9n^2) }

\newcommand{\radiusformulaconstant}{48} %
\newcommand{\radiusformulaconstantdivtwo}{24} %

\newcommand{\radiusformulanomodulus}{\radiusformulaconstant \cdot \radpar \cdot \blocksize^{2n/\blocksize}   \cdot n^{7/2} \cdot |\dcrk|^{\frac{3}{2n}}}
\newcommand{\radiusformulanomodulusnoomega}[1]{\radiusformulaconstant \cdot \blocksize^{2n/\blocksize}   \cdot n^{7/2} \cdot |\dcrk|^{\frac{3}{2n}} \cdot \norm(#1)^{\frac{1}{n}}}
\newcommand{\radiusformulamacro}[1]{\radiusformulanomodulus \cdot \norm(#1)^{\frac 1 n}}
\newcommand{\radiusformula}{\radiusformulamacro{\moduz}}

\newcommand{\radiusformulaxmb}{\radiusformulamacro{\kx \mbb \moduz}}
\newcommand{\radiusformulaR}{\radiusformulaconstant \cdot \radpar \cdot \blocksize^{2n/\blocksize}   \cdot n^{7/2} \cdot |\Delta(\order)|^{\frac{3}{2n}} \cdot \norm(\moduz)^{\frac{1}{n}}}

\newcommand{\kesslerformulainverse}{\kesslerconstant \cdot \sqrt{n} \cdot \log(n)^3}
\newcommand{\kesslerformula}{\frac{1}{\kesslerformulainverse}}
\newcommand{\kesslerformulainline}{(\kesslerformulainverse)^{-1}}
\newcommand{\kesslerconstant}{1000}  %

\newcommand{\defphi}{|\Kmodumoduz/\Kmoduz|}

\newcommand{\Kmodu}{K^{\modu,1}}
\newcommand{\Kmodumodu}{K^\modu}
\newcommand{\Kmoduz}{K^{\moduz,1}}
\newcommand{\Kmodumoduz}{K^{\moduz}}

\newcommand{\OKmodu}{\mathcal{O}_{K^{\modu,1}}^\times}
\newcommand{\muKmodu}{\mu_{\Kmodu}}
\newcommand{\Tmodu}{T^{\modu}}

\newcommand{\rayDiv}{\Div_{K^\modu}}
\newcommand{\rayPic}{\Pic_{K^\modu}}

\newcommand{\rayideals}{\mathcal{I}_{K}^{\modu}}

\newcommand{\DivKS}{{\Div_{K,\mS}}}
\newcommand{\varN}{N}
\newcommand{\dDivKS}{{\Div_{K,\mS,\varN}}}

\newcommand{\OKstar}{\OK^{\times}}

\newcommand{\rayclassgroup}{\Cl_K^\modu}
\newcommand{\rayclassgroupz}{\Cl_K^{\modu_0}}
\newcommand{\classgroup}{\Cl_K}
\newcommand{\classnumber}{h_K}
\newcommand{\idealsmodu}{\mathcal{I}_K^\modu}
\newcommand{\prid}{\mathrm{Princ}_K}
\newcommand{\pridmodu}{\prid^\modu}

\newcommand{\finitepart}[1]{{#1}_\mathrm{f}}
\newcommand{\secideal}{d^0}%
\newcommand{\infinitepart}[1]{{#1}_\infty}
\newcommand{\infpart}[1]{{#1}_\infty}
\newcommand{\finpart}[1]{{#1}_\mathrm{f}}

\newcommand{\subpic}{G}

\newcommand{\unif}{\mathcal{U}}

%% file: ideal_sampling/00-IS-preamble.tex
 \usepackage{algorithmic}

 \newcommand{\noopsort}[1]{}

\newcommand{\hkztime}{ \blocksize^\blocksize}
\newcommand{\redgamma}{\gamma_{\mathrm{red}}}

\newcommand{\countfun}[2]{\lvert #1(#2) \rvert }
\newcommand{\idset}{\mathcal{S}}

\newcommand{\idsetmoduG}{\idset}%

\newcommand{\smoothset}{\idset_B}

\newcommand{\smoothnessvarBintro}{B}

\newcommand{\rwB}{B_{\mathrm{rw}}}

\newcommand{\smoothdensityB}{B_{\mathrm{sm}}}

\newcommand{\onerelationB}{B_{\mathrm{max}}}

\newcommand{\tdedres}{\tilde{\rho}_K}

\newcommand{\dedrescut}{\rho_K^{\mathrm{cut}}}

\newcommand{\ky}{y}
\newcommand{\kx}{x}

\newcommand{\ha}{a}
\newcommand{\hb}{b}
\newcommand{\hc}{c}

\newcommand{\dH}{\ddot{\hyper}}

\newcommand{\radpar}{\omega}

\renewcommand{\paragraph}[1]{\textit{#1}}
\newcommand{\dExp}[1]{\Exp(#1)}
\newcommand{\raygen}[1]{{\dExp{#1}}^\times_{\tau}}
\newcommand{\raygenone}[1]{{\dExp{#1}}^\times_{1}}
\newcommand{\raygenx}[2]{{\dExp{#1}}^\times_{#2}}
\newcommand{\dederesidue}{\rho_K}
\newcommand{\expval}{\mathbb{E}}
\newcommand{\psucc}{p_{\mathrm{success}}}

\newcommand{\tbeta}{\tilde{\beta}}

\newcommand{\nf}{K}
\newcommand{\tmb}{\bar{\mb}}%

\newcommand{\volsim}[1]{C(r,#1)}
\newcommand{\volsimN}[1]{C(r,\norm(#1))}
\newcommand{\rou}{\mu_K}

\newcommand{\idprob}{p}

\usepackage{comment}%

%% file: abstract.tex
\begin{abstract}
We present the first algorithm for computing class groups and unit groups of 
arbitrary number fields that provably runs in probabilistic subexponential 
time, assuming the Extended Riemann Hypothesis (ERH).
Previous subexponential algorithms were either restricted to imaginary quadratic fields, or 
relied on several heuristic assumptions that have long resisted rigorous analysis.

The heart of our method is a new general strategy to provably solve a recurring computational problem in number theory (assuming ERH): given an ideal class $[\mathfrak{a}]$ of a number field $K$, sample an ideal $\mathfrak b \in [\mathfrak{a}]$ belonging to a particular family of ideals (e.g., the family of smooth ideals, or near-prime ideals).
More precisely, let $\mathcal{S}$ be an arbitrary family of ideals, and $\mathcal{S}_B$ the family of $B$-smooth ideals. We describe an efficient algorithm that samples ideals $\mathfrak b \in [\mathfrak{a}]$ such that $\mathfrak b \in \mathcal{S} \cdot\mathcal{S}_B$ with probability proportional to the density of $\mathcal{S}$ within the set of all ideals.

The case where $\mathcal{S}$ is the set of prime ideals yields the family $\mathcal{S}\cdot\mathcal{S}_B$ of near-prime ideals, of particular interest in that it constitutes a dense family of efficiently factorable ideals.
The case of smooth ideals $\mathcal{S} = \mathcal{S}_B$ regularly comes up in index-calculus algorithms (notably to compute class groups and unit groups), where it has long constituted a theoretical obstacle overcome only by heuristic arguments.
\end{abstract}

%% file: introduction.tex
\section{Introduction} \label{section:generalintro}
\noindent
Many number theoretic algorithms resort to heuristic assumptions for their analysis.
This issue concerns even the most fundamental problems of the field, such as the computation of class groups in subexponential time. 
This persistent need for heuristic assumptions often stems from a step of this form: given an ideal class $[\mathfrak a]$ of a number field $K$, find a representative $\mathfrak b \in [\mathfrak a]$ belonging to a particular family $\idset$ of ideals (for instance, the family of smooth ideals).
It is relatively simple to design an algorithm for this task: sample a random representative $\mathfrak b \in [\mathfrak a]$, and hope that it belongs to the desired family $\idset$. One then heuristically argues that the probability that $\mathfrak b \in \idset$ should be proportional to the density of $\idset$. For instance, the subexponential density of smooth ideals heuristically implies that one can find smooth representatives in subexponential time. This is the heart of state-of-the-art algorithms to compute class groups, unit groups, or generators of principal ideals in number fields~\cite{ANTS:BiasseFiecker14,buchmann1988subexponential,lenstra1993development}, and has long constituted a theoretical obstacle overcome only by heuristic arguments (with the exception of quadratic fields~\cite{HM89}).

In the first part of this paper, we propose a general strategy to solve these ideal sampling tasks rigorously and efficiently, assuming only the extended Riemann hypothesis (henceforth, ERH).
In the second part, we apply this new technique to present the first algorithm for computing class groups and unit groups of arbitrary number fields that provably runs in probabilistic subexponential time. 
These two parts of the paper can be read essentially independently.

 Part~\partref{part2b} consists of roughly two subjects: an slightly extended analysis and an application of a known provable variant of the BKZ algorithm for ideal lattices \cite{HPS11}, which is used in Part~\partref{part1} for lattice basis reduction; and an extended analysis of the algorithm of Buchmann, Pohst and Kessler \cite{buchmann87,buchmann96} which is required for the post-processing part in the class group and unit group computation in Part~\partref{part2}. These extra analyses are necessary to apply these known results to the specific use-cases of the present work.

\subsection*{Main result of Part~\partref{part1}: sampling ideals} Let $\idset$ be an arbitrary family of ideals, and $\idset_B$ the family of $B$-smooth ideals (i.e., products of prime ideals of norm at most $B$). In the first part of the paper, we describe an efficient algorithm that samples $\mathfrak b \in [\ma]$ such that $\mathfrak b \in \idset\cdot\idset_B$ with probability proportional to the density of $\idset$. 
The set $\idset_B$ is used to randomize the input, and $B$ can be chosen as small as $(\log |\Delta_K|)^{O(1)}$, where $\Delta_K$ is the discriminant of the field $K$.
This result is formalized in \Cref{theorem:ISmain}, page~\pageref{theorem:ISmain}, and allows to work with arbitrary ray class groups, and to restrict $\idset_B$ to ideals whose prime factors fall in a prescribed subgroup.

For concreteness, \Cref{thm:sampling-simplified} below is a specialization of \Cref{theorem:ISmain} to the simplest case, without ray nor subgroups. Here, $\delta_{\idset}[r^n]$ is the local density of $\idset$ (\Cref{def:idealdensity}), i.e., the proportion of ideals of norm about $r^n$ that belong to $\idset$.

\input{simplifiedsamplingtheorem}

\begin{remark}
Note that the algorithm is described in a slightly different way than the above discussion: given $\mathfrak a$, we find $\beta \in \mathfrak a$ such that $\beta\mathfrak a^{-1} \in \idealset \cdot \smoothset$. The ideal $\beta\mathfrak a^{-1}$ is in the inverse class of $\mathfrak a$, so up~to~an inversion, this problem is equivalent to the ideal sampling problem discussed above.
\end{remark}

\subsubsection*{Technique.} 
The folklore strategy to solve ideal sampling tasks is the following. The input ideal $\mathfrak a$ can be seen as a lattice, via the Minkowski embedding. One may find a reasonably short basis of $\mathfrak a$ (for instance, by means of LLL~\cite{lenstra82:_factor}), which then allows one to repeatedly sample reasonably short random elements $\beta \in \mathfrak a$, until the ideal $\mathfrak b = \beta\mathfrak a^{-1}$ belongs to the desired family $\idset$.
One then typically argues (heuristically!) that the probability of success is proportional to the density of $\idset$.

To obtain a rigorous sampling algorithm, we proceed in two steps. First, we prove that a fairly straightforward strategy as above indeed has the desired probability of success \emph{when the input $\mathfrak a$ is treated as a random ideal lattice with uniformly random Arakelov class}. More precisely, we prove in \Cref{prop:elementdensity} that there is a reasonably small box $\boxnf$ (in the embedding space) such that the \emph{expected} density (over the randomness of $\mathfrak a$) of elements $\beta \in \mathfrak a \cap \boxnf$ such that $\beta \mathfrak a^{-1} \in \idset$ is proportional to the density of $\idset$.

Second, we deal with arbitrary input $\mathfrak a$ by randomizing its Arakelov class via a generalization of the random walks introduced in~\cite{dBDPW}. Concretely, the input $\mathfrak a$ is multiplied by random ideals of small prime norm (the \emph{discrete} part of the random walk), and is randomly distorted according to some Gaussian distribution (the \emph{continuous} part of the random walk).
We prove in \Cref{thm:random-walk-prime-sampling} that the result is close to uniformly distributed in the Arakelov ray class group. The discrete part of the random walk introduces small prime factors, hence our method samples ideals in $\idset \cdot \idset_B$ instead of $\idset$. In all applications we are aware of, $\idset = \idset \cdot \idset_B$.

\subsection*{Main result of Part~\partref{part2}: computing class groups and unit groups}
The case of smooth ideals $\idset = \idset_B$ regularly comes up in index-calculus algorithms, such as the aforementioned algorithms for computing class groups or unit groups. In these cases, this sampling task is not the only source of heuristics, so significantly more work is required to obtain a rigorous algorithm. This is the object of the second part of the article. %

Let $K$ be a number field of degree $n$ and discriminant $\Delta_K$. The determination of the structure of its class group $\Cl(K)$, together with a system of fundamental units, is one of the main problems of computational number theory~\cite[p.~217]{Cohen1993}.
It has long been believed that this task can be solved in probabilistic subexponential time. Such algorithms have been described and analyzed under a variety of heuristic assumptions~\cite{buchmann1988subexponential,ANTS:BiasseFiecker14}.
Despite decades of investigation, only imaginary quadratic fields have been amenable to a rigorous analysis~\cite{HM89}, assuming ERH. 
The history of class group computation is discussed in further detail in \Cref{subsec:historyclassgroup}.
In Part~\partref{part2} of this paper, we present the first general algorithm for this problem that provably runs in probabilistic subexponential time, assuming ERH.
We use the classical $L$-notation
$$L_x(\alpha,c) = \exp\left((c + o(1))(\log x)^\alpha(\log\log x)^{1-\alpha}\right),$$
and $L_x(\alpha) = L_x(\alpha,O(1))$. We prove the following theorem.

\begin{theorem}[ERH]\label{thm:main-thm-intro}
There is a probabilistic algorithm which, on input a number field $K$ of degree $n$ and discriminant $\Delta_K$
and an LLL-reduced basis of its ring of integers, 
computes its ideal class group and a compact representation of a fundamental system of units, and runs in expected time polynomial in the length of the input, in $L_{|\Delta_K|}(1/2)$, in $L_{n^n}(2/3)$, and in  $\min(\dedres, L_{|\dcrk|}( 2/3 + o(1)))$,
where $\dedres$ is the residue at $1$ of the Dedekind zeta function $\zeta_K$.
\end{theorem}

It has been conjectured since Buchmann's 1988 heuristic algorithm~\cite{buchmann1988subexponential} that this problem can be solved in subexponential time $L_{|\Delta_K|}(1/2)$ for any family of fields of fixed degree.
\Cref{thm:main-thm-intro}, together with the upper bound $\dedres = (\log |\Delta_K|)^{O(n)}$ (see Equation~\eqref{eq:bounddedres} below), implies this conjecture, assuming ERH.

Then, it was conjectured by Biasse and Fieker's 2014 algorithm~\cite{ANTS:BiasseFiecker14} that this problem can be solved in subexponential time  even for varying degree. 
Again, \Cref{thm:main-thm-intro} implies this conjecture, assuming ERH. However, Biasse and Fieker conjectured a complexity as in \Cref{thm:main-thm-intro} where the quantity $\dedres$ is replaced with $L_{n^n}(2/3)$. 
In our analysis, the quantity $\dedres$ arises from the best known estimates on the density of bounded smooth ideals. %
It seems $\dedres$ should appear in the same way in the heuristic complexity of~\cite{ANTS:BiasseFiecker14}, unless one expects a better bound on the density of smooth ideals.

\subsubsection*{Computing $\mT$-units}
While we stated our main result as an algorithm for computing units and class groups in \Cref{thm:main-thm-intro}, our algorithm actually does slightly more than that: it computes the so-called Log-$\mT$-unit lattice for any set $\mT$ of prime ideals.
It is well known that such an algorithm for $\mT$-units can be used to compute the class group and the unit group. Combined with \Cref{thm:compute-1-rel}, which allows to decompose any integral ideal as a product of prime ideals in a sufficiently large set $\mT$, this can also be used to solve other algorithmic problems, such as the principal ideal problem, or the class group discrete logarithm problem.

\subsection*{Main results of Part~\partref{part2b}: Detailed analyses of provable lattice techniques}
~\\ 
Part~\partref{part2b} consists of an extended analysis of two already known 
results in lattice theory; namely the existence of a BKZ-algorithm variant \cite{HPS11}  that 
has a provable running time, and the existence of an LLL-algorithm variant (called the Buchmann-Pohst-Kessler algorithm \cite{buchmann87,buchmann96}) that 
has a relatively well numerical stability, so that it can be used on `approximated' bases.
The extended analysis consists, in the case of the BKZ-algorithm, mainly of making explicit the techniques described in \cite[Section 3, `Cost of BKZ'']{HPS11} and applying this BKZ-variant to ideal lattices. In the case of the Buchmann-Pohst-Kessler algorithm \cite{buchmann87,buchmann96}, the present work required an extended analysis because in our use-case (in contrast to theirs) some lattice invariants like the rank and the determinant are unknown, thus requiring a slight extension to their algorithm.

The BKZ-variant a with provable running time is used in Part~\partref{part1} for lattice basis reduction of ideal lattices, and the extended algorithm of Buchmann, Pohst and Kessler is required for the post-processing part in the class group and unit group computation in Part~\partref{part2}. Note that both of these algorithms being non-heuristic and having a provable running time is essential for the main results of the present work: a rigorous analysis of a common number-theoretic technique (Part~\partref{part1}) and a rigorous algorithm for computing the unit group and the class group that has a provable upper bound on the run time (Part~\partref{part2}).

\subsection*{Further applications}
Sampling smooth ideals is a task that regularly arises in computational number theory. In Part~\partref{part2}, we focus on the problem of class group and unit group computation, but it is more generally a common component of index-calculus algorithms, like the general number field sieve for integer factorization or the computation of discrete logarithms in finite fields. We do not investigate this direction further in the present paper.

Applying the method of Part~\partref{part1} to the case where $\idset$ is the set of prime ideals allows one to sample in the family $\idset\cdot\idset_B$ of near-prime ideals, of particular interest in that it constitutes a dense family of efficiently factorable ideals.
Therefore, our sampling method provides a rigorous way to transform any ideal $\mathfrak a$ into an equivalent ideal $\mathfrak b$ of known factorization.
Obtaining such factorable ideals (or elements) is a key step in algorithms to compute power residue symbols.
Specifically, it allows to perform the `principalization step' in \cite[\textsection 5.2]{de2017calculating} efficiently. 
The first author of the present article has developed this idea in his PhD dissertation~\cite{KoenThesis}, applying the main result of Part~\partref{part1} to construct the first polynomial time algorithm to compute power residue symbols.

\subsection*{Related work}
The aforementioned difficulties of ideal sampling and class group computation have already been overcome in the special case of 
\emph{imaginary quadratic} number fields. Building on a result of Seysen~\cite{Sey87}, Hafner and McCurley~\cite{HM89} gave a provable algorithm for computing class groups and unit groups of imaginary quadratic fields, assuming ERH.
This case distinguishes itself by the finiteness of the unit group and the existence of \emph{reduced} representatives of ideal classes.
This algorithm exploits random walks in the class group to find $B$-smooth \emph{principal} ideals.
The idea of performing a random walk in the class group was reused in the algorithms of Buchmann~\cite{buchmann1988subexponential} and Biasse and Fieker~\cite{ANTS:BiasseFiecker14}, in a heuristic way.
Rather than random walks in class groups, our method exploits the much richer Arakelov (ray) class groups. Random walks in Arakelov class groups were first studied in~\cite{dBDPW} to prove the random self-reducibility of computational problems in ideal lattices. Their technique to study the convergence of these walks plays a key role in the present paper.
We note that Schoof~\cite{Schoof2008computing} rephrased Buchmann's algorithm in terms of Arakelov theory, and we borrow from his formalism.

\section{Preliminaries}\label{sec:main-preliminaries}

\input{preliminaries/notation}

\input{preliminaries/lattices}

\input{preliminaries/number_theory}

\input{preliminaries/ideal_lattices}

\input{preliminaries/log_embedding}

\input{preliminaries/probabilities}

\input{preliminaries/gaussians} %

%% file: simplifiedsamplingtheorem.tex
\begin{restatable}[\normalfont ERH]{theorem}{samplingsimplified}
\label{thm:sampling-simplified}
Assuming ERH, there is a randomized algorithm $\mathcal A$ such that the following holds.
Let $K$ be a number field, with degree $n$, discriminant $\Delta_K$, and 
let an LLL-reduced basis of the ring of integers $\OK$ be given.
 Let $\mathfrak a \subseteq \OK$ be an integral ideal. %
 Let $\varepsilon \in \R_{>0}$, let $\blocksize \geq 2$ be an integer, and
 let $r \geq \radiusformulanomodulusnoomega{\ma}$.

 Given the above data, the algorithm $\mathcal A$ outputs $\beta \in \mathfrak a$ such that 
 $(\beta) \cdot {\mathfrak a}^{-1} \in \idealset\cdot \idealset_{\smoothnessvarBintro}$
with probability at least $\delta_{\idset}[r^n]/3 -\eps$, for some smoothness bound %
$\smoothnessvarBintro = (\log |\dcrk| + \log\log(1/\eps))^{O(1)}$ and for any set $\idsetmoduG$ of integral ideals. Furthermore, the algorithm runs in expected polynomial time in
$\log |\dcrk|$, in $\log(\norm(\mathfrak a))$, in 
$\allowbreak \log(1/\eps)$, in $\hkztime$, and in the length of the input.
\end{restatable}

%% file: preliminaries/notation.tex
\subsection{Notation}
We denote by $\N, \Z, \Q, \R, \C$ the natural numbers, the integers, the rationals, the real numbers, and the complex numbers respectively. The notation $\log$ refers to logarithms in base $e$, whereas the notation $\log_2$ denotes logarithms in base $2$. 
For finite sets $X$ we denote by $\volabs{X}$ the number of elements in $X$. For infinite sets $X$ with a well-defined volume, we use both notations $\vol(X) = \volabs{X}$
for the volume of $X$. The transpose of a matrix $M$ is denoted by $M^\top$.
For a ring $R$, we write $R^\times$ the set of invertible elements of $R$, and $R^* := R \setminus \{0\}$. 

We use the classic big $O$ and $\Omega$ asymptotic notations, and all hidden constants are absolute (in particular, they never depend on the choice of a field $K$). We also use the notation $f = \poly(g)$ as a synonym for $f = g^{O(1)}$.
We use the notation $O_\varepsilon$ to signify that the hidden constants depend on $\varepsilon$.
As already mentioned, to denote the running time of subexponential algorithms, we use the asymptotic $L$-notation $$L_x(\alpha,c) = \exp\left((c + o(1))(\log x)^\alpha(\log\log x)^{1-\alpha}\right),$$
where the $o(1)$ is asymptotic in $x$ (and does not depend on other parameters). We also write $L_x(\alpha) = L_x(\alpha,O(1))$.

For a real vector space $V \subseteq \R^m$, we consider the Euclidean norm $\| \cdot \|$, and the infinity-norm $\| \cdot \|_\infty$.
We sometimes write $\| \cdot \|_2$ for $\| \cdot \|$ to emphasize the type of norm. We will occasionally use the notation $\mathcal{B} = \mathcal{B}_2$ and $\mathcal{B}_\infty$ for the unit ball with respect to the Euclidean and infinity norm respectively. In particular, the Euclidean ball of radius $r$ in $V$ is typically denoted
\[r\mathcal B = \{ \bv \in V~|~ \| \bv \| < r \}.\]
The vector space $V$ is either clear from context or explicitly mentioned when introducing $\mathcal{B}$ or $\mathcal{B}_\infty$.

\subsection{The Extended Riemann Hypothesis} All statements that mention (\normalfont ERH), such as \Cref{thm:sampling-simplified}, assume the \emph{Extended Riemann Hypothesis}, which refers to the Riemann Hypothesis for Hecke $L$-functions (see \cite[\textsection 5.7]{iwaniec2004analytic}).

%% file: preliminaries/lattices.tex
\subsection{Euclidean lattices}
\label{preliminaries:lattices}
A lattice $\Lambda$ is a discrete subgroup of a real vector space $V = \R^m$.
We write $\Span_\R(\Lambda)$ the real vector subspace of $V$ spanned by the vectors of $\Lambda$.
The rank of the lattice is the dimension of $\Span_\R(\Lambda)$, and
we say the lattice is full-rank if $\Span_\R(\Lambda) = V$.
A lattice (of rank $n$) can be represented by a basis $\mB = (\bb_1, \cdots, \bb_n)$ such that $\Lambda = \{\sum_i x_i \bb_i\,,\,x_i \in \Z\}$. For a given basis $\mB \in \R^{m \times n }$ we denote by $\lat(\mB)$ the lattice spanned by the columns of this basis. 

\subsubsection*{Geometric invariants.}
The covolume of $\Lambda$, denoted by $\covol(\Lambda)$, is the volume of the quotient $\Span_\R(\Lambda) / \Lambda$. If $\mB \in \R^{m \times n}$ is a basis of $\Lambda$ (with column vectors), then the covolume of $\Lambda$ is equal to $\covol(\Lambda) = \sqrt{\det(\mB^{\top} \mB)}$. This quantity is also the volume of any fundamental domain of the lattice. One such domain is of special interest: the Voronoi cell.

\begin{definition} \label{def:voronoi} Let $\Lambda \subseteq V$ be a full-rank lattice. We denote by
\[ \voronoi(\Lambda) = \{ \bx \in V ~|~ \|\bx\| \leq \|\bx-\ell\| \text{ for all } \ell \in \Lambda \} \]
the Voronoi cell of $\Lambda$ around zero. It is a fundamental domain for the lattice $\Lambda$ (up to a set of `faces' of measure zero), thus has volume $\covol(\Lambda)$.
\end{definition}

The $i$-th successive minimum of $\Lambda$ is denoted by $\lambda_i(\Lambda)$. More precisely, each $\lambda_i(\Lambda)$ is the smallest real number such that there exist at least $i$ linearly independent vectors of euclidean norm at most $\lambda_i(\Lambda)$ in $\Lambda$. We call $\lambda_1(\Lambda)$ the first minimum, and $\lambda_n(\Lambda)$ the last minimum (where $n$ is the rank of the lattice).
The covering radius $\covtwo(\Lambda)$ is the smallest $r > 0$ such that for any element $\bx \in \Span_\R(\Lambda)$ there exists a lattice point at distance at most $r$ from $\bx$. The analogous notions with respect to the maximum norm $\| \cdot \|_\infty$ instead of the Euclidean norm are denoted by $\lambda_i^{(\infty)}(\Lambda)$ and $\covinf(\Lambda)$.

The following notion, the \emph{generating radius} $\rr(\Lambda)$, is closely related to $\lambda_n(\Lambda)$ and $\cov(\Lambda)$, but not as standard.
\begin{definition}[Generating radius] \label{definition:rr} For a lattice $\Lambda$, the \emph{generating radius} $\rr(\Lambda)$ is the smallest real number such that the set of all vectors of $\Lambda$ of euclidean norm $\leq \rr(\Lambda)$ generates the lattice $\Lambda$ (as a $\Z$-module). In other words,
\[\rr(\Lambda) = \min\Big(r > 0\,|\, \{\bx \in \Lambda\,|\, \|\bx\| \leq r\} \text{ generates } \Lambda \Big).\]
\end{definition}

\begin{lemma}
\label{lemma:bl-leq-2cov}
For any lattice $\Lambda$, we have %
\[ \lambda_n(\Lambda) \leq \rr(\Lambda) \leq 2 \cdot \cov(\Lambda) \leq \sqrt{n} \cdot \lambda_n(\Lambda). \]
\end{lemma}

\begin{proof} The left-most inequality follows from the very definitions of $\lambda_n(\Lambda)$ and $\rr(\Lambda)$, and the right-most inequality can be found in \cite[Theorem 7.9]{MGbook}.
We prove the middle inequality by 
using the concept of \emph{Voronoi-relevant vectors} of the lattice $\Lambda$.
The Voronoi cell \(\voronoi(\Lambda)\) of $L$ contains all vectors of $\Span_\R(\Lambda)$ that are closer or equally close to $0$ than to any other lattice point (see \Cref{def:voronoi}).
We know (e.g.,~\cite[Proposition 8.4]{MGbook}) that $\voronoi(\Lambda)$ is compact and convex and that $\vol_{\Span_\R(\Lambda)}(\voronoi(\Lambda)) = \covol(\Lambda)$. From the definition of the Voronoi cell and of the covering radius of $\Lambda$, we also know that all $\bx \in \voronoi(\Lambda)$ satisfy $\| \bx \| \leq \cov(\Lambda)$ %
(see also~\cite[Proposition 8.4]{MGbook}).

The Voronoi-relevant vectors of $\Lambda$ are the vectors $S_{\voronoi} = \{\bv \in \Lambda\, | \, \exists \bx \in \voronoi(\Lambda),\, \|\bx\| = \|\bx - \bv\|\}$; these are the vectors defining the facets of the Voronoi cell. From the definition of these vectors, it holds that $\voronoi(\Lambda) = \{\bx \in \Span_\R(\Lambda)\, | \, \forall \bv \in S_{\voronoi},~ \|\bx \| \leq \|\bx - \bv\|\}$. Moreover, since $\|\bx \| \leq \cov(\Lambda)$ for all $\bx \in \voronoi(\Lambda)$, we have that $\|\bv\| \leq 2 \cov(\Lambda)$ for all Voronoi-relevant vectors $\bv$. We will show that $S_{\voronoi}$ generates the lattice $\Lambda$, which will conclude the proof.

Let $\Lambda'$ be the lattice generated by the vectors in $S_{\voronoi}$. We know that $\Lambda' \subseteq \Lambda$ and that $\Lambda'$ has rank $n$ (since $\Span_\R(\Lambda') = \Span_\R(\voronoi(\Lambda)) = \Span_\R(\Lambda)$). Moreover, since $S_{\voronoi} \subseteq \Lambda'$, we know that $\voronoi(\Lambda') \subseteq \{\bx \in \Span_\R(\Lambda)\, |\, \forall \bv \in S_{\voronoi},\, \|\bx \| \leq \|\bx - \bv\|\} = \voronoi(\Lambda)$. From this we conclude that $\covol(\Lambda') \leq \covol(\Lambda)$ and so $\Lambda' = \Lambda$ as desired.
\end{proof}

The next two results regarding  counting lattice points in sets, \Cref{fact:average_count} and \Cref{lemma:divisorcapinfinityball}, are folklore. We include proofs for completeness, as they play an important role in the article.

\subsubsection*{Estimation on the \emph{average} number of lattice points in a measurable volume}
\begin{restatable}[]{lemma}{latticeaveragecount}
\label{fact:average_count}
Let $V$ be a Euclidean vector space and let $\Lambda \subseteq V$ be a full rank lattice. Let $S \subseteq V$ be a measurable set, and let $\bc \in V/\Lambda$ be chosen uniformly. Then
\[ \underset{\bc \from V/\Lambda}{\mean}[|(\Lambda + \bc) \cap S|] = \vol(S)/\covol(\Lambda). \]
\end{restatable}
\begin{proof} By integrating the measurable indicator set $1_S$, and choosing a fundamental domain $F$ of $V/\Lambda$ (which has volume $\covol(\Lambda)$), we obtain
\[ \underset{\bc \from V/\Lambda}{\mean}[|(\Lambda + \bc) \cap S|] = \frac{1}{\vol(F)} \int_{\bc \in F} \sum_{\ell \in \Lambda} 1_S(\bc + \ell) d\bc = \frac{1}{\covol(\Lambda)} \int_{\bv \in V} 1_S(\bv) d\bv = \frac{\vol(S)}{\covol(\Lambda)}.  \]
\end{proof}

\subsubsection*{Estimation on the number of lattice points in a convex measurable volume}%
For the ideal sampling algorithm (\Cref{alg:samplerandom}) we need to efficiently sample in a shifted box (see \Cref{sec:sampling-box}).
\Cref{lemma:divisorcapinfinityball}, which shares some similarities with \cite[\textsection 4.2]{PP21}, provides
means to estimate the number of lattice elements in such a box. This estimate
is essential in the proof in \Cref{sec:sampling-box}. To prepare for
the proof of this lemma, we will need some facts on Minkowski sums of sets.

\input{ideal_sampling/PRELIM/proof_lattice_cap_box}

%% file: ideal_sampling/PRELIM/proof_lattice_cap_box.tex
\begin{definition} \label{def:minkowskisum}Let $V$ be a Euclidean vector space. For two sets $X,Y \subseteq V$, we define
the Minkowski sum $X \minksum Y$ as follows.
\[ X \minksum Y = \{ \bx + \by ~|~ \bx \in X, \by \in Y \}.\]
For $c \in \R_{>0}$ we denote by $cX$ the set
\[ cX = \{ c \cdot \bx ~|~ \bx \in X \}. \]
\end{definition}

\begin{lemma} \label{lemma:minkowskisum} Let $V$ be a Euclidean vector space and let $r,s > 0$ and let $X \subseteq V$ be a convex volume.
Then 
\[ (rX) \minksum (sX) = (r + s) X. \]
\end{lemma}
\begin{proof} We start with inclusion to the right. Suppose $\by \in (rX) \minksum (sX)$, i.e., $\by = r\bx + s\bx'$ where $\bx,\bx' \in X$.
Then $\tfrac{\by}{r+s} = \tfrac{r\bx + s\bx'}{r+s} \in X$, since it is a weighted average of two
points in $X$ and $X$ is convex. So $\by \in (r+s) X$. Inclusion to the left holds because $\by \in (r+s)X$ means that $\by = (r+s)\bx = r\bx + s\bx \in (rX) \minksum (sX)$.
\end{proof}
\begin{lemma} \label{lemma:minkowskisymmetric} Let $V$ be a Euclidean vector space, let $r > 0$, let $X,Y \subseteq V$ be sets 
and let $S \subseteq V$ be a symmetric set, i.e., $\bx \in S \Leftrightarrow -\bx \in S$. Then
\[  (X \minksum S) \cap Y \subseteq [X \cap (Y \minksum S)] \minksum S . \]
\end{lemma}
\begin{proof} Suppose $\bx + \bs = \by  \in (X \minksum S) \cap Y$. Then $\bx =\by - \bs \in X \cap (Y \minksum S)$, so $\by = \bx + \bs \in [X \cap (Y \minksum S)]\minksum S$.
\end{proof}

\begin{restatable}[]{lemma}{latticecapx}
\label{lemma:divisorcapinfinityball}%
Let $V$ be a $n$-dimensional Euclidean vector space,
let $\Lambda \subseteq V$ be a full-rank lattice,
let $X \subseteq V$ be a convex measurable volume for which $\voronoi \subseteq cX$ for some $c \in \R_{>0}$, where $\voronoi$ is the
Voronoi cell of $\Lambda$ (see \Cref{def:voronoi}). Then, for all $\bt,\bt' \in V$ and all $r >2 c$,
\[ |(\Lambda + \bt) \cap r( X + \bt')| \in
[e^{-2nc/r},e^{2nc/r}]
\cdot \frac{r^n \cdot \vol(X)}{\covol(\Lambda)},\]
where $r(X+ \bt') = \{ r \cdot (x + \bt') ~|~ x \in X \}$ is the scaling of the (translated) set $X + \bt'$ by $r \in \R_{>0}$.
\end{restatable}

\begin{proof} As $|(\Lambda + \bt) \cap (rX + r\bt')| = |(\Lambda + \bt - r\bt') \cap rX|$, we just assume, without loss of generality, that $\bt' = 0$.
Note that $\voronoi \subseteq cX$, and that $X$ is convex. So, by \Cref{lemma:minkowskisum}, we have
$(rX) \minksum \voronoi \subseteq (rX) \minksum (cX) = (r+c) X$. Similarly, $(r-c) X \minksum \voronoi \subseteq rX$. Therefore
\begin{equation} [(\Lambda + \bt) \cap rX] \minksum \voronoi \subseteq (r+c)X. \label{eq:minklower} \end{equation}
Note that $\voronoi$ is symmetric and $(\Lambda + \bt) \minksum \voronoi = V$, the whole vector space. So, by \Cref{lemma:minkowskisymmetric}  and $(r-c)X \minksum \voronoi \subseteq rX$,
\begin{align}  (r-c)X & = [(\Lambda + \bt) \minksum \voronoi] \cap (r-c)X  \\ & \subseteq  [(\Lambda + \bt) \cap ((r-c)X \minksum \voronoi)] \minksum \voronoi \subseteq  [(\Lambda + \bt) \cap rX] \minksum \voronoi  \label{eq:minkupper}    \end{align}
By  \Cref{eq:minklower,eq:minkupper} and the fact that $\voronoi$ is a fundamental domain of $\Lambda$ with
volume $\covol(\Lambda)$, we obtain
\[ (r-c)^n \vol(X) \leq |(\Lambda + \bt) \cap rX | \cdot \covol(\Lambda) \leq (r+c)^n \vol(X) .\]
Dividing by $\vol(\Lambda)$ and using the estimate $e^{-2nc/r} \leq (1-c/r)^n \leq (1+c/r)^n \leq e^{2nc/r}$ (note that $r > 2c$) we
arrive at the final claim.
\end{proof}

%% file: preliminaries/number_theory.tex
\subsection{Number fields} \label{prelim:numberfields}
Throughout this paper, we consider a number field $K$ of rank $n$ over $\Q$, having ring of integers $\OK$, discriminant $\disc$, regulator $\reg$, class number $\classnumber$ and group of roots of unity $\rou$. Additionally, we consider a \emph{modulus}: a formal product $\modu = \moduz \moduinf$, where $\moduz \subseteq \OK$ is an integral ideal and $\moduinf$ is a formal product of infinite places (see more details below).
We know by Minkowski's theorem~\cite[pp.\ 261--264]{Min1} 
that\footnote{We have $|\dcrk| \geq \left(\frac{\pi}{4}\right)^n n^{2n}/(n!)^2 \geq \left(\frac{\pi}{2}\right)^n$ for $n \geq 2$, since $n^{2n}/(n!)^2 \geq 2^n$ for $n \geq 2$.}
$\log |\dcrk| \geq \log(\pi/2)n \geq 0.4n$ for $n \geq 2$.

The number field $K$ has $n$ field embeddings into $\C$, which are divided in $\rem$ real embeddings and $\cem$ conjugate pairs of complex embeddings, with $n = \rem + 2 \cem$. These embeddings combined yield the so-called Minkowski embedding 
\begin{align*}
K &\longrightarrow K_\R \subseteq \bigoplus_{\sigma:K \hookrightarrow \C} \C
\\ \alpha &\longmapsto (\sigma(\alpha))_\sigma,
\end{align*}
where
\[ K_\R = \bigg \{ x \in \bigoplus_{\sigma : K \hookrightarrow \C} \C ~\bigg |~ x_{\overline{\sigma}} = \overline{x_\sigma} \bigg \} .\]
Here, $\overline{\sigma}$ equals the conjugate embedding of $\sigma$ whenever $\sigma$ is a complex embedding and it is just $\sigma$ itself whenever it is a real embedding. We index the components of the vectors in $K_\R$ by the embeddings of $K$, i.e., we write $x = (x_\sigma)_\sigma \in K_\R$.
Embeddings up to conjugation are called infinite places, denoted by $\nu$. With any embedding $\sigma$ we denote by $\nu_\sigma$ the associated place; and for any place we choose a fixed embedding~$\sigma_\nu$. We will sometimes write $\sigma \mid \moduinf$ to mean that the associated place $\nu_\sigma$ divides $\moduinf$. For any radius $r \in \R_{>0}$ we write $r \Ballinf = \{ (x_\sigma)_\sigma \in \nfr ~|~ |x_\sigma| \leq r \mbox{ for all } \sigma \}$, if it is clear from the context that $\nfr$ is
the vector space at hand.

\subsubsection*{Ideals}
The group of fractional ideals of $K$ is denoted by $\ideals$. Fractional ideals are denoted by $\ma, \mb, \ldots$, and the symbols $\mp, \mathfrak{q}$ are generally reserved for integral prime ideals of $\OK$. 
Any $\ma \in \ideals$ factors uniquely as a product of prime ideals (with possibly negative exponents), and we denote by $\ord_\mp(\ma)$ the exponent of $\mp$ in the factorization of $\ma$; by extension, for any elements $\alpha \in K^*$ we define $\ord_\mp(\alpha) = \ord_\mp(\ma)$ for $\mathfrak a = (\alpha)$ the principal ideal generated by $\alpha$. 
Two fractional ideals $\ma,\mb$ are said to be coprime if their unique decomposition into a product of prime ideals (with possibly negative coefficients) do not share any prime ideals.
The algebraic norm of a fractional ideal $\ma$ or an element $\alpha \in K$ is denoted $\norm(\ma)$ and $\norm(\alpha)$, respectively. When $\ma = (\alpha)$, we have $\norm(\ma) = |\norm(\alpha)|$.

\subsubsection*{The class number formula.} Let $h_K$ be the class number of $K$, $\reg$ be its regulator and $\rou$ be its group of roots of unity. Let also $\zeta_K$ be the Dedekind zeta function of $K$, and $\dedres$ be its residue at $1$.
The class number formula \cite[VII.\textsection 5, Cor 5.11]{neukirch2013algebraic} states that
\begin{equation} \label{eq:classnumberformula} \dedres := \lim_{s \rightarrow 1} (s-1)\zeta_K(s) = \frac{2^\rem \cdot (2 \pi)^\cem \cdot \reg \cdot \classnumber}{|\mu_K| \cdot \sqrt{|\disc|}}. \end{equation}
The Dedekind residue $\dedres \in \R_{>0}$ is bounded above%
\footnote{We use here that $\log(ex/2) \leq x$ for all $x > 0$ and the fact that $\log(|\dcrk|)/(n-1) > 0$ for $n \geq 2$.}~\cite{louboutin00}: for $n = [K:\Q] > 1$, we have %

\begin{equation} \label{eq:bounddedres}
 \log(\dedres) \leq (n-1) \cdot \log \left( \frac{e \log |\disc|}{2(n-1)}  \right) \leq \log |\disc|.  
\end{equation}
It can be approximated in polynomial time up to some factor~\cite{Bach95}, and so can the related quantity $\reg \cdot \classnumber$. More precisely, we have the following proposition.
\begin{proposition}[ERH]
\label{prop:approx-rho}
There exists a polynomial time algorithm (in $\log |\disc|$) that takes as input any number field $K$,
and an LLL-reduced basis of its ring of integers $\OK$, 
and outputs $\rho_0 \in \Q$ and $\eta_0 \in \Q$ such that
\begin{align*}
\rho_0 & \in [\tfrac{3}{4},\tfrac{5}{4}] \cdot \rho_K \\
\eta_0 &\in [\tfrac{3}{4},\tfrac{5}{4}] \cdot \reg \cdot \classnumber.
\end{align*}
\end{proposition}

\subsubsection*{Density of ideals}
In this article, we consider families of ideals, like smooth ideals or prime ideals. Given a specifically randomly generated ideal, we want to estimate the probability that it belongs to a given family.
The notion of local density provides a certain approximation of this probability, for uniformly random ideals of bounded norm.

\begin{definition} For any set of ideals $\idset$, we define $\idset(t) = \{ \mb \in \idset ~| ~ \norm(\mb) \leq t \}$. %
\end{definition}

\begin{definition}[Local density of an ideal set] \label{def:idealdensity}
Let $x > 0$ a positive real number, and let $\idset$ be a set of integral ideals of $K$. We define the \emph{local density} of $\idset$ at $x$ as
\[ \delta_\idset[x] = \min_{t \in [x/e^n,x]} \frac{\countfun{\idset}{t}}{\dederesidue \cdot t} = \min_{t \in [x/e^n,x]} \frac{ |\{ \mb \in \idset ~| ~ \norm(\mb) \leq t \}| }{\dederesidue \cdot t} ,\]
where $\dederesidue = \lim_{s \rightarrow 1} (s-1)\zeta_K(s)$ (see \Cref{eq:classnumberformula}).
\end{definition}
Note that the local density tends to the (asymptotic) `natural density' of  $\idset$ as $x\rightarrow \infty$, since $|\{ \ma \subseteq \OK ~|~ \norm(\ma) < t\}| \sim \dederesidue \cdot t$ \cite[\textsection 9.5]{overholt2014course}. %
However, this notion of natural density is not fine enough as it is simply $0$ for families of interest like smooth ideals or prime ideals.

\subsubsection*{Ray} 
The group of fractional ideals coprime with the finite part $\moduz$ of the modulus $\modu$ is denoted by $\idealsmodu$.
We denote by 
\[ \Kmodu = \langle \alpha \in \OK ~|~ \alpha \equiv 1 \mbox{ mod } \moduz \mbox{ and } \sigma(\alpha) > 0 \mbox{ for all real } \emb \mid \moduinf \rangle \]
 the \emph{ray} modulo $\modu$, a multiplicative subgroup of $K^*$. Here, the notation $\langle \cdot \rangle$ means that $\Kmodu$ is multiplicatively generated by the elements $\alpha \in \OK$ satisfying $\alpha \equiv 1 \mbox{ mod } \moduz \mbox{ and } \sigma(\alpha) > 0 \mbox{ for all real } \emb \mid \moduinf$, which also includes fractions. %

The \emph{ideal ray class group} $\rayclassgroup$ modulo $\modu$ is defined as the quotient of $\rayideals$ by the subgroup $\mbox{Princ}_K^\modu := \{(\alpha) \in \idealsmodu ~|~  \alpha \in \Kmodu\}$.
One retrieves the (ordinary) ideal class group by taking an empty modulus $\modu = \OK$, for which $\Kmodu = K^*$.
We also define $\Kmodumodu = \langle \alpha \in \OK ~|~  (\alpha) + \moduz = \OK  \rangle$, the multiplicative subgroup of $K^*$ generated by elements coprime to $\moduz$. Note that $\Kmodu \subseteq \Kmodumodu$ and that $\Kmodumodu = K^*$ if $\moduz = \OK$. The number of \emph{real places} $\nu \mid \moduinf$ is denoted by $|\modur|$. We denote $\norm(\modu) = \norm(\moduz) \cdot 2^{|\modur|}$. We will also use the generalized Euler totient $\phi(\moduz) = \defphi$ which equals $|(\OK/\moduz)^\times|$ for $\moduz \subsetneq \OK$ and equals $1$ for $\moduz = \OK$.

We denote $\nfrm = \{ (x_\sigma)_\sigma \in \nfr ~|~  x_{\sigma} > 0 \mbox{ for real } \sigma \mid \moduinf\}$
for the `positive part' of $\nfr$ with respect to the modulus $\modu$. For any $\tau \in \Kmodumodu$, we denote
 $\taunfrm =\{ (x_\sigma)_\sigma \in \nfr ~|~  x_{\sigma}/\sigma(\tau) > 0 \mbox{ for real } \sigma \mid \moduinf\}$, 
 which is the part of $\nfr$ that has the same sign as $\tau$ at the real embeddings $\sigma \mid \moduinf$.

%% file: preliminaries/ideal_lattices.tex
\subsection{Ideal lattices}
Ideals can be viewed as lattices in the real vector space $K_\R$, where $K_\R$ has its (Euclidean or maximum)
norm inherited from the complex 
vector space it lives in. Explicitly, the Euclidean and maximum norm of $\alpha \in K$ are 
respectively defined by the rules $\|\alpha\|^2 = \sum_\sigma |\sigma(\alpha)|^2$ and $\|\alpha\|_\infty = \max_\sigma |\sigma(\alpha)|$,
where $\sigma$ ranges over all embeddings $K \rightarrow \C$.

For any ideal $\ma$ of $K$, we define the associated lattice $\ma \subseteq \nfr$ to be the image of $\ma \subseteq K$
under the Minkowski embedding, which is clearly a discrete additive subgroup of $K_\R$. Abusing notation,
we denote both the ideal and the associated lattice with the same symbol $\ma$.
In particular, $\OK$ is a lattice.
Note that we have $\vol(\ma) \! =\! \sqrt{|\dcrk|} \norm(\ma)$ for ideals $\ma \in \ideals$.
The notion of \emph{ideal lattices} extends to a larger family of lattices in $K_\R$ as follows.

\begin{definition}[Ideal lattices] \label{def:ideallattices} Let $\nf$ be a number field with ring of integers $\OK$. An \emph{ideal lattice} of $\nf$ is a lattice in $\nfr$ of the form $x \ma$ where $x \in \nfr^\times$ is invertible and $\ma$ is a
fractional ideal of $K$.
We denote the group of ideal lattices by $\idlat_\nf$.
\end{definition}
The set of ideal lattices is a group with product $(x \ma)(y \mb) = (xy) (\ma \mb)$,
inverse $x^{-1} \ma^{-1}$ and unit $\OK$.

\subsubsection*{Bounds on invariants of ideal lattices}
\noindent
Denote $\Gamma(\Lambda) = \lambda_n(\Lambda)/\lambda_1(\Lambda)$, and define, for a fixed number field~$K$:
\begin{equation} \label{eq:gammadef} \Gamma_{K} = \sup_{x\ma \in \idlat_K} \Gamma(x\ma) \end{equation}
We have the following bounds. 
\begin{lemma} \label{lemma:idlatfacts} For any ideal lattice $x\ma \in \idlat_K$, %
\begin{enumerate}[(i)]
	\item \label{item:lower-bound-lambda1} $\lambda_1(x\ma) \geq \sqrt{n} \cdot \norm(x\ma)^{1/n}$.
	\item \label{item:gap-bound} $\lambda_n(\OK)/\sqrt{n} \leq \Gamma_K \leq \lambda^\infty_n(\OK) \leq |\dcrk|^{1/n}$. 
	\item \label{item:gap-bound-cyc} For cyclotomic number fields $K$, $\Gamma_K = 1$.
	\item \label{item:covering-bound} $\lambda_n(x\ma) \leq \sqrt{n} \cdot \Gamma_K \cdot \vol(x\ma)^{1/n}$.
	\item \label{item:covering-bound2}  $\covinf(x\ma) \leq \covtwo(x\ma) \leq n/2 \cdot \Gamma_K \cdot \vol(x\ma)^{1/n}$.
\end{enumerate}

\end{lemma}
\begin{proof} 
The first item follows from the fact that for any non-zero element $z \in x \ma$, it holds that $\norm(x \ma) \leq |\norm(z)|$ and that $|\norm(z)|^{2/n} \leq \frac{1}{n} \cdot \|z\|^2$ by the inequality of arithmetic and geometric means, applied to the $(|\sigma(z)|^2)_\sigma$. Applying this to a $z$ reaching $\lambda_1(x \ma)$ yields $\lambda_1(x \ma) \geq \sqrt{n} |\norm(z)|^{1/n} \geq \sqrt{n} \norm(x \ma)^{1/n}$. 
For the second item, the inequality $\lambda_n(\OK)/\sqrt{n} \leq \Gamma_K$ follows from the definition of $\Gamma(\OK)$ and the fact that $\lambda_1(\OK) = \sqrt{n}$ (using that $\lambda_1(\OK) \geq \sqrt{n}$ by the first item and that this lower bound is reached by $1 \in \OK$).
To obtain the bound $\Gamma_{K} \leq \lambda_n^\infty(\OK)$, pick an arbitrary ideal lattice $x\ma \in \idlat_K$ and choose a shortest element $x\alpha \in x\ma$ with $\alpha \in \ma \in \ideals$.
That means $\| x\alpha \| = \lambda_1(x\ma)$. Then $x\ma \supset x \cdot (\alpha)$, and
therefore, %
\begin{align} \lambda_n(x\ma) &  \leq \lambda_n(x \cdot \alpha \OK) \leq \| x\alpha \| \cdot \lambda_n^\infty(\OK) 
\leq \lambda_1(x\ma) \cdot \lambda_n^\infty(\OK)
.
\label{eq:ideallatticelambdanbound} \end{align}
The bound $ \lambda^\infty_n(\OK) \leq |\dcrk|^{1/n}$ is obtained from \cite[Theorem 3.1]{bhargava2020} and is tailored to our purposes in \Cref{theorem:nthminimum}.\footnote{We note that the inequality $\Gamma_K \leq |\dcrk|^{1/n}$ was already proven in~\cite[Propositions 4.1 and 4.2]{bayer2006upper}, and is the only one we actually need in this article. However, we believe that the intermediate inequality $\lambda^\infty_n(\OK) \leq |\dcrk|^{1/n}$ could be useful on its own.}
Part (iii) follows from part (ii) and the fact that $\| \zeta \| = \|1 \|$ for roots of unity $\zeta \in K$. Part (iv) is essentially  Minkowski's bound $\lambda_1(x\ma) \leq \sqrt{n} \vol(x\ma)^{1/n}$ combined with the definition of $\Gamma_K$. Finally, the last item follows from the fact that $\covtwo(\Lambda) \leq \sqrt{n}/2 \cdot  \lambda_n(\Lambda)$ \cite[Theorem 7.9]{MGbook}.
\end{proof}

\subsection{Representation of elements and ideals}
\label{sec:representation}

We assume throughout this paper that the number field $K$ is represented by a monic irreducible polynomial $f \in \Z[x]$ satisfying $\size(f) := \sum_i \log_2 |f_i| \leq \poly(\log |\dcrk|)$. This restriction is very mild, indeed we can prove that such a polynomial always exists (see \Cref{A:boundf}) and there are \emph{heuristic} polynomial time algorithms computing such polynomials, e.g.,~\cite{cohen1991polynomial,gelin2016reducing}.

Additionally, we assume throughout this paper that we know an LLL-reduced basis $(\mathbf{b}_1, \cdots, \mathbf{b}_n)$ of the
ring of integers $\OK$ of $K$.
Such a basis has vectors (represented as polynomials in $\Q[x]/(f(x))$) whose size 
is polynomially bounded%
\footnote{
Let $\theta$ denote the class of $x$ in $\Q[x]/(f(x))$, and write $b_i = \sum_{j=1}^n q_{i,j} \theta^{j-1}$ for the elements of the LLL-reduced basis of $\OK$ (with $q_{i,j} \in \Q$). Let $\mathbf B \in \C^{n \times n}$ be the matrix whose columns correspond to the Minkowski embedding of the $b_i$'s, and $\mathbf T$ be the one corresponding to the $\theta^{j-1}$'s. Then $(q_{i,j})_{i,j} = \mathbf T^{-1} \cdot \mathbf B$. By LLL reducedness of the $b_i$'s, it holds that the log of the coefficients of $\mathbf B$ are polynomially bounded in $\log |\dcrk|$. The coefficients of $\mathbf T^{-1}$ correspond to the coefficients of the Lagrange polynomials associated to the roots of $f$, so their logarithm is also polynomially bounded in $\size(f) = \poly(\log |\dcrk|)$ (using e.g., Mignotte's root separation lower bound and Cauchy's upper bound on the roots). Hence, we conclude that $|\log(q_{i,j})| \leq \poly(\log |\dcrk|)$.

It remains to prove that the denominators  of the $q_i$ are bounded. 
As the integral basis $(1,\theta,\ldots,\theta^{n-1})$ has discriminant $\Delta(f)$, any denominator of $q_i$ must divide $\Delta(f)/\dcrk$ (since $\beta$ is integral). Since the size of $\Delta(f)$ is bounded by $\poly(\log |\dcrk|)$ (per assumption), any denominator of $q_i$ must so, too.
} 
in $\log |\disc|$. Note that such a basis can be computed in unconditional, probabilistic subexponential time in $\size(f)$ (by factoring $\mathrm{disc}(f)$ with~\cite{Pomerance1987}, computing a basis of $\OK$ with~\cite[Theorem~1.4]{buchmannlenstra_roi}, and reducing it with the LLL algorithm~\cite{lenstra82:_factor}).

For the main result of Part~\partref{part1} such a basis 
is not required per se (see the discussion in \Cref{sec:arbitraryorder}), as any 
sub-order of $\OK$ would suffice as well. This choice for a basis of the ring of integers 
is done purely because it simplifies the description
and analysis of this main result.

All elements and ideals in $K$ can then be
represented by vectors and bases with rational
coefficients, by using this basis of $\OK$ as 
a coordinate system; which we do in this paper.

The \emph{size} $\size(k)$ of a number $k \in \Z$ is defined to be $\log_2 |k|$
and is extended to rationals $\frac{a}{b}$ as $\size(a) + \size(b)$ for reduced 
fractions $\frac{a}{b}$. The size of a $\Q$-vector is just the sum of the
sizes of its entries, which allows to define the size of an element $\alpha \in K$ 
as the size of the rational vector representing $\alpha$ in the basis of $\OK$. The size
of a matrix $M_\ma$ defining an ideal $\ma$ is defined as the sum of the sizes 
of the matrix entries. %

Basic operations, such as addition, multiplication, inversion and approximate
computation of a complex embedding of elements in $K$, are all polynomial in $\log |\disc|$
and the (just defined) size of the elements involved.

Elements $x = (x_\sigma)_\sigma \in \nfr$ are represented with rational coefficients, that is, $x \in \bigoplus_{\sigma} (\Q + i\Q) \subseteq \nfr$,
and the size of such an element is the size of the vector $(x_\sigma)_\sigma$ (which can be seen as a vector in $\Q^n$).

\subsubsection*{Representation of ideals.} 
We assume throughout the paper that
all integral ideals $\ma \subseteq \OK$ are represented
by their Hermite Normal form (HNF) basis by default. This is possible since
their matrix, using the basis of $\OK$ as coordinate system, has integral coefficients.

This requirement allows us to abuse terminology and write sentences like ``the algorithm takes as input an integral ideal $\ma$'', without having to specify which basis of the ideal is given to the algorithm.
As the HNF basis 
of an ideal is a canonical representation of that ideal, it cannot lead to any confusion.

Addition, multiplication and inversion of ideals given in HNF basis can be performed in time polynomial in the input size and in $\log |\disc|$.
Additionally, we have that the HNF basis of an integral 
ideal $\ma$ is bounded in size by $\poly(n, \log \norm(\ma))$. So, by using the HNF basis of an integral ideal $\ma$ by default, one avoids the problem of the specific representation
of $\ma$ bearing an influence on the running time of algorithms involving $\ma$. 
We therefore define $\size(\ma)$ to be the size of its HNF basis.
For a finite set $\mT$  consisting of 
ideals we define $\size(\mT) = |\mT| \cdot \max_{\ma \in \mT} \size(\ma)$.
For a modulus $\modu = \moduz \moduinf$, where $\moduz \subseteq \OK$ and $\moduinf$ is a formal product of infinite places, we define $\size(\modu) = \size(\moduz)$. Namely, as higher powers of 
infinite places do not have any influence as compared to a power of one, we will assume throughout this work that the infinite part $\moduinf$ only consists of single powers of places of $K$.

\subsubsection*{Compact representation of elements.} \label{sec:sunitrepresentation}
In Part \partref{part2} of this article, an algorithm is
discussed that computes a fundamental system of $\mS$-units.
In this algorithm, the output elements are given in a so-called \emph{compact representation},
which we will explain presently.

Given a set of elements $\{\gamma_1,\ldots,\gamma_k\}$ of $K$, 
we can write the element $\eta = \prod_{j = 1}^k \gamma_j^{n_j}$
in compact representation by the pair of vectors
\begin{equation} (n_1,\ldots,n_k) \in \Z^k, ~~~ (\gamma_1,\ldots,\gamma_k) \in K^n  \label{eq:compactrep} \end{equation}
In this way, $\eta$ is not explicitly computed in terms of the 
basis of $K$ (or $\OK$), but rather, the product is left implicit,
allowing for much larger elements $\eta \in K$ to be described.
Indeed, $\eta= \prod_{j = 1}^k \gamma_j^{n_j}$ written in an $\OK$-basis might require $\sum_j n_j \size(\gamma_j)$
bits to write down, whereas the compact representation in \Cref{eq:compactrep}
only requires at most $\sum_j \size(n_j) + \sum_j \size(\gamma_j)$ bits to write down. The drawback of this representation is that it only allows polynomial time multiplication and inversion of elements, but generally not efficient addition.

%% file: preliminaries/log_embedding.tex
\newcommand{\hyperacc}{\hyper'}
\newcommand{\hypermin}{\hyper'_-}

\subsection{The logarithmic embedding} \label{section:logembedding}
The logarithmic embedding of $K^\times$ is the map
\[ \Log : K^\times \longrightarrow \bigoplus_{\nu} \R :  \alpha \longmapsto (\npl \log | \embn(\alpha)|)_\pl ,\]
where $\npl = 2$ if $\pl$ is a complex place and $1$ otherwise.
This map naturally extends to all invertible elements of $\K_\R$, and this extension is surjective, i.e., $\Log \nfrstar = \bigoplus_{\nu} \R$. 
Denoting $\nfr^0 = \{ (x_\sigma)_\sigma  \in \nfr ~|~ \prod_\sigma x_\sigma = 1 \}$, we use this logarithmic map to define 
the vector space $\hyper = \Log(\nfr^0)$ (called the hyperplane) and the \emph{logarithmic unit lattice}
$$\logunits \subseteq \hyper,$$
a full-rank
lattice in $\hyper$.
We write its dimension by $\dimh = \dim(\hyper) = n_\R + n_\C - 1$. Explicitly, we have
\[ \hyper = \Log \nfr^0 = \{ (x_\pl)_\pl \in \bigoplus_{\pl} \R  ~|~  \sum_\pl  x_\pl = 0  \}
.\]
The volume of the logarithmic unit lattice is given by the following formula (see, e.g., \cite[I.\textsection7 \& I.\textsection5, p.~33]{neukirch2013algebraic})

\begin{equation}
\label{eq:unittorusvolume}
\vol(\logunits) = \sqrt{\cem + \rem}\cdot \reg.
\end{equation}

\subsubsection*{The first minimum of $\logunits$}
\noindent By a result of Kessler~\cite{kessler}, we have a lower bound on the first minimum of the lattice $\logunits$.
\begin{lemma}
\label{le:lower-bound-first-minimum-log-unit}
We have $\lambda_1(\logunits) \geq \kesslerformula$.
\end{lemma}

\subsection{Divisors} \label{section:divisors}
We define the divisor group $\Div_{K}$ of $K$ as 
 \[ \Div_{K} := \bigoplus_{\mp} \Z \times \bigoplus_{\nu} \R,\]
 where $\nu$ ranges over the set of all infinite places (embeddings into the complex numbers up to possible conjugation), and $\mp$ ranges over all prime ideals of $\OK$ (also referred to as the \emph{finite places} of $K$).
We denote the canonical basis elements with the symbols $\ldb\mathfrak p\rdb$ and $\ldb\nu\rdb$ (the divisor with value $1$ at $\mathfrak p$ or $\nu$ respectively, and $0$ everywhere else).
Then, an arbitrary divisor can be written as
\[\ba = \sum_{\mp } \ha_\mp \cdot  \ldb \mp \rdb  + \sum_{\nu} \ha_\nu \cdot \ldb \nu \rdb, \]
with only finitely many non-zero $\ha_\mp$ and with $\ha_\nu \in \R$. 
We will consistently use the symbols $\ba,\bb,\be,\ldots$ for such divisors.
Given such a divisor $\ba$, we often write $\finitepart{\ba} = \sum_{\mp } \ha_\mp \cdot  \ldb \mp \rdb$ for its ``finite part'' and $\infinitepart{\ba} = \sum_{\nu} \ha_\nu \cdot \ldb \nu \rdb$ for its ``infinite part''.

\subsubsection*{The degree map}
The degree map is defined as
\begin{align} 
\nonumber \deg:\Div_{K} &\longrightarrow \R \\
\label{eq:defdeg}  \sum_{\mp} \ha_\mp \cdot  \ldb \mp \rdb  + \sum_{\nu} \ha_\nu \cdot \ldb \nu \rdb &\longmapsto \sum_{\mp} \ha_\mp \cdot \log(\norm(\mp)) + 
\sum_{\nu} \ha_\nu
.\end{align}
The kernel of this map is the subgroup $\Div^0_{K} = \ker(\deg)$ of degree-zero divisors.

\subsubsection*{From field elements to divisors} Denoting $\ord_\mp$ for the valuation at the prime $\mp$, there is a canonical homomorphism
\begin{align*}
\ldb \cdot \rdb: K^\times &\longrightarrow \Div_{K}\\
\alpha &\longmapsto \ldb \alpha \rdb =  \sum_{\mp}\mathrm{ord}_{\mp}(\alpha)\ldb \mp \rdb - \underbrace{\sum_{\nu} \npl \log |\sigma_\nu(\alpha)| \cdot \ldb \nu \rdb}_{\Log(\alpha)},
\end{align*}
with $\npl = 2$ if $\pl$ is complex and $1$ otherwise.
The \emph{product formula} states that for any $\alpha \in K^\times$, we have $\ldb \alpha \rdb \in \Div^0_{K}$, i.e., $\ldb K^\times\rdb \subset \Div_K^0$.

\subsubsection*{The exponential maps}
The finite part $\finitepart{\ba} \in \bigoplus_{\mp} \Z$ of a divisor $\ba$ naturally corresponds to an ideal in the group $\ideals$ of fractional ideals via the map
\[  \finpart{\Exp} : \bigoplus_{\mp} \Z \longrightarrow \ideals :  \sum_{\mp } \ha_\mp \cdot  \ldb \mp \rdb \longmapsto  \prod_{\mp} \mp^{\ha_\mp}. \]
This map is an isomorphism with inverse
\begin{equation} \label{def:d} d:\ideals \longrightarrow \bigoplus_{\mp} \Z : \ma \longmapsto  \sum_{\mp  } \ord_\mp(\ma) \cdot  \ldb \mp \rdb.\end{equation}
We will often use a normalized section defined as 
\begin{equation} \label{def:dzero} \secideal:\ideals \longrightarrow \Div_K^0 : \ma \longmapsto  \sum_{\mp  } \ord_\mp(\ma) \cdot  \ldb \mp \rdb  - \frac{ \log(\norm(\ma))}{n} \sum_{\nu} \npl \ldb \nu \rdb,\end{equation}
to map into $\Div_K^0$ instead of $\Div_K$.\\

The infinite part $\infinitepart{\ba} \in \bigoplus_{\nu} \R$ of $\ba$ can be mapped into $\nfr^\times$ via
\[ \infpart{\Exp} : \bigoplus_{\nu} \R \longrightarrow \nfr^\times : \sum_{\nu} \ha_\nu \cdot \ldb \nu \rdb \longmapsto (e^{\nplemb^{-1} \cdot \ha_{\nu_\sigma}})_\sigma \in \nfr^\times. \]
The map $\infpart{\Exp}$ is injective, but not surjective.
Note that the logarithmic embedding $\Log$ introduced in \Cref{section:logembedding} is a retraction of $\infpart{\Exp}$ (i.e., $\Log \circ \infpart{\Exp}$ is the identity).
Furthermore, for any $x = (x_\emb)_\emb \in \nfrstar$, we have $\infpart{\Exp}(\Log(x)) = (|x_\emb|)_\emb \in \nfrstar$. \\

The main reason for us to consider divisors is that they encode ideal lattices. This naturally follows by combining the above exponential maps into the following:
\[
\Exp: \Div_{K} \longrightarrow \idlat_K :
\ba \longmapsto \infpart{\Exp}(\infpart{\ba}) \cdot \finpart{\Exp}(\finpart{\ba}).
\]
For a divisor $\ba = \sum_{\mp} \ha_\mp \cdot  \ldb \mp \rdb  + \sum_{\nu} \ha_\nu \cdot  \ldb \nu \rdb $, the associated ideal lattice is
\[ \Exp(\ba) =
(e^{\nplemb^{-1} \cdot\ha_{\pl_\emb}})_\emb  \cdot \prod_\mp \mp^{\ha_\mp} = \left\{( e^{\nplemb^{-1} \cdot  \ha_{\pl_\emb}} \cdot \emb(\alpha))_\emb \in \nfr ~|~  \alpha \in \prod_\mp \mp^{\ha_\mp} \right\} \subseteq K_\R.\]
It is a group homomorphism sending the additive operation in $\Div_{K}$ to the multiplicative operation in $\idlat_K$.
For any divisor $\ba \in \Div_{K}$, we have
\[ \vol( \dExp{\ba}) \! =  \! \sqrt {|\dcrk|} \cdot \prod_\nu e^{\ha_\nu} \cdot \norm\left( \prod_{\mp}  \mathfrak {\mp}^{\ha_\mp} \right) = \sqrt {|\dcrk|} \cdot e^{\deg(\ba)} .\]

\subsubsection*{Euclidean norm on $\Div_{K}$} \label{def:normondiv}
For $\ba \in \Div_{K}$ written as
\[ \ba =  \sum_{\mp} \ha_\mp \cdot \ldb \mp \rdb+  \sum_{\nu} \ha_\nu \cdot \ldb \nu \rdb ,\] we define the norm
\[ \divnorm{\ba} = \left(\sum_\mp \ha_\mp^2 + \sum_\pl \ha_{\pl}^2\right)^{1/2}, \]
Note that, naturally, $H \subseteq \Log(\nfrstar) \hookrightarrow \Div_K$, where the defined Euclidean norms are compatible.

\subsubsection*{Subgroups of $\Div_{K}$}
Given a set $\mS$ of prime ideals of $\OK$, we define the subgroup
 \[ \Div_{K,\mS} = \bigoplus_{\mp \in \mS} \Z \times \bigoplus_{\nu} \R \subseteq \Div_{K}.\]
In the next two parts of this article, we will work with two important cases.
In \Cref{part1}, we will consider a modulus $\modu = \moduz \moduinf$, and let $\mS$ be the set of all prime ideals that do not divide $\moduz$. The resulting group $\rayDiv := \bigoplus_{\mp \nmid \moduz} \Z \times \bigoplus_{\nu} \R$ is called the \emph{Arakelov $\mathfrak m$-ray divisor group}.
In \Cref{part2}, we will consider cases where the set $\mS$ is finite. We then call $\Div_{K,\mS}$ the $\mS$-divisor group, and it relates to the classical notion of $\mS$-units.
Further details on each specific case are discussed in the relevant sections.

%% file: preliminaries/probabilities.tex
\subsection{Probabilities and (Discrete) Gaussian distributions}
Given a distribution $D$ over a discrete set $X$ and $x \in X$, we denote $D(x)$ for the probability that $D$ outputs $x$ and $D(A) = \sum_{x \in A} D(x)$ for any $A \subseteq X$.

\begin{definition}[Statistical distance] \label{def:statisticaldistance}
Let $(\Omega,\mathcal{S})$ be a measurable space with probability measures $P,Q$. The statistical distance (or total variation distance) between $P$ and $Q$ is defined by the rule 
\[ SD(P,Q) = \sup_{X \in \mathcal{S}} |P(X) - Q(X)|. \]
\end{definition}
For a discrete space $\Omega$, we have 
\[ SD(P,Q) =  \frac{1}{2} \sum_{x \in \Omega} |P(x) - Q(x)| =: \frac{1}{2} \|P - Q\|_1. \] 
For a continuous space $\Omega$ with probability densities $P,Q$, we have
\[ SD(P,Q) =  \frac{1}{2} \int_{x \in \Omega} |P(x) - Q(x)| =: \frac{1}{2} \|P - Q\|_1. \] 
For spaces that are partially discrete and partially continuous, some well-defined mix of an integral and a sum will define the statistical distance.
Due to the equivalence of these notions (up to a constant $\frac{1}{2}$) we will often describe closeness of probability distributions in terms of the distance metric $\| \cdot \|_1$, instead of $SD(  \cdot, \cdot)$. 
 \\ 
~\\
\noindent
The data processing inequality captures the idea that an algorithm (by just processing a single query) cannot
increase the statistical distance between two probability distributions. A proof can be found for example in \cite[\textsection 2.8]{Cover2006}.
\begin{theorem}[Data processing inequality] \label{theorem:dataprocessinginequality}
Let $(\Omega,\mathcal{S})$ be a measurable space with probability measures $P,Q$. Let $f$ be a (potentially probabilistic) function on $\Omega$. Then
\[  \|f(P) - f(Q) \|_1 \leq \|P-Q\|_1. \]
\end{theorem}

%% file: preliminaries/gaussians.tex
In the rest of this section, we recall known results about Gaussian distributions and prove some lemmas that will be useful for the rest of the article.

\subsubsection*{Gaussian distributions} \label{subsec:gaussiandistr}

We denote $\gaussian_s(\vec{x}) = \exp(-\pi \|\vec{x}\|^2/s^2)$ the
Gaussian function defined on any Euclidean vector space $V$. 
Since $\int_{\vec{x} \in V} \gaussian_s(\vec{x}) d\vec{x}= s^{\dim V}$,
we define the continuous Gaussian distribution on a Euclidean vector space $V$ with parameter $s$ and center $\vec{c}$ by 
the density function
\[  \Gaussian_{V,s,\vec{c}}(\vec{v}) = s^{-\dim(V)} \cdot \gaussian_s(\vec{v}-\vec{c}).  \]
For discrete sets $X \subseteq V$, specifically lattices, we use the same notation for
the \emph{discrete Gaussian}, where the subscript $X$ indicates that the 
Gaussian is discrete.
\[ \Gaussian_{X,s,\vec{c}}(\vec{x}) = \gaussian_{s}(\vec{x} - \vec{c})/\gaussian_{s}(X-\vec{c}), \]
where $\gaussian_{s}(X-\vec{c}) = \sum_{\vec{x} \in X} \gaussian_{s}(\vec{x}-\vec{c})$. This notation $f(X) = \sum_{\vec{x} \in X} f(\vec{x})$
for functions $f$ and discrete sets $X$ will be used frequently. Whenever the 
Gaussian is centered at the origin, i.e., $\vec{c} = 0$, we will suppress the subscript $\vec c$ from the notation, like this: $\Gaussian_{V,s}$.

\subsubsection*{Tail bounds of Gaussian distributions}

The following lemma originates from Banaszczyk's paper on transference theorems in lattices~\cite[Lemma 1.5]{Banaszczyk1993}.

\begin{lemma}[{\cite[Lemma 1.5]{Banaszczyk1993}}]
\label{lemma:Bana-bound}
For any $c > 1/\sqrt{2\pi}$ and any $n$ dimensional lattice $\Lambda$, $\gaussian_1(\Lambda \setminus c\sqrt{n}\mathcal{B}) \leq C^n \gaussian_1(\Lambda)$, where $\mathcal{B}$ is the euclidean ball of radius 1 and $C = c\sqrt{2\pi\mathrm{e}} \cdot \mathrm{e}^{-\pi c^2} < 1$.
\end{lemma}
\begin{corollary}
\label{coro:bana}
Let $\Lambda$ be a lattice of rank $n$ and $s > 0$. For any $\eps \in (0,1]$, it holds that $\Pr_{\vec x \leftarrow \Gaussian_{\Lambda,s}}(\|\vec x\| \geq s \cdot \sqrt{\log(1/\eps) + 2n}) \leq \eps$.
\end{corollary}

We also have a similar result in the case of a continuous Gaussian distribution.
\begin{lemma}
\label{lemma:bound-gaussian}
Let $V$ be a real vector space of dimension $n$ and $s > 0$. For any $\eps \in (0,1]$, it holds that $\Pr_{\vec x \leftarrow \Gaussian_{V,s}}(\|\vec x\| \geq s \cdot \sqrt{2n \cdot \log(2n/\eps)}) \leq \eps$.
\end{lemma}

\begin{proof}
Let $\vec B$ be an orthonormal basis of $V$ and write $\vec x = (x_1, \cdots, x_n)$ the coordinates of $\vec x$ in this basis. Then the random variables $x_i$ are linearly independent Gaussian distributions over $\R$ with standard deviation $s$. Moreover, for any $t > 0$, if $\|\vec x \| \geq t$, there should exist some $i$ such that $|x_i| \geq t/\sqrt{n}$. Hence, we obtain
\[\Pr_{\vec x \leftarrow \Gaussian_{V,s}}\big(\|\vec x\| \geq t\big)   \leq n \cdot \Pr_{x\leftarrow \Gaussian_{\R,s}}\big(|x| \geq t/\sqrt{n}\big) \leq 2n \cdot \exp\Big(-\frac{t^2}{2n \cdot s^2}\Big),\]
where the first inequality comes from the union bound and the last one comes from Chernoff's bound. Taking $t = s \cdot \sqrt{2n \cdot \log(2n/\eps)}$ leads to the desired result.
\end{proof}

\subsubsection*{Smoothing on lattices}\label{subsec:smoothingdefinition}
For any $\eps > 0$ and a lattice $\Lambda$ we define the smoothing parameter \cite{MicciancioRegev2007} $\eta_\eps(\Lambda)$ to
be the smallest $s \in \R_{>0}$ such that $\gaussian_{1/s}(\Lambda^\vee \backslash \{ 0\}) \leq \eps$, where
$\Lambda^\vee = \{ x \in \mbox{span}_\R(\Lambda) ~|~ \langle x , \ell \rangle \in \Z \mbox{ for all } \ell \in \Lambda \}$ is the
dual lattice of $\Lambda$. This smoothing parameter satisfies the following bound \cite[Lemma 3.3]{MicciancioRegev2007}
\begin{align}
\label{eq:smoothing-upper}
\eta_{\eps}(\Lambda) \leq \sqrt{\frac{\log(2n(1+1/\eps))}{\pi}}\cdot \lambda_n(\Lambda).
\end{align} 

We will use the following classical results on the total Gaussian weight of a (shifted) lattice $\Lambda$.
\begin{lemma}[Proof of {\cite[Lemma 4.4]{MicciancioRegev2007}}]
\label{lemma:total-gaussian-weight} \label{lemma:smoothing}
Let $\Lambda$ be an $n$-dimensional lattice and $s \geq \eta_{\eps}(\Lambda)$ for some $\eps > 0$. We have
$$ \gaussian_s(\Lambda) \in \left[1-\eps, 1+\eps\right] \cdot \frac{s^n}{\covol(\Lambda)}.$$
\end{lemma}

\begin{lemma}[{\cite[Lemma 2.9]{MicciancioRegev2007}}]
\label{lemma:upper-bound-total-weight-shift}
Let $\Lambda$ be an $n$-dimensional lattice, $s>0$ be a positive real number and $\vec c \in \Span_\R(\Lambda)$ be a vector, then $\gaussian_s(\Lambda+\vec c) \leq \gaussian_s(\Lambda)$.
\end{lemma}

\subsubsection*{Approximating and computing discrete Gaussians}

We will use the following lemma which says that one can sample from a distribution statistically close to a discrete Gaussian distribution. 
It is an adaptation of \cite[Theorem 4.1]{DBLP:conf/stoc/GentryPV08} for which the Gaussian can be approximated within any error and in which only bit operations are used (as opposed to real number operations as in the original article).

\begin{lemma}[Adapted from {\cite[Theorem 4.1]{DBLP:conf/stoc/GentryPV08}}]
\label{lemma:gpv}
There exists a probabilistic algorithm that takes as input a basis $\vec B = (\bb_1, \cdots, \bb_n) \in \Q^{n \times n}$ of an $n$-dimensional lattice $\Lambda$, an error bound $\geps \in (0,1]$, a parameter $s \geq \sqrt{\frac{\log(1/\geps)+2\log(n)+3}{\pi}} \cdot \max_i \|\vec \bb_i\|$ and a center $\bc \in \Span_{\R}(\Lambda) \cap \Q^n$ and outputs a sample from a distribution $\Klein_{\mB,\geps,s,\bc}$ such that $\SD(\Klein_{\mB,\geps,s,\bc}, \Gaussian_{\Lambda,s,\bc}) \leq \geps$.
This algorithm runs in expected time (bit complexity) polynomial in the size of the input and in $\log(1/\geps)$. Additionally, the output of $\bv \from \Klein_{\mB,\geps,s,\bc}$ always satisfies $\|\bv - \bc\| \leq s \cdot \sqrt{n\log(2n^2/\geps)}$.
 \end{lemma}

\begin{proof}
The proof can be adapted from the one of \cite[Theorem 4.1]{DBLP:conf/stoc/GentryPV08}, which proves the result for $\geps = n^{-\omega(1)}$. In this proof, the statistical distance between the sampled distribution $\Klein_{\mB,\geps,s, \vec c}$ and the ideal distribution $\Gaussian_{\Lambda,s,\bc}$ comes from two places.
The first one is that the algorithm makes $n$ calls to a sub-routine algorithm $\texttt{Sample}_\Z$, sampling from a Gaussian distribution over $\Z$ which is only statistically close to $\Gaussian_{\Z,s',c'}$ (for some parameters $c'$ and $s'$ depending on the input). Lemmas 4.2 and 4.3 from~\cite{DBLP:conf/stoc/GentryPV08} show that the statistical distance between $\texttt{Sample}_\Z$ and $\Gaussian_{\Z,s',c'}$ can be made as small as any arbitrary $\delta > 0$, to the cost of increasing the running time of $\texttt{Sample}_\Z$ in a way that is polynomial in 
$\log(1/\delta)$. We choose $\delta = \geps/(2n)$, which provides a running time polynomial in $\log(1/\geps)$ and $\log n$, and ensures that the $n$ calls to $\texttt{Sample}_\Z$ are responsible for a change of statistical distance that is at most $\geps /2$.

The second reason why the algorithm only provides a distribution that is statistically close to $\Gaussian_{\Lambda,s,\vec c}$ comes from the fact that $\gaussian_{s'}(\Z+c)$ is only approximately close to $\gaussian_{s'}(\Z)$ for an arbitrary $c \in \R$ (where $s' \geq s/(\max_i \|\vec b_i\|)$). In the proof of Theorem 4.1, the authors of~\cite{DBLP:conf/stoc/GentryPV08} show that, provided that $s \geq \eta_{\delta}(\Z) \cdot \max_i \|\bb_i\|$, then this statistical distance is at most $\delta'/2$, where $1+\delta' = \left(\frac{1+\delta}{1-\delta}\right)^n$.
We choose $\delta = \geps/(8n) \leq 1/(8n)$.
Using the fact that\footnote{%
We have $\frac{1+\delta}{1-\delta} \leq 1 + 4\delta$ for $\delta \leq 1/2$. Also, $(1 + x)^n \leq (e^x)^n = e^{nx} \leq 1 + 2nx$ for $x< 1/(2n)$. Hence, for $\delta \leq 1/(8n)$ and $n \in \N_{>0}$, $\left( \frac{1+\delta}{1-\delta} \right)^n \leq (1 + 4\delta)^n \leq 1 + 8n \delta$.
}
$\left(\frac{1+\delta}{1-\delta}\right)^n \leq 1+8n\delta$ for all $\delta \leq 1/(8n)$ and $n \geq 1$, we obtain that the statistical distance between the approximate distribution and the ideal one is at most $4n\delta = \geps/2$.
Using \Cref{eq:smoothing-upper}, we see that $\eta_{\delta}(\Z) \leq \sqrt{\frac{\log(1/\geps)+2\log(n)+3}{\pi}}$, hence $s \geq \eta_{\delta}(\Z) \cdot \max_i \|\bb_i\|$ as desired.

For the bound on $\vec{v} \from \Klein_{\mB,\geps,s,\bc}$ observe the output $\bv_0$ of the algorithm in \cite[Section 4.2]{DBLP:conf/stoc/GentryPV08}. By \cite[Lemma 4.4]{DBLP:conf/stoc/GentryPV08}, the output of this algorithm satisfies $\bv - \bc = \sum_{i \in [n]} (\hat{z}_i -c'_i) \cdot \tilde{\bb}_i$, where the values $\hat{z}_i$ and $c_i'$ are from the algorithm, and the $\tilde{\bb}_i$ are the Gram-Schmidt vectors associated to the $\bb_i$'s.
Hence, by the Pythagorean theorem, $\|\bv - \bc\|^2 =  \sum_{i \in [n]} |\hat{z}_i -c'_i|^2 \cdot \|\tilde{\bb}_i\|^2$. In part 1(b) of the algorithm \cite[Section 4.2]{DBLP:conf/stoc/GentryPV08}, $D_{\Z,s'_i,c'_i}$ is implemented as in \cite[Section 4.1]{DBLP:conf/stoc/GentryPV08}, i.e., the algorithm \texttt{Sample$\Z$} is called (see the proof of \cite[Theorem 4.1]{DBLP:conf/stoc/GentryPV08}). Therefore $\hat{z}_i$ can be shown to lie in $\Z \cap [c'_i - t(n) s'_i, c'_i + t(n) s'_i]$, where $t(n) = \sqrt{\log(n/\delta)}$ (with $\delta = \geps/(2n)$, as before). As a consequence,
\[ \|\bv - \bc\|^2 =  \sum_{i \in [n]} |\hat{z}_i -c'_i|^2 \cdot \|\tilde{\bb}_i\|^2 \leq \sum_{i \in [n]} t(n)^2 (s_i')^2 \cdot \|\tilde{\bb}_i\|^2 = \sum_{i \in [n]} t(n)^2 s^2 = n t(n)^2 s^2, \]
by the definitions of $s_i' = s/\|\tilde{\bb}_i\|$. Hence $\| \bv - \bc\| \leq \sqrt{n\log(n/\delta)}s$ with $\delta = \geps/(2n)$, as was required to prove. 

We show now that the algorithm in \cite{DBLP:conf/stoc/GentryPV08} can be readily adapted into one without real number operations, but just bit operations. There are two places in the algorithm in \cite[Section 3.2]{DBLP:conf/stoc/GentryPV08} where real arithmetic is used, namely in the subroutine \texttt{Sample$\Z$} and in the computation of $\sd/\|\tilde{\bb}_k\|$ where $\tilde{\bb}_k$ are the Gram-Schmidt vectors. Note that the Gram-Schmidt orthogonalization itself (without normalizing) can be done with rational arithmetic in polynomial time and is thus not altered in this adapted version.

The subroutine \texttt{Sample$\Z$} can be amended to avoid real operations by just approximating the sampling probabilities $\rho_\sd(x-c)$ with $x \in \Z \cap [c - s \cdot t(n), c + s \cdot t(n)]$ well enough. This can be done within polynomial bit complexity in the size of the input, and, by maybe slightly increasing $t(n)$, without loss in the approximation error $\delta$.

The value of $\sd/\|\tilde{\bb}_k\|$ is only used in the subroutine \texttt{Sample$\Z$}, where actually its \emph{square} is used to compute $\rho_{\sd/\|\tilde{\bb}_k\|}(x-c)$ and hence there is no need to compute $\sd/\|\tilde{\bb}_k\|$ but rather its square $\sd^2/\|\tilde{\bb}_k\|^2$ which consists of rational numbers.

Hence we can conclude that this slight adaptation of \cite[Section 3]{DBLP:conf/stoc/GentryPV08} has polynomial bit-complexity in the size of its input.
\end{proof}

%% file: acknowledgements.tex
\section{Acknowledgements}
\noindent
First and foremost, we would like to express our gratitude to L\'eo Ducas, who took part in early discussions that lead to this project, and suggested key ideas for the randomisation of ideals.

We thank Aurel Page and Wessel van Woerden for insightful discussions. We thank Sameera Vemulapalli for bringing to our attention the bound on $\lambda_n(\OK)$ of Bhargava et al. \cite{bhargava2020}. We thank Pascal Autissier for his useful comments on the first version of this article and for bringing to our attention the article of Bayer Fluckiger~\cite{bayer2006upper}; we also owe him the new shorter version of the proof of Lemma~\ref{lemma:smallf}.

Alice Pellet-Mary and Benjamin Wesolowski were supported by the Agence
Nationale de la Recherche under grants ANR-22-PETQ-0008 (PQ-TLS) and ANR-21-
CE94-0003 (CHARM). Benjamin Wesoloswki was supported by the European Research Council under grant No. 101116169 (AGATHA CRYPTY). This work started when the authors were visiting the Simon Institute during the lattice semester in 2020.

%% file: ideal_sampling/A01-introduction.tex
\section{Introduction}
\noindent
In this first part of the article, we propose a general strategy to provably solve a recurring computational problem in number theory (assuming the extended Riemann hypothesis, ERH): given an ideal class $[\ma]$ of a number field $K$, sample an ideal $\mathfrak c \in [\ma]$ belonging to a particular family of ideals (e.g., the family of smooth ideals, or near-prime ideals).
While there is a simple heuristic algorithm for this task, it has proved notoriously difficult to resolve it rigorously. It has thereby been a central roadblock explaining the heuristic nature of many major algorithms in computational number theory.

The main result of this part is \Cref{theorem:ISmain} (page \pageref{theorem:ISmain}). Its formulation in full generality is postponed, as it first requires the introduction of several notions. However, a simplified statement, which may already suit many needs, is available in \Cref{thm:sampling-simplified} (page \pageref{thm:sampling-simplified}).

\subsection*{Roadmap}%
Fix an arbitrary family of ideals $\idset$.
For convenience, we consider the input to be an ideal $\mb \in [\ma]^{-1}$, and are looking for an ideal $\mathfrak c \in \idset$ in the inverse class of $\mb$. 
The folklore strategy consists in considering $\mb$ as an \emph{ideal lattice} via the Minkowski embedding, and sampling random elements $\beta\in\mb$ (within some bounds, say in a ``box'' $r\mathcal B$ of radius $r$) until $\mathfrak c = \beta \mb^{-1}$ falls in the desired family of ideals $\idset$. Heuristically, one \emph{expects} that for $\beta$ ``sufficiently random'', the ideal $\beta \mb^{-1}$ falls in $\idset$ with probability proportional to the ``density'' of the family (think about the set of prime ideals that have norm around $x$ having ``density'' $\approx 1/(\dedres\log(x))$). This, of course, cannot be literally true for arbitrary families (e.g., principal ideals), since $\beta \mb^{-1}$ is confined to one ideal class.
Instead, we will solve the problem for $\mathfrak c \in \idset \cdot \idset_B$, where $\idset_B$ is the family of $B$-smooth ideals for some bound $B$. In all applications we are aware of, $\idset = \idset \cdot \idset_B$ (smooth ideals, near-prime ideals).

\subsubsection*{The Arakelov class group} 
The notion of \emph{ideal lattice} plays a key role in this sampling strategy.
The space of ideal lattices up to isomorphism is naturally isomorphic to the so-called \emph{Arakelov class groups}
We open this part of the article in \Cref{sec:preliminaries_arak} with an introduction to Arakelov ray class groups. They can be thought of as a ``combination'' of the ray class group and the ray unit group of a number field. In \Cref{section:randomwalk}, we state a useful result on these Arakelov ray class groups: certain \emph{random walks} in them rapidly converge to the uniform distribution. This is a generalization of \cite{dBDPW} from Arakelov class groups to Arakelov \emph{ray} class groups.

\subsubsection*{Average densities}
In the folklore strategy, one is hoping that for (uniformly) random $\beta\in\mb \cap r\mathcal B$ (the intersection of an ideal lattice and a ``box''), the probability that $\beta \mb^{-1} \in \idset$ is proportional to the density of $\idset$. This is generally not true.
However, in \Cref{section:densitiessampling}, we prove that it \emph{is} true \emph{on average} when the ideal lattice $\mb$ is also random, uniformly distributed in the Arakelov ray class group.
This might sound too weak for our goal: the input $\mb$ of the algorithm is not random. It will actually be sufficient when properly combined with a randomization step.

\subsubsection*{Ideal sampling} We then turn these density results into an algorithm. First, we show in \Cref{section:samplingboxideallattice} how to sample uniformly a random element in a set of the form $\mb \cap r\mathcal B$, the intersection of an ideal lattice and a ``box''.

Then, \Cref{section:sampling} culminates with the main algorithm, \Cref{alg:samplerandom}.
It addresses the final obstacle: the density result of \Cref{section:densitiessampling} only holds on average, for uniformly random ideal lattices. The input $\mb$ is not random, so we need to randomize it. This is where random walks in the Arakelov class group come in. Essentially, \Cref{alg:samplerandom} starts by multiplying $\mb$ with a few ``small'' prime ideals, resulting in a random ideal $\mathfrak w \mathfrak b$ where $\mathfrak w$ is smooth (i.e., $\mathfrak w \in \idset_B$ for some bound $B$). 

\Cref{alg:samplerandom} then samples a uniformly random $\beta\in (\mathfrak w\mb) \cap r\mathcal B$. Then:
\begin{itemize}
\item From \Cref{section:randomwalk} and \cite{dBDPW} (on the rapid equidistribution of random walks), the random  ideal lattice $\mathfrak w\mb$ is close to uniformly distributed.
\item From \Cref{section:densitiessampling}, we obtain that $\beta (\mathfrak w\mb)^{-1} \in \idset$ with probability proportional to the density of $\idset$. 
\end{itemize}
Upon the event $\beta (\mathfrak w\mb)^{-1} \in \idset$, we get
\(\beta \mb^{-1} = \mathfrak w \cdot (\beta (\mathfrak w\mb)^{-1}) \in \idset_B \cdot \idset,\)
as desired.

\subsubsection*{Further properties}
Finally, in \Cref{section:propertiessampling} and \Cref{section:estimationmodu}, we develop tools to ease the applicability of this ``ideal sampling'' algorithm. These tools are indispensable for the application presented in Part \partref{part2}, and may be useful in other contexts.
More precisely, 
in \Cref{section:propertiessampling}, three important properties of the ideal sampling algorithm are stated and proved: the \emph{shifting property},  \emph{boundedness} and
 \emph{almost-Lipschitz-continuity}. In \Cref{section:estimationmodu}, we estimate quantities related to the modulus $\modu$, which affect the behavior of the ideal sampling algorithm.

%% file: ideal_sampling/A02-prelim.tex
\section{Background on the Arakelov class group}

\noindent
Recall that throughout this paper, we consider a number field $K$ of rank $n$ over $\Q$, having ring of integers $\OK$, discriminant $\disc$, regulator $\reg$, class number $\classnumber$ and group of roots of unity $\rou$. It has $\rem$ real embeddings and $\cem$ conjugate pairs of complex embeddings, with $n = \rem + 2 \cem$.

\label{sec:preliminaries_arak}

\input{ideal_sampling/PRELIM/arakelovtheory}

\input{ideal_sampling/PRELIM/ideallattices}

%% file: ideal_sampling/PRELIM/arakelovtheory.tex
\subsection{The Arakelov Ray Class Group}

\label{section:arakelovclassgroup}
In this section, we rely heavily on the notation introduced in \Cref{section:divisors} for the divisor group $\Div_K$.
The \emph{Arakelov ray divisor group} with respect to a modulus $\modu = \moduz \moduinf$ is a subgroup $\rayDiv \subseteq \Div_K$ defined as
 \[ \rayDiv = \bigoplus_{\mp \nmid \moduz} \Z \times \bigoplus_{\nu} \R \]
where $\mp$ ranges over the set of all prime ideals of $\OK$ that do not divide the finite part $\moduz$ of the modulus, and $\nu$ over the set of all infinite places  (embeddings into the complex numbers up to conjugation). 
The case $\modu = \OK$ yields the standard divisor group $\Div_K$.
Recall that we write an arbitrary element in $\rayDiv$ as 
\begin{equation} \label{eq:arakelovfiniteinfinite} \ba = \underbrace{\sum_{\mp \nmid \moduz} \ha_\mp \cdot  \ldb \mp \rdb}_{\finitepart{\ba}}  + \underbrace{\sum_{\nu} \ha_\nu \cdot \ldb \nu \rdb}_{\infinitepart{\ba}}. \end{equation}
with only finitely many non-zero $\ha_\mp$ and with $\ha_\nu \in \R$. 
The map $\ldb \cdot \rdb : K^\times \to \Div_K$ naturally restricts and co-restrict to
\begin{align*}
\ldb \cdot \rdb: \Kmodumodu &\longrightarrow \rayDiv\\
\alpha &\longmapsto \ldb \alpha \rdb =  \sum_{\mp \nmid \moduz}\mathrm{ord}_{\mp}(\alpha)\ldb \mp \rdb - \sum_{\nu} \npl \log |\sigma_\nu(\alpha)| \cdot \ldb \nu \rdb,
\end{align*}
where $\npl = 2$ whenever $\pl$ is a complex place and $1$ otherwise.
The divisors of the form $\ldb \alpha \rdb$ for $\alpha \in \Kmodu$ are called  
\emph{principal $\modu$-ray divisors}.

Just as the ideal ray class group is the group of ideals coprime with $\modu$
quotiented by the ray $\Kmodu$, the \emph{Picard ray group} is the group
of Arakelov ray divisors quotiented by the group of principal ray Arakelov divisors. In other
words, the Picard ray group $\rayPic$ is defined by the following exact sequence, where $\muKmodu = \mu_K \cap \Kmodu$ are the roots of unity in the ray:
\[  0 \rightarrow \muKmodu \rightarrow \Kmodu \xrightarrow{\ldb \cdot \rdb} \rayDiv \rightarrow \rayPic \rightarrow 0. \]
For any Arakelov ray divisor $\ba = \sum_{\mp \nmid \moduz} \ha_\mp \cdot  \ldb \mp \rdb  + \sum_{\nu} \ha_\nu \cdot \ldb \nu \rdb$ , we denote its %
class in the Picard ray group $\rayPic$ by $[\ba]$, in the same fashion that
$[\ma]$ denotes the ideal class of the ideal $\ma$ in $\rayclassgroup$.

Since principal ray divisors $\ldb \alpha \rdb$ for $\alpha \in \Kmodu$ are in the kernel of the degree map,
the degree factors through $\rayPic$. We can therefore define the \emph{degree-zero Arakelov ray divisor group} 
$\rayDiv^0 = \{ \ba \in \rayDiv ~|~ \deg(\ba) =0 \}$ and the \emph{Arakelov ray class group} $\rayPic^0 = \{ [\ba] \in \rayPic ~|~ \deg([\ba]) = 0 \}$. 
\\
The maps $\finpart{\Exp}$, $\infpart{\Exp}$, $d$ and $\secideal$ defined in \Cref{section:divisors} naturally restrict and co-restrict to
\begin{alignat}{4}
\finpart{\Exp} &: \bigoplus_{\mp \nmid \moduz} \Z &\longrightarrow& \idealsmodu &:&  \ba \longmapsto \prod_{\mp \nmid \moduz} \mp^{\ha_\mp},\\
\infpart{\Exp} &: \bigoplus_{\nu} \R &\longrightarrow& \nfr^\times &:& \ba \longmapsto  \left(e^{\nplemb^{-1} \cdot \ha_{\nu_\sigma}}\right)_\sigma,\\
\label{def:d2} d &:\rayideals &\longrightarrow& \bigoplus_{\mp \nmid \moduz} \Z &:& \ma \longmapsto  \sum_{\mp \nmid \moduz } \ord_\mp(\ma) \cdot  \ldb \mp \rdb,\\
\label{def:dzero2} \secideal &:\rayideals &\longrightarrow& \rayDiv^0 &:& \ma \longmapsto  d(\ma)  - \frac{\npl \log(\norm(\ma))}{n} \sum_{\nu} \ldb \nu \rdb.
\end{alignat}

The groups and their relations treated above fit nicely in the diagram of exact sequences given in \Cref{fig:commdiagray}, 
where the middle row sequence splits with the section $\secideal$. In this diagram we use the notation $\OKmodu = \units \cap \Kmodu$ and $\mu_{\Kmodu} = \mu_K \cap \Kmodu$.
The group $\Tmodu = H/\logunitsmodu$ is the \emph{logarithmic ray unit torus}, with $\logunitsmodu$ the \emph{logarithmic ray unit lattice}. %
\begin{figure}[t]
\begin{center}
\begin{tikzpicture}
  \matrix (m) [matrix of math nodes,row sep=2em,column sep=2em,minimum width=2em]
  {
      & 0                                   &  0          &  0               &  \\
    0 & \OKmodu/\muKmodu                         &  \Kmodu/\muKmodu  &  \pridmodu  & 0 \\
    0 & H &  \rayDiv^0   & \rayideals   & 0 \\
    0 & \Tmodu                                 &  \rayPic^0   &  \Cl_K^\modu           & 0 \\
      & 0                                   &  0          &  0               & \\    };
  \path[-stealth]
    (m-2-1) edge (m-2-2)
    (m-3-1) edge (m-3-2)
    (m-4-1) edge (m-4-2)
    (m-2-4) edge (m-2-5)
    (m-3-4) edge (m-3-5)
    (m-4-4) edge (m-4-5)
    (m-1-2) edge (m-2-2)
    (m-1-3) edge (m-2-3)
    (m-1-4) edge (m-2-4)
    (m-4-2) edge (m-5-2)
    (m-4-3) edge (m-5-3)
    (m-4-4) edge (m-5-4)
    
  (m-2-4) edge (m-3-4)
  (m-3-4) edge (m-4-4)
  (m-2-3) edge node [left] {$\ldb \cdot \rdb$} (m-3-3)
  (m-3-3) edge (m-4-3)
  (m-2-2) edge node [left] {$\Log$} (m-3-2)
  (m-3-2) edge node [left] {$ $} (m-4-2)
  
  (m-2-2) edge (m-2-3)
  (m-2-3) edge (m-2-4)
  (m-3-2) edge node [below] {$ $} (m-3-3)
  (m-3-3) edge node [below] {\scriptsize{$\ba \mapsto \finpart{\Exp}(\finitepart{\ba})$}} (m-3-4)
  (m-3-4) edge[bend right] node [above] {$d^0$} (m-3-3)

  (m-4-2) edge node [below] {$ $} (m-4-3)
  (m-4-3) edge node [below] {$ $} (m-4-4)
;
\end{tikzpicture} 
\end{center}
\vspace{-2em}
\caption{A commutative diagram of short exact sequences involving the Arakelov ray class group.}
\label{fig:commdiagray}
\end{figure}

\subsubsection*{Relations between different ray groups}
The (ray) unit groups $\units,\OKmodu$, the (ray) class groups $\classgroup, \rayclassgroup$, and the
ray groups $\Kmodu$ and $\Kmodumodu$ are tightly related by
an exact sequence (e.g., \cite[Chapter VI, \textsection 1]{lang1994algebraic}), relating
the (relative) cardinalities of these groups. Namely,
\begin{equation} \label{eq:relateraynonray} |\units/\OKmodu| \cdot |\rayclassgroup| = \phi(\moduz) \cdot 2^{|\modur|} \cdot |\classgroup|, \end{equation}
where $\phi(\moduz) =\defphi$ and $|\modur|$ is the number of real places dividing $\moduinf$, (note that $|\Kmodumodu/\Kmodu| =  \phi(\moduz) \cdot 2^{|\modur|}$).
Also, by observing the kernel-cokernel sequence of the inclusions $\OKmodu \subseteq \OKstar \subseteq H = \log \nfr^0$, we obtain,
\begin{equation} \label{eq:relateraynonraytorus} |\muKmodu| \cdot |\units/\OKmodu| = |\mu_K| \cdot \vol(\Tmodu)/\vol(\unittorus). \end{equation}

\subsubsection*{Geometric properties of the divisor group and the Arakelov class group}
Recall that the norm of a divisor is defined as
\[ \divnorm{\ba} = \left(\sum_\mp \ha_\mp^2 + \sum_\pl \ha_{\pl}^2\right)^{1/2}.\]
In the following lemma, we show that the volume of the
Arakelov ray class group roughly follows the square root of the absolute value of the field discriminant times $\phi(\moduz) \cdot 2^{|\modur|}$; with
a possible additional correction factor due to the roots of unity.

\begin{lemma}[Volume of $\rayPic^0$] \label{lemma:boundvolraypicappendix} %
For $n = [K:\Q] > 1$, we have
\begin{align} \label{eq:raypicvolumenolog}  \vol\left({\rayPic^0}\right) & = |\Cl_K^{\modu}| \cdot \vol(T^{\modu}) = \frac{|\muKmodu|}{|\mu_K|} \cdot \phi(\moduz) \cdot 2^{|\modur|} \cdot \classnumber \cdot \vol(T) \\ & =\frac{|\muKmodu|}{|\mu_K|} \cdot \phi(\moduz) \cdot \classnumber \cdot \reg \cdot \sqrt{\dimh+1} \cdot  2^{|\modur|} , \end{align}
where $\phi(\moduz) = \defphi$ and $|\modur|$ is the number of real places dividing $\moduinf$. Additionally,
\begin{equation} \label{eq:simplepicbound} \log\left(\vol\left(\rayPic^0\right)\right) \leq \log\left(\norm(\moduz) \cdot 2^{|\modur|}\right) +  \log |\dcrk|. \end{equation}
\end{lemma}

\begin{proof} The first identity involving the volume of the Arakelov ray class group follows from the exact sequence in \Cref{fig:commdiagray}. The second one can be deduced from \Cref{eq:relateraynonray,eq:relateraynonraytorus}. %
The third one follows from the volume computation of $T$ in
\Cref{eq:unittorusvolume}.

The bound on the logarithm is obtained by using $\frac{|\muKmodu|}{|\mu_K|} \leq 1$, applying the class number formula \cite[VII.\textsection 5, Cor 5.11]{neukirch2013algebraic} and Louboutin's bound \cite{louboutin00} on the residue $\dederesidue$ of the Dedekind zeta function at $ s= 1$:
\begin{align*} \volabs{\rayPic^0} & \leq \phi(\moduz) \cdot \classnumber \cdot \reg \cdot \sqrt{\dimh+1} \cdot 2^{|\modur|} =  \frac{\phi(\moduz) \cdot \dederesidue \cdot \sqrt{|\dcrk|} \cdot |\mu_K| \cdot \sqrt{\rem + \cem}}{2^{\rem - |\modur|} \cdot (2 \pi)^{\cem}} \\ & \leq \phi(\moduz) \cdot 2^{|\modur|} \cdot \sqrt{|\dcrk|}  \cdot \dederesidue
  \leq \phi(\moduz)\cdot 2^{|\modur|} \cdot \sqrt{|\dcrk|} \cdot \left( \frac{e\log |\dcrk|}{2(n-1)} \right)^{n-1}
  \\ & \leq \phi(\moduz)\cdot 2^{|\modur|} \cdot |\dcrk|.
  \end{align*}
The first inequality follows from $\frac{|\mu_K| \sqrt{\rem+\cem}}{2^{\rem} (2\pi)^\cem} \leq \frac{n^{3/2}}{2^n}\leq 1$. 
The last inequality follows from the fact that $\tfrac{e\log|x|}{|x|} \leq 1$ for all $x \in \R$. This inequality instantiated with $x = |\dcrk|^{\frac{1}{2(n-1)}}$ then yields $\left( \frac{e\log |\dcrk|}{2(n-1)} \right)^{n-1} \leq \sqrt{|\dcrk|}$.
\end{proof}

%% file: ideal_sampling/PRELIM/ideallattices.tex
\subsection{Divisors and ideal lattices} 
The $\Exp$ map introduced in \Cref{section:divisors} restricts to the group homomorphism
\[
\Exp: \rayDiv \longrightarrow \idlat_K :
\ba \longmapsto \infpart{\Exp}(\infpart{\ba}) \cdot \finpart{\Exp}(\finpart{\ba}),
\]
sending each Arakelov divisor to an ideal lattice. Recall that for any divisor $\ba \in \rayDiv$, we have
\( \vol( \dExp{\ba}) \! = \sqrt {|\dcrk|} \cdot e^{\deg(\ba)} .\)

\subsection{\texorpdfstring{$\tau$}{t}-equivalent elements and generators of an Arakelov ray divisor}
Suppose that $\modu$ is a modulus, and let $\ma$ be any ideal coprime to $\moduz$.
An element $\alpha \in \ma$ is said to be $\tau$-equivalent (with respect to $\modu$)
if $\alpha \equiv \tau \bmod \moduz$ and $\embn(\alpha)/\embn(\tau) \in \R_{>0}$
for all real $\pl \mid \moduinf$.
If additionally, $\ma$ is \emph{principal}, any $\tau$-equivalent element $\alpha$
such that $\ma = (\alpha)$ is called a $\tau$-equivalent \emph{generator} of $\ma$.

These notions generalize to Arakelov ray divisors.
As we can see Arakelov ray divisors as ideal lattices 
$\kx\ma$, an element of such a divisor is just an element of the shape $\kx \alpha$
where $\alpha \in \ma$ and $\kx \in \nfr$. Similarly, a generator of such divisor
is an element in $\nfrstar$
of the shape $\kx \alpha$, where $\alpha$ is a generator of $\ma$.
The precise definitions are as follows.

\begin{definition}[$\tau$-equivalent elements of an Arakelov ray divisor] \label{def:arakelts} \label{def:rayelts} Let $\tau \in \Kmodumodu$ and let $\ba \in \rayDiv$ be an Arakelov ray divisor with
an infinite part $\infinitepart{\ba}$ and a finite part $\finitepart{\ba}$ (see \Cref{eq:arakelovfiniteinfinite}). We define the set of $\tau$-equivalent elements
$\rayelt{\ba} \subseteq \nfr$ of $\ba$ by the following rule
\[ \rayelt{\ba} :=  \infpart{\Exp}(\infpart{\ba}) \cdot ( \finpart{\Exp}(\finpart{\ba}) \cap \tau \Kmodu)
 \]
 Equivalently, we can write
 \[ \rayelt{\ba} = \{ \alpha \in \dExp{\ba} ~|~  \infpart{\Exp}(-\infpart{\ba}) \cdot \alpha \in \tau \Kmodu \}. \]
\end{definition}

\begin{definition}[$\tau$-equivalent generators of an Arakelov ray divisor] \label{def:generators} \label{def:raygenerators} Let $\tau \in \Kmodumodu$ and let $\ba \in \rayDiv$ be an Arakelov ray divisor with
an infinite part $\infinitepart{\ba}$ and a finite part $\finitepart{\ba}$ (see \Cref{eq:arakelovfiniteinfinite}). We define the set of $\tau$-equivalent generators  
$\raygen{\ba} \subseteq \nfr$ of $\ba$ by the following rule
\[ \raygen{\ba} := \begin{cases} 
                        \infpart{\Exp}(\infpart{\ba}) \cdot (\kappa \cdot \OKstar \cap \tau \Kmodu) \subseteq \dExp{\ba} & \mbox{ if } \finpart{\Exp}(\finpart{\ba}) = (\kappa)   \\
                        & \mbox{ for some } \kappa \in \Kmodumodu \\ 
                    \emptyset & \mbox{ otherwise } 
                       \end{cases}
 \]
 Equivalently, we can write
 \[ \raygen{\ba} = \{ \alpha \in \dExp{\ba} ~|~ \infpart{\Exp}(-\infpart{\ba}) \cdot \alpha  \in \tau \Kmodu \text{ is a generator of } \finpart{\Exp}(\finpart{\ba})\}. \]
\end{definition}

\subsection{Uniform sampling of prime ideals in a certain Arakelov class}
In the main result of this part, we will need to sample primes $\mp$ that satisfy $[d^0(\mp)] \in \subpic$ for a certain finite-index subgroup of $\rayPic^0$. The procedure explaining this and the respective running time can be found in this lemma; the proof is provided in \Cref{A:primegsample}.
\begin{restatable}[Uniform sampling of prime ideals, \protect{\normalfont ERH}]{lemma}{primegsampling}
 \label{lemma:sampleGprimes} Let a basis of $\OK$ be known, and let $\moduz \subseteq \OK$ be a modulus.
Let $\subpic \subseteq \rayPic^0$ a finite-index subgroup and let $\mathbf{O}_\subpic$ be an oracle that on input an ideal $\mc$ returns whether $[d^0(\mc)] \in \subpic$ or not.
Let $\mathcal{P}_B = \{ \mp \mbox{ prime ideal of K } ~|~ \norm(\mp) \leq B, \mp \nmid \moduz  \mbox{ and } [d^0(\mp)] \in \subpic \}$. 

There exists a bound
\[ B_0 = \widetilde O\Big( [\rayPic^0:\subpic]^2 %
\cdot \big [n^2 (\log \log(1/\varepsilon))^2 + n^2 (\log(1/\sprime))^2 +  (\log( |\dcrk|\norm(\modu)))^2 \big] \Big) \]
such that for all $B \geq B_0$, one can sample uniformly from $\mathcal{P}_B$ in expected time $O([\rayPic^0:\subpic] \cdot n^3 \log^2 B)$ and using $O( [\rayPic^0:\subpic] \cdot n \log B)$ queries to $\mathbf{O}_\subpic$.
\end{restatable}

%% file: ideal_sampling/A03-RWRAY.tex
\section{Random walks in Arakelov ray class groups} \label{chapter:randomwalkray} \label{chapter:randomwalk}

\label{section:randomwalk}

\noindent
In this section, we present a theorem on the rapid mixing of random walks on finite-index subgroups of Arakelov ray class groups.
It is a generalization of the main theorem of~\cite{dBDPW}, from Arakelov class groups to finite-index subgroups $\subpic$ of Arakelov ray class groups.
Starting with a point in the
hyperplane $\hyper \hookrightarrow \rayDiv^0$, sampled according to a Gaussian distribution, we prove that multiplying
this point sufficiently often by small random prime ideals whose Arakelov classes are in $\subpic \subseteq \rayPic^0$ 
yields a random ray divisor that is very close to uniformly distributed in $\subpic$, (the concerning subgroup of the Arakelov ray class group).

\begin{definition}[Random Walk Distribution in $\rayDiv^0$] \label{def:rwdiv} For a number field $K$ and a finite-index subgroup $\subpic \subseteq \rayPic^0$, 
we denote by $\Walk_{\subpic}(B,N,\sd) \in L_1(\rayDiv^0)$ the distribution on $\rayDiv^0$ that is obtained by the following random walk procedure. 

Sample %
$\ha  = (\ha_\pl)_\pl \in \hyper \subseteq \rayDiv^0$
according to a centered Gaussian distribution with standard deviation $\sd > 0$ (see \Cref{subsec:gaussiandistr}). Subsequently, sample $N$ ideals $\mp_j$ uniformly from the set of all prime ideals coprime with $\moduz$, with norm bounded by $B$ and whose Arakelov class $[d^0(\mp_j)]$ lies in $G$. Finally, output $\ha + \sum_{j = 1}^N d^0(\mp_j)$, where $\ha \in \rayDiv^0$ is understood via the inclusion $H \subseteq \rayDiv^0$.
\end{definition}
\begin{definition} \label{def:perioddistr} For any distribution $\distr$ on $\rayDiv^0$, we define the distribution $[\distr]$
on $\rayPic^0$ by the following rule:
\[ [\distr]( \cdot ) = \sum_{\kappa \in \Kmodu/\muKmodu} \distr( \cdot + \ldb \kappa \rdb ). \]
This distribution $[\distr]$ on $\rayPic^0$ arises whenever one samples $\ba \from \distr$ and subsequently takes the Arakelov ray class $[\ba] \in \rayPic^0$. 
\end{definition}

\begin{definition}[Random Walk Distribution in $\rayPic^0$] \label{def:rwpic}
We denote by $[\Walk_{\subpic}(B,N,\sd)]$ the distribution on the Arakelov class group obtained by sampling $\ba$ from $\Walk_{\subpic}(B,N,\sd)$ and taking the Arakelov ray class $[\ba] \in \rayPic^0$ (as in \Cref{def:perioddistr}).
\end{definition}

\begin{restatable}[Random Walks in finite-index subgroups of the Arakelov Ray Class Group, \protect{\normalfont ERH}]{theorem}{randomwalktheorem}
\label{thm:random-walk-weak}%
Let $\varepsilon > 0$ and $s > 0$ be any positive real numbers and let   $k \in \R_{>0}$ be a positive real number as well. Let $\subpic \subseteq \rayPic^0$ 
be a finite-index subgroup of the Arakelov ray class group.
Putting\footnote{Recall that for any lattice $\Lambda$, we write $\Lambda^\vee$ for its dual, and $\eta_1(\Lambda)$ for its smoothing parameter (see page~\pageref{subsec:smoothingdefinition}).} $\sprime = \min(\sqrt{2} \cdot s,1/\eta_1(\logunitsmodu^\vee))$,
there exists a bound 
\[ B = \widetilde O\Big( [\rayPic^0:\subpic]^2 \cdot n^{2k} \big [n^2 (\log \log(1/\varepsilon))^2 + n^2 (\log(1/\sprime))^2 +  (\log( |\dcrk|\norm(\modu)))^2 \big] \Big) \]
such that for any integer

\begin{equation} \label{eq:sufflowerboundN_rest} N \geq \left \lceil \frac{1}{2k\log n} \cdot ( \dimh \cdot \log(1/\sprime) + 2\log(1/\varepsilon) + \log \volabs{\rayPic^0} - \log[\rayPic^0:\subpic]+ 2) \right \rceil, \end{equation}

the random walk distribution $[\Walk_{\subpic}(B,N,\sd)]$ is $\varepsilon$-close to uniform in $L_1(\subpic)$, i.e.,
\[ \left\| [\Walk_{\subpic}(B,N,\sd)] - \unif(\subpic)\right\|_1  \leq \varepsilon ,\]
where $\norm(\modu) = \norm(\moduz) \cdot 2^{|\modur|}$, $\dimh = \dim(H) = n_\R + n_\C - 1$ and $|\modur|$ is the number of different real places dividing $\modu$ (see \Cref{prelim:numberfields}).
\end{restatable}

 \begin{proof}
 The proof is very similar to that of~\cite{dBDPW}, adapted to account for the ray and the finite-index subgroup. Details can be found in \Cref{sec:appendixRW}.
 \end{proof}

The following instantiation of the random walk theorem 
gives the appropriate parameters for the main application
of this paper, the sampling algorithm (see \Cref{alg:samplerandom} and \Cref{theorem:ISmain})
\begin{corollary}[\protect{\normalfont ERH}] \label{thm:random-walk-prime-sampling}
\label{thm:random-walk-ideal-sampling}
Let $n = [K:\Q] \geq 2$, $\sd = 1/n^2$, let $\subpic \subseteq \rayPic^0$ be a finite-index subgroup and let $\eps > 0$ be an error parameter.
There exists a bound $B = \widetilde O \Big ( [\rayPic^0:\subpic]^2 \cdot n^{2} \cdot \big[  n^2(\log \log (1/\eps))^2 + (\log (|\dcrk|\norm(\modu)))^2 \big] \Big)$
such that for $N = \lceil 7n  + 2 \log(1/\eps) +  %
\log \volabs{\rayPic^0} - \log[\rayPic^0:\subpic] + 2 \rceil$ %
the random walk distribution $[\Walk(B,N,\sd)]$ on $\rayPic^0$ is $\eps$-close to uniform in $L_1(\subpic)$, i.e.,
\[ \left\| [\Walk(B,N,\sd)] - \unif(\subpic)\right\|_1  \leq \eps.\]
\end{corollary}

\begin{proof}
This formulation of the random walk theorem is obtained by instantiating \Cref{thm:random-walk-weak}
with
$\sd = 1/n^{2}$ and $k = 1$. To obtain the bounds on $B$ and $N$, we use the inequality $1/\sprime = \max(n^2/\sqrt{2}, \eta_1(\lograyunits^\vee)) \leq \stupidfactor \cdot n^2$, which we will verify at the end of this proof.

One then gets the bound on $B$ by applying \Cref{thm:random-walk-weak} and simply moving the $n^2(\log(1/\sprime))^2$ into the polylogarithmic factors. For the lower bound on $N$, note that $\tfrac{1}{2\log(2)} \leq 1$ and%
\footnote{Using the bound $1/\sprime \leq \stupidfactor \cdot n^2$, we have $\frac{ \dimh \cdot \log(1/\sprime)}{2 \log n} \leq \frac{\dimh \cdot \log(\stupidfactor \cdot n^2)}{2 \log n} = \frac{\dimh \cdot [\log(\stupidfactor) + 2\log n]}{2 \log n} \leq \dimh \cdot (\frac{ \log(\stupidfactor)}{2 \log n} + 1) \leq (\frac{ \log(\stupidfactor)}{2 \log 2} + 1) n \leq 7n$.}
$\dimh \cdot \log(1/\sprime)/(2 \log n) \leq 7 n$. Hence, a sufficient lower bound on $N$ is the one in \Cref{thm:random-walk-weak} with $\frac{1}{2k \log n}$ removed and $\dimh \cdot \log(1/\sprime)$ replaced by $7n$.

As promised, we finish the proof by showing $1/\sprime \leq  \stupidfactor \cdot n^2$.
As $\lograyunits \subseteq \logunits$, we have $\lograyunits^\vee \supseteq \logunits^\vee$. Therefore, the smoothing parameter of $\lograyunits^\vee$ satisfies
\[ \eta_1(\lograyunits^\vee) \leq \eta_1(\logunits^\vee) \leq \frac{\sqrt{\dimh}}{\lambda_1(\logunits)}
\leq \kesslerconstant \cdot n \cdot \log(n)^3
\leq  \stupidfactor \cdot n^2.  \]
Here, the first inequality follows from the fact that $\eta_1(\Lambda) \leq \eta_1(\Lambda')$ if $\Lambda \supseteq \Lambda'$. The second inequality holds for general lattices \cite[Lemma 3.2]{MicciancioRegev2007}, the third inequality holds by the fact that $1/\lambda_1(\logunits) \leq \kesslerformulainverse$ (see \Cref{le:lower-bound-first-minimum-log-unit}) 
and $\dimh \leq n$,
 and the last inequality by
 $x\log(x)^3 \leq 1.4 \cdot x^2$ for all $x > 0$.
\end{proof}

%% file: ideal_sampling/A04-correspondence.tex
\section{Average densities and sampling} \label{section:densitiessampling}

\label{sec:correspondence}
\subsection{Result}
Let $\idset \subseteq \ideals$ be a set of integral ideals, and assume that $\ba$ is an Arakelov ray divisor whose
class is uniformly distributed. In this section, we show that one can sample elements $\beta$ such that $\beta \ba^{-1}$ corresponds to an ideal of the family $\idset$ with the probability one would naturally expect, i.e., proportional to the \emph{density} of $\idset$.
This is made precise in \Cref{prop:elementdensity} via the notion of \emph{local density}.
Recall that the local density of a set of ideals $\idset$ is defined (\Cref{def:idealdensity}) as
\[ \delta_\idset[x] = \min_{t \in [x/e^n,x]} \frac{\countfun{\idset}{t}}{\dederesidue \cdot t} = \min_{t \in [x/e^n,x]} \frac{ |\{ \mb \in \idset ~| ~ \norm(\mb) \leq t \}| }{\dederesidue \cdot t} ,\]
where, $\dederesidue = \lim_{s \rightarrow 1} (s-1)\zeta_K(s)$ (see \Cref{eq:classnumberformula}).

Recall that for $r \in \R_{>0}$, the box of radius $r$ in $\nfr$ is
$r\Ballinf = \{ (x_\emb)_\emb \in \nfr ~|~  |x_\emb| \leq r \}$.

\begin{theorem} \label{prop:elementdensity} 
Let $G \subseteq \rayPic^0$ be a finite-index subgroup of the Arakelov ray class group, and let $[\bb] \in \rayPic^0$ arbitrary, and
let $\distr$ be a distribution on $\rayDiv^0$ such that $[\distr]$ is uniform on the coset $\subpic + [\bb]$.
Let $r \geq  8 \cdot n^2 \cdot \Gamma_K \cdot |\dcrk|^{\frac{1}{2n}} \cdot \norm(\modu)^{1/n}$, where %
$\Gamma_K \leq \lambda^\infty_n(\OK) \leq |\dcrk|^{1/n}$ is defined in \Cref{eq:gammadef}, let $\tau \in \Kmodumodu$ and
let $\idsetmoduG$ be a set of integral ideals coprime to $\moduz$ for which holds $[d^0(\idsetmoduG)] \subseteq G + [\ldb \tau \rdb] - [\bb]$, %
with local density $\delta_{\idsetmoduG}[r^n]$ at $r^n$.
Then
\begin{align} \label{eq:elementdensity}  
\underset{\ba \from \distr}{\mathbb{E}} \left[ \underset{\alpha \from \dExp{\ba} \cap r \Ballinf}{\Pr} \left[  (\alpha) \cdot \dExp{-\ba} \in \idsetmoduG ~\Big \lvert~ \alpha \cdot \infpart{\Exp}(-\infpart{\ba}) \in \tau \Kmodu \right] \right] 
\\  \geq \frac{\norm(\moduz)}{\phi(\moduz)} \cdot  \frac{[\rayPic^0:\subpic]}{3} \cdot \delta_{\idsetmoduG}[r^n] \label{eq:elementdensity2}  \\
 \geq  \frac{[\rayPic^0:\subpic]}{3} \cdot \delta_{\idsetmoduG}[r^n] \label{eq:elementdensity3}
\end{align}
where $\alpha \from \dExp{\ba} \cap r \Ballinf$ is uniformly sampled from the finite set $\dExp{\ba} \cap r \Ballinf$ and $\phi(\moduz) = \defphi \leq \norm(\moduz)$.
\end{theorem}
\begin{remark} The factor $1/3$ in \Cref{eq:elementdensity2,eq:elementdensity3} can be made arbitrarily close to $1$ by increasing the radius $r \in \R$ and widening the density interval $[x/e^n,x]$ in \Cref{def:idealdensity}.
\end{remark}

\begin{remark} It is possible, with essentially the same proof, to rephrase this theorem in such a way that it concerns an \emph{intersection} of events instead of a conditional probability.
The probability then also depends on the number 
$\norm(\modu) = \norm(\moduz) \cdot 2^{|\modur|}$, 
where $|\modur|$ is the number of real places dividing $\moduinf$.
Under the same conditions as in \Cref{prop:elementdensity},  one can prove that
 \begin{align} \label{eq:elementdensityintersection} %
\underset{\substack{ \ba \from \distr \\ \alpha \from \dExp{\ba} \cap r \Ballinf } }{\Pr} & \left[ 
\begin{array}{c}
(\alpha) \cdot \dExp{-\ba} \in \idsetmoduG \mbox{, and }\\
 \alpha \cdot \infpart{\Exp}(-\infpart{\ba}) \in \tau \Kmodu 
\end{array} \right] 
 \geq \frac{\norm(\moduz)}{\phi(\moduz)} \cdot   \frac{[\rayPic^0:\subpic]}{3 \cdot \norm(\modu)} \cdot \delta_{\idsetmoduG}[r^n] . 
 \end{align}
\end{remark}

\subsection{Proof of \texorpdfstring{\Cref{prop:elementdensity}}{theorem}}
\subsubsection{%
Fixing the ideal 
\texorpdfstring{$\mc \in \idsetmoduG$}{c in S}
and the Arakelov divisor 
\texorpdfstring{$\ba \in \rayDiv^0$}{a in DivKm0}
}
We concentrate on the `inner probability' of \Cref{eq:elementdensity} in \Cref{prop:elementdensity}
in the case where $\idsetmoduG = \{ \mc \}$ consists of a \emph{single} integral ideal. 
We denote
\begin{equation} \label{eq:fixeddivisorprobability} \idprob_{\ba,\mc} =  \underset{\alpha \from \dExp{\ba} \cap r \Ballinf}{\Pr} \left[  (\alpha) \cdot \dExp{-\ba} = \mc ~\Big \lvert~ \alpha \cdot \infpart{\Exp}(-\infpart{\ba}) \in \tau \Kmodu \right], \end{equation}
where we leave the dependency on $r \in \R_{>0}$, the modulus $\modu$ and $\tau \in \Kmodumodu/\Kmodu$ implicit.
By the law of conditional probability, we have that $\idprob_{\ba,\mc}$ in \Cref{eq:fixeddivisorprobability} equals
\begin{equation} \label{eq:lawconditional} \scalebox{0.9}{$\underset{ \alpha \from \dExp{\ba} \cap r \Ballinf }{\Pr} \left [  \begin{matrix} (\alpha) \cdot \dExp{-\ba} = \mc \\ \mbox{ and } \\
\alpha \cdot \infpart{\Exp}(-\infpart{\ba}) \in \tau \Kmodu \end{matrix}  \right ] \Bigg/
\underset{ \alpha \from \dExp{\ba} \cap r \Ballinf }{\Pr} \left [  \alpha \cdot \infpart{\Exp}(-\infpart{\ba}) \in \tau \Kmodu  \right] $}  \end{equation}
The following lemma addresses the probability values of the numerator and denominator in \Cref{eq:lawconditional} separately.
\begin{lemma} \label{lemma:fixeddivisorprob} Let $\modu = \moduz \moduinf$ be a modulus, let $\tau \in \Kmodumodu/\Kmodu$, let $\ba \in \rayDiv^0$ be a fixed Arakelov ray divisor, and let $\mc \in \idealsmodu$ be an integral ideal. Then 
 \begin{equation} \label{eq:fixeddivisorprobabilityinlemma} \underset{ \alpha \from \dExp{\ba} \cap r \Ballinf }{\Pr} \left [  \begin{matrix} (\alpha) \cdot \dExp{-\ba} = \mc \\ \mbox{ and } \\
\alpha \cdot \infpart{\Exp}(-\infpart{\ba}) \in \tau \Kmodu \end{matrix}  \right ] = \frac{|\raygen{\ba + d(\mc)}\cap r\Ballinf|}{|\dExp{\ba} \cap r \Ballinf|}, \end{equation}
 and, there exists some $\tilde{\tau} \in \nfr$ such that 
  \begin{equation} \label{eq:fixeddivisorprobabilityinlemma2} \underset{ \alpha \from \dExp{\ba} \cap r \Ballinf }{\Pr} \left [ \alpha \cdot \infpart{\Exp}(-\infpart{\ba}) \in \tau \Kmodu \right ] = \frac{|(\dExp{\ba}\moduz  + \tilde{\tau} )\cap r\Ballinf \cap \taunfrm|}{|\dExp{\ba} \cap r \Ballinf|} , \end{equation}
 where the sampling $\alpha \from \dExp{\ba} \cap r \Ballinf$ is uniform in both expressions, and where
 $\taunfrm = \{  (x_\emb)_\emb \in \nfr ~|~ x_\emb/\emb(\tau) > 0 \mbox{ for real } \emb \mid \moduinf \}$.
\end{lemma}
\begin{proof} By examining \Cref{def:raygenerators} closely, noting that $\finpart{\Exp}(\finpart{(\ba + d(\mc))})$ $= \finpart{\Exp}(\finpart{\ba}) \cdot \mc \in \idealsmodu$, 
we see that for all $\alpha \in \dExp{\ba}$, 
\[ (\alpha) \cdot \dExp{-\ba} = \mc \mbox{ and  } \alpha \cdot \infpart{\Exp}(-\infpart{\ba}) \in \tau \Kmodu ~~\Longleftrightarrow ~~ \alpha \in \raygen{\ba + d(\mc)}. \]
As the number of choices for $\alpha \in \dExp{\ba} \cap r\Ballinf$ equals $|\dExp{\ba} \cap r \Ballinf|$, the number of good choices equals $|\raygen{\ba + d(\mc)}\cap r\Ballinf|$ and since the sampling procedure is uniform, we arrive at the first probability claim.
For the second probability claim, write $\ma = \finpart{\Exp}(\finpart{\ba}) \in \idealsmodu$, for conciseness. We note that for $\alpha \in \dExp{\ba}$, 
$\alpha \cdot \infpart{\Exp}(-\infpart{\ba}) \in \tau \Kmodu$ is equivalent to 
\[ \alpha \cdot \infpart{\Exp}(-\infpart{\ba}) \in \finpart{\Exp}(\finpart{\ba}) \cap \tau \Kmodu  = \ma \cap \tau \Kmodu = (\ma\moduz + \tau' ) \cap \taunfrm , \]
where $\tau' \in \ma$ is such that $\tau' \equiv \tau$ modulo $\moduz$ (note that $\ma$ and $\moduz$ are coprime).

So, for $\alpha \in \dExp{\ba} \cap r \Ballinf$, the statement 
$\alpha \cdot \infpart{\Exp}(-\infpart{\ba}) \in \tau \Kmodu$ is equivalent to 
\[ \alpha \in \infpart{\Exp}(\infpart{\ba})\big( (\ma\moduz + \tau')\cap \taunfrm \big) \cap \Ballinf = ( \dExp{\ba}\moduz + \tilde{\tau}) \cap \Ballinf \cap \taunfrm \]
where we use the fact that $\dExp{\ba} \taunfrm = \taunfrm$ and where we put $\tilde{\tau} = \infpart{\Exp}(\infpart{\ba}) \tau' \in \nfr$. This proves the claim.
\end{proof}

By combining \Cref{eq:lawconditional,eq:fixeddivisorprobabilityinlemma,eq:fixeddivisorprobabilityinlemma2} and 
scratching redundant factors, one concludes that there exists $\tilde{\tau} \in \dExp{\ba}$ such that
\begin{align} %
\idprob_{\ba,\mc} =  \frac{|\raygen{\ba + d(\mc)}\cap r\Ballinf|}{|(\dExp{\ba}\moduz  + \tilde{\tau} )\cap r\Ballinf \cap \taunfrm|}. \label{eq:combinedcombined22}
\end{align}

\subsubsection{Estimating 
\texorpdfstring{$|(\dExp{\ba}\moduz + \tilde{\tau}) \cap r \Ballinfmodu \cap \taunfrm|$}{|(Exp(a)m0 + tau) cap rBinf cap tauKrm|}
}  \label{subsubsec:estimateidealcapr}
When the radius $r$ is sufficiently large compared to the lattice $\dExp{\ba}\moduz \subseteq \nfr$, 
one can deduce that for $\ba \in \rayDiv^0$ the number of points in $(\dExp{\ba}\moduz+ \tilde{\tau}) \cap r \Ballinf \cap \taunfrm$
is approximately $\vol(r \Ballinf) \cdot 2^{-|\modur|} /\det(\moduz)$, where the $2^{-|\modur|}$ accounts for the fact that
we only take the halves of the box for which $x_\emb/\emb(\tau) > 0$ at the real embeddings $\emb \mid \moduinf$. %
More precisely, we instantiate \Cref{lemma:divisorcapinfinityball} with 
\begin{itemize}
 \item $\Lambda = \dExp{\ba + d(\moduz)} \subseteq \nfr$, for which holds $\covol(\Lambda) = \norm(\moduz) \cdot \sqrt{|\disc|}$,
 \item $t = \tilde{\tau} \in \nfr$,
\item $(t')_\emb = \mathrm{sgn}(\emb(\tau)) \cdot 1/2$ for real $\pl \mid \moduinf$, and $(t')_\emb = 0$ otherwise. Here $\mathrm{sgn}(\emb(\tau))$ is the sign of $\emb(\tau)$.
 \item $X = \{ (x_\emb)_\emb \in \nfr ~|~  |x_\emb| \leq 1 \mbox{ and } |x_\emb| \leq \tfrac{1}{2} \mbox{ if $\emb$ is real and } \emb \mid \moduinf \}$, such that $X + t' = \Ballinf \cap \taunfrm$. Note that  $\vol(X) = 2^{\rem - |\modur|} \cdot (2\pi)^{\cem}$, due to the hybrid complex-real nature of $\nfrm$ and taking account for the required positivity (or negativity, depending on $\tau$) at the real places $\nu \mid \moduinf$.
 \item $c =  n \cdot \norm(\moduz)^{\frac{1}{n}}  \cdot \Gamma_K \cdot |\dcrk|^{\frac{1}{2n}} \geq 2 \cdot \covtwo(\dExp{\ba + d(\moduz)})$ (see \Cref{lemma:idlatfacts}(i) and (iv)), so that the Voronoi cell $\voronoi$ of $\dExp{\ba + d(\moduz)}$ lies in $cX$.
\end{itemize}
This yields, for
 $r > 8 \cdot n^2 \cdot \norm(\moduz)^{\frac{1}{n}}  \cdot \Gamma_K \cdot |\dcrk|^{\frac{1}{2n}} \geq 8nc$,
\begin{equation} |(\dExp{\ba}\moduz  + \tilde{\tau} )\cap r\Ballinf \cap \taunfrm|  \in [e^{-1/4},e^{1/4}] \cdot \frac{r^n \cdot 2^{\rem - |\modur|} \cdot (2\pi)^{\cem} }{\norm(\moduz) \cdot \sqrt{|\dcrk|}}.  \label{eq:estimateidealcapr} \end{equation}
Applying this to the denominator of \Cref{eq:combinedcombined22}, we directly deduce that 
\begin{align} \label{eq:fixeddivisorprobability3} \idprob_{\ba,\mc} 
\in [e^{-1/4},e^{1/4}] \cdot  \frac{\sqrt{|\dcrk|} \cdot \norm(\moduz) \cdot 2^{|\modur|}}{r^n \cdot 2^{\rem} \cdot (2\pi)^{\cem} }  \cdot  |\raygen{\ba + d(\mc)}\cap r\Ballinf| 
\end{align}

\subsubsection{Estimating the probability of sampling a fixed ideal for a \underline{random} Arakelov divisor}
Still focusing on the simplified case where $\idset = \{ \mc \}$, the goal of this proof is to find a lower bound for $\underset{\ba \from \distr}{\mathbb{E}} [ \idprob_{\ba,\mc} ]$. By linearity of expectation, we have
\begin{align} \underset{\ba \from \distr }{\expval}\left [ \idprob_{\ba,\mc} \right ] 
\in [e^{-1/4},e^{1/4} ] \cdot  \frac{\sqrt{|\dcrk|} \cdot \norm(\moduz) \cdot 2^{|\modur|}}{r^n \cdot 2^{\rem} \cdot (2\pi)^{\cem} }  \cdot \underset{\ba \from \distr }{\expval} \left [ |\raygen{\ba + d(\mc)}\cap r\Ballinf| \right].
\label{eq:expvaluebox}
\end{align}
So it remains to focus on the expected value of $|\raygen{\ba + d(\mc)}\cap r\Ballinf|$ for $\ba \from \distr$.

 \subsubsection{The number \texorpdfstring{$|\raygen{\ba + d(\mc)}\cap r\Ballinf|$}{|Exp(a + d(c)) cap rB|} 
 only depends on the Arakelov ray \underline{class} of 
 \texorpdfstring{$\ba \in \rayDiv^0$}{a in DivKm0}
 }
A fact that plays a large role in the full proof, is that the number $|\raygen{\ba + d(\mc)} \cap r \Ballinf|$
of $\tau$-equivalent generators in a box depends on the Arakelov ray \underline{class} $[\ba]$ rather than the divisor $\ba$
itself. This has as a consequence that the involved probability distribution changes from $\distr$ to $[\distr] = \unif(\subpic)$, uniform on $\subpic \subseteq \rayPic^0$,
which is easier to analyze. This fact, among others, is proven in the following lemma.
\begin{lemma} \label{lemma:helplemma} For all ray divisors $\ba \in \rayDiv^0$, elements %
$\tau,\tau' \in \Kmodumodu$, ideals $\mc \in \idealsmodu$ and real numbers $r > 0$ we have the following list
of facts.
\begin{enumerate}[(i)]
\item \label{lemma:helplemmaii} $|\raygen{\ba} \cap r \Ballinf| = |\raygenx{\ba + \ldb \tau' \rdb}{\tau \tau'} \cap r \Ballinf|$, i.e., the number of $\tau$-equivalent ray generators of $\ba$ in a fixed box of radius $r$ is equal to the number of $\tau\tau'$-equivalent ray generators of $\ba + \ldb \tau' \rdb$ in the same box.
 \item \label{lemma:helplemmaiv}$|\raygen{\ba + d(\mc)} \cap r \Ballinf| = |\raygen{\ba + d^0(\mc)} \cap \frac{r}{\norm(\mc)^{1/n}} \Ballinf|$,
 since the maps $d^0$ and $d$ only differ by some scaling $\norm(\mc)^{1/n}$.
 \item \label{lemma:helplemmav}Writing %
 $\infpart{\ba} = \sum_{\nu} \ha_\nu \cdot \ldb \nu \rdb \in \rayDiv^0$, we have
 \begin{equation} \label{eq:logarithmicmaptosimplex} |\raygenone{\infpart{\ba}} \cap r \Ballinf| = |\muKmodu|\cdot |(\logunitsmodu + (a_{\pl})_\pl) ~\cap ~ S_{\log(r)}|,  \end{equation}
 where $S_{\log r} =  \{ (\hb_\pl)_\pl \in \bigoplus_{\nu} \R~|~ \hb_\pl \leq \npl \log(r)\,,\, \sum_\pl  \hb_\pl = 0 \} \subseteq \hyper$ is a simplex (as in \Cref{lemma:volume-simplex}) and $\logunitsmodu = \Log(\OKstar \cap \Kmodu)$. Here, $\npl = 2$ if $\pl$ is complex and $1$ otherwise.
\end{enumerate}
\end{lemma}
\begin{proof} 
For part (i), observe that multiplying by $\left (\frac{\emb(\tau')}{|\emb(\tau')|} \right)_\emb \in K_\R$ yields a bijection from $\dExp{\ba}$ to $\dExp{\ba + \ldb \tau' \rdb}$, preserving the maximum norm. It remains to show that this bijection sends $\raygenx{\ba}{\tau}$ to $\raygenx{\ba + \ldb \tau' \rdb}{\tau'\tau}$. Using \Cref{def:raygenerators} and assuming $\finpart{\Exp}(\finpart{\ba}) = \kappa \OK$ (and therefore $\finpart{\Exp}(\finpart{[\ba + \ldb \tau' \rdb]}) = \tau' \kappa \OK$), we have
\begin{align*} & \left (\frac{\emb(\tau')}{|\emb(\tau')|} \right)_\emb \cdot \raygenx{\ba}{\tau} 
=  \left (\frac{1}{|\emb(\tau')|} \right)_\emb \cdot (\tau') \cdot \underbrace{\infpart{\Exp}(\infpart{\ba}) \cdot ( \kappa \OKstar \cap \tau \Kmodu)}_{\raygenx{\ba}{\tau}} \\
= & \underbrace{\left (\frac{1}{|\emb(\tau')|} \right)_\emb \cdot \infpart{\Exp}(\infpart{\ba})}_{\infpart{\Exp}(\infpart{(\ba + \ldb \tau' \rdb)})} \cdot (\tau' \kappa \OKstar \cap \tau'\tau \Kmodu) = \raygenx{\ba + \ldb \tau' \rdb}{\tau'\tau} \end{align*}

For part (ii), recall that multiplying the ideal lattice $\dExp{d(\mc)} = \mc \subseteq K_\R$ 
by the scalar $\norm(\mc)^{-1/n}$ results in the ideal lattice $\dExp{d^0(\mc)}$. Applying this scalar multiplication to the set $\dExp{\ba + d(\mc)} \cap r \Ballinf$ yields a bijective correspondence with $\dExp{\ba + d^0(\mc)} \cap \frac{r}{\norm(\mc)^{1/n}} \Ballinf$. 

In part (iii) it is enough to show that the logarithm $\Log:\nfrstar \rightarrow \Log(\nfrstar)$ takes $\raygenone{\infpart{\ba}}$ to the shifted lattice $\Log(\OKmodu) + (\ha_{\pl})_\pl \subset H$ and takes $r \Ballinf \cap \nfr^0$ to the simplex $S_{\log(r)} \subset H$. %
This logarithmic map is $|\muKmodu|$-to-one
on $\raygenone{\infpart{\ba}}$, as it sends roots of unity to the all-zero vector in $\Log(\nfrstar)$ (which is the unit in that group), yielding the extra factor $|\muKmodu|$ in \Cref{eq:logarithmicmaptosimplex}. Here, $\muKmodu = \mu_K \cap \Kmodu$, i.e., the roots of unity in $\Kmodu$.
\end{proof}

As a corollary of \Cref{lemma:helplemma}(\itemref{lemma:helplemmaii}) 
we deduce that $|\raygen{\ba} \cap r \Ballinf| = |\raygenx{\ba + \ldb \kappa \rdb}{\tau} \cap r \Ballinf|$ for $\kappa \in \Kmodu$,
i.e., the number of elements $|\raygen{\ba} \cap r \Ballinf|$ only depends on the Arakelov ray \underline{class} of $\ba$ (next to $r \in \R$, $\modu$ and $\tau \in \Kmodumodu$). Choose a (measurable) fundamental domain $F \subseteq \rayDiv^0$ of the quotient group $\rayPic^0$. Put $F_{\subpic} = \{ \ba \in F ~|~ [\ba] \in \subpic \}$, which is a subdomain of $F$ for the subgroup $\subpic \subseteq \rayPic^0$ and likewise put $F_{\Tmodu} = \{ \ba \in F ~|~ [\ba] \in \Tmodu \}$, a fundamental domain of $\Tmodu$ in $\rayPic^0$. Note that $F_{\Tmodu} \subseteq F_{\subpic}$. By the assumption that $[\distr]$ is uniform on $\subpic + [\bb]$, and $[d^0(\mc)] \in [d^0(\idset)] \subseteq \subpic  + [ \ldb \tau \rdb]- [\bb]$, we deduce, writing $\tilde{r} = r \norm(\mc)^{-1/n}$,
 \begin{align}
  & ~~~~~\underset{\ba \from \distr }{\expval} \left [ |\raygen{\ba + d(\mc)}\cap r\Ballinf| \right] =  \underset{\ba \from \distr }{\expval} \left [ |\raygen{\ba + d^0(\mc)}\cap \tilde{r}   \Ballinf| \right] \\
  = & \underset{\ba \from \unif(F_{\subpic} + [\bb]) }{\expval} \left [ |\raygen{\ba + d^0(\mc)}\cap \tilde{r}\Ballinf| \right] = \underset{\ba \from \unif(F_{\subpic}) }{\expval} \left [ |\raygen{\ba + \ldb \tau \rdb}\cap \tilde{r}\Ballinf| \right] \\
  = & \underset{\ba \from \unif(F_{\subpic}) }{\expval} \left [ |\raygenx{\ba}{1} \cap \tilde{r}\Ballinf| \right] = \frac{1}{|\subpic/\Tmodu|} \underset{\ba \from \unif(F_{\Tmodu}) }{\expval} \left [ |\raygenx{\ba}{1} \cap \tilde{r}\Ballinf| \right] \label{eq:endclassexp}
 \end{align}
where the first equality follows from scaling (\Cref{lemma:helplemma}(\itemref{lemma:helplemmaiv})) and the second one by the fact that the random variable is an Arakelov ray class invariant (\Cref{lemma:helplemma}(\itemref{lemma:helplemmaii})) and that $[\distr]$ is uniform on $\subpic  + [\bb]$.
The third equality holds because the class $[\ba + d^0(\mc)]$ for $\ba \in \rayPic^0$ uniformly distributed over $F_\subpic + [\bb]$ is distributed as $[\ba' + \ldb \tau \rdb]$ for
$\ba' \from \unif(F_\subpic)$. This follows from the assumption $d^0(\mc) \in \subpic + [\ldb \tau \rdb] - [\bb]$.
The fourth equality
follows directly from \Cref{lemma:helplemma}(\itemref{lemma:helplemmaii}), and the last equality follows from
\Cref{def:raygenerators}. Namely, an Arakelov divisor $\ba$ can only have generators if the ideal class
of $\dExp{\finpart{\ba}}$ is trivial, i.e., if $[\ba] \in \Tmodu$. So, instead, $\ba$ can be chosen
uniformly from a fundamental domain $F_{\Tmodu}$ of $\Tmodu$ in $\rayDiv^0$, with a correction factor of $\tfrac{1}{|\subpic/\Tmodu|}$
in the expected value.

\subsubsection{Volume of the simplex in $H$}
For the next step in the proof, we need to know
the volume of the simplex $S_{\log r} \subseteq H$,
hence the following lemma.
\begin{lemma}
\label{lemma:volume-simplex}
The volume of the simplex $S_{\alpha} = \{ (\hb_\pl)_\pl \in H  ~|~ \hb_\pl \leq \npl \alpha \} \subseteq H= \Log K_\R^0$ for $\alpha \in \R_{>0} $ is given by
\[ \vol(S_{\alpha}) = \frac{\sqrt{\dimh+1}  \cdot (n\alpha)^\dimh}{\dimh!},\]
where $\dimh = n_\R + n_\C - 1$ and where $\npl = 2$ whenever $\pl$ is complex and $\npl = 1$ when $\pl$ is real.
\end{lemma}
\begin{proof}
By applying to $S_\alpha \subseteq H$ the translation $\hc_\nu = \npl \alpha - \hb_\nu$, one can see that $S_\alpha$ is a regular $\dimh$-simplex $\{ \hc \in \prod_{\nu} \R ~|~ \sum_\nu \hc_\nu = n \cdot \alpha \mbox{ and } \hc_\nu \geq 0\}$ with edge length $\sqrt{2} \cdot n \cdot \alpha$. Therefore, the volume of $S_\alpha$ equals $\vol(S_\alpha)  = \frac{(n\alpha)^\dimh \sqrt{\dimh+1}}{\dimh!}$ \cite{simplexvolume}.
\end{proof}

\subsubsection{Taking the logarithmic map into \texorpdfstring{$H = \Log \nfr^0$}{H = log KR}}

Applying the logarithmic map on the set $\raygenx{\ba}{1}\cap r\cdot \norm(\mc)^{-1/n} \Ballinf$, 
sends $\raygenx{\ba}{1}$ to a shift of the logarithmic ray unit lattice $\logunitsmodu \subseteq H$ and 
$r\cdot \norm(\mc)^{-1/n} \Ballinf$ to a simplex $S_{n \log r- \log \norm(\mc)}$, %
 where $S_{x} = \Log( x \Ballinf) \subseteq H =  \Log \nfr^0$ as in \Cref{lemma:volume-simplex} (see also \Cref{lemma:helplemma}(\itemref{lemma:helplemmav})). Note that $\vol(\Tmodu) = \covol(\logunitsmodu)$.

The expected value as in \Cref{eq:endclassexp} %
then equals  
the average number of points 
of a randomly shifted logarithmic ray unit lattice into this simplex,
which equals $\vol(S_{n \log r - \log \norm(\mc)})/\vol(\Tmodu)$ (see \Cref{fact:average_count}).
Therefore, 
 \begin{align}  & \frac{1}{|\subpic/\Tmodu|} \underset{\ba \from \unif(F_{\Tmodu})}{\expval} [ |\raygenx{\ba}{1} \cap r \norm(\mc)^{-1/n}\Ballinf|] 
 = \frac{|\mu_{\Kmodu}|\cdot \vol(S_{n \log r- \log \norm(\mc)})}{|\subpic/\Tmodu|\cdot |\Tmodu|} \\ = & \frac{[\rayPic^0:\subpic] \cdot |\mu_K|\cdot \vol(S_{n \log r- \log \norm(\mc)})}{\phi(\moduz) \cdot 2^{|\modur|} \cdot \classnumber \cdot \volabs{\torus}} =\frac{[\rayPic^0:\subpic] \cdot |\mu_K|\cdot \volsimN{\mc}}{\phi(\moduz) \cdot 2^{|\modur|} \cdot \classnumber \cdot \reg} \label{eq:fullexpvalue} \end{align}
 where, for the second equation, we use that $|\subpic/\Tmodu| \cdot |\Tmodu| = |\subpic| = |\rayPic^0|/[\rayPic^0:\subpic]$ and \Cref{eq:raypicvolumenolog} of \Cref{lemma:boundvolraypicappendix}. For the third equation we use the
 fact that $\frac{\vol(S_{n \log r - \log \norm(\mc)})}{\volabs{\torus}} = \frac{\volsimN{\mc}}{\reg}$,
 where we define $\volsimN{\mc} := (n \log r - \log \norm(\mc))^\dimh/\dimh!$ (see \Cref{lemma:volume-simplex} and \Cref{eq:unittorusvolume}).

 \subsubsection{Applying the Abel summation formula to get the probability for the ideal set $\idsetmoduG$}
By combining \Cref{eq:expvaluebox,eq:endclassexp,eq:fullexpvalue}, %
using the class number formula (see \Cref{eq:classnumberformula}) and the fact that $\frac{\norm(\moduz)}{\phi(\moduz)} = \frac{|\OK/\moduz|}{|(\OK/\moduz)^\times|} \geq 1$, one obtains, %
\begin{align}
  \underset{\ba \from \distr}{\mathbb{E}} \left[ \idprob_{\ba,\mc} \right] &  \geq e^{-1/4} \cdot \frac{\sqrt{|\dcrk|} \cdot \norm(\moduz) \cdot 2^{|\modur|}}{r^n \cdot 2^{\rem} \cdot (2\pi)^{\cem} } \cdot \frac{[\rayPic^0:\subpic] \cdot|\mu_K|\cdot \volsimN{\mc}}{\phi(\moduz) \cdot 2^{|\modur|} \cdot h_K \cdot\reg} \\ & = e^{-1/4} \cdot \frac{[\rayPic^0:\subpic] \cdot\volsimN{\mc}}{r^n \cdot \dederesidue} \cdot \frac{\norm(\moduz)}{\phi(\moduz)} 
\end{align}
where $|\modur|$ is the number of real places dividing $\moduinf$ and where $\volsimN{\mc} = (n \log r - \log \norm(\mc))^\dimh/\dimh!$.
By taking the sum over all $\mc \in \idsetmoduG$ (note that $[d^0(\idsetmoduG)] \subseteq G$), using linearity of the expected value operator, one can achieve the following lower bound.
\begin{align}
 \underset{\ba \from \distr}{\mathbb{E}} \left[ \sum_{\mc \in \idsetmoduG}  \idprob_{\ba,\mc} \right] 
  &= \sum_{\mc \in \idsetmoduG}  \underset{\ba \from \distr}{\mathbb{E}} \left[  \idprob_{\ba,\mc} \right ] \\
  & \in [e^{-1/4},e^{1/4}] \cdot \frac{\norm(\moduz)}{\phi(\moduz)}  \cdot [\rayPic^0:\subpic] \cdot \sum_{\mc \in \idsetmoduG} \frac{\volsimN{\mc}}{\dederesidue \cdot r^n}  & \label{eq:shortabel1}
  \end{align}
By an application of the Abel summation formula, one can relate 
the sum  $\sum_{\mc \in \idsetmoduG} C(r,\norm(\mc))$ to an integral involving the counting
function $\countfun{\idsetmoduG}{t} = |\{ \mc \in \idsetmoduG ~|~ \norm(\ma) \leq t \}|$ of the ideal set $\idsetmoduG$
and the derivative of the volume function $\volsim{N} = (n \log r - \log N)^\dimh/\dimh!$ with respect to the variable $N$. 
More precisely, we have 
\begin{align} \label{eq:shortabel2} \sum_{\mc \in \idsetmoduG} \frac{\volsimN{\mc}}{\dederesidue \cdot r^n}  = - &\int_{t = 1}^{r^n} \frac{\countfun{\idsetmoduG}{t}}{\dederesidue \cdot r^n} \cdot \left[ \tfrac{d}{dN} \volsim{N}\Big|_{N = t} \right]   dt \\
 = \frac{1}{(\dimh-1)!} &\int_{t = 1}^{r^n} \frac{\countfun{\idsetmoduG}{t}}{\dederesidue \cdot t} \cdot \left[ \tfrac{d}{dN} \Gamma\big(\dimh,n \log r - \log N\big) \Big|_{N=t} \right] dt, \label{eq:shortabel3}
\end{align}
Where the first equality is the Abel summation formula \cite[Theorem 4.2]{apostol1998introduction} and the second equality follows from computing
the derivative of the upper incomplete Gamma function $\Gamma(\dimh,x) = \int_{x}^\infty u^{\dimh - 1} e^{-u} du$.
\[  -\tfrac{d}{dN} \volsim{N}\Big|_t = \frac{(n \log r - \log t)^{\dimh - 1}}{t \cdot (\dimh-1)!}  = \frac{r^n}{t \cdot (\dimh -1)!} \cdot   \left[ \tfrac{d}{dN} \Gamma\big(\dimh,n \log r - \log N\big) \right]\!\!\Big|_t . \]
Using \Cref{def:idealdensity} about ideal density  and the fact that the integrand is positive, \Cref{eq:shortabel3} is lower bounded by
\begin{align} & \frac{1}{(\dimh-1)!} \int_{t = (r/e)^n}^{r^n} \frac{\countfun{\idsetmoduG}{t}}{\dederesidue \cdot t} \cdot \left[ \tfrac{d}{dN} \Gamma\big(\dimh,n \log r - \log N\big) \Big|_{N=t} \right] dt \\
\geq ~ & \frac{\delta_{\idsetmoduG}[r^n]}{(\dimh-1)!} \int_{t = (r/e)^n}^{r^n}  \left[ \tfrac{d}{dN} \Gamma\big(\dimh,n \log r - \log N\big) \Big|_{N=t} \right] dt \geq \tfrac{1}{2} \cdot \delta_{\idsetmoduG}[r^n] \label{eq:lowerboundabelthing}, \end{align}
where the last inequality (\Cref{eq:lowerboundabelthing}) follows from the definition of the upper 
incomplete Gamma function,
\begin{align*} \frac{1}{ (\dimh-1)!} \int_{t = (r/e)^n}^{r^n} \! \! \left( \frac{d}{dt} \Gamma(\dimh, n \log r - \log N)\big \lvert_{N = t} \right) dt & = \frac{1}{(\dimh-1)!} \cdot (\Gamma(\dimh,0) - \Gamma(\dimh,n)) \\ & = 1 - e^{-n} \sum_{k = 0}^{\dimh-1} \frac{n^k}{k!} \geq 1/2,\end{align*}
where we used the fact that $e^{-n} \sum_{k = 0}^{\dimh-1} \frac{n^k}{k!}$ equals the probability 
that a Poisson distribution with parameter $n$ yields at most $\dimh-1 \leq n-1$ occurrences, which is bounded by a half.

By combining \Cref{eq:shortabel1,eq:shortabel3,eq:lowerboundabelthing} and using $e^{-1/4}/2 > 1/3$, we obtain
\begin{align*}
\underset{\ba \from \distr}{\mathbb{E}} \left[ \underset{\alpha \from \dExp{\ba} \cap r \Ballinf}{\Pr}  \left [ (\alpha) \cdot \dExp{-\ba} \in \idsetmoduG ~\big \lvert~ \alpha \cdot \infpart{\Exp}(-\infpart{\ba}) \in \tau \Kmodu \right ] \right ]  \\ \geq \frac{\norm(\moduz)}{\phi(\moduz)} \cdot [\rayPic^0:\subpic] \cdot \delta_{\idsetmoduG}[r^n]/3.\end{align*}    
which finishes the proof.\qed

%% file: ideal_sampling/A05-samplinginabox.tex
\section{Uniform sampling in a box intersected with an ideal lattice} \label{section:samplingboxideallattice} \label{sec:sampling-box}
\subsection{Introduction}
In this section, we explain how one can efficiently sample in a (shifted) ideal lattice $\kx (\mb + \gamma)$ intersected with a (partially positive) box, provided that the dimensions of the box are sufficiently large.
More precisely, let $r > 0$ and let $r \Ballinf$ denote the box
\[ r \Ballinf := \{(\kx_\emb)_\emb \in K_\R\,|\, |\kx_\emb| \leq r,\, \forall \sigma\}.\]
Now, the box of interest is %
$r \Ballinf \cap \taunfrm$, whose coordinates at $\sigma$ have the same sign as $(\emb(\tau))_\emb$ for the real embeddings $\emb \mid \moduinf$. Our aim of this section is to perfectly uniform sample from the intersection $r \Ballinf \cap \taunfrm \cap \kx (\mb + \gamma)$.

The algorithm we use to sample uniformly in this box follows the framework described by Plançon and Prest in~\cite{PP21}, and is very similar to their instantiation of that framework \cite[Sec. 4.2]{PP21}. The main two differences are that~\cite{PP21} analyzes the running time of their algorithm only heuristically, whereas here we provide a provable running time; also,~\cite{PP21} assumes perfect uniform samples from real intervals, whereas here we only rely on perfectly uniform bits.

\subsubsection*{Technical idea of the algorithm}
The core idea of the algorithm applies to general lattices and relies on two observations. The first observation: for lattices $\Lambda$ satisfying $\Lambda \subseteq \frac{1}{N} \Z^n$ and sufficiently large, bounded convex sets $S \subseteq \R^n$, the task of uniformly sampling from $S \cap \Lambda$ reduces to uniformly sampling in $(1 + c)S \cap \frac{1}{N} \Z^n$ for some constant $c>0$, where we understand $(1+c) \cdot S := \{ (1+c)\cdot s ~|~ s \in S\}$. Indeed, we can use a `good' fundamental domain $F$ of $\Lambda$ to round a uniform sample of $(1 + c) \cdot S \cap \frac{1}{N} \Z^n$ to a point in $\Lambda$. Since $(F + \ell) \cap \frac{1}{N} \Z^n$ contains the same number of elements for all $\ell \in \Lambda \cap S$ (since $\Lambda \subseteq \frac{1}{N} \Z^n$), this yields a perfect uniform distribution over $S \cap \Lambda$, if $c>0$ is chosen adequately. This observation is formalized in \Cref{lemma:perfectsampleq}.

The second observation is that each lattice $\Lambda$ given by a basis $\vec{B}$ can be \emph{approximated} by a lattice $\tilde{\Lambda}$ given by a basis $\tilde{\vec{B}}$ that
satisfies $\tilde{\Lambda} \subseteq \frac{1}{N} \Z^n$. Hence, by the first observation,
one can perfectly uniformly sample from $\tilde{\Lambda} \cap S$ for suitable $S$. If the approximation $\tilde{\vec{B}} \approx \vec{B}$
is good enough, one can transform this perfectly uniform sample from $\tilde{\Lambda} \cap S$ to a perfectly uniform sample from $\Lambda \cap S$. This observation is formalized in \Cref{lemma:perfectsampler}.

Note that in the first observation, a `good' fundamental domain of the lattice is required.
For such a fundamental domain,
we need to know a sufficiently good basis of the ideal we want to sample from.
We obtain such a basis by using one of the provable variants of the BKZ algorithm~\cite{DBLP:journals/tcs/Schnorr87,HPS11,gama2008finding,ALNS20}. \Cref{lemma:shortvectorsinideallattice} captures the end result of using a provable variant of the BKZ algorithm; for an in-depth proof of this lemma, we refer to \Cref{section:latprelim,section:BKZinteger,section:BKZapproximate}.
\begin{remark} We note that this choice of basis reduction algorithm (BKZ) does not exploit the ideal structure of the lattice. There do exist algorithms that exploit the ideal structure~\cite{CDW21, PHS19} and might outperform BKZ in some context. However, these algorithms are heuristic, require a quantum computer and/or some exponential pre-processing on the number field and in the case of~\cite{CDW21} are restricted to cyclotomic number fields. Hence, we chose to use the BKZ algorithm here in order to obtain a non-heuristic algorithm. If one allows heuristics and quantum algorithms, it may be possible to diminish the block-size dependent quantity $\hkztime$ that appears in the running time of the basis reduction algorithm.
\end{remark}

\subsubsection*{Ray}
Actually, in this section, a slightly more general algorithm is described. In this more general algorithm, we sample elements from
an intersection of a (shifted) ideal lattice $\kx (\mb + \gamma)$
and a box $r\Ballinf$, that also fall in the $\tau$-\emph{ray} mod $\modu$, for some $\tau \in \Kmodumodu$.

That is, we will sample from $\kx( (\mb + \gamma) \cap \tau \Kmodu)$ intersected with $r\Ballinf$; those are the elements of the shape $x(\beta + \gamma)$ for which $\beta + \gamma \in \tau\Kmodu$ and $|x_\emb \emb(\beta + \gamma)|\leq r$ for all embeddings $\sigma$. So, for the algorithm including the modulus $\modu$, the output additionally satisfies certain modular conditions depending on the modulus $\modu$.

One retrieves the simpler algorithm (without a modulus) that samples in $\kx(\mb + \gamma) \cap r \Ballinf$ by just putting $\modu = \OK$ (so that $\tau\Kmodu = \Kmodu = \Kmodumodu = K^*$ for all $\tau \in K^*$).

\subsection{On efficiently deciding whether an algebraic number is greater or equal to some rational number}

\label{sec:liouville}
To be able to perfectly 
sample from an ideal lattice intersected with a box, 
one needs to be able to efficiently decide, for an algebraic number $\alpha \in K$
in symbolic representation, whether $|\emb(\alpha)| > r$ or $|\emb(\alpha)| \leq r$ 
(or, for real embeddings, whether $\emb(\alpha)>r$ or $\leq r$)
for some rational $r \in \Q$ and embedding $\emb$. Namely, the combined 
information over all places  
signifies whether the algebraic number $\alpha$ is inside some box $r \Ballinf$ or not.

The challenge in this question is that the algebraic number is given 
in some symbolic representation, like $\alpha = 408 \sqrt{2} - 577 (\approx  -0.0008665)$. 
In order to decide, for example, whether $\alpha > 0$ or $\leq 0$, 
we need to know \emph{how well} to approximate this number (in this particular case 
at least up to 3 decimals).

The result of this section, in a nutshell, is that for any algebraic number $\alpha \in K$ we can efficiently decide
whether $|\emb(\alpha)|$ (or $\emb(\alpha)$ in the case of a real place)
is larger than some rational $r \in \Q$ or not. This is proven in the following
sequence of lemmas. The next lemma is coined Liouville's lemma, since the technique comes from the tendency
of algebraic numbers to avoid rationals, which was originally used to show the explicit 
existence of transcendental numbers (Liouville numbers). 

\begin{lemma}[Liouville's lemma] \label{lemma:liouville} Let $f(x) \in \Z[x]$ be a degree $n$ polynomial. Let $\omega \in \R$ be a real root of $f$. Then, if $\omega \notin \Z$, we have, for all $z \in \Z$,
\[  | \omega - z | >  \Big [ (n+1)^3 \cdot 2^n \cdot (2 + \|f\|_\infty)^{n} \Big]^{-1} =: \mathfrak{A}(f). \]
\end{lemma}
\begin{proof} 
Let $h \mid f$ be an irreducible factor of $f$ and let $\omega$ be a real root of $h(x) \in \Z[x]$. If $h$ is linear, 
$\omega = \frac{c}{d} \in \Q$, and by the rational root theorem, $d$ is an integer factor of the leading coefficient of $h$ (and thus of that of $f$). Hence, if $\omega \notin \Z$, we have $|\omega - z| \geq 1/d \geq 1/\|f\|_\infty > \mathfrak{A}(f)$ for every $z \in \Z$.
 
Assume now that $h$ is non-linear. 
By the mean value theorem, 
we have, for any $t \in [\omega-1,\omega+1]$,
 \begin{equation} |h(t)| = |h(t) - h(\omega)| \leq  \max_{t_0 \in [\omega-1,\omega+1]} |h'(t_0)| \cdot |t - \omega| 
 \end{equation}
Additionally, for any \emph{integer} $z \in \Z$, we have $|h(z)| \geq 1$, since $h(x) \in \Z[x]$ is irreducible and of degree $>1$ (and hence cannot have integer roots). This immediately yields, for any \emph{integer} $z \in \Z$,
\begin{align} |z- \omega| &\geq  \frac{|h(z)|}{\max_{x_0 \in [\omega-1,\omega+1]} |h'(x_0)|}  \geq  \Big (\max_{x_0 \in [\omega-1,\omega+1]} |h'(x_0)| \Big)^{-1} \\
& \geq \Big (\max_{h \mid f} \max_{\omega \mbox{ \scriptsize{root of} } h} \max_{x_0 \in [\omega-1,\omega+1]} |h'(x_0)| \Big)^{-1} \label{eq:complexquantityofh}
 \end{align}
 where, in the last line, the first maximum is over all divisors of $f$ (also the possible linear ones) and the second maximum is over all real roots.
We now aim to find a easy-to-compute lower bound on \Cref{eq:complexquantityofh}, which we call $\mathfrak{A}(f)$, 
in terms of the 
coefficients of $f$ alone, using the Landau-Mignotte bound \cite[Theorem 6.32]{von1999modern} and Cauchy's bound \cite[Section 3.6]{conte80}.

By the triangle inequality, writing $h(x) = \sum_{j = 0}^{\deg(h)} h_j \cdot x^j$, $\|h\|_\infty  = \max_j |h_j|$ and $\|f\|_\infty = \max_j |f_j|$,  we can deduce
\begin{align*} \max_{x_0 \in [\omega-1,\omega+1]} |h'(x_0)|  &\leq \sum_{j = 1}^{\deg(h)} j \cdot |h_j| \cdot \max_{x_0 \in [\omega-1,\omega+1]} |x_0|^{j-1}
\\ & \leq \deg(h)^2 \cdot \|h\|_\infty \cdot \max_{x_0 \in [\omega-1,\omega+1]} |x_0|^{\deg(h)-1} 
\\ & \leq \deg(h)^2 \cdot \|h\|_\infty \cdot (|\omega|+1)^{\deg(h)-1} 
\\ & \leq n^2 \cdot(n + 1) \cdot 2^n \cdot \|f\|_\infty \cdot  (|\omega|+1)^{n-1}  ~~ \mbox{(Landau-Mignotte bound)}
\\ & \leq n^2\cdot(n + 1) \cdot 2^n \cdot \|f\|_\infty \cdot  (2+\|f\|_\infty)^{n-1} ~~ \mbox{(Cauchy's bound)}
\\ & < (n+1)^3 \cdot 2^n \cdot (2 + \|f\|_\infty)^{n} = \mathfrak{A}(f)^{-1}.
\end{align*}
where we recall that $n = \deg(f) \geq \deg(h)$.
Thus, for any $z \in \Z$ and any real root $\omega$ of $f$ satisfying $\omega \notin \Z$, we have
\begin{align} \label{eq:avoidance} |\omega - x| \geq \mathfrak{A}(f) = \left[  (n+1)^3 \cdot 2^{n} \cdot (2 + \|f\|_\infty)^{n} \right]^{-1}.
\end{align} 
\end{proof}

\begin{corollary} \label{cor:decisionefficient} There exists an algorithm that, given 
as input a degree $n$ polynomial $f(x) \in \Z[x]$ and some rational $r \in \Q$, 
decides whether the roots $\omega_i \in \C$ of $f(x)$ satisfy $|\omega_i| > r$ 
or $|\omega_i| \leq r$ for complex roots $\omega_i$, respectively, $\omega_i > r$ or $\omega_i \leq r$ for real roots $\omega_i$. The running time of the algorithm
is polynomial in the size of $f(x)$ and $r$.
\end{corollary}
\begin{proof} Denote $r = \frac{a}{b} \in \Q$. We replace the roots $\omega_i$ by $b \cdot \omega_i$, the polynomial $f(x)$ by $b^n \cdot f(b^{-1} \cdot x)$ and the rational $r$ by $a \in \Z$. Hence, we may without loss of generality assume that $r \in \Z$, with the expense of an increase in the size of the polynomial $f$ by $n \cdot \size(r)$. 

We start with the case of a real root $\omega$ of $f$, and we then proceed
with the case of a complex root.\\
\textbf{Real root case.} \\
According to \Cref{lemma:liouville}, we have for any real but non-integral root $\omega$ of $f(x)$ and any $z \in \Z$,  $|\omega - z| > \mathfrak{A}(f)$. This lower bound gives us means to decide whether a real root of $f$ is larger (or equal) or smaller than some integer. Namely, 
approximate the real roots of $f$ within precision $<\mathfrak{A}(f)/2 = \left[  (n+1)^3 \cdot (2 + \|f\|_\infty)^{n} \right]^{-1}/2$, which can be done in time $\poly(n, \log \|f\|_\infty) = \poly(n,\size(f),\size(r))$. Denote the approximations of these real roots by $\tilde{\omega}_i \approx \omega_i$. There are two cases:
\begin{itemize}
\item $|\tilde{\omega}_i - z| < \mathfrak{A}(f)/2$ for some integer $z \in \Z$. In that case, $\omega_i = z$. Indeed, if $\omega_i$ were not an integer, we must have 
\[ \mathfrak{A}(f) \leq |\omega_i - z | \leq |\omega_i - \tilde{\omega_i}| + |\tilde{\omega_i} - z| < \mathfrak{A}(f), \]
which is  a contradiction. Hence, $\omega_i$ is an integer and equal to $z \in \Z$. In such case, deciding whether $|\omega_i| > r$ or $\leq r$ is easily and efficiently done.
\item $|\tilde{\omega_i} - z| \geq \mathfrak{A}(f)/2$ for all $z \in \Z$. Then, by the triangle inequality, we have $\omega_i \notin \Z$, since, for all $z \in \Z$,
\[ |\omega_i - z | \geq \underbrace{|\tilde{\omega_i} - z|}_{\geq \mathfrak{A}(f)/2} - \underbrace{|\omega_i - \tilde{\omega}_i|}_{< \mathfrak{A}(f)/2} >0.   \] 
Hence, by \Cref{lemma:liouville}, $|\omega_i - z| \geq \mathfrak{A}(f)$ for any integer $z \in \Z$. We claim that $\tilde{\omega_i} > r$ if and only if $\omega_i > r$. That is, by using the approximation $\tilde{\omega_i}$, we can decide whether $\omega_i > r$ or $\leq r$. 

Indeed, suppose (as to achieve a contradiction) that $\omega_i \leq r$ but $\tilde{\omega_i} > r$. Then, since $|\omega_i - r| \geq \mathfrak{A}(f)$, and thus $r - \omega_i \geq \mathfrak{A}(f)$.  
\[ \underbrace{r-\omega_i}_{\geq \tilde{\mathfrak{A}}(f)}  = \underbrace{(r-\tilde{\omega}_i)}_{\leq 0} +  \underbrace{(\omega_i - \tilde{\omega}_i)}_{\in [-\mathfrak{A}(f)/2,\mathfrak{A}(f)/2]} \leq  \mathfrak{A}(f)/2 \]
which is a contradiction. The other case ($\omega_i > r$ but $\tilde{\omega_i} \leq r$) can be excluded similarly.
\end{itemize}
\textbf{Complex root case.} \\
For the complex root case, we take the companion matrix $M_f$ of $f = \sum_{j = 0}^n f_j x^j$. This $n \times n$ matrix $M_f$ is defined by $(M_f)_{j+1,j} = 1$ for $j \in \{1,\ldots,n-1\}$ and $(M_f)_{j,n} = -f_{j-1}/f_n$ for $j \in \{1,\ldots,n\}$ (this is the companion matrix of the monic polynomial $f_n^{-1} \cdot f(x)$). It is known that $M_f$ has the roots $\omega_i$ of $f$ as eigen values. Therefore, the matrix resulting from the Kronecker product $M_f \otimes M_f$ has all products $\omega_i \omega_j$ ($i,j \in \{1,\ldots,n\}$) as eigen values, in particular the square absolute values $|\omega_i|^2$ for all original roots of $f$.

Computing the characteristic polynomial of $M_f \otimes M_f$, and clearing denominators of that polynomials yields the $n^2$-degree integer polynomial $g(x) \in \Z[x]$. Identifying the roots of $g$ that are the square absolute values of those of $f$ (by approximating the roots of both), one can apply the same technique as in the real root case to decide whether $|\omega_i|^2 > r^2$ or not. 

It remains to be shown that the polynomial $g$ is polynomially bounded in size by the size of $f$, so that all operations can be done in polynomial time in $\size(f)$ and $\size(r)$. The $j$-th coefficient $p_j$ of the characteristic polynomial $p(x)$ of $M_f \otimes M_f$ can be computed (up to sign) by taking the sum of all $\binom{n^2}{j}$ principal minors of $M_f \otimes M_f$ of size $j$ \cite[\textsection 1.2, Equation (1.2.13)]{horn2012matrix}. By Hadamard's bound, each of these principal minors has size at most $\|f\|_\infty^{2j}$ and can only have a denominator that is a divisor of $f_n^{2j}$. Hence, $f_n^{2n} \cdot p(x)$ must be a multiple of $g(x)$. This means that $\|g\|_\infty \leq \|p\|_\infty \leq (n^2)! \|f\|_\infty^{2n}$. Hence, $\size(g) \leq \poly(n, \size(f))$, which was to be demonstrated. 
\end{proof}

\begin{corollary} \label{cor:boxdecide} There exist an algorithm
that on input a degree $n$ number field $K = \Q(\theta)$, an ideal $\mb$, an element $x = (x_\emb)_\emb \in \nfrstar$ represented by rational numbers, an element $\alpha \in \mb$ represented by the $\Q$-basis of $K$ (the power-basis formed by $\theta$), a rational $g \in \Q$ and an integer $k \in \Z_{>0}$, decides whether $x_\emb |\emb(\alpha)| > g^{1/k}$ or $x_\emb |\emb(\alpha)| \leq g^{1/k}$, for every complex embedding $\emb$; and decides whether $x_\emb \emb(\alpha) > g^{1/k}$ or $x_\emb \emb(\alpha) \leq g^{1/k}$, for every real embedding $\emb$.
Moreover, this algorithm runs in time $\poly(\size(\alpha),\size(x),\size(\mb),\size(g),n,\log|\dcrk|,k)$.
\end{corollary}
\begin{proof} Throughout this paper we assume that the polynomial $f$ defining the number field $K$ satisfies $\size(f) \leq \poly(\log |\dcrk|)$ (see \Cref{sec:representation}). As $x_\emb |\emb(\alpha)| > g^{1/k}$ is equivalent to $|\emb(\alpha^k)| > g/x_\emb^k$ (with a possible sign change), we can reduce this result to that of \Cref{cor:decisionefficient}.

The sole challenge is to find and bound the polynomial of which $\alpha^k$ is a root. For this we start by bounding the polynomial for which $\alpha \in \mb$ is a root. Write down the multiplication matrix of $\alpha$ in the $\theta$-basis. That is, the matrix defined by the coefficients $q_{ij}$ occurring in the identities $\alpha \theta^j = \sum_{i =0}^{n-1} q_{ij} \theta^i$ for $j \in \{0,\ldots,n-1\}$. Then the characteristic polynomial of this matrix is one of which $\alpha \in \mb$ must be a root. 

The size of the multiplication matrix of $\alpha$ is bounded by $\poly(\size(\alpha), \size(f)) = \poly(\size(\alpha),\log |\dcrk|)$. Hence, the size of the characteristic polynomial (whose $j$-th coefficient can be shown to be the sum of all principal minors of size $j$ of the multiplication matrix, which can be bounded by Hadamard's bound) must be polynomially bounded in $\size(\alpha)$, $\log|\dcrk|$ and $n$ as well. By a very similar reasoning, the size of the characteristic polynomial of $\alpha^k$ must be $\poly(n,\log|\dcrk|,\size(\alpha),k)$. By applying \Cref{cor:decisionefficient}, we obtain the final result.
\end{proof}

\subsection{Lattice reduction} \label{subsec:latticereduction}
Being able to successfully sample in an ideal lattice intersected with a box largely depends
on the maximum vector length (i.e., the quality) of the basis of the ideal lattice compared to the dimensions of the box. If these dimensions of the box are somewhat larger than the basis vectors, this sampling can be done efficiently. For the applications and algorithms of the present work, we generally require the dimensions of this box to be as small as possible, since this benefits their complexity. As a consequence, we would like to obtain a basis of the ideal lattice with as small as possible maximum vector length (i.e., best basis quality), but without paying too much time. 

The standard algorithm for finding lattice bases of good quality is the Block Korkine Zolotarev (BKZ) algorithm, which allows for a controllable trade-off between the output basis quality and the heuristic run time. Unfortunately, no published versions of the BKZ-algorithm exist that have a \emph{provable} run-time, which is what we require in the present work (as our end goal is an algorithm with a provable run-time). Hence,  \Cref{section:latprelim,section:BKZinteger,section:BKZapproximate} of this paper is devoted to showing that there exists a variant of the BKZ-algorithm with the same trade-off between output basis quality and run time as the textbook version, but in which the run time is actually proven and non-heuristic. 

The proof of this result requires both careful numerical analysis (for the ideal lattice case) and precise monitoring of the sizes (`potential') of the bases occurring intermediately in the algorithm. To stay relevant with the current topic, this proof is therefore postponed to \Cref{part2b}. The overall result can be summarized in the following lemma.

\begin{restatable}{lemma}{approxbkz}
\label{lemma:shortvectorsinideallattice}
Let $K$ be a number field of degree $n$. Let $\kx \ma$ be an ideal lattice where $\kx \in \nfrstar$ is represented by rational numbers, and where
$\ma \in \ideals$ is represented by a rational Hermite Normal Form matrix $M_\ma$ with respect to a given LLL-reduced basis of $\OK$.

Then there exists an algorithm that computes
a $\Z$-basis $(\kx \alpha_1,\ldots,\kx \alpha_n)$ of $\kx \ma$ with $\alpha_i \in K$ such that
\[ \|  \kx \cdot \alpha_i \| \leq 2n \cdot \blocksize^{2n/\blocksize} \cdot \lambda_n(\kx \cdot \ma) ,\]
using time at most $T = \poly(\hkztime, \size(\ma), \log |\dcrk|, \size(\kx)) $.
\end{restatable}
\begin{proof} The proof can be found in \Cref{subsec:BKZ:conclusion}.
\end{proof}
\subsection{Perfectly uniform sampling in a bounded convex set intersected with a lattice}
\label{section:perfectlyuniform}

\begin{notation} We denote $B_{r} = \{ \vec{x} \in \R^n ~|~ \|x\| < r \}$. For any bounded convex set $S \subseteq \R^n$ and $c \in \R_{>0}$, we denote $cS := \{ c \cdot s ~|~ s \in S \}$. For a basis $\mB = (\bb_1,\ldots,\bb_n)$ of $\R^n$, we denote $\lat(\mB) = \mB \Z^n := \{ \sum_{i = 1}^n \bb_i z_i ~|~ z_i \in \Z \mbox{ for all $i$} \}$ (which is a lattice).
\end{notation}

\begin{lemma} \label{lemma:perfectsampleq} Let $\mB = (\bb_1,\ldots,\bb_n)$ with $\bb_i \in \frac{1}{N} \Z^n$ be a basis, let $D = \sum_i \|\bb_i\|$, let $U,c,\eps \in \R_{>0}$ such that $c >\eps > 0$. Suppose $S$ is a convex set satisfying $B_{D/\eps} \subseteq S \subseteq B_{U}$ for which we can perfectly uniformly sample in $(c + \eps) S \cap \frac{1}{N} \Z^n$ within expected time $T$. Additionally, suppose we can efficiently decide membership in $cS$.

Then we can perfectly uniformly sample in $cS \cap \mB \Z^n$ within time $\poly(T,\size(\mB),\allowbreak\log U, \allowbreak \log N)$ with success probability lower bounded by
\begin{equation} \frac{|(c-\eps)S \cap \frac{1}{N} \Z^n|}{|(c+\eps)S \cap \frac{1}{N} \Z^n|} \label{eq:lowerboundproba} \end{equation}
\end{lemma}
\begin{proof} (\textbf{Algorithm description}) Sample a perfectly uniform $u \in (c + \eps) S \cap \frac{1}{N} \Z^n$. Compute $v = \mB^{-1}u$ (which can be done in time $\poly(T,\size(\mB),\log U)$) and put $w_i = \lfloor v_i \rceil$, rounding to the nearest integer. Output $\mB w$ if $\mB w \in cS$, otherwise output `failure'.

(\textbf{Analysis})
We will prove that this procedure yields a perfectly uniform sample in $\mB \Z^n \cap cS$ and that the
success probability of a single iteration is lower bounded by the quantity in \Cref{eq:lowerboundproba}.

For a fixed $\bb \in (\mB \Z^n) \cap cS$, the probability that $\bb$ is outputted is proportional to the number $|(\bb + \mB \cdot [-\tfrac{1}{2},\tfrac{1}{2})^n) \cap (c + \eps)S \cap \frac{1}{N}\Z^n|$. We have $\bb \in cS$; and each $\mathbf{x} \in \mB \cdot [-\tfrac{1}{2},\tfrac{1}{2})^n$ satisfies $\|\mathbf{x}\|\leq D/2 \leq D$, hence $\mathbf{x} \in B_{D} \subseteq \eps \cdot S$ (since $B_{D/\eps} \subseteq S$). So, we obtain that $\bb + \mB \cdot [-\tfrac{1}{2},\tfrac{1}{2})^n \subseteq (c + \eps)S$ by the convexity of $S$. Therefore
\begin{align*} |(\bb + \mB \cdot [-\tfrac{1}{2},\tfrac{1}{2})^n) \cap (c + \eps)S \cap \frac{1}{N} \Z^n| & = |(\bb + \mB \cdot [-\tfrac{1}{2},\tfrac{1}{2})^n) \cap \frac{1}{N} \Z^n |
\\ & = N^n \det(\mB),   \end{align*}
for all $\bb \in (\mB \Z^n) \cap S$. Note that this relies critically on $\frac{1}{N} \Z^n \subseteq \mB \Z^n$. So the sampling probability is the same for all $\bb \in (\mB \Z^n) \cap cS$, which is the property of a perfect uniform distribution. It remains to show that the probability of success (a sample in $\mB \Z^n \cap cS$) is lower bounded.

If the initial uniform sample $u$ satisfies $u \in (c-\eps)S \cap \frac{1}{N} \Z^n$, then $\mB w = \mB \lfloor \mB^{-1} u \rceil =  u + y$ for some $y \in \mB [-\tfrac{1}{2},\tfrac{1}{2})^n \subseteq \eps S$ (where we mean with $\lfloor \cdot \rceil$ that every coordinate is rounded to the nearest integer). For such $u$ holds that $u+ y$ must lie in $cS$. Therefore, the success probability is lower bounded by
$\frac{|(c-\eps)S \cap \frac{1}{N} \Z^n|}{|(c+\eps)S \cap \frac{1}{N} \Z^n|}$.
\end{proof}

In our context, the basis $\mB_{\mb}$ of an ideal lattice $\mb$ (as embedded in the Minkowski space) cannot be represented by rational numbers, but is rather symbolically represented. In the following proposition we show that this is not a real complication; namely, it is enough to be able to compute a sufficiently good rational approximation $\mC \approx \mB_{\mb}$ in order to be able to perfectly uniformly sample in this lattice $\mb$ intersected with a certain convex set $S$. 

As a consequence, the algorithm of the following proposition takes as an input a sufficiently good approximation $\mC$ of $\mB$ and uses only this rational matrix $\mC$ to compute with. To be able to check whether an element of the lattice $\mB \Z^n$ lies in the convex set $S$ or not, an auxiliary oracle is needed. In our use-case, this oracle can be implemented by using the results of \Cref{sec:liouville}.

\begin{proposition} \label{lemma:perfectsampler} Let $\mB = (\bb_1,\ldots,\bb_n)$ with $\bb_i \in \R^n$ be a basis of lattice $\Lambda$ and
let $D = 2 \sum_i \|\bb_i\|$.
\begin{enumerate}
 \item Let $t \in \R^n$ satisfying $\|t\| < U$ for some $U \in \R_{>0}$.
 \item Let $N \in \Z_{>0}$ satisfying $\frac{n}{N} \leq D$ and let $\mC \in \frac{1}{N} \Z^n$ be an approximation of $\mB$ satisfying
$\|\mC - \mB\| \leq \|\mB^{-1}\|^{-1} \cdot D/U$ and $\sum_i \|\bc_i\| \leq D$.
\item Let $\tilde{t} \in \frac{1}{N} \Z^n$ an approximation of $t \in \R^n$ satisfying $\| t - \tilde{t}\|_\infty \leq \frac{1}{2N}$ and $\|\tilde{t}\| < 2U$.
\item Let $1/5 > \eps > 0$ and let $S \subseteq \R^n$ be a convex set satisfying $B_{D/\eps} \subseteq S \subseteq B_{U}$.
\item Suppose we can perfectly uniformly sample in $(1 + 4\eps) S \cap \frac{1}{N} \Z^n$ within expected time $T$. 
\item Additionally, suppose we have an oracle $O_{\mB}$ that can decide, on input $v \in \Z^n$, whether $\mB v \in S + t$ or not.
\end{enumerate}

Then we can sample $v \in \Z^n$ such that $\mB v$ is perfectly uniformly distributed in $(S + t) \cap \mB \Z^n$, within expected time $\poly(T, \allowbreak \size(\mC), \allowbreak \log U, \allowbreak \log N)$, using a single call to $O_{\mB}$, and with success probability lower bounded by
\[ \frac{|(1-5\eps)S \cap \frac{1}{N} \Z^n|}{|(1+5\eps)S \cap \frac{1}{N} \Z^n|} \]
\end{proposition}
\begin{proof} (\textbf{Algorithm description)}  \\
\textbf{(1)} We start with processing the shift $t \in \R^n$. We use the approximation $\tilde{t} \in \frac{1}{N} \Z^n$ that satisfies $\|t - \tilde{t}\|_\infty \leq \frac{1}{2N} \leq D/(2n)$ and hence $\|t - \tilde{t}\|_2 \leq D/2$. We compute
$w_0 = \lfloor \mC^{-1} \tilde{t} ~\rceil$ (where $\lfloor \cdot \rceil$ means that each component is rounded to the nearest integer).
\\
\textbf{(2)} Imitate the approach in \Cref{lemma:perfectsampleq}, (taking $c = 1 + 3\eps$) by sampling $u \in \frac{1}{N} \Z^n \cap (1 + 4\eps) S$ and computing $\bc = \mC \lfloor \mC^{-1} u \rceil$ until $\bc \in (1 + 3\eps)S$. By the exact same reasoning as in \Cref{lemma:perfectsampleq}, one then obtains a perfectly uniform sample $\bc \in \mC \cdot \Z^n \cap (1 + 3\eps)S$. \\
\textbf{(3)} We compute $v = \mC^{-1}\bc = \lfloor \mC^{-1} u \rceil$ and we output $v + w_0 \in \Z^n$ %
if $\mB (v + w_0) \in S + t$ (using the oracle $O_{\mB}$). Otherwise output `failure'. \\

(\textbf{Analysis})
In this analysis we will denote $\bb = \mB v, \bb_0 = \mB w_0, \bc = \mC v$ and $\bc_0 = \mC w_0$.
We will first show that the output $v + w_0 \in \Z^n$ has the property that $\mB(v + w_0)$ is perfectly uniformly distributed in $\mB \Z^n \cap (S+t)$.  We will finish the proof with a lower bound on the success probability and a run time analysis. 

By \Cref{lemma:perfectsampleq}, we know that $v + w_0$ has the property that $\mC(v + w_0) = \bc + \bc_0$ is perfectly uniform in $\mC \Z^n \cap [ (1 + 3\eps)S + \bc_0]$. Hence $
\mB(v + w_0) = \mB \mC^{-1} (\bc + \bc_0)$ is perfectly uniform in $\mB \Z^n \cap [\mB \mC^{-1}(1 + 3\eps)S + \bb_0]$. So, since we reject those $v + w_0$ for which $\mB(v + w_0) \notin S + t$, it is sufficient to show that $S + t \subset \mB \mC^{-1}(1 + 3\eps)S + \bb_0 $. This is equivalent to $\mC \mB^{-1} (S + t) \subset (1  +3\eps)S + \bc_0$ which is in turn equivalent to $\mC\mB^{-1}(S+t)-\bc_0 \subset (1+3\eps)S$. To show this, take an arbitrary $s \in S$, and put
\begin{equation} \mC \mB^{-1} (s + t) - \bc_0 = s + t - \bc_0 + (\mC\mB^{-1}-I)s + (\mC\mB^{-1} -I)t .\label{eq:mcmb} \end{equation} 
We will now show that $\|t - \bc_0\| \leq D$, $\|(\mC\mB^{-1}-I)s\| \leq D$ and $\|(\mC\mB^{-1}-I)t\| \leq D$, so that the right-hand side of \Cref{eq:mcmb} lies in $S \minksum B_{3D} \subseteq S \minksum (3\eps)S = (1 + 3\eps)S$ (where $\minksum$ denotes the Minkowski sum). 
We have $\|t - \bc_0\| \leq \| t - \tilde{t}\| + \|\tilde{t} - \bc_0\| \leq D/2 + D/2$ since $\|t - \tilde{t}\| \leq D/2$ by construction and $\tilde{t} - \bc_0 = \tilde{t} - \mC \lfloor \mC^{-1} \tilde{t} \rceil = \tilde{t} - \mC (\mC^{-1} \tilde{t} + u) = \mC u$ with $u \in [-1/2,1/2)^n$. Hence $\|\tilde{t} - \bc_0\| = \|\mC u \| \leq D/2$, by assumption. Using the properties of matrix norms, we obtain $\|(\mC\mB^{-1}-I)\| = \| (\mC - \mB) \mB^{-1} \| \leq \| \mC - \mB\| \|\mB^{-1}\| \leq D/U$. Combining this with the assumptions $\|s\| \leq U$ and $\|t\| \leq U$, we obtain that all the summands' norms are bounded by $D$.

(\textbf{Success probability})
For the success probability, note that, if $u \in \frac{1}{N} \Z^n$ from step (2) of the algorithm description were to be in $(1-5\eps) S$, a similar reasoning as above shows that, for $v =  \lfloor \mC^{-1} u  \rceil$, we surely have $\bc = \mC v \in (1 - 4\eps)S$.

We have, by the Neumann series of $(I - (I - \mC^{-1}\mB))^{-1}$, noting that $\| I - \mC^{-1}\mB \| \leq D/U \leq \eps < 1/5$,
\begin{align*} \|\mB \mC^{-1} - I  \| &= \| (I - (I - \mC^{-1}\mB))^{-1} - I\| = \big \|  \sum_{j = 1}^\infty (I - \mC^{-1} \mB)^{j}  \big | \leq \sum_{j = 1}^\infty \| I - \mC^{-1} \mB  \|^j \\ 
 & \leq \tfrac{5}{4} D/U.
\end{align*}
Hence, by similar computations as above, we can bound the norms of the following summands: 
\begin{align*} \mB(v + w_0) &=  \mB \mC^{-1} (\bc + \bc_0) - t  \\ 
& = \bc + \underbrace{(\mB\mC^{-1} - I) \bc}_{\leq \tfrac{5}{4} D} + \underbrace{(\mB\mC^{-1} - I) t}_{\leq \tfrac{5}{4} D} + \underbrace{(\mB\mC^{-1} - I) (\bc_0 - t)}_{\leq \tfrac{5}{4}D^2/U \leq D/4}  + \underbrace{\bc_0 - t}_{\leq D}. \end{align*}
So, since $\bc = \mC v \in (1 - 4\eps )S$ and all the other summands together have norm at most $4D$ and their sum thus lies in $B_{4D} \subseteq 4\eps S$, we see that $\mB(v + w_0) \in (1 - 4\eps)S \minksum 4\eps S \subseteq S$.

Therefore, the success probability of the entire procedure is lower bounded by
$\frac{|(1-5 \eps)S \cap \frac{1}{N} \Z^n|}{|(1+4 \eps)S \cap \frac{1}{N} \Z^n|} \geq \frac{|(1-5 \eps)S \cap \frac{1}{N} \Z^n|}{|(1+5 \eps)S \cap \frac{1}{N} \Z^n|}$.

(\textbf{Run time analysis}) Note that the additional running time of this algorithm compared to \Cref{lemma:perfectsampleq} is caused by the (possible) extra computations on $\tilde{t} \in \frac{1}{N} \Z^n$ (which satisfies $\|\tilde{t}\| \leq 2U$, which can be at most $\poly(\size(\mC), \allowbreak \log U, \allowbreak \log N)$. The dependency on $n$ is hidden in $\size(\mC)$.
\end{proof}

\begin{notation}  \label{notation:prodsigmaz} We define $\phi: \prod_{\sigma} \Z \hookrightarrow \nfr$, $(n_\sigma)_\sigma \mapsto (x_\sigma)_\sigma$
by %
putting
\begin{equation} \label{eq:embeddingofZn} \begin{cases} x_{\sigma_\nu} = n_{\sigma_{\nu}} + i n_{\bar{\sigma}_{\nu}} & \mbox{ if $\nu$ is complex} \\
x_{\bar{\sigma}_\nu} = n_{\sigma_{\nu}} - i n_{\bar{\sigma}_{\nu}} & \mbox{ if $\nu$ is complex} \\
x_{\sigma_\nu} = n_{\sigma_{\nu}} & \mbox{ if $\nu$ is real}
\end{cases} \end{equation}
Abusing notation, we will just denote $\Z^n$ or $\prod_{\sigma} \Z$ for $\phi(\prod_{\sigma} \Z)$.
\end{notation}

\begin{corollary} \label{cor:perfectsamplenf}Let $K$ be a degree $n$ number field, let $x \in \nfrstar$ be represented by rational numbers, let $\mb \in \ideals$ be an ideal and let $\modu = \moduz \moduinf$ be a modulus. Let $\radpar \in \Q_{\geq 1}$ and let $r = \radiusformulaxmb$. Let $\mB_{\mb \moduz} = (\bb_1,\ldots,\bb_n)$ with $\bb_j \in \nfr$ (i.e., in the Minkowski space) be a basis of $\mb \moduz$ satisfying
\[ \| x \cdot \bb_i \| \leq 2n \cdot \blocksize^{2n/\blocksize} \cdot \lambda_n(x \cdot \mb \cdot \moduz) %
\]
Let $\tau \in \Kmodumodu$ satisfying $\|x \tau\| \leq r$. %
Then we can perfectly uniformly sample from
\[  x (\mb \moduz + \tau) \cap  r\Ballinftau. \]
 within expected time $\poly(n,\log |\dcrk|, \size(\mb), \size(\modu), \size(x), \size(\radpar))$
\end{corollary}
\begin{proof} We apply \Cref{lemma:perfectsampler}. For this, we need to satisfy all requirements (1) - (6) in \Cref{lemma:perfectsampler}. We start with (1) and (4), then proceed with (2), (3), (5) and (6). Note that \Cref{lemma:perfectsampler} only outputs some $v \in \Z^n$, but because the elements in $x(\mb \moduz  + \tau)$ can be symbolically represented, we will see that such an output $v = (v_1,\ldots,v_n) \in \Z^n$ can be converted into $x \sum_{j =1}^n \beta_j v_j + x \tau$ where $\beta_j \in K$ is the element in $\mb \moduz$ associated with the basis element $\bb_j$.

\textbf{Requirements (1) and (4)}. 
Since a uniform sample from $x (\mb \moduz + \tau) \cap  r\Ballinftau$ can be simply obtained by taking a uniform sample of $x\mb\moduz \cap (r\Ballinftau - x \tau)$ and adding $x\tau$ afterwards, we concentrate on taking a uniform sample from this latter set.
Recall that \[ r \Ballinftau = \{ (x_\emb)_\emb \in \nfr ~|~  |x_\emb| \leq r \mbox{ and } x_\emb/\emb(\tau) > 0 \mbox{ for real }  \emb \mid \modu \}. \]
Hence we can write $r \Ballinftau =  (r'_\sigma)_\sigma \Ballinf + t' = S + t'$ with $r'_\emb = r/2$ for real $\emb \mid \moduinf$ and $r'_\emb = r$ otherwise, and $t'_\emb = r/2$ for real $\emb \mid \moduinf$ and $t'_\emb = 0$ otherwise; and $S = (r'_\sigma)\Ballinf$ (which is a convex set). Then, putting $t := t' - x \tau$, we have
\[ r \Ballinftau - x \tau = S + t' - x \tau = S + t \]
Note that  $\| x \cdot \bb_i \| \leq r/(\radiusformulaconstantdivtwo n^2)$ (by the definition of $r$ and \Cref{lemma:idlatfacts}(\itemref{item:gap-bound}) and (\itemref{item:covering-bound}); see also \Cref{eq:radiusboundofD}).
Hence, writing $D =  2\sum_{i} \|x \cdot \bb_i \| \leq r/(12 n)$, we see that 
 $B_{6nD} \subseteq B_{r/2} \subseteq S \subseteq B_{r}$ and $\|t\|_\infty \leq \|t'\| + \|x \tau\| \leq 2nr$. So, taking $\eps := 1/(6n)$ and $U := 2nr$ satisfies the requirements (1) and (4) of \Cref{lemma:perfectsampler}.

\textbf{Requirement (2) and (3)}. Let us first compute an upper bound on $U/D$. We have (see \Cref{lemma:idlatfacts}) $\lambda_1(x \moduz \mb) \geq \sqrt{n} \norm(x\moduz\mb)^{1/n}$ and hence $D \geq 2 n^{3/2} \norm(x\moduz\mb)^{1/n}$. Hence, since $U = 2nr$, and $\blocksize^{1/\blocksize} \leq e^{1/e}$,
\begin{equation} U/D \leq \frac{(2 \cdot n) \cdot \radiusformulaxmb }{2 n^{3/2}\cdot \norm(x\moduz\mb)^{1/n}} =  \radiusformulaconstant \cdot \radpar \cdot e^{2n/e} \cdot n^3 \cdot |\dcrk| \label{eq:UD}
\end{equation}

We need to (efficiently) approximate $\mC \in \frac{1}{N} \Z^n$ with $\|\mC - \mB\| \leq \|\mB^{-1}\|^{-1} \cdot D/U$, with $\mB = x \mB_{\mb \moduz}$. For that it is sufficient\footnote{We have $\|A\| \leq \|A\|_F = (\sum_{ij} |A_{ij}|^2 )^{1/2}$ (Frobenius norm) and hence approximating $\mB$ by $\mC \in N^{-1} \Z^{n \times n}$ yields a $\mC$ for which $|\mC_{ij} - \mB_{ij}|\leq N^{-1}$ and hence $\|\mB - \mC\| \leq \|\mB - \mC\|_F \leq n N^{-1}$.} to choose $N \geq  n \cdot \|\mB^{-1}\| \cdot U/D$. 
Using \Cref{lemma:wellconditioned}, noting that $\lambda_n(x \cdot \mb \cdot \moduz)/\lambda_1(x \cdot \mb \cdot \moduz) \leq |\dcrk|^{1/n}$ (see \Cref{lemma:idlatfacts}) and $\blocksize^{1/\blocksize} \leq e^{1/e}$, we obtain
\begin{align} \| \mB^{-1} \| & \leq n^{n/2+1} \lambda_1(x \mb \moduz) \cdot \left(  \prod_{j=1}^n \frac{ 2n \cdot \blocksize^{2n/\blocksize} \cdot \lambda_n(x \cdot \mb \cdot \moduz)}{\lambda_j(x \mb \moduz)} \right) \nonumber \\ & \leq n^{n/2 + 1} |\dcrk| \cdot (2n)^n \cdot e^{2n^2/e} . \label{eq:matrixBbound} \end{align}
Combining \Cref{eq:UD} and \Cref{eq:matrixBbound}, instantiating 
\begin{align} \label{eq:instantiateN} N := \lceil n \cdot n^{n/2 + 1} |\dcrk| \cdot (2n)^n \cdot e^{2n^2/e} \cdot \radiusformulaconstant \cdot \radpar \cdot e^{2n/e} \cdot n^3\cdot |\dcrk|  \rceil \geq n \cdot \|\mB^{-1}\|^{-1} \cdot U/D \end{align}
we can certainly compute such approximation $\mC$ within time $\poly(n, \log (N), \log(r))$ for such $N$ (since $\|x \bb_i\| < r$). Similarly, for requirement (3), we can efficiently compute $\tilde{t} \in \frac{1}{N} \Z^n$ satisfying $\|t - \tilde{t}\| \leq 1/(2N)$. Since $1/(2N) \leq D/2 \leq U$ we certainly have $\|\tilde{t}\| \leq 2U$.

\textbf{Requirement (5)}.
We can efficiently and perfectly uniformly sample in $\frac{1}{N} \Z^n \cap (1 + c)S$ for every $c > 0$ (where we understand $\Z^n \hookrightarrow \nfr$ as in \Cref{notation:prodsigmaz}) by simple rejection sampling (where the `circles' associated to the complex places are done one-by-one). 
Indeed, for this $N$, the uniform random sampling in
$(1 + 4\eps)S \cap \frac{1}{N} \Z^n$ (with $\eps = 1/(6n)$) has bit complexity $O(n\log(rN))$. 

\textbf{Requirement (6)}.
For the implementation of the oracle $O_{x\mB_{\mb \moduz}}$, it is enough to show that we can efficiently decide for any rational number $g \in \Q$, integer $k \in \Z_{>0}$ and any embedding $\sigma$ of $K$, whether an algebraic number $x \beta$ with $\beta \in \mb \moduz$ satisfies $|\sigma(x \beta)| > g^{1/k}$ or $|\sigma(x \beta)| \leq g^{1/k}$. Indeed, this holds since the set $S$ is entirely defined in terms of these (absolute values of) embeddings. By putting $k = 2 \cdot \blocksize \cdot n$ we can write $r = g^{1/k}$ for some rational $r$. The existence and the effectiveness of such an oracle is precisely the object of \Cref{cor:boxdecide}; it is clear that it runs in expected time $\poly(n, \log |\dcrk|, \size(\mb), \size(\modu), \size(x), \size(\radpar))$.

\textbf{Bit complexity}
Since $\log(r), \log(N)$ are both $\poly(n,\log |\dcrk|, \size(\mb),\allowbreak \size(\modu),\allowbreak \size(x),\allowbreak \size(\radpar))$, a single run of the algorithm takes bit complexity  $\poly(n,\allowbreak\log |\dcrk|,\allowbreak \size(\mb),\allowbreak \size(\modu),\allowbreak \size(x), \allowbreak\size(\radpar))$. If the success probability of a single run is bounded from below by a constant, the entire algorithm (until success) has the same expected bit complexity. It thus remains to show that, for the box $S = (r'_\sigma)\Ballinf$, $\eps = 1/(6n)$, and $N$ as in \Cref{eq:instantiateN}, the success probability is lower bounded by a constant. We have that the success probability is lower bounded by (using \Cref{lemma:divisorcapinfinityball} with $c = (\min_\sigma | r'_\sigma|)^{-1}$ and (the $r$ of that lemma) $r = (1 \pm 5 \eps)  \in [\tfrac{1}{6},2]$)
\[ \frac{|(1-5\eps)S \cap \frac{1}{N} \Z^n|}{|(1+5\eps)S \cap \frac{1}{N} \Z^n|} \geq \frac{(1-5\eps)^n |S|}{(1+5\eps)^n |S| }  \frac{N^n e^{-12nc}}{N^n e^{4nc}} \geq \frac{(1-1/n)^n e^{-16nc}}{(1+1/n)^n  } \geq e^{-2} \cdot e^{-1} \geq e^{-3},\]
since $c =(\min_\sigma | r'_\sigma|)^{-1} \leq (r/2)^{-1}  \leq 1/(\radiusformulaconstantdivtwo n)$.
\end{proof}

\subsection{Algorithm for sampling in a box intersected with an ideal lattice}

\begin{figure}
\begin{algorithm} [H]
    \caption{Uniform sampling in $\kx \cdot ((\mbb + \gamma) \cap \tau \Kmodu) \cap r \Ballinf$}
    \label{alg:sample_in_a_box}
    \begin{algorithmic}[1]
    	\REQUIRE ~\\
         \begin{itemize}
          \item An LLL-reduced basis of $\OK$,
          \item A modulus $\modu = \moduz \moduinf$ of the degree $n$ number field $K$.
          \item An ideal $\mbb \in \idealsmodu$,
          \item An element $\gamma \in K$,%
          \item An element $\tau \in \Kmodumodu$,%
          \item A block size parameter $\blocksize \in \Z$, with $2 \leq \blocksize \leq n$.
          \item A element
           $\kx \in \nfrstar$,
           \item A real number $\radpar \in \Q_{\geq 1}$.
         \end{itemize}

    	\ENSURE 
    	A uniformly distributed element $\beta \in x \cdot ((\mbb + \gamma) \cap \tau \Kmodu) \cap r \Ballinf$ with $r = \radiusformulaxmb$.

    \vspace{2mm}
    		\STATE \textbf{Defining the radius of the ball.} ~\\
    		Define $r = \radiusformulaxmb$. \label{algbox:defr}
        \STATE \textbf{Obtaining a short basis of $\kx \mbb \moduz$.} ~\\ Apply \Cref{lemma:shortvectorsinideallattice} to obtain a $\blocksize$-BKZ reduced basis $B_{\mbb\moduz} = (\bb_1,\ldots,\bb_n)$ of $\mbb\moduz$ that satisfies
\begin{align} \| \kx \cdot \bb_i \| \nonumber & \leq 2n \cdot \blocksize^{2n/\blocksize} \cdot \lambda_n(\kx \cdot  \mbb \cdot \moduz) \\ & \leq 2n \cdot \blocksize^{2n/\blocksize} \cdot \sqrt{n} \cdot  |\dcrk|^{3/(2n)} \cdot \norm(\kx\mbb \moduz)^{1/n} \leq r/( \radiusformulaconstantdivtwo n^2),  \label{eq:radiusboundofD} \end{align}
where the second inequality follows from  \Cref{lemma:idlatfacts}(\itemref{item:gap-bound}) and (\itemref{item:covering-bound}) and the third from the definition of $r$.%
        \label{algbox:reducebasis}
        \STATE \textbf{Computing a new shift $\gammamodu$ to take account for the modulus.}~\\
        Compute $\gammamodu \in \mbb + \gamma$ such that $\gammamodu \equiv \tau$ modulo $\moduz$, which is possible by the fact that $\mbb$ and $\moduz$ are coprime.
        If $\moduz = \OK$, we put $\gamma_\modu = \gamma$. %
        \label{algbox:recomputeshift}
        \STATE \textbf{Reducing the shift $\gammamodu$ modulo the short basis of $\mbb \moduz$.}~\\ Reduce
$\gammamodu \in K$
modulo this $\blocksize$-BKZ-reduced basis $(\bb_1,\cdots,  \bb_n)$ of $\mbb \moduz$, yielding $\redgamma \in \gammamodu + \mbb \moduz$. That is, write $\gammamodu = \sum_{i} t_i \bb_i$ and put $\redgamma =  \sum_{i} (t_i - \lfloor t_i \rceil) \bb_i$. \label{algbox:reduceshift}

        \STATE \textbf{Sampling an element in $x\cdot((\mbb\moduz + \redgamma) \cap r \Ballinftau$.} \\
        Using \Cref{cor:perfectsamplenf}, sample $\beta \in x\cdot((\mbb\moduz + \redgamma) \cap r \Ballinftau$.  \label{algbox:sampleuniform}
        \RETURN $\beta$.
    \end{algorithmic}
\end{algorithm}
\end{figure}

\label{subsec:samplingboxintersected}
\begin{lemma} \label{lemma:modulusbox} Let $K$ be a number field,
let $\modu = \moduz \moduinf$ be a modulus of $\OK$,
let $\tau \in \Kmodumodu$.
Let $\mbb \in \idealsmodu$ be an ideal of $K$ coprime with $\moduz$ and
let $\gamma \in K$ be a shift. %
Let $\gammamodu \in \mb + \gamma$ such that $\gammamodu \equiv \tau$ modulo $\moduz$.

Then
\[ (\mbb + \gamma) \cap \tau \Kmodu = (\mbb \moduz + \gammamodu) \cap \taunfrm\]

\end{lemma}

\begin{proof} By scaling up $\mbb,\gamma,\tau$ by an integer $M \in \N_{>0}$
coprime with $\moduz$, we can assume that they lie in $\OK$. It is then enough
to show that $(M\mbb + M\gamma) \cap M\tau \Kmodu = (M\mbb \moduz + M\gammamodu) \cap M\taunfrm$.

We start with inclusion to the right. Since $M\mbb + M\gamma \subseteq \OK$, we have that $\alpha \in (M\mbb + M\gamma) \cap M\tau \Kmodu$ satisfies $\alpha = M\beta + M\gamma = M\tau + \mu$ with $\beta \in \mbb$ and $\mu \in \moduz \subseteq \OK$. Additionally $\sigma(\alpha)/\sigma(\tau) > 0$ for real $\sigma \mid \moduinf$ (since multiplying with $M$ does not change the sign of $\sigma(\tau)$).

By the definition of $\gammamodu \in K$%
, we have $\gammamodu \in \mbb + \gamma$ and $\gammamodu \in \moduz + \tau$. Hence $\alpha - M\gammamodu \equiv 0$ mod both $M\mbb$ and $\moduz$. Hence $\alpha \equiv M\gammamodu$ modulo $M\mbb \moduz$, i.e., $\alpha \in M\mbb \moduz + M\gammamodu$. Also, $\sigma(\alpha)/\sigma(\tau) > 0$ for real $\sigma \mod \moduinf$, so $\alpha \in \taunfrm$.

The inclusion to the left can be done similarly. Any $\alpha \in (M\mbb \moduz + M\gammamodu) \subseteq \OK$ satisfies $\alpha \in M\mbb + M\gamma$. Also, since $\alpha \equiv M\gammamodu \equiv M\tau$ mod $\moduz$ and $\alpha \in \taunfrm$, we have $\alpha \in M\tau\Kmodu$. This concludes the proof.
\end{proof}

\begin{proposition}[Correctness and efficiency of \Cref{alg:sample_in_a_box}] \label{proposition:samp-box-correctness}
Let $K$ be a number field of degree $n$,
let $\modu = \moduz \moduinf$ be a modulus of $\OK$,
let $\mbb \in \idealsmodu$ be an
ideal,
let $\gamma \in K$ be a shift,
let $\tau \in \OK$ be coprime to $\moduz$,
let $\blocksize \in \{1,\ldots,n\}$ be a block size parameter,
let $\kx \in \nfrstar$ be represented by rational numbers,
let $\radpar \in \Q_{\geq 1}$ and put $r = \radiusformulaxmb$.

Then, the randomized algorithm
\Cref{alg:sample_in_a_box} samples from a
uniform distribution over
$\kx \cdot ((\mbb + \gamma) \cap \tau \Kmodu) \cap r \Ballinf$.
Moreover, this algorithm runs in expected time $\poly(n,\log|\dcrk|,\size(\modu), \size(\mbb), \size(\gamma), \size(\tau),\size(\kx), \size(\radpar))$.
\end{proposition}
\begin{proof} The correctness of the algorithm follows from \Cref{cor:perfectsamplenf} and \Cref{lemma:modulusbox}. The only thing that needs to be checked is that $\redgamma$ satisfies $\|x\redgamma\| \leq r$. We have $\redgamma = \sum_{i} c_i \bb_i$ with $c_i \in [-1/2,1/2)$, hence $\|x \redgamma\| \leq \sum_i |c_i| \|x \bb_i\| \leq r$, since the basis was BKZ-reduced (see line \lineref{algbox:reducebasis}).

We conclude the proof with the estimate on the running time. In line \lineref{algbox:defr}, a mere instantiation of the radius $r$ is given.

In line \lineref{algbox:reducebasis}, a short basis of $\kx \mbb \moduz$ is computed using a BKZ-like algorithm. This takes, by \Cref{lemma:shortvectorsinideallattice},
time $\poly(\hkztime, \size(\mb), \size(\moduz), \log |\dcrk|, \size(\kx))$.

In line \lineref{algbox:recomputeshift}, a $\gamma_\modu \in K$ is computed satisfying $\gammamodu \in (\mbb + \gamma) \cap (\moduz + \tau)$. For this, it suffices to compute elements $\beta \in \mbb$ and $\mu \in \moduz$ such that $\beta + \mu = 1$ (by putting $\gammamodu = \tau \beta + \gamma \mu$). Such a pair $(\beta,\mu) \in \mbb \times \modu$ can be found by applying the Hermite normal form to the concatenated basis matrices of $\mbb$ and $\modu$ \cite[Proposition 1.3.1]{cohen1999advanced}. This requires $\tilde{O}(n^5 \log( M)^2)$ time \cite{storjohann}, where $M$ is the maximum entry occurring in the basis matrices. Clearly this overall process takes time polynomial in $n,\size(\moduz),\size(\gamma)$ and $\size(\mb)$.

In line \lineref{algbox:reduceshift}, this element $\gamma_\modu$ is reduced modulo $\mb \moduz$, which takes time at most $\poly(\size(\mb),\size(\moduz),\size(\gamma))$, since $\size(\gamma_\modu) =  \poly(\size(\mb),\size(\moduz),\size(\gamma))$.

Lastly, in line \lineref{algbox:sampleuniform}, a uniform sample is taken, following \Cref{cor:perfectsamplenf}. The expected bit complexity then follows.
\end{proof}

%% file: ideal_sampling/A06-ideal-sampling-algorithm.tex
\section{Ideal sampling} \label{section:sampling}

\subsection{Introduction}
In this section, we prove the main result of this part, \Cref{theorem:ISmain}.
Recall the task at hand. 
Fix a family of ideals $\idset$. Given an ideal $\ma$, find $\beta \in \ma$ such that $\beta\ma^{-1} \in \idset \cdot \idset_B$ with probability proportional to the density of $\idset$.

In \Cref{chapter:randomwalk} is proven that the input ideal lattice $\ma$ can be randomized so that its Arakelov class is uniformly distributed. For such random ideal lattices, by \Cref{sec:correspondence}, the event $\beta\ma^{-1} \in \idset \cdot \idset_B$ happens with the anticipated probability (the ideal density of $\idset$) when $\beta$ is sampled in a large enough box. Subsequently, we proved in \Cref{sec:sampling-box} that one can efficiently sample from such a box. Combining these results together leads to \Cref{alg:samplerandom} and \Cref{theorem:ISmain}, of which the latter can be informally rephrased as follows: There is an efficient way to sample $\beta \in \ma$ satisfying $\beta \ma^{-1} \in \idset \cdot \idset_B$ with a provable lower bound on the sampling probability.
\\ \\ \noindent
In the proof of \Cref{theorem:ISmain} we require three technical results,
which are in the later separate \Cref{section:helplemmas} and \Cref{subsection:statdiff}. 
The first of these three results is a lemma that
states that part of \Cref{alg:samplerandom} is exactly a random walk as in \Cref{thm:random-walk-weak},
applied to the input ideal lattice $y \cdot \mb$. This allows to apply the density result as in \Cref{prop:elementdensity}.
The second result consists of a lemma that states that a distribution $\distr$ on $\rayDiv^0$
for which holds that the `folded' distribution $[\distr]$ is close to uniform in $\rayPic^0$,
the original distribution $\distr$ is close to some $\distr_U$ on $\rayDiv^0$ satisfying $[\distr_U] = U(\rayPic^0)$.
This allows for the statistical argument that we may assume the `randomized' ideal lattice being drawn
from $\distr_U$ rather than the original distribution $\distr$, for the cost of a small error coming
from the statistical distance.

The last of these three results is \Cref{lemma:closeness_continuous}, which shows that 
the output distribution of 
\Cref{alg:samplerandom} and that of the `continuous variant of \Cref{alg:samplerandom}' (which 
is to be specified precisely later) are close. This result is very useful because some algorithmic properties 
are much easier to prove for this `continuous variant'; this closeness of the output distributions then 
show that these properties must also hold for the original variant of \Cref{alg:samplerandom}, though 
with a small error due to the slight difference between the output distributions of the two variants.

\subsection{Ideal sampling}
\begin{definition} We denote by $\idset_B^{\subpic}$ the set of $B$-smooth integral ideals coprime with $\moduz$ and whose prime divisors lie in the subgroup $\subpic \subseteq \rayPic^0$, i.e.,
\[ \smoothset^\subpic = \big\{ \ma \mbox{ ideal of } \OK ~\big|~ \mp \mid \ma \mbox{ implies } \mp \nmid \moduz,  \norm(\mp) \leq B \mbox{ and } [d^0(\mp)] \in \subpic \}. \]
For $G = \rayPic^0$ we just get the set of smooth ideals $\smoothset$ coprime with $\moduz$.
\end{definition}

\begin{notation} \label{notation:taupositive} Let $\modu = \moduz \moduinf$ be a modulus of $K$ and let $\tau \in \OK$ be an element coprime with $\moduz$. Then we denote
\[ \taunfrm := \{ (x_\emb)_\emb \in \nfr ~|~ \sign(x_\emb) = \sign(\emb(\tau)) \mbox{ for all real } \emb \mid \moduinf \}. \]
Alternatively, this set consists of all $(x_\emb)_\emb \in \nfr$ for which holds $x_\emb/\emb(\tau) > 0$ for all real $\emb \mid \moduinf$.
\end{notation}

\begin{notation} \label{notation:ZH} Let $\hyper = \Log(\nfr^0)$ be the hyper space of the number field $K$, i.e., $H = \{ (x_\pl)_\pl \in H ~|~ \sum_\pl x_\pl = 0 \}$. Let $(\pl_1,\ldots,\pl_{\dimh+1})$ be an ordering of the places (with $\dimh + 1 = \rem + \cem$). We denote $\Z_H \subseteq H$ for the integral lattice with basis $\mB_H = ( e_{\pl_1} - e_{\pl_2},\ldots, e_{\pl_i} -  e_{\pl_{i+1}}, \ldots,  e_{\pl_{\dimh}} - e_{\pl_{\dimh+1}})$, where $e_{\pl_i}$ is the standard basis. It
 satisfies $\lambda_1(\Z_H) = \lambda_{\dimh}(\Z_H) = \sqrt{2}$ and hence $\cov(\Z_H) \leq n$.

We will denote $\dH = \frac{\delta}{n} \Z_H$ for the `discretized' hyper space, where $\delta > 0$ is some grid parameter. We then have $\cov(\dH) \leq \delta$.
\end{notation}

\begin{lemma} \label{lemma:alg2independentnorm} The output distribution of \Cref{alg:samplerandom} is independent of the absolute value of the norm $|\norm(\ky)|$ of $\ky$, and independent of the signs (and complex phase) of the entries $\ky_\emb$ of $\ky = (\ky_\emb)_\emb \in \nfrstar$. 
\end{lemma}
\begin{proof}

The variable $\ky$ only occurs in lines \lineref{line:alg1:sample} and \lineref{line:alg1:return}. The element $\tilde{\beta}$ is perfectly uniformly random over $\big((A_\emb \cdot y_\emb)_\emb
        \cdot \big[ \tmb \cap \tau \Kmodu \big] \big) \cap r \cdot \norm(\ky\tmb)^{1/n} \cdot \Ballinf$. Hence, $\beta := (A_\emb^{-1} \cdot y_\emb^{-1})_\emb \cdot \tilde{\beta}$ is uniformly distributed over
        \[ \big( \big[ \tmb \cap \tau \Kmodu \big] \big) \cap r \cdot (A_\emb^{-1} \cdot y_\emb^{-1})_\emb
        \cdot |\norm(\ky)|^{1/n} \cdot \norm(\tmb)^{1/n} \cdot \Ballinf \] \[ = \big( \big[ \tmb \cap \tau \Kmodu \big] \big) \cap r \cdot (A_\emb^{-1} \cdot (y^0)_\emb^{-1})_\emb \cdot \norm(\tmb)^{1/n} \cdot \Ballinf, \]
        where $y^0 :=  \ky/|\norm(\ky)|^{1/n}$. Hence, the algorithm depends on $y^0$ and in particular not on $|\norm(\ky)|$. Since the set $\Ballinf$ is symmetric around zero, the signs (and complex phases) of $y^0_\emb$ (and hence of $\ky$) do not have any influence on the set $(A_\emb^{-1} \cdot (y^0)_\emb^{-1})_\emb \cdot \norm(\tmb)^{1/n} \cdot \Ballinf$.
\end{proof}

\begin{theorem}[\normalfont{ERH}, Ideal sampling theorem]\label{theorem:ISmain} 
  Let $K$ be a number field of degree $n$, let $\modu = \moduz \moduinf$ be a modulus, let $\subpic \subseteq \rayPic^0$ a finite-index subgroup of $\rayPic^0$
 and let $\mathbf{O}_\subpic$ be an oracle that on input an ideal $\mc$
 outputs whether $[d^0(\mc)] \in \subpic$ or not.
 Let $\mbb \in \ideals$ be an ideal coprime with $\moduz$,
 let $\tau \in \Kmodumodu$ (see \Cref{prelim:numberfields}) satisfying $[\ldb \tau \rdb] \in G$,
and let $\ky \in \nfrstar$ be represented by rational coordinates.
 Let $\blocksize \in \{2, \cdots, n\}$ be an integer, let $0< \eps < \min(1,\tfrac{20}{n})$ be an error parameter and let $\radpar \in \Q_{\geq 1}$. Let $\idset$ be a set of integral ideals coprime with $\moduz$, satisfying $[d^0(\idset)] \subseteq G - [d^0(\mbb)]$. 
 \\
\noindent\textbf{(A) Correctness.}
\Cref{alg:samplerandom} outputs
an element $\beta \in \mbb$ such that
 \begin{itemize}
  \item $(\beta) \cdot \mbb^{-1} \in \idset\cdot \smoothset^\subpic$,
  \item $\beta \in \tau \Kmodu$ (i.e.,
  $\ord_\mp(\beta- \tau) \geq \ord_\mp(\moduz)$ for all $\mp \mid \moduz$
  and $\sigma(\beta/\tau) > 0$ for all real $\sigma \mid \moduinf$),
  \item $|\norm(\beta)| \leq \norm(\mbb) \cdot B^N \cdot r^n$
 \end{itemize}
 with probability at least
 \begin{equation} \label{eq:lowerboundsampleprob} \frac{\norm(\moduz)}{\phi(\moduz)}  \cdot \frac{[\rayPic^0:\subpic]}{3} \cdot \delta_{\idset}[r^n] -\eps \geq  \frac{[\rayPic^0:\subpic]}{3} \cdot \delta_{\idset}[r^n] -\eps . \end{equation}
 Here, $B = \widetilde O \Big ([\rayPic^0:\subpic]^2 \cdot n^{2} \cdot \big[ n^2 \cdot (\log \log (1/\eps))^2 + (\log (|\dcrk|\norm(\modu)))^2 \big]  \Big)$,
    $N = \lceil 7n + \log(\norm(\modu)) + \log|\rayPic^0| - \log[\rayPic^0:\subpic] + 2 \log(1/\eps) + 1 \rceil$ as in \Cref{thm:random-walk-prime-sampling},
    $r = \radiusformulamacro{\moduz}$,
    and $\norm(\modu) = 2^{|\modur|} \cdot \norm(\moduz)$ with $|\modur|$ being the number of different real embeddings dividing $\modu$ (see \Cref{prelim:numberfields}).

\noindent\textbf{(B) Running time.}
Furthermore, \Cref{alg:samplerandom} has a bit complexity of $\poly(\hkztime, \allowbreak \log |\dcrk|,\allowbreak \size(\mbb), \allowbreak \log(1/\eps), \allowbreak \log(\norm(\modu)),\allowbreak [\rayPic^0:\subpic],\allowbreak  \size(\ky), \allowbreak \size(\tau), \allowbreak \size(\radpar))$ and uses at most
$\poly(\log|\dcrk|,\allowbreak \log \norm(\modu)) \cdot [\rayPic^0:\subpic]$
queries to $\mathbf{O}_\subpic$.
\end{theorem}

\begin{proof}[Proof of \Cref{theorem:ISmain} (A) Correctness] 
This proof is structured as follows. We will assume, purely for the sake of argument, that lines \lineref{line:alg1:distortion} and \lineref{line:alg1:rationalA} in \Cref{alg:samplerandom} are replaced by the following,
\begin{equation} \mbox{``\texttt{Sample $a = (a_\sigma)_\sigma \from \Gaussian_{H,\sd}$ and put $(A_\sigma)_\sigma = (e^{\nplemb^{-1} a_\sigma})_\sigma$}''.} \label{eq:alg1change} \end{equation}
and that $\ky \in \nfrstar$ is replaced\footnote{Note that this particular change of $\ky$ by $\ky^0$ does not impact the output distribution, by \Cref{lemma:alg2independentnorm}.} by $\ky^0 = \ky/|\norm(\ky)|^{1/n}$ (so that $\hyb = \Log(\ky) \in \hyper$). This is indeed purely 
for the sake of argument, since these changes renders this algorithm unprocessable by a computer, due to the real arithmetic. We will show 
two things: (1)  After these changes, which we will call the `continuous version of \Cref{alg:samplerandom}', both the correctness and the success probability claim 
as stated in \Cref{eq:lowerboundsampleprob} of (A) do hold (but with $\eps/2$ instead of $\eps$); and (2) the output distribution of this `continuous version of \Cref{alg:samplerandom}' is $\eps/2$-statistically close 
to the output distribution of the ordinary (or discrete) version of \Cref{alg:samplerandom} (that is, without the changes on $\ky$ and line \lineref{line:alg1:distortion} and \lineref{line:alg1:rationalA}). 
Together, we may then conclude that the correctness and success probability claim holds for the ordinary version of \Cref{alg:samplerandom}, which is 
what we aimed to show. 

\textbf{Part (1): Showing the correctness and success probability for the `continuous version of \Cref{alg:samplerandom}'}.
We assume that lines \lineref{line:alg1:distortion} and \lineref{line:alg1:rationalA} from \Cref{alg:samplerandom} are replaced by \Cref{eq:alg1change}
and that $\ky \in \nfr^0$. We will denote $\hyb = \Log(\ky) \in \hyper$. Then, by \Cref{lemma:algo_distr}, which we will treat later, the ideal-element pair $\big((\beta)\tmb^{-1}, \beta \big) \allowbreak \in \allowbreak \idealsmodu \times \mb$ from \Cref{alg:samplerandom} is distributed as 
\[\big( (\alpha)\cdot\dExp{-\ba}, \alpha \cdot \infpart{\Exp}(-\infpart{\ba}) \big),\]
with $\ba \from \Walk = \Walk_G(N,B,\sd) + d^0(\mbb) + \hb$  and subsequently $\alpha \from \rayelt{\ba} \cap r \Ballinf$ uniformly. Here $\Walk = \Walk_\subpic(N,B,\sd) + d^0(\mbb) + \hb$ is the random walk distribution starting on the point $d^0(\mbb) + \hb \in \rayDiv^0$ (see \Cref{def:rwdiv}). Here, $\rayelt{\ba}$ is defined in \Cref{def:rayelts} and $\hb \in H \subseteq \rayDiv^0$.

For the random walk distribution $\Walk = \Walk_\subpic(N,B,\sd)+ d^0(\mbb) +\hb$ on $\rayDiv^0$ with these parameters $(N,B,\sd)$ holds that $[\Walk]$ on $\rayPic^0$ is
$\eps/2$-close
to the uniform distribution $\unif(\subpic + d^0(\mbb))$ on the coset $\subpic + d^0(\mbb)$ in the total variation distance. So, allowing an error of
$\eps/2$
we may as well assume that $\ba$ instead comes from a distribution $\distr$ on $\rayDiv^0$ that satisfies $[\distr] = \unif(\subpic+ d^0(\mbb))$ (see \Cref{lemma:lifting}).

By applying \Cref{prop:elementdensity}, using that $[\ldb \tau \rdb] \in G$, one then obtains that the probability that $(\beta) \cdot \tmb^{-1} = (\alpha) \dExp{-\ba} \in \idset$ given that $\beta = \alpha \infpart{\Exp}(-\infpart{\ba}) \in \tau\Kmodu$ (which is precisely the way how $\alpha$ is sampled) is at least $\frac{\norm(\moduz)}{\phi(\moduz)}\frac{[\rayPic^0:\subpic]}{3} \cdot \delta_{\idset}[r^n] -\eps/2$. From the fact that $\tmb = \mbb \cdot \prod_{j} \mp_j$ with $\mp_j \nmid \modu$, $[d^0(\mp_j)] \in \subpic$, %
and $\norm(\mp_j) \leq B$, we have that $(\beta) \cdot \mbb^{-1} \in \idset \cdot \smoothset^\subpic$ in that case, and the probability claim of \Cref{eq:lowerboundsampleprob} (with $\eps/2$ instead of $\eps$) follows.

We finish part (1) of this proof by showing that the output $\beta \in \mb$ of \Cref{alg:samplerandom} satisfies all bullet points of the theorem.
By lines \lineref{line:alg1:sample} and \lineref{line:alg1:return} it follows that
$\beta \in \tau\Kmodu$. By line \lineref{line:alg1:sample} it follows that $\tbeta \in r \cdot \norm(\tmb)^{1/n} \cdot \Ballinf$, hence $|\norm(\tbeta)| \leq r^n \cdot \norm(\tmb) \leq r^n \cdot B^N \cdot \norm(\mb)$. Since the Gaussian distortion $(A_\sigma)_\sigma$ (in the `continuous version') does not change the norm (as $\ha \in \hyper$), we have $|\norm(\beta)| = |\norm(\tbeta)|$.

\textbf{Part (2): Showing that the output distribution of the `continuous version of \Cref{alg:samplerandom}' and the discrete version are $\eps/2$-close}. By the proposition in \Cref{subsection:statdiff}, together with the fact that the discrete Gaussian in line \lineref{line:alg1:distortion} is approximated within statistical distance $\eps/4$, we conclude that with the choice of the `grid parameter' $\delta$ of $\dH$ in line \lineref{line:alg1:distortion} of \Cref{alg:samplerandom}, the statistical difference between the `continuous variant' and the ordinary variant of \Cref{alg:samplerandom} is at most $\eps/4 + \eps/4 = \eps/2$.  

Hence, the success probability of the ordinary \Cref{alg:samplerandom} as in the theorem statement is lower bounded by
\[ \frac{\norm(\moduz)}{\phi(\moduz)}\frac{[\rayPic^0:\subpic]}{3} \cdot \delta_{\idset}[r^n] -\eps/2 - \eps/2 = \frac{\norm(\moduz)}{\phi(\moduz)}\frac{[\rayPic^0:\subpic]}{3} \cdot \delta_{\idset}[r^n] - \eps, \]
as was required to prove.
\end{proof}

\begin{proof}[Proof of \Cref{theorem:ISmain} (B) Running time]
In the following complexity analysis, any complexity that is within $\poly(\log |\dcrk|,\allowbreak \size(\mbb), \allowbreak \log(1/\eps),\allowbreak \log(\norm(\modu)),\allowbreak [\rayPic^0:\subpic], \allowbreak \size(\ky), \allowbreak \size(\tau))$ we will call `polynomial in the size of the input'.
 Note that $\log B$ and $N$ are $\poly(\log |\dcrk|$, $\size(\mbb),$ $\log(1/\eps),$ $\log(\norm(\modu)),$ $\log([\rayPic^0:\subpic]))$,
so any complexity polynomially bounded by $\log B$ and $N$ must be polynomial in the size of the input as well. We can omit the dependency here on $|\subpic|$ since $\log|\subpic| \leq \log|\rayPic^0| \leq \log(|\dcrk| \norm(\modu))$ (see \Cref{lemma:boundvolraypicappendix}).

We go through lines \lineref{line:alg1:BNinit} to \lineref{line:alg1:return} of \Cref{alg:samplerandom}.
Line \lineref{line:alg1:BNinit} just initializes $B$ and $N$ and $\sd$. Line \lineref{line:alg1:sampleprimes} uses \Cref{lemma:sampleGprimes} $N$ times to obtain
the random primes $\mp_1,\ldots,\mp_N$. This takes complexity $O([\rayPic^0:\subpic] \cdot n^3 \cdot \log^2(B) \cdot N)$ and $O(N \cdot [\rayPic^0:\subpic] \cdot n \log B) = \poly(\log|\dcrk|,\log \norm(\modu))) \cdot [\rayPic^0:\subpic]$ queries to $\mathbf{O}_\subpic$.

In line \lineref{line:alg1:primes}, $\mb$ is multiplied by the $N$ sampled prime ideals.
Multiplication of two ideals can be done by LLL-reducing the $n^2 \times n$ matrix involving all products of the $\Z$-generators of the respective ideals, taking bit complexity $\tilde{O}(n^{11} \log (M)^3)$ \cite{nguyen2009lll}, where $M$ is the maximum entry of the matrix involved. This multiplication is done with $N$ ideals for which $\log(M)$ is bounded by $ \poly(\log |\dcrk|, \size(\mb), \log(B))$, which means
that the total bit complexity of this ideal multiplication is in polynomially bounded in the size of the input. An alternative way
to see this is by using the two-element representation of ideals (e.g., \cite[\textsection 4.7]{cohen2008computational}).

Line \lineref{line:alg1:distortion} requires to sample $\ddot{\ha}$ from an approximate discrete Gaussian $\Klein_{\frac{\delta}{n}\mB_H,\eps/4,s,0}$ in $\dH = \frac{\delta}{n} \Z_\hyper$ as in \Cref{lemma:gpv} (with $\mB_H$ as in \Cref{notation:ZH}), which is within statistical distance $\geps := \eps/4$ of an exact  
discrete Gaussian. Indeed, observing that 
the maximum length of the basis of $\frac{\delta}{n} \Z_\hyper$ (see \Cref{notation:ZH}) is at most $\sqrt{2}\delta/n$ we see that certainly (see line \lineref{line:alg1:distortion})
 \[ \sd \geq \delta \sqrt{n \log(2n/(\eps/40))} \geq \sqrt{\frac{\log(4/\eps) + 2\log(n) + 3}{\pi}} \cdot \frac{\sqrt{2} \delta }{n}. \]
This, according to  \Cref{lemma:gpv}, takes expected time polynomial in the input size and $\log(1/\eps)$. Since the input size is $\poly(1/\delta) = \poly( n, \log |\dcrk|, \log(1/\eps), \log(\radpar))$ 
this Gaussian sampling algorithm has a expected bit complexity of $\poly( n, \allowbreak\log |\dcrk|, \allowbreak\log(1/\eps), \allowbreak\log(\radpar))$. Note that, by \Cref{lemma:gpv}, the length of $\ddot{a}$ cannot exceed $\sd \cdot \sqrt{n \log(2n^2/\eps)}$.

In line \lineref{line:alg1:rationalA} we need to approximate $(e^{\nplemb^{-1} \ddot{a}_\sigma})_\sigma$ by rational $(A_\sigma)_\sigma$ such that the relative multiplicative error is small enough. Since the length of $\ddot{a}$ does not exceed $\sd \cdot \sqrt{n \log(2n^2/\eps)}$ (note that $\sd = 1/n^2$) we clearly see that $\min_\sigma \ddot{a}_\sigma$ cannot be smaller than $\eps/(2n^2) < e^{-\sqrt{\log(2n^2/\eps)}}$. Hence, approximating $(A_\sigma)_\sigma$ as in line \lineref{line:alg1:rationalA} can be done in time $\poly(n,\log (1/\delta), \log(1/\eps))$ which is $\poly(n, \log |\dcrk|, \log(1/\eps), \log(\radpar))$. Note that the size of $(A_\sigma)_\sigma$ itself is also bounded polynomially in these latter parameters.

Line \lineref{line:alg1:sample} requires
\Cref{alg:sample_in_a_box},
which uses expected bit complexity at most
$\poly(\hkztime,\allowbreak\size(\tmb),\allowbreak \size(\moduz),\allowbreak \log |\disc|,\allowbreak\size(y), \allowbreak \size( (A_\sigma)_\sigma), \allowbreak \size(\tau))$. Note that $\size(\tmb) = \poly(\size(\mbb),\allowbreak N, \allowbreak \log(B))$. \Cref{proposition:samp-box-correctness} %
can be applied here because we have
$r \cdot \norm(\tmb)^{1/n} \cdot |\norm(\ky)| =  \radiusformula \cdot \norm(\tmb)^{\frac{1}{n}} \cdot |\norm(\ky)|$.

The last line, line \lineref{line:alg1:return}, only multiplies $\tilde{\beta}$ with $(A_\sigma^{-1} \cdot y_\sigma^{-1})_\sigma$, which has bit complexity polynomial in $\size( (A_\sigma)_\sigma,\allowbreak (y_\sigma)_\sigma, \allowbreak \size(\tilde{\beta}))$ which is polynomial in the size of the input, as $\size(\tilde{\beta})$ must be $\poly(r,\size(\tmb),\size(y),\size((A_\sigma)_\sigma)$.

Therefore, all steps require a bit complexity at most polynomial in the size of the input and $\hkztime$. The total number of  queries to $\mathbf{O}_\subpic$ is at most $\poly(\log|\dcrk|,\log \norm(\modu))) \cdot [\rayPic^0:\subpic]$.
\end{proof}

\begin{remark} If $\modu = \OK$, we have $\norm(\modu) = \phi(\modu) = 1$, and $\Kmodumodu = \Kmodu = K^*$. In this case, $\tau \Kmodu = K^*$ 
\end{remark}

\begin{remark} The requirement $[\ldb \tau \rdb] \in \subpic$ is not necessary. The same proof and algorithm applies for arbitrary $[\ldb \tau \rdb]$; but, as $\beta \in \tau \Kmodu$, we have $[\ldb \beta \rdb] = [\ldb \tau \rdb]$. Hence, if we do not assume $[\ldb \tau \rdb] \in \subpic$, the property $[\ldb \beta \rdb] \in \subpic$ is  generally not valid anymore.
\end{remark}

\begin{figure}
\begin{algorithm} [H]
    \caption{Sampling of $\beta \in \mbb$ such that $\beta \in \tau \Kmodu$}
    \label{alg:samplerandom}
    \begin{algorithmic}[1]
    	\REQUIRE ~\\
         \begin{itemize}
          \item An LLL-reduced basis of $\OK$,
          \item A modulus $\modu = \moduz \moduinf$ of the degree $n$ number field $K$.
          \item A finite-index subgroup $\subpic \subseteq \rayPic^0$,
          \item An ideal $\mbb \in \ideals$ coprime with $\moduz$,
          \item An element
           $\ky \in \nfrstar$,
          \item An element $\tau \in \Kmodumodu$, satisfying $[\ldb \tau \rdb] \in \subpic$,
          \item An error parameter $\varepsilon > 0$.
          \item A block size parameter $\blocksize \in \Z$, with $2 \leq \blocksize \leq n$.
          \item An oracle $\mathbf{O}_\subpic$ that on input an ideal $\mc$  outputs whether $[d^0(\mc)] \in \subpic$.
          \item An element $\radpar \in \Q_{\geq 1}$.
         \end{itemize}

    	\ENSURE An element $\beta \in \mbb$ 
    	
        \STATE  Let $B, N$ be as in \Cref{thm:random-walk-prime-sampling}, put $\sd = 1/n^2$ and put $r = \radiusformula$. \label{line:alg1:BNinit}
        \STATE Sample $N$ random prime ideals $\mp_1,\ldots,\mp_N$ uniformly from the set $\{ \mp \mbox{ prime ideal of } K ~|~ N(\mp) \leq B, \mp \nmid \moduz , [d^0(\mp)] \in \subpic \}$, using the oracle $\mathbf{O}_\subpic$ and \Cref{lemma:sampleGprimes} \label{line:alg1:sampleprimes}
        \STATE  Multiply $\mbb$ by these $N$ random prime ideals $\mp_j$, obtaining $\tmb = \mbb \cdot \prod_j \mp_j$. \label{line:alg1:primes}

        \STATE Sample, using \Cref{lemma:gpv}, an approximation of the discrete Gaussian 
        \[ \ddot{a} = (\ddot{a}_\sigma)_\sigma \from \Klein_{\frac{\delta}{n} \mB_{\hyper}, \geps,\sd,0} \] 
         with error parameter $\geps := \eps/4$, where $\frac{\delta}{n} \mB_{\hyper}$ is a basis of $\dH = \frac{\delta}{n} \Z_H$ (see \Cref{notation:ZH}) and \label{line:alg1:distortion}
        \[ \delta :=  \frac{(\eps/40)^{4n^2\sd + 1} \cdot \sd }{\radpar^n \cdot e^{10n^2} \cdot |\dcrk| \cdot \sqrt{n \log(2n/(\eps/40))} }.\]
        \STATE Compute a rationally represented $(A_\sigma)_\sigma \in \nfr^0$ such that 
        \[ \max_{\sigma} |A_\sigma/e^{\nplemb^{-1} \ddot{a}_{\nu_\emb}} - 1| \leq \delta/(2n),\]
        where $\nplemb = 2$ if $\emb$ is complex and $1$ otherwise. 
        \label{line:alg1:rationalA}

        \STATE Sample, following \Cref{alg:sample_in_a_box} with input $\gamma = 0$, a uniformly random element \[  \tilde{\beta} \in 
        \big((A_\sigma \cdot y_\sigma)_\sigma
        \cdot \big[ \tmb \cap \tau \Kmodu \big] \big) \cap r \cdot |\norm(\ky)|^{1/n} \cdot \norm(\tmb)^{1/n} \cdot \Ballinf \] \label{line:alg1:sample} 
        \RETURN 
        $\beta = (A_\sigma^{-1} \cdot y_\sigma^{-1})_\sigma \cdot \tilde{\beta} \in \mb$.
       \label{line:alg1:return}
    \end{algorithmic} 
\end{algorithm}
\end{figure}

\subsection{Two help lemmas} \label{section:helplemmas}
The following two lemmas are used in the proof of
\Cref{theorem:ISmain}.
The first of these lemmas shows that part of \Cref{alg:samplerandom} consists of a random walk, whereas the second one shows that
one can replace the random walk distribution over $\rayDiv^0$ with a close distribution $\distr_U$ whose `folded analogue' $[\distr_U]$
in $\rayPic^0$ is uniform.

\begin{lemma}[The `continuous version of \Cref{alg:samplerandom}' mimicks a random walk]  \label{lemma:algo_distr} Let $\modu$ %
a modulus, let $N,B, \sd$ and $r$  as in \Cref{alg:samplerandom} and let $\Walk_\subpic = \Walk_\subpic(N,B,\sd)$ be the random walk distribution on $\subpic \subseteq \rayPic^0$ (see \Cref{def:rwdiv}). Let $\Walk_{r}$ be the distribution on $(\alpha, \ba) \in K_\R \times \rayDiv^0$ obtained by sampling $\ba \from \Walk_\subpic(N,B,\sd)  + d^0(\mbb) + \hb$ with $\hb = \Log(y)$ where $\ky \in \nfr^0$ followed by sampling
\[ \alpha \in \dExp{\ba }_\tau \cap r \Ballinf  \]   %
uniformly.
Then the pair $\big((\beta) \cdot \tmb^{-1}, \beta \big) \in \idealsmodu \times \mb$ obtained by running \Cref{alg:samplerandom} where lines \lineref{line:alg1:distortion} and \lineref{line:alg1:rationalA} are replaced by \Cref{eq:alg1change} (coined the `continuous version'), follows the same distribution as $\big( (\alpha) \cdot \dExp{-\ba}, \alpha \cdot \infpart{\Exp}(-\infpart{\ba})\big)$ with $(\alpha,\ba ) \from \Walk_{r}$.%
\end{lemma}
\begin{proof}Both the distribution $\Walk_\subpic$ and the `continuous version of \Cref{alg:samplerandom}' (where lines \lineref{line:alg1:distortion} and \lineref{line:alg1:rationalA} are replaced by \Cref{eq:alg1change}, as in the proof of \Cref{theorem:ISmain}) involve the following two random processes: picking $N$ uniformly random primes from $\{ \mp \in \idealsmodu \mbox{ prime } ~|~ \norm(\mp) \leq B  \mbox{ and } [ d^0(\mp)] \in \subpic \}$ and sampling a Gaussian $\ha = (\ha_{\pl})_\pl \in \hyper$; both with the exact same parameters. Without loss of generality, we can therefore focus on one fixed sample $\{ \mp_j ~|~  1 \leq j \leq N \}$ of primes and one fixed vector $\ha = (\ha_{\pl})_\pl \in \hyper$.
This means that we consider the fixed $\ba = \sum_{j = 1}^N d^0( \mp_j ) + \ha + d^0(\mbb) + \hb \in \rayDiv^0$ for the procedure involving $\Walk_r$ and the fixed ideal $\tmb = \mbb \prod_{j = 1}^N \mp_j$ and distortion 
$\Exp(\ha + \hb) = (e^{\nplemb^{-1} (\ha_{\pl_\emb}+\hb_{\pl_\emb})}  )_\emb$
(recall that $(\ky_\sigma)_\sigma = (e^{\nplemb^{-1} \hb_{\pl_\emb}})_\emb$; note that, by \Cref{lemma:alg2independentnorm}, we may assume that the entries of $\ky_\sigma$ are all positive) for the procedure involving the `continuous version of \Cref{alg:samplerandom}'. Then, writing $\bar{b} = \norm(\tmb)^{1/n}$ and $\tmb = \dExp{\finpart{\ba}}$, we have
\[ \dExp{\ba}_\tau =  
\Exp(\ha + \hb)
(\tmb \cap \tau \Kmodu) / \bar{b}  ~\mbox{ and }  ~ \infpart{\Exp}(\infpart{\ba}) = 
\Exp(\ha + \hb)
/ \bar{b} \]
Thus, $\alpha \dExp{-\infpart{\ba}}$ for uniformly random $\alpha \in \dExp{\ba}_\tau \cap r \Ballinf$ is distributed as
\begin{align} & \infpart{\Exp}(-\infpart{\ba}) \cdot \unif \big( \dExp{\ba}_\tau \cap r \Ballinf\big)
\\ = & \underbrace{ \Exp(-\ha - \hb) \cdot \bar{b}}_{\infpart{\Exp}(-\infpart{\ba})} \cdot \unif \Big( 
\underbrace{ 
\Exp(\ha + \hb)
(\tmb \cap \tau \Kmodu) / \bar{b}}_{\dExp{\ba}_\tau} \cap r \Ballinf \Big)
\\ = & \unif \Big(  (\tmb \cap \tau \Kmodu) \cap 
\Exp(-\ha- \hb)
\cdot \bar{b} \cdot r \Ballinf \Big)
 \label{eq:enddistribution}
 \end{align}
 which is exactly %
 the distribution of $\beta \in \tmb$ in \Cref{alg:samplerandom} for fixed $\tmb$ and $(\ha_{\pl})_\pl$. It follows that 
 \[(\alpha) \dExp{-\ba} = (\alpha)\infpart{\Exp}(-\infpart{\ba}) \finpart{\Exp}(-\finpart{\ba}) = (\alpha)\infpart{\Exp}(-\infpart{\ba})/ \tmb\]
 is distributed as $(\beta) \tmb^{-1}$, which finishes the proof.
\end{proof}

We will need the following `lifting' lemma, which states that the random walk distribution $\Walk_{\subpic}(N,B,\sd)$ over $\rayDiv^0$ is close to a distribution $\distr$ such that the class $[\distr]$ is uniform over a coset $\subpic \subseteq \rayPic^0$.
\begin{lemma}[Lifting property of distributions]
\label{lemma:lifting} Let $\subpic \subseteq \rayPic^0$ be a finite-index subgroup and suppose that a distribution $\distr:\rayDiv^0 \rightarrow \R_{\geq 0}$ satisfies $\| [\distr] - \unif(\subpic+ [\bb])\|_{1} < \eps$ for some coset $\subpic + [\bb]$, where $\bb \in \rayPic^0$.
Then there exists a `lifted' distribution $\distr_U:\rayDiv \rightarrow \R_{\geq 0}$ such that $[\distr_U] = \unif(\subpic + [\bb])$ and $\|\distr - \distr_U\|_{1} < \eps$.
\end{lemma}

\begin{proof} %
Put
 \[ \distr_U(\ba) = \left \{ \begin{matrix} \frac{1}{\vol(\subpic) } \cdot \frac{\distr(\ba)}{[\distr]([\ba])} & \mbox{ if } [\distr]([\ba]) \neq 0 \\
 u  & \mbox{ otherwise }    \end{matrix} \right. , \]
 for some $u:\rayDiv^0 \rightarrow \R_{\geq 0}$ that satisfies $[u] =  \frac{1}{\vol(\subpic)} \cdot 1_{\subpic} \in L_1(\rayPic^0)$. Then, one can check that $[\distr_U] = \frac{1}{\vol(\subpic)} \cdot 1_{\subpic}$ is uniform on $\subpic + [\bb]$. Furthermore, writing $F$ for a fundamental domain in $\rayDiv^0$ for $\rayPic^0$, we have
  \begin{align*} \| \distr - \distr_U \|_{1} &= \int_{\ba \in F} \sum_{\alpha \in \Kmodu/\muKmodu} | \distr(\ba + \ldb \alpha \rdb) - \distr_U(\ba + \ldb \alpha \rdb) |  d \ba \\
  & = \int_{[\ba] \in \rayPic^0}\left| [\distr]([\ba]) - \frac{1_{\subpic + [\bb]}([\ba])}{\vol(\subpic)} \right | d([\ba]) = \| [\distr] - \unif(\subpic + [\bb])\|_1 \leq \eps .
  \end{align*}
  The first equation holds by definition, the second equation by the fact that the sign of $(\distr(\ba + \ldb \alpha \rdb) - \distr_U(\ba + \ldb \alpha \rdb))$ depends (by construction) solely on the coset $[\ba]$.%
\end{proof}

\subsection{Statistical distance between the `continuous variant' and the ordinary variant of \Cref{alg:samplerandom}}
\label{subsection:statdiff}
\newcommand{\tunif}{\tilde{\unif}}
\newcommand{\ssig}{\varsigma}
In this subsection we show that there is only a small statistical distance between the output distributions 
of the `continuous variant of \Cref{alg:samplerandom}' (where lines \lineref{line:alg1:distortion} and \lineref{line:alg1:rationalA} are 
replaced by \Cref{eq:alg1change}) and the ordinary variant of \Cref{alg:samplerandom} (with no changes). 
This result is used in the proof of part (A) of \Cref{theorem:ISmain} and in the proof that \Cref{alg:samplerandom} is almost-Lipschitz in the later \Cref{lemma:idealsamplinglipschitz}.

To prove the closeness of these two distributions, we fix all the input parameters and compare the two output distributions; 
where we note that we may assume that $\ky \in \nfr^0$ and that the entries of $\ky$ are positive, by \Cref{lemma:alg2independentnorm}.
By the law of total probability, it is enough to show the closeness of these two distributions for a fixed sample of the 
primes $\{\mp_j \}$ (as in line \lineref{line:alg1:primes}), and hence a fixed ideal $\tmb$. Indeed, the difference between
the two distributions can be attributed solely to the way the Gaussian $a = (\ha_\pl)_\pl \in \hyper$ (discrete or continuous) is sampled and the influence 
it bears on the uniform distribution of $\beta$.

Though the input of \Cref{alg:samplerandom} is given in terms of an
ideal $\mb$ and an element $\ky = \dExp{(\hb_{\pl})_\pl} \in \nfr^0$ with $\hb \in \hyper$
we will often consider the
divisor $\ba = d^0(\mb) + \hb$ as input instead.
As this is a one-to-one correspondence, and is used only theoretically (for cleanness of notation), 
no harm is done.

By a similar computation as in the proof of \Cref{lemma:algo_distr}, one can deduce that the output $\beta$ is distributed as (denoting $\hb = \Log(\ky) \in \hyper$)
\begin{align}
&  \unif \Big((\tmb \cap \tau \Kmodu) \cap  
\Exp( - \Log( (A_\sigma)_\sigma) - b)
\cdot \norm(\tmb)^{1/n} \cdot r \Ballinf \Big) \\ 
 = & y^{-1} \cdot \unif \Big( (y \cdot (\tmb \cap \tau \Kmodu)) \cap  
\Exp( - \Log( (A_\sigma)_\sigma)) 
 \cdot \norm(\tmb)^{1/n} \cdot r \Ballinf \Big) \\ 
  = & y^{-1} \cdot \unif \Big( (y \cdot (\tmb \moduz + \gamma)) \cap
\Exp( - \Log( (A_\sigma)_\sigma)) 
  \cdot \norm(\tmb)^{1/n} \cdot r \Ballinftau \Big) \label{eq:firstdistro}
\end{align}
both in the `continuous version' as in the ordinary version (since the definition of $(A_\sigma)_\sigma \in \nfr$ changes accordingly). Equation \Cref{eq:firstdistro} follows 
from \Cref{lemma:modulusbox}; where $\Ballinftau = \Ballinf \cap \tau \Kmodumoduz$ (see \Cref{notation:taupositive}). Note that from \Cref{eq:firstdistro} follows that $(A_\sigma)_\sigma$ has only an influence on the shape (`skewness') of the box $\norm(\tmb)^{1/n} \cdot r \Ballinftau$
where the uniform sampling takes place.
All other things are equal; hence we introduce the following notation.

\newcommand{\tildota}{\tilde{\ddot{a}}}
\newcommand{\tildotx}{\tilde{\ddot{\vx}}}
\newcommand{\hyperx}{a}
\newcommand{\hyperdx}{\ddot{\hyperx}}
\newcommand{\tildothyperx}{\tilde{\ddot{\hyperx}}}

\begin{notation} \label{notation:easethelemma} For $\hyperx \in \hyper$, we denote $\unif_\hyperx$ for the uniform distribution over the set 
$(y (\tmb \moduz + \gamma)) \cap \Exp(\hyperx) \cdot \norm(\tmb)^{1/n} \cdot r \Ballinftau$, where $r = \radiusformula$ with $\radpar \in \Q_{\geq 1}$; as in \Cref{alg:samplerandom}, i.e., 
\[  \unif_\hyperx := \unif \Big( (y \cdot (\tmb \moduz + \gamma)) \cap  \Exp(\hyperx) \cdot \norm(\tmb)^{1/n} \cdot r \Ballinftau \Big) =  \unif \Big( (y \cdot (\tmb \moduz + \gamma)) \cap \Exp(\hyperx) \cdot r_0 \Ballinftau \Big). \]
where we write $r_0 := \norm(\tmb)^{1/n} \cdot r = \radiusformulamacro{\tmb\moduz }$.
Note that, for $y\beta \in y(\tmb\moduz + \gamma) \cap \taunfrm$, we have that 
\[ \unif_\hyperx[y \beta] > 0 \Longleftrightarrow  \log |\embn(y\beta)| \leq \log r_0 + \npl^{-1} \hyperx_\pl \mbox{ for all } \pl \] %
where $\npl = 2$ if $\pl$ is complex and $1$ otherwise.
We denote 
\[ S_{y\beta} = \{ \hyperx \in H ~|~ \npl^{-1} \hyperx_\pl \geq \log |\embn(y\beta)| - \log r_0 \mbox{ for all } \pl \}. \]  I.e., we have, for $y\beta \in y(\tmb\moduz + \gamma) \cap \taunfrm$ and $\hyperx \in H$,
\[ \hyperx \in S_{y\beta}  \Leftrightarrow  y\beta \in  \Exp(\hyperx) \cdot r_0 \Ballinftau \Leftrightarrow \unif_\hyperx[\beta] > 0. \]
We denote $\partial S_{y\beta}$ for the boundary of $S_{y\beta}$; note that this boundary is a subset of $\dimh = \dim(\hyper)$ hyperplanes in $\hyper$ of dimension $\dimh-1$.

In the ordinary variant of \Cref{alg:samplerandom}, an approximation $(A_\sigma)_\sigma$ of 
$(e^{\nplemb^{-1} \ddot{a}_{\pl_\emb}})_\emb$ 
with $\ddot{a} \in \dH$ is computed, hence leading to a slight error. For ease of notation we denote $(A_\sigma)_\sigma = \Exp(\tildota) =  (e^{\nplemb^{-1} \tildota_\sigma})_\sigma$. By construction (see line \lineref{line:alg1:distortion}) we have (for $\delta < 1/2$)
\begin{align} \label{eq:closetildota}
 \| \ddot{a} - \tildota \|_2 & \leq n \max_{\sigma} | \log(A_\sigma) - \log(e^{\nplemb^{-1} \ddot{a}_{\pl_\sigma}})| = n \max_{\pl} |\log(A_{\embn}/  e^{\nplemb^{-1} \cdot \ddot{a}_\pl})| \\ & \leq n \max_{\sigma} \log|1 - \delta| \leq 2n \cdot \delta/(2n) = \delta.
\end{align}
\end{notation}

So, the ordinary variant of \Cref{alg:samplerandom}, with fixed choices of primes $\{ \mp_j \}$ leading to $\tmb$, has output distribution 
\[ \sum_{\hyperdx \in \dH} \unif_{\tildothyperx} \Gaussian_{\dH,\sd}(\hyperdx) =  \sum_{\hyperdx \in \dH} \unif_{\tildothyperx}(\placeholder) \cdot \Gaussian_{\dH,\sd}(\hyperdx), \]
whereas the `continuous variant of \Cref{alg:samplerandom}' (with the same choices of $\{ \mp_j \}$), has output distribution 
\[ \int_{\hyperx \in \hyper} \unif_{\hyperx} \Gaussian_{\hyper,\sd}(\hyperx) =  \int_{\hyperx \in \hyper} \unif_{\hyperx}(\placeholder) \cdot \Gaussian_{\hyper,\sd}(\hyperx). \]

\begin{lemma} \label{lemma:almostalwaysclose} Using  \Cref{notation:easethelemma}, let $\delta > 0$ be a distance parameter and denote $B_{\delta}(\hyperx) \subseteq \hyper$ for the $2$-ball of radius $\delta$ around $\hyperx \in \hyper$ (where we denote $B_{\delta} := B_{\delta}(0)$).
Then, for $\hyperx \notin \bigcup_{y\beta \in y (\tmb \moduz + \gamma)} (\partial S_{y\beta} \minksum B_\delta)$ and $\hyperx' \in B_{\delta}(\hyperx)$, 
\[ \unif_\hyperx = \unif_{\hyperx'} . \]
\end{lemma}
\begin{proof} For $\hyperx \notin \partial S_{y\beta} \minksum B_\delta$ holds that $B_\delta(\hyperx) \cap \partial S_{y\beta} = \emptyset$. Hence for all $\hyperx' \in B_{\delta}$, the status of $y\beta \in e^{\hyperx + \hyperx'} \cdot  \norm(\tmb)^{1/n} \cdot r_0   \Ballinftau$ (where $r_0 = \norm(\tmb)^{1/n} \cdot r$) does not change. Hence, if $\hyperx \notin \bigcup_{y\beta \in y (\tmb \moduz + \gamma)} (\partial S_{y\beta} \minksum B_\delta)$, the entire distribution $\unif_{\hyperx- \hyperx'} = \unif_{\hyperx}$ is constant for $\hyperx' \in B_{\delta}$. This concludes the proof.
\end{proof}

\begin{proposition} \label{lemma:closeness_continuous}  Using the notation of \Cref{notation:easethelemma}, let  $1/(2n) > \eps>0$ be an error parameter. Let $\sd > 0$ be a Gaussian parameter and let $\dH \subseteq \hyper$ be a full-rank lattice whose Voronoi domain $F$ satisfies
\begin{equation} \delta := \max_{\hyperx \in F} \|\hyperx\|  \leq \frac{\eps^{4n^2\sd + 1} \cdot \sd }{\radpar^n \cdot e^{10n^2} \cdot |\dcrk| \cdot \sqrt{n \log(2n/\eps)} }  \label{eq:deltainstantiation} \end{equation}
Then 
\[ \left \| \int_{\hyperx \in H} \unif_x \Gaussian_{H,\sd}(\hyperx) d\hyperx -  \sum_{\hyperdx \in \dH} \unif_{\tildothyperx} \Gaussian_{\dH,\sd}(\hyperdx)    \right\|_1 \leq 9 \eps \]
\end{proposition}
\begin{proof} Writing $F$ for the Voronoi fundamental domain of the lattice $\dH$, using the identity $AB - A'B' = (A - A')B + A'(B - B')$, the triangle inequality and the fact that $\| \unif_{\tildothyperx}\|_1 = 1$, we obtain
 
\begin{align} & \left \| \int_{\hyperx \in H} \unif_x \Gaussian_{H,\sd}(\hyperx) d\hyperx -  \sum_{\hyperdx \in \dH} \unif_{\tildothyperx} \Gaussian_{\dH,\sd}(\hyperdx)    \right\|_1 \nonumber \\ 
\leq & \left \| \int_{\hyperx \in F} \sum_{\hyperdx \in \dH} \unif_{\hyperdx + \hyperx} \Gaussian_{H,\sd}(\hyperdx  +\hyperx) -  \unif_{\tildothyperx} \frac{\Gaussian_{\dH,\sd}(\hyperdx)}{\vol(F)} d\hyperx    \right\|_1 
\nonumber \\ 
\leq &  \underbrace{ \int_{\hyperx \in F} \sum_{\hyperdx \in \dH} \|\unif_{\hyperdx + \hyperx} - \unif_{\tildothyperx} \|_1 \Gaussian_{H,\sd}(\hyperdx  +\hyperx) d \hyperx}_{(I)} + \underbrace{\int_{\hyperx \in F} \sum_{\hyperdx \in \dH} \left| \Gaussian_{H,\sd}(\hyperdx  +\hyperx)-  \frac{\Gaussian_{\dH,\sd}(\hyperdx)}{\vol(F)} \right| d\hyperx}_{(II)}. 
\label{eq:splitboundIII} \end{align}
We will bound the parts (I) and (II) separately, starting with (I). \\
\textbf{Bound on (I)}.\\%
In the following sequence of inequalities we use respectively \Cref{lemma:almostalwaysclose} (in combination with \Cref{eq:closetildota}),  $\cov(\dH) = \max_{x \in F} \|x\| = \delta$ and the fact that total variation distances are always bounded by $2$. Also, we use a Gaussian tail bound (see \Cref{lemma:bound-gaussian}) with radius $R := \sd \sqrt{n \log(2n/\eps)} > 2\delta >0$
(since we can deduce that $\delta < \sd/2$ from the assumption in \Cref{eq:deltainstantiation}).
Combining this, writing $T_{\delta} = \bigcup_{y\beta \in y (\tmb \moduz + \gamma)} (\partial S_{y\beta} \minksum B_\delta)$, we obtain 
\begin{align}  (I) & \leq \int_{\hyperx \in F}  \sum_{\substack{\hyperdx \in \dH \\ \hyperdx \in T_{\delta} }} \|\unif_{\hyperdx + \hyperx} - \unif_{\tildothyperx} \|_1 \Gaussian_{H,\sd}(\hyperdx  +\hyperx) d \hyperx 
 \leq  2\int_{\hyperx \in F} \sum_{\substack{\hyperdx \in \dH \\ \hyperdx \in T_{\delta} }} \Gaussian_{H,\sd}(\hyperdx  +\hyperx) d \hyperx \\
& \leq 2\int_{\hyperx \in T_{2\delta} }  \Gaussian_{H,\sd}(\hyperx) d \hyperx \leq \eps + 2\int_{\substack{\hyperx \in T_{2\delta} \\ \|\hyperx\| < R} }  \Gaussian_{H,\sd}(\hyperx) d \hyperx \label{eq:epsboundgauss}
\end{align}
The third inequality follows from the observation that $\hyperdx + \hyperx \in \partial S_{y\beta} \minksum B_{2\delta}$ if $\hyperdx \in \partial S_{y\beta} \minksum B_{\delta}$ (since $\hyperx \in F \subseteq B_{\delta}$).

If $\hyperx \in \partial S_{y\beta} \minksum B_{2\delta}$ and $\|\hyperx\| < R$, then (for all $\pl$) $\npl^{-1} \hyperx_\pl \geq \log|\embn(y\beta)| -\log r_0  - R$ (since $R > 2\delta$), and hence $\log |\embn(y\beta)| \leq \npl^{-1} \hyperx_\pl + \log r_0 + R$, i.e., 
$\max_\sigma |\sigma(y\beta)| \leq r_0 \cdot \exp( \max_\pl \hyperx_\pl ) e^{R} \leq r_0 e^{2R}$ (since we assumed $\|\hyperx\|<R$). Recall, by definition, $r_0 = \norm(\tmb)^{1/n} \cdot r$. Hence, by the condition $\|\hyperx \| < R$, we may take the union $T_{2\delta} = \bigcup_{y\beta \in y (\tmb \moduz + \gamma)} (\partial S_{y\beta} \minksum B_{2\delta})$ only over those $y\beta$ satisfying $y\beta \in r_0 e^{2R} \Ballinf$. 

We concentrate for the moment on the Gaussian weight on a single $S_{y\beta} \minksum B_{2\delta}$ in the union of $T_{2 \delta}$.
Note that $\partial S_{y\beta}$ lies in a union of $\dimh = \dim(\hyper) \leq n$ hyperplanes in $\hyper$. Denoting $P$ for a hyperplane in $H$ of dimension $\dim(H)-1$ that goes through the origin, we have, for all $y\beta \in y (\tmb \moduz + \gamma)$, 
\begin{equation} \label{eq:boundonsingleintegral} \int_{\hyperx \in \partial S_{y\beta} \minksum B_{2\delta}} \Gaussian_{H,\sd}(\hyperx) d\hyperx \leq n 
 \int_{\hyperx \in P \minksum B_{2\delta}} \Gaussian_{H,\sd}(\hyperx) d\hyperx \leq n \cdot \mbox{erf}(2\delta/\sd) \leq 4n \cdot \delta/\sd,
\end{equation}
where $\mbox{erf}$ is the standard error function of the standard normal distribution that satisfies $\mbox{erf}(x) \leq \frac{2}{\sqrt{\pi}} x$. The inequality in \Cref{eq:boundonsingleintegral} holds because the hyperplane $P \subseteq H$ going through the origin is the one where $P \minksum B_{2\delta}$ has the largest Gaussian weight.

Writing $e^{2nR} \cdot N_0$ for an upper bound for $|(y (\tmb \moduz + \gamma)) \cap r_0 e^{2R} \Ballinf|$ for which we will instantiate $N_0$ later, we combine \Cref{eq:epsboundgauss} and \Cref{eq:boundonsingleintegral}, and the assumption  in \Cref{eq:deltainstantiation} on $\delta$, to obtain
\begin{equation} (I) \leq \eps + 2 \int_{\substack{\hyperx \in \bigcup_{\beta \in \mb}(\partial S_\beta \minksum B_{2\delta})\\ \|\hyperx \| < R}} \Gaussian_{H,\sd}(\hyperx) d\hyperx  \leq \eps +  \frac{8n \cdot e^{2nR} \cdot N_{0}\cdot \delta}{\sd} \leq 2 \eps . \label{eq:halfIbound}\end{equation}
It remains to show that $\delta$ as in the assumptions of this proposition indeed satisfies $\delta \leq \eps \cdot \sd \cdot (8n)^{-1} e^{-2nR}\cdot N_{0}^{-1}$. We have $\exp(-2nR) = \exp(-2 n^{3/2} \sd \sqrt{\log(2n/\eps)}) \geq \exp(-4n^2 \sd \log(1/\eps)) = \eps^{4n^2 \sd}$ (where we use that $\eps < 1/(2n)$). Also, by a similar reasoning as in \Cref{subsubsec:estimateidealcapr}, using $r_0 = \norm(\tmb)^{1/n} r$, and the definition $r = \radiusformula$, we have 
\begin{align*}  N_0 & \in [e^{-1/4},e^{1/4}] \cdot \frac{\norm(\tmb) \cdot 2^{\rem} \cdot (2\pi)^{\cem} \cdot r^n}{\norm(\tmb\moduz) \cdot \sqrt{|\dcrk|}} \\ & \in [e^{-1/4},e^{1/4}] \cdot  \radiusformulaconstant^n \cdot \radpar^n \cdot 2^{\rem} \cdot (2\pi)^{\cem} \cdot \blocksize^{2n^2/\blocksize} \cdot n^{7n/2} \cdot |\dcrk|  \end{align*}
And hence $N_0 \leq \omega^n \cdot e^{6n^2} \cdot |\dcrk|$ since $\radiusformulaconstant^n \cdot \blocksize^{2n^2/\blocksize} \cdot n^{7n/2} \leq e^{n \log(\radiusformulaconstant) + 2n^2/e + 7n\log(n)/2} \leq e^{6n^2}$ (because $\log(\radiusformulaconstant) + (2/e) \cdot n + 7\log(n)/2 \leq 6n$ for all $n > 0$). 
Therefore, indeed,
\begin{align} \nonumber \delta &\leq \frac{\eps^{4n^2\sd + 1} \cdot \sd }{\radpar^n \cdot e^{10n^2} \cdot |\dcrk| \cdot \sqrt{n \log(2n/\eps)} } \\ 
& = \eps \cdot \sd  \cdot e^{-4n^2} \cdot \eps^{4n^2 \sd} \cdot \radpar^{-n} \cdot e^{-6n^2} \cdot |\dcrk|^{-1}  \cdot  (n \log(2n/\eps))^{-1/2} \label{eq:scrutinizedelta} \\ 
& \leq  \eps \cdot \sd \cdot (8n)^{-1} e^{-2nR} \cdot N_0^{-1} \nonumber
\end{align}
\textbf{Bound on (II)} \\ 
We still write $F$ for the Voronoi domain of $\dH$. We apply \Cref{lemma:closegaussians} to obtain a bound on (II). Indeed, by putting $\Lambda = \dH$ and noting that 
$\lambda_n(\dH) \leq 2 \delta = 2\cov_2(\dH)$ (see \Cref{lemma:bl-leq-2cov}) %
we surely have 
\[  \lambda_n(\dH) \leq 2 \delta \leq 2 \cdot \frac{\eps^{4n^2\sd + 1} \cdot \sd }{e^{10n^2} \cdot |\dcrk| \cdot \sqrt{n \log(2n/\eps)} } \leq \frac{\eps \cdot \sd }{5 \pi n \cdot 2 \cdot \sqrt{n \log(2n/\eps)} }\]
where we use that $e^{10n^2} > 20 \pi n$ for $n \geq 1$. Now, since $2 \cdot \sqrt{n \log(2n/\eps)} > \sqrt{n \log(4n/\eps)}$ the parameter $\sd$ satisfies the assumptions of \Cref{lemma:closegaussians}, hence we can bound (II) by $7\eps$.

Combining \Cref{eq:splitboundIII} and the bound on (II), we obtain
\[ \left \| \int_{\hyperx \in H} \unif_x \Gaussian_{H,\sd}(\hyperx) d\hyperx -  \sum_{\hyperdx \in \dH} \unif_{\hyperdx} \Gaussian_{\dH,\sd}(\hyperdx)    \right\|_1 \leq 9 \eps, \]
as required.

\end{proof}

\begin{lemma} \label{lemma:closegaussians} Let $\Lambda \subseteq V$ be a full-rank lattice,  let $F$ be the Voronoi domain of $\Lambda$, let $\eps < 1/2$ be some error parameter and let 
 $\sd > 5 \pi \cdot \sqrt{n\log(4n/\eps)} \cdot \eps^{-1}  \cdot n \cdot \lambda_n(\Lambda)$. Then
\[ \int_{x \in F} \sum_{\ell \in \Lambda} \big|\Gaussian_{V,\sd}(\ell + x) - |F|^{-1} \Gaussian_{\Lambda,\sd}(\ell)\big| dx \leq 7 \eps \]
\end{lemma}
\begin{proof} Note that $\cov_2(\Lambda) = \max_{x \in F} \|x\| \leq n \cdot \lambda_n(\Lambda)/2$ %
(see \Cref{lemma:bl-leq-2cov}). 
Using smoothing arguments, (using that $\sd > \sqrt{\log(2n(1+1/\eps))/\pi}\cdot \lambda_n(\Lambda)$) writing $d = \dim(V)$, 
we see that (see \Cref{lemma:total-gaussian-weight})
\[ |F|^{-1} \cdot \Gaussian_{\Lambda,\sd}(\ell) = \gaussian_{\sd}(\ell)\cdot \frac{1}{|F| \cdot \gaussian_{\sd}(\Lambda)} \in [1-2\eps,1+2\eps] \cdot \sd^{-d} \cdot \gaussian_{\sd}(\ell) = [1-2\eps,1+2\eps] \cdot\Gaussian_{V,\sd}(\ell) , \] 
where we used that $1 - 2 \eps \leq (1 + \eps)^{-1} \leq (1-\eps)^{-1} \leq 1 + 2 \eps$ for $\eps < 1/2$. Using this, in combination with the tail bound (both using \Cref{coro:bana} and \Cref{lemma:bound-gaussian}) with tail cut parameter $R:= \sd \sqrt{n \log(2n/\eps)} > \cov_2(\Lambda)$, we obtain 
\begin{align} 
&\int_{x \in F} \sum_{\ell \in \Lambda} \big|\Gaussian_{V,\sd}(\ell + x) - |F|^{-1} \Gaussian_{\Lambda,\sd}(\ell)\big| dx  \\
\leq 2\eps + &\int_{x \in F} \sum_{\substack{\ell \in \Lambda \\ \|\ell\| \leq 2R}} \left| \Gaussian_{V,\sd}(\ell + x)-  |F|^{-1}\cdot\Gaussian_{\Lambda,\sd}(\ell) \right| dx \nonumber \\
\leq 2\eps + 4\eps + & \int_{\hyperx \in F}  \sum_{\substack{\ell \in \Lambda \\ \|\ell\| \leq 2R}} \Big|  \Gaussian_{V,\sd}(\ell+x)  - \Gaussian_{V,\sd}(\ell) \Big| dx \nonumber \\ 
\leq   6\eps + &\int_{x \in F}  \sum_{\substack{\ell \in \Lambda \\ \|\ell\| \leq 2R}}   \Gaussian_{V,\sd}(\ell + x)|1  - e^{5\pi \cdot \cov_2(\Lambda) \cdot R /\sd^2}| dx  \label{eq:computationgauss} \\ 
 & \leq 6\eps + \frac{10\pi \cov_2(\Lambda) R}{\sd^2} \leq 6\eps + \frac{5\pi n \lambda_n(\Lambda) R}{\sd^2} < 7 \eps.  
\label{eq:prestatgausbetter}
\end{align}
The inequality in \Cref{eq:computationgauss} follows from the fact that 
\begin{align*} -\frac{\sd^2}{\pi} \log \gaussian_{\sd}(\ell + x) &= \| \ell + x\|^2 =  \|\ell\|^2 + 2\langle \ell, x\rangle + \|x\|^2 \\ & \in [\|\ell \|^2 - 4R \cov_2(\Lambda), \|\ell\|^2+ 5 R \cov_2(\Lambda)] \end{align*}  which follows by Cauchy-Schwarz and the fact that $\max_{x \in F} \|x\| \leq \cov_2(\Lambda) \leq R$. The last inequalities in \Cref{eq:prestatgausbetter} follow from $|1 - e^{x}|\leq 2x$ for $x < 1/2$, and 
\[ \frac{5\pi n \lambda_n(\Lambda) R}{\sd^2} = \frac{5\pi n \lambda_n(\Lambda) \sqrt{n \log(2n/\eps)}}{\sd} < \eps. \]
\end{proof}

\subsection{The sampling theorem in arbitrary orders of a number field} \label{sec:arbitraryorder}
The sampling theorem (\Cref{theorem:ISmain}) is not restricted to the ring of integers $\OK$
of a number field $K$. Below we give a quick sketch how to amend the reasoning as to
obtain a similar result for arbitrary orders within $K$.

 Write $\order \subseteq \OK$ for an order,  $\condord = \{ x \in K ~|~  x \OK \subseteq \order\}$ for its conductor
and $\discord$ for its discriminant. Then we can in fact apply the sampling algorithm (\Cref{alg:samplerandom}) for $R$ in place of $\OK$ if $\condord \mid \modu$.
Indeed, for all ideals and elements involved are assumed to be coprime with $\condord$, there is a one-to-one correspondence between ideal arithmetic
in $\order$ and $\OK$, by the maps $\mathfrak{A} \mapsto \mathfrak{A} \OK$ and $\ma \mapsto \ma \cap \order$ for ideals $\mathfrak{A}$ of $\order$ and
ideals $\ma$ of $\OK$. These maps are compatible with ideal multiplication, as long as all ideals involved are coprime to $\condord$ (e.g., combine \cite[Chapter I, Proposition 12.10]{neukirch2013algebraic} and \cite[Chapter 1, Extension and Contraction]{atiyah69}).

\begin{lemma} Let $\order \subseteq \OK$ an order of a number field $K$ with conductor $\condord$. Suppose that $\condord \mid \modu$ and that $\tau \in \order$. Then replacing every occurring ideal $\ma$ (that is, $\mb,\mp_j,\moduz$) by the ideal $\mathfrak{A} = \ma \cap \order$ in \Cref{alg:samplerandom} and using only arithmetic in $\order$, does not change the output distribution.
\end{lemma}

\begin{proof}
We go through every step of \Cref{alg:samplerandom}, where we will use that $\mb$ and the primes $\mp_j$ are coprime with $\moduz$. In step \lineref{line:alg1:primes} we multiply $\mathfrak{B} = \mb \cap \order$
by random prime ideals $\mathfrak{P}_j = \mp_j \cap \order$ satisfying $[d^0(\mp_j)] \in \subpic$. Hence we obtain $\bar{\mathfrak{B}} = \tmb \cap \order$ with $\tmb = \mbb \cdot \prod_j \mp_j$. The Gaussian distortion in step \lineref{line:alg1:distortion}, is not affected by the specific order $\order$.

In step \lineref{line:alg1:sample}, we have
$\tmb \cap \tau \Kmodu = (\tmb \moduz + \tilde{\tau}) \cap \taunfrm$ for $\tilde{\tau} \in \tmb$ satisfying $\tilde{\tau} \equiv \tau$ mod $\moduz$ (see \Cref{lemma:modulusbox}).
Since $\condord \mid \moduz$, we have that $\mathfrak{M}_0 := \moduz \cap \order = \moduz$, by the very definition of the conductor; it is the largest ideal of $\OK$ that is also an ideal in $\order$. Therefore, $\tilde{\tau} \in \tau + \moduz \subseteq \order$, as $\tau$ is assumed to be in $\order$.

So, $\tmb \cap \tau \Kmodu \cap \order = (\tmb \moduz + \tilde{\tau}) \cap \order$. Since $\condord \mid \moduz$, we have that $\bar{\mathfrak{B}} \mathfrak{M}_0 = (\tmb\moduz) \cap \order = \tmb\moduz$. Hence, $\tmb \moduz + \tilde{\tau} = \bar{\mathfrak{B}} \mathfrak{M}_0  + \tilde{\tau} \subseteq \order$.

Then it follows that the sampling an output distribution of steps \lineref{line:alg1:sample} and \lineref{line:alg1:return} are the same as in the original algorithm.
\end{proof}

\begin{remark} Note that, if only a sub-order $\order \subseteq \OK$ is known, the call to \Cref{alg:sample_in_a_box} in line \lineref{line:alg1:sample} of \Cref{alg:samplerandom} cannot be done with $r = \radiusformula$. Instead, the larger $r = \radiusformulaR$ must then be used, where $|\Delta(\order)| = [\OK:\order]^2 |\dcrk|$ is the discriminant of $\order$.
 
\end{remark}

%% file: ideal_sampling/A07-properties-of-sampling-algorithm.tex
\section{Properties of the ideal sampling algorithm}

\label{section:propertiessampling}
\subsection{Introduction}
In this section we will treat several important properties
of the ideal sampling algorithm (\Cref{alg:samplerandom}). These
properties are phrased in terms of the \emph{output distribution} of \Cref{alg:samplerandom},
that is, the distribution of $\beta \in K$ on line 5.
This distribution has three important properties: having the \emph{shifting property}, being \emph{bounded}
and being \emph{almost Lipschitz continuous}.

The boundedness and the shifting property are useful to show certain
\emph{randomness properties} of the output distribution of \Cref{alg:samplerandom} on a random input.
In \Cref{part2} of this paper, we will use
\Cref{alg:samplerandom}
on a `Gaussian' input
 to rigorously
compute $\mS$-unit groups. For this, these two properties will be of fundamental
importance to show that sampling sufficiently many `relations' will eventually
lead to the full $\mS$-unit group.

The (almost-) Lipschitz continuous property is there
to show that rounding or slightly disturbing the input of the ideal sampling algorithm
does not impact the output of this algorithm significantly
(this concerns the variable $\Log(y) = \hb = (\hb_\pl)_\pl \in \hyper$).
From this one can conclude that rounding or finite precision issues do not
play a role at all for the input of \Cref{alg:samplerandom}.

By \Cref{lemma:alg2independentnorm}, the output distribution of \Cref{alg:samplerandom}
is independent on the norm of $\ky \in \nfrstar$ and the signs of the entries. Hence, in theoretical arguments, we are always
allowed to replace $\ky$ by $\ky^0 = \ky/|\norm(\ky)|^{1/n} \in \nfr^0$ and to assume that the entries of $\ky^0$ are positive. In the actual running algorithm
we will not do this, because in general $\ky^0$ does not consist of rational numbers.

\subsubsection*{The shifting property}
The shifting property of \Cref{alg:samplerandom} relates a certain change
in the input to a similar change in the output. For this property to be phrased succinctly
it is useful to consider the input of
the algorithm
to be the zero-degree Arakelov divisor $\ba = d^0(\mb) + \hb$ (one retrieves
$\mb = \dExp{\finpart{\ba}}$ and $\hb = \ba - d^0(\mb)$, where $\hb = \Log(y)$; here we use $\hyper \subseteq \rayDiv^0$).

Shifting the input $\ba$ by a principal divisor, yielding $\ba + \ldb \alpha \rdb$ (with $\alpha \in \Kmodu$),
has as a result that the output distribution of \Cref{alg:samplerandom} is `multiplied' by $\alpha$.
More precisely, the probability of sampling $\eta \in K$ in the ideal sampling algorithm on input divisor $\ba$ is
precisely the same as the probability of sampling $\alpha \cdot \eta$ on input $\ba + \ldb \alpha \rdb$.
Writing $\distr_{\ba}$ for the output distribution of \Cref{alg:samplerandom} (see \Cref{notation:distribution}) this can be more succinctly written as:
\[ \distr_{\ba}( \eta ) =  \distr_{\ba + \ldb \alpha \rdb}( \alpha \cdot \eta ).\]
This is called the `shifting' property because
a shift of the input divisor by a principal divisor $\ldb \alpha \rdb$ causes
the output distribution to `shift' (multiplicatively) by the same element $\alpha$.

This property of \Cref{alg:samplerandom} is proven in \Cref{section:shifting}.

\subsubsection*{Boundedness}
The output distribution of \Cref{alg:samplerandom} being bounded means here that the output
cannot exceed a certain size.
This property is almost immediate by the
fact that the output element is sampled in a (slightly) distorted box, where the distortion is dictated 
by a truncated discrete Gaussian (see \Cref{lemma:gpv}). Exactly the dimensions of this box
give the upper bound on the size of the elements; this is proven in \Cref{section:bounded}.

\subsubsection*{Lipschitz continuous}
The output distribution being almost-\emph{Lipschitz} is useful when one considers the input
to be discretized or rounded. Again, using the notation $\distr_{\ba}$ for
the output distribution of \Cref{alg:samplerandom} on input divisor $\ba$ (see \Cref{notation:distribution}),
being almost Lipschitz means that the total variation distance between $\distr_{\ba}$ and $\distr_{\ba + \ha}$ with $\ha = (\ha_\pl)_\pl \in \hyper$
is linearly upper bounded by $\|\ha\|$, allowing some slack by adding a small error. More precisely, there exist $L \in \R_{>0}$ and a small $\eta \in [0,1)$ such that for all $\ba,\ha \in \hyper$,
\[  \|\distr_{\ba + \ha} -  \distr_{\ba} \|_1 \leq L \cdot \| \ha \| + \eta. \]
This fact is proven in \Cref{section:lipschitz}.

\subsection{Preliminaries}
The input of \Cref{alg:samplerandom} is given in terms of an
ideal $\mb$ and an element $\ky = \dExp{(\hb_{\pl})_\pl} \in \nfr^0$ with $\hb \in \hyper$. In this section
about properties of this algorithm, it is cleaner to use an Arakelov
divisor $\ba = d^0(\mb) + \hb$ as input instead\footnote{This is interchangeable, as the ideal $\mb = \dExp{\finpart{\ba}}$ and element $\hb = \ba - d^0(\mb)$
are readily retrieved (where $\hb = \Log(\ky)$).
It is important to note, however, that $\dExp{\ba} \neq \ky \cdot \mb$. Instead, $\dExp{\ba} \cdot \norm(\mb)^{1/n} = \ky \cdot \mb$. We chose for this convention to make $\ba$ a degree-zero Arakelov divisor.}.

\begin{notation}[Output distribution] \label{notation:distribution} For a fixed number field $K$, with modulus $\modu$, finite-index subgroup $\subpic \subseteq \rayPic^0$, element $\tau \in \OK$
coprime with $\moduz$ and fixed error parameter $\eps > 0$, block size $\blocksize$, $\radpar \in \Q_{\geq 1}$ and $\ba = d^0(\mb) + \hb \in \rayDiv^0$ (where $\hb$ is understood via the inclusion $\hyper \subseteq \rayDiv^0$), we denote
\[ \distr_{\ba} \in L_1(K^*) \]
for the output distribution of $\beta$ from \Cref{alg:samplerandom} on input $\mb$
(coprime with $\moduz$) and
$\ky = \dExp{\hb} \in \nfr^0$.
We will use the notation $\distr_{\ba }(\alpha)$ for the probability of sampling $\alpha \in K$ from $\distr_{\ba}$.
\[ \distr_{\ba}(\alpha) = \prob{x \from \distr_{\ba}}[x = \alpha]. \]
\end{notation}

\subsection{Shifting Property} \label{section:shifting}
In order to prove the shifting property of \Cref{alg:samplerandom}, 
we first start by the following lemma, that shows what impact 
multiplication by an element $\alpha \in K^*$ has on a ($\tau$-equivalent)
ideal lattice $\rayelt{\ba}$. This will then be used to show that 
the output distribution of \Cref{alg:samplerandom} 
`shifts' (i.e., multiplies) by $\alpha$ if the input 
divisor is shifted by the principal divisor $\ldb \alpha \rdb$ in \Cref{lemma:com_mult_arak}.

\begin{lemma} \label{lemma:bijectionbyprincipaldivisor} For all $\ba \in \rayDiv^0$, multiplication by $\alpha \in \Kmodu$ gives an additive group isomorphism
\[ \rayelt{\ba} \rightarrow \rayelt{\ba + \finpart{\ldb \alpha \rdb}  },~ \gamma \mapsto (\sigma(\alpha))_\sigma \cdot \gamma  \]
and additionally, restricted to $r \Ballinf$, a bijection
\[ \rayelt{\ba} \cap r \Ballinf \rightarrow \rayelt{\ba + \finpart{\ldb \alpha \rdb}} \cap  \dExp{-\infpart{\ldb \alpha \rdb}} \cdot r \Ballinf.  \]
\end{lemma}
\begin{proof} Very similar as in \Cref{lemma:helplemma} part (\itemref{lemma:helplemmaii}), multiplication by $(\sigma(\alpha))_\sigma \in \nfr$ yields a bijection from $\rayelt{\ba}$ to $\rayelt{\ba + \ldb \alpha \rdb} \dExp{- \infpart {\ldb \alpha \rdb}} = \rayelt{\ba + \finpart{\ldb \alpha \rdb}}$. Multiplication by an element is straightforwardly a group morphism, so the first claim follows by the fact that bijective group morphisms are isomorphisms.

For the second part, observe that multiplication by $(\sigma(\alpha))_\sigma$ transforms the box $r \Ballinf$ into $(\sigma(\alpha))_\sigma \cdot r \Ballinf = (|\sigma(\alpha)|)_\sigma \cdot r \Ballinf = \infpart{\Exp}(-\infpart{\ldb \alpha \rdb}) \cdot r \Ballinf$.
\end{proof}

\begin{lemma} \label{lemma:com_mult_replete} \label{lemma:com_mult_arak} \label{lemma:shifting}
For all degree-zero Arakelov divisors $\ba \in \rayDiv^0$,
the distribution $\distr_\ba$
satisfies the following \emph{shifting property} for all $\alpha \in \Kmodu$,
\[ \distr_{\ba + \ldb \alpha \rdb }(\placeholder \cdot \alpha ) = \distr_{\ba}(\placeholder) \]

\end{lemma}
\begin{proof}
Take an arbitrary $\ba \in \rayDiv^0$ and write $\ba = d^0(\mb) + \hb \in \rayDiv^0$ for an ideal $\mb$ and an $\hb = (\hb_\pl)_\pl \in \hyper$ (where we see $\hb \in H \subseteq \rayDiv^0$). To prove the statement it is enough to show for every $\beta \in \Kmodumodu$
\[ \distr_{\ba + \ldb \alpha \rdb }( \beta \cdot \alpha )= \distr_{\ba}(  \beta )\]

As can be proved in the same fashion as in \Cref{lemma:algo_distr}, the distribution $\distr_{\ba}$ is the same as the distribution resulting from
\[ \alpha \from \infpart{\Exp}(- \infpart{(\ba+\bb)}) \cdot \unif(\rayelt{\ba + \bb} \cap r \Ballinf) \mbox{ with } \bb =  d^0(\prod_{j} \mp_j) + \ha, \] where  $\mp_j$ are sampled according to step \lineref{line:alg1:sampleprimes} and $\ha = \Log((A_\sigma)_\sigma)$ is sampled from a discrete Gaussian, as in step \lineref{line:alg1:distortion}. 
Note that this $\ha$ is slightly distorted, due to the computation of $(A_\sigma)_\sigma$, but that is not problematic; the property used here is that this $\ha$ is distributed independently of the value of $\ba$. We write $\bb = d^0(\prod_j \mp_j) + \ha$. This divisor corresponds to the `random walk part' of the algorithm.

We seek to compare the distributions $\distr_{\ba}$ and $\distr_{\ba + \ldb \alpha \rdb}$. By the law
of total probability,
 since the sampling of $\bb = d^0(\prod_j \mp_j) + \ha$ is independent of
 $\ba$ and $\ldb \alpha \rdb$,
it is sufficient to consider a fixed sample
$\bb = d^0(\prod_j \mp_j) + \ha$
and compare the resulting distributions.
These resulting distributions, for a fixed $\bb$, denoted $\distr_{\ba | \bb}$ and $\distr_{\ba + \ldb \alpha \rdb | \bb}$, are
\begin{equation} \distr_{\ba | \bb} = \infpart{\Exp}(-\infpart{(\bb+\ba)}) \cdot  \unif( \rayelt{\bb + \ba} \cap r \Ballinf) \label{eq:distrfixedb} \end{equation}
and
\begin{align} \distr_{\ba + \ldb \alpha \rdb | \bb} &=  \infpart{\Exp}(-\infpart{(\bb+\ba + \ldb \alpha \rdb)}) \cdot  \unif( \rayelt{\bb + \ba + \ldb \alpha \rdb} \cap r \Ballinf) \\
&=  \infpart{\Exp}(-\infpart{(\bb+\ba)}) \cdot  \unif( \rayelt{\bb + \ba + \finpart{\ldb \alpha \rdb}} \cap \infpart{\Exp}(\infpart{-\ldb \alpha \rdb}) \cdot  r \Ballinf) \\ 
& = \infpart{\Exp}(-\infpart{(\bb+\ba)}) \cdot  \unif\Big( (\sigma(\alpha))_\sigma \cdot (\rayelt{\bb + \ba} \cap    r \Ballinf) \Big) \\
& =  (\sigma(\alpha))_\sigma \cdot \distr_{\ba | \bb}
\end{align}
where the first equality holds by definition, the second equality follows from pulling $\infpart{\Exp}(-\infpart{\ldb \alpha \rdb})$ into the uniform distribution and using the relation $\ldb \alpha \rdb = \finpart{\ldb\alpha \rdb} + \infpart{\ldb \alpha \rdb}$, the third equality follows from \Cref{lemma:bijectionbyprincipaldivisor} and the last equality by \Cref{lemma:algo_distr}.

The result quickly follows from the law of total probability and the fact that multiplying by $(\sigma(\alpha))_\sigma$ and $\alpha$ is the same  in $\nfr$.
\end{proof}

\begin{lemma}[Post-selection shifting lemma] \label{lemma:postselection} Let $\idset$ be a set of ideals coprime with $\moduz$ and let $\distr^\idset_{\ba}$ be the distribution obtained by running \Cref{alg:samplerandom} on input $\ba$ and only outputting $\beta$ if $\beta/\tmb \in \idset$ and otherwise outputting $\perp$.

Then, this distribution satisfies
\[ \distr^\idset_{\ba + \ldb \alpha \rdb}( \placeholder \cdot \alpha )= \distr^\idset_{\ba}( \placeholder ),  \]
for all $\alpha \in \Kmodu$.
\end{lemma}
\begin{proof} This uses the same proof structure as \Cref{lemma:com_mult_arak}, again using \Cref{lemma:algo_distr}.
Reusing the notation $\distr_{\ba | \bb}$ for
$\bb =  d^0(\prod_{j} \mp_j) + \ha$
as in \Cref{lemma:com_mult_arak},
we proved that, for a fixed sample $\bb$, we have
\begin{equation}  \distr_{\ba | \bb}( \placeholder) = \distr_{\ba + \ldb \alpha \rdb | \bb}( \alpha \cdot \placeholder) . \label{eq:shiftingB} \end{equation}
The ideal $\beta/\tmb$ (where $\beta,\tmb$ are from \Cref{alg:samplerandom} on input $\ba$) is distributed as\footnote{The distribution $\distr_{\ba | \bb}$ concerns the right part $\alpha \infpart{\Exp}(-\infpart{\ba})$ of \Cref{lemma:algo_distr}, whereas $\beta/\tmb$ concerns $\alpha \dExp{-\ba} = \alpha \infpart{\Exp}(-\infpart{\ba}) \finpart{\Exp}(-\finpart{\ba})$.} $\finpart{\Exp}(-\finpart{\ba}) \cdot \gamma$ with $\gamma \from \distr_{\ba | \bb}$ (on input $\ba$ with fixed $\bb$).

But replacing $\ba$ with $\ba + \ldb \alpha \rdb$ here yields that the distribution of $\beta/\tmb$ of \Cref{alg:samplerandom} on input $\ba + \ldb \alpha \rdb$ is equal to (for a fixed sample $\bb$)
\[  \finpart{\Exp}(- \finpart{(\ba + \ldb \alpha \rdb)}) \cdot \gamma \mbox{ with }  \gamma \from \distr_{\ba + \ldb \alpha \rdb | \bb}   \]
which is, by \Cref{eq:shiftingB}, equal to 
\[ \finpart{\Exp}(- \finpart{(\ba + \ldb \alpha \rdb)}) \cdot \alpha \cdot \gamma \mbox{ with }  \gamma \from \distr_{\ba | \bb}    \]
which is, by manipulation, equal to
\[ \finpart{\Exp}(- \finpart{\ba}) \cdot \gamma \mbox{ with }  \gamma \from \distr_{\ba | \bb} .   \]
So, for a fixed sample of
$\bb =  d^0(\prod_{j} \mp_j) + \ha$
the probabilities of obtaining $\beta/\tmb$ are the same, no matter
whether the input was $\ba$ or $\ba + \ldb \alpha \rdb$.

As a result,  $\distr_{\ba | \bb}^\idset = \alpha \cdot \distr^\idset_{\ba + \ldb \alpha \rdb | \bb}$, where $\distr_{\ba | \bb}^\idset$ denotes the output distribution of \Cref{alg:samplerandom} on input $\ba$ with fixed sample $\bb \from \Walk(N,B,s)$ and the additional post-selection (i.e., outputting $\beta$ whenever  $\beta/\tmb \in \idset$ and $\perp$ otherwise). 

By the law of total probability the distributions of $\beta / \tmb$
for $\ba$ or $\ba + \ldb \alpha \rdb$ are the same; and, thus the rejection of $\beta/\tmb$ for not being in $\idealset$ happens for the same occurrences. Therefore, $\distr^\idset_{\ba + \ldb \alpha \rdb}( \placeholder \cdot \alpha ) = \distr^\idset_{\ba}( \placeholder )$ for all $\alpha \in \Kmodu$.
\end{proof}

\subsection{Boundedness Property}
\label{section:bounded}
In this section we show that
the output distribution $\distr_\ba$ is bounded, i.e., 
no arbitrarily large inputs occur. This can be seen 
by the fact that the output is sampled from a distorted box,
where the distortion is distributed as a truncated discrete Gaussian (a bounded distribution by itself). 
Thus, the output must be confined to a certain box, slightly larger than the dimensions of the original
undistorted box. This is made more formal in the following text.

Recall the Euclidean distance notion on $\Div_K$ in \Cref{def:normondiv}.

\begin{lemma} \label{lemma:boundeddivisor}
Let $\ba = d^0(\mb) + \hb \in \rayDiv^0$ with
$\mb \in \ideals$ and $\hb = \Log(y) \in \hyper$
and let $B,N, \sd,r,\eps$ be as defined in
\Cref{alg:samplerandom}.
Then, we have, for $\beta \from \distr_\ba$,
\[  \divnorm{ \ldb \beta \rdb} \leq 5 \log(B^N \cdot r^n) + \divnorm{\ba} +  \sd \cdot \sqrt{n \log(8n^2/\eps)}. \]

\end{lemma}

\begin{proof}  ~\\
\textbf{Bound on the norm of $\beta$.} \\
We start with a bound on $|\norm(\beta)|$, which is of use for the rest of the proof.
The output of $\beta \from \distr_\ba$ with $\ba = d^0(\mb) + \hb$ lies in the box 
$\Exp(-\ha-\hb)
\cdot r \cdot \norm(\tmb)^{1/n} \cdot \Ballinf \subseteq \nfr$ (see lines \lineref{line:alg1:sample} and \lineref{line:alg1:return} in \Cref{alg:samplerandom}%
), where $\ha = (\ha_{\pl})_\pl = \Log((A_\sigma)_\sigma) \in \hyper$ is an approximation a sample of a (truncated) Gaussian distribution (see line \lineref{line:alg1:distortion} and \lineref{line:alg1:rationalA} of \Cref{alg:samplerandom}); and $\hb = (\hb_\pl)_\pl = \Log(\ky)$. Hence, since $\ha,\hb \in \hyper$ do not change the norm (we may assume this by \Cref{lemma:alg2independentnorm}), we obtain $|\norm(\beta)| \leq \norm(\tmb) \cdot r^n \leq B^N \cdot r^n \cdot \norm(\mb)$; a bound on the norm of $\beta$.
~\\
\textbf{Split the bound on $\divnorm{\ldb \beta \rdb}$ into three parts.} \\
We have, by the triangle inequality,
\begin{equation} \label{eq:divnormtriangle} \divnorm{ \ldb \beta \rdb } \leq \divnorm{ \ldb \beta \rdb - \ba } + \divnorm{\ba}.  \end{equation}
We seek to bound $\divnorm{\ldb \beta \rdb}$. To do so, we concentrate on $\divnorm{ \ldb \beta \rdb - \ba }$, which we can split in a `finite part', a `constant part' and an `infinite part'.
\begin{equation} \ldb \beta \rdb - \ba = \underbrace{ \sum_{\mp} m_\mp \cdot \ldb \mp \rdb}_{\mbox{\scriptsize{finite part }} \mathbf{f}} - \underbrace{\sum_{\nu} \frac{\npl \log \norm(\prod_{\mp} \mp^{m_\mp})}{n} \ldb\nu \rdb}_{\mbox{\scriptsize{constant part }} \mathbf{c}} +  \underbrace{\sum_{\nu} \iota_\nu \ldb\nu \rdb}_{\mbox{\scriptsize{infinite part }} \mathbf{i}}. \label{eq:splitintofininf} \end{equation}
where $\npl = 2$ if $\pl$ is complex and $1$ otherwise. \\ \noindent
\textbf{Bounding the finite part $\mathbf{f}$ and constant part $\mathbf{c}$.}\\
Since $\beta \in \tmb \subseteq \mb$, we see that $(\beta) = \fa \cdot \mb$ for an \emph{integral ideal} $\fa$ (recall that $\mb = \dExp{\finpart{\ba}}$). Hence, $\sum_{\mp} m_\mp \cdot d^0(\mp) = d^0(\ma)$, has only positive coefficients $m_\mp \geq 0$. Since $\norm(\ma) = |\norm(\beta)|/\norm(\mb) \leq B^N \cdot r^n$ by the norm bound on $\beta$ in the very beginning of this proof, we quickly deduce
\[ B^N \cdot r^n \geq \norm(\ma) \geq \prod_{\mp} \norm(\mp)^{m_\mp} \geq 2^{\sum_{\mp} |m_\mp|} \geq 2^{(\sum_{\mp} |m_\mp|^2)^{1/2}}, \]
which implies
\begin{equation} \divnorm{\mathbf{f}} \leq (\sum_{\mp} |m_\mp|^2)^{1/2} \leq \log(B^N \cdot r^n)/\log(2) \leq 2 \log(B^N \cdot r^n)  \label{eq:boundingfinitepart} \end{equation}
Similarly, for the constant part $\mathbf{c}$, using $\ma = \prod_{\mp} \mp$, we have,  %
\begin{equation} \divnorm{\mathbf{c}} \leq\sqrt{ \sum_{\nu} n_\nu^2}  \cdot  \frac{\log(\norm(\ma))}{n} \leq \log(\norm(\ma)) \leq \log(B^N \cdot r^n) \label{eq:boundingconstantpart}  \end{equation}
\textbf{Bounding the infinite part $\mathbf{i}$.}\\
By definition (see \Cref{eq:splitintofininf}), and by the inclusion of $\hyper \subseteq \rayDiv^0$  we have
\begin{align} \label{eq:hcdef} \iota_\pl &= - \npl \log |\embn(\beta)| + \npl \frac{\log(|\norm(\beta)|)}{n} - \hb_\sigma \in \hyper
\end{align}
since $\ba = d^0(\mb) + \hb$. %
By the fact that $\beta$ lies in the box 
$\Exp(-\ha-\hb) \cdot
r \cdot \norm(\tmb)^{1/n} \cdot \Ballinf$, we have $\log |\embn(\beta)| \leq - \npl^{-1} \ha_\sigma -  \npl^{-1} \hb_\sigma + \log r + \log (\norm(\tmb)^{1/n})$. Therefore, %
\begin{align} - \log |\sigma(\beta)|
&\geq  \npl^{-1} \ha_\sigma +  \npl^{-1} \hb_\sigma - \log r - \log (\norm(\tmb)^{1/n})
\label{eq:computesigmanubeta}
\end{align}
Hence, substituting \Cref{eq:computesigmanubeta} into \Cref{eq:hcdef}, and using that $\beta \in \tmb$ (and hence $|\norm(\beta)| \geq \norm(\tmb)$), we obtain
\[ \iota_\pl \geq \ha_\pl  + \npl \log(|\norm(\beta)|^{1/n} \cdot
\norm(\tmb)^{-1/n} \cdot r^{-1}) \geq \ha_\pl - \npl \log(r). \]
So, writing $\iota - \ha \in H$, it satisfies $\sum_{\pl} (\iota-\ha)_\pl = 0$ and $\iota_\pl - \ha_\pl \geq -\npl  \log(r)$ for all $\pl$.
Then $\| \iota - \ha \| =  \sqrt{\sum_{\pl} |\iota_\pl - \ha_\pl|^2 }$ is bounded by%
\footnote{
Suppose a sequence $(x_1,\ldots,x_n)$ satisfies $\sum_{j =1}^n x_j = 0$ and $x_j \geq -t$ for all $j \in \{1,\ldots,n\}$ for some $t \in \R_{>0}$. Then $\sum_{j = 1}^n x_j^2 \leq n(n-1) t^2$. This maximum is attained by, for example, $x_1 = (n-1)t$ and $x_j = -t$ for $1< j \leq n$.

We prove this by induction, where the case $n = 1$ is trivial (since then $x_1^2 = x_1 = 0 \leq n(n-1)t^2$). For $n>1$, we can assume that a non-zero sequence $(x_1,\ldots,x_n)$ has a positive co-ordinate (by the zero-sum requirement). Moreover, we can assume it has only a single positive co-ordinate. Indeed, if there are multiple positive coordinates, say $x_1,x_2$, then the sequence $(y_1,\ldots,y_{n-1}) = (x_1 + x_2,x_3,\ldots,x_n)$ satisfies $\sum_{j=1}^{n-1} y_j = 0$ and $y_j \geq -t$. Hence, $\sum_{j = 1}^n x_j^2 \leq \sum_{j = 1}^{n-1} y_j^2 \leq (n-1)(n-2)t^2 \leq n(n-1)t^2$, by induction. For a single positive co-ordinate it is clear that $x_1 = (n-1)t$ and $x_j = -t$ for $1 < j \leq n$ is the optimal solution.

Applying this to the sequence $q_\sigma = \npl^{-1} (\iota_{\nu_\sigma} - a_{\nu_\sigma})$, we see that 
$\| \iota - a \| \leq 2\|(q_\sigma)_\sigma\| \leq 2 \sqrt{n(n-1)} \log(r)$.
}
$2\sqrt{n(n-1)} \log(r) \leq 2\log(r^n)$. Hence, writing $\mathbf{i} = (\iota_\pl)_\pl$, we have
\begin{equation} \label{eq:boundinginfinitepart}  \divnorm{\mathbf{i}}=  \divnorm{ \iota } \leq \divnorm{ \ha  } +2 \log(r^n) \leq \divnorm{\ha} + 2\log(B^N \cdot r^n). \end{equation}

\noindent
\textbf{Combining the bounds on $\divnorm{\mathbf{i}}$ and $\divnorm{\mathbf{f}}$.}\\
Concluding, by combining \Cref{eq:divnormtriangle,eq:splitintofininf,eq:boundingfinitepart,eq:boundingconstantpart,eq:boundinginfinitepart}, we obtain
\begin{align} \divnorm{\ldb \beta \rdb} &\leq \divnorm{ \ldb \beta \rdb - \ba } + \divnorm{\ba} \leq \divnorm{\mathbf{f}} + \divnorm{\mathbf{c}} +  \divnorm{\mathbf{i}} + \divnorm{\ba} \\
&\leq 5 \log(B^N \cdot r^n) + \| \ha \| + \divnorm{\ba} .\end{align}
The final bound then comes from the fact that $\ha = (\ha_\pl)_\pl \in \hyper$ is (an approximation of a) discrete Gaussian distributed with deviation $s$ computed according to the algorithm in \Cref{lemma:gpv} with $\geps := \eps/4$ (see \lineref{line:alg1:distortion} of \Cref{alg:samplerandom}). Hence, $\|\ha\| \leq  \sd \cdot \sqrt{n \log(8n^2/\eps)}$.
\end{proof}

\subsection{(Almost)-Lipschitz Property} \label{section:lipschitz}
The almost-Lipschitz-continuous property of the output distribution of
\Cref{alg:samplerandom} implies that it does not matter 
too much if the input parameter $\ky = \dExp{\hb} \in \nfr^0$ is slightly
inaccurate or rounded to a vector of rational numbers due
to machine precision. 

This is relevant for \Cref{part2} of this paper, 
where we will feed \Cref{alg:samplerandom} $y = \dExp{(\hb_{\pl})_\pl}$ with
a continuous Gaussian distributed
$\hb \in \hyper$. The Lipschitz-continuous property in this subsection essentially
implies that 
using a discrete Gaussian distribution instead for $\hb \in \hyper$ will
only cause a very small error overhead.

\begin{lemma} \label{lemma:idealsamplinglipschitz} The output distribution of the sampling algorithm (\Cref{alg:samplerandom}) is $\eps$-almost Lipschitz continuous
in the input parameter $\hb \in \hyper$ with $\ky = \dExp{\hb}$. That is, for $\ba = d^0(\mb) + \hb$ and $\ba' = d^0(\mb) + \hb' = \ba + (\hb'-\hb)$ we have
\[ \tfrac{1}{2} \|\distr_{\ba} - \distr_{\ba'} \|_1\leq
\frac{n^2}{2} \cdot
\|\hb - \hb'\| + \eps, \]
where $\tfrac{1}{2} \|\distr_{\ba} - \distr_{\ba'} \|_1$ is the total variation distance of the two distributions.
\end{lemma}
\begin{proof}
In this proof, like in the proof of \Cref{theorem:ISmain}, we will again assume that line \lineref{line:alg1:distortion} in \Cref{alg:samplerandom} is replaced by the line in \Cref{eq:alg1change}, and assume that $y \in \nfr^0$. Again we will refer to this as the `continuous 
version of \Cref{alg:samplerandom}'. This proof will be structured as follows. First we will show that this `continuous version of \Cref{alg:samplerandom}' is actually Lipschitz continuous with Lipschitz constant $n^2/2$. After that we will use \Cref{lemma:closeness_continuous} to show that the Lipschitz continuity also `almost' holds for the ordinary version of \Cref{alg:samplerandom}, with an additional error of $\eps$.

In the following, we will reason about the `continuous version of \Cref{alg:samplerandom}'.
The input $\hb \in \hyper$ is only used in \Cref{alg:samplerandom} (as $\hb = \Log(y)$) in combination with
the variable $\ha = (\ha_\pl)_\pl \in \hyper$, which sampled according to a continuous Gaussian distribution with deviation $\sd = 1/n^2$.
The sum $\ha  + \hb$ is then used combined in the rest of the algorithm.

So, for two different inputs $\hb,\hb' \in \hyper$, the only thing that differs in the algorithm are the distributions of
$\ha + \hb$ and $\ha' + \hb'$, for which $\ha,\ha'$ are (independently) sampled according to the same Gaussian distribution.

We denote
the distributions respectively $\Gaussian_{\hb}$ and $\Gaussian_{\hb'}$; these are Gaussian distributions with deviation $\sd = 1/n^2$
centered at $\hb$ and $\hb'$ respectively (see \Cref{subsec:gaussiandistr}).

Therefore, by the
data processing inequality (\Cref{theorem:dataprocessinginequality}), one immediately deduces that
the statistical distance between the output distribution of the `continuous version of \Cref{alg:samplerandom}' on input $\mb,\hb$ and $\mb,\hb'$ is bounded above by 
\begin{align*}%
\tfrac{1}{2} \| \Gaussian_{\hyper,\hb} - \Gaussian_{\hyper,\hb'} \|_1. \end{align*}
The total variation distance of two continuous (multivariate) Gaussian distributions with different centers can be bounded by the Kullback-Leibler divergence \cite[Ch.~1, Ex.~11]{Pardo:996837},
\begin{align*}  \tfrac{1}{2} \| \Gaussian_{\hyper,\hb} - \Gaussian_{\hyper,\hb'} \|_1 &\leq   \tfrac{1}{2} \sqrt{ \tfrac{1}{2} D_{KL}(\Gaussian_{\hyper,\hb} ~\|~ \Gaussian_{\hyper,\hb'}) }
 \\ & \leq \tfrac{1}{2} \sqrt{ \tfrac{1}{4} \|\hb - \hb'\|^2/\sd^2  } = \frac{\|\hb - \hb'\|}{4 \cdot \sd} = \frac{n^2 \cdot \|\hb - \hb'\|}{4}
\end{align*}
The last equality is obtained by instantiating $\sd = 1/n^2$.

As sketched in the introduction of this proof, we use that the ordinary version and the continuous version of \Cref{alg:samplerandom} only differ by $\eps/2$, by the very same argument as in the proof of \Cref{theorem:ISmain}. Hence, the statistical distance between the ordinary version of \Cref{alg:samplerandom} input $\mb,\hb$ and $\mb,\hb'$ is bounded by $\frac{n^2}{2} \cdot
\|\hb - \hb'\| + \eps$, as required.
\end{proof}

In most use-cases, \Cref{alg:samplerandom} is \emph{repeated until success}, for which
the following lemma is more useful: it shows that the output distribution of the algorithm
\emph{conditioned on a successful outcome} is also almost-Lipschitz. This distribution conditioned 
on a successful outcome is exactly the same as the output distribution resulting from repeating the algorithm
until success.
\begin{lemma} \label{lemma:idealsamplinglipschitzpostselection} The output distribution $\distr^{\mbox{\scriptsize{suc}}}_{\ba}$, resulting from repeating the sampling algorithm (\Cref{alg:samplerandom}) until success, is $\frac{\eps}{\delta_{\idset}[r^n]}$-almost Lipschitz continuous
in the input parameter $\hb \in \hyper$ with $\ky = \dExp{\hb}$. That is, for $\ba = d^0(\mb) + \hb$ and $\ba' = d^0(\mb) + \hb' = \ba + (\hb'-\hb)$ we have
\[ \tfrac{1}{2} \|\distr^{\mbox{\scriptsize{suc}}}_{\ba} - \distr^{\mbox{\scriptsize{suc}}}_{\ba'} \|_1\leq
\frac{n^2}{2\delta_{\idset}[r^n]} \cdot
\|\hb - \hb'\| + \frac{\eps}{\delta_{\idset}[r^n]}, \]
where $\tfrac{1}{2} \|\distr^{\mbox{\scriptsize{suc}}}_{\ba} - \distr^{\mbox{\scriptsize{suc}}}_{\ba'} \|_1$ is the total variation distance of the two distributions.
\end{lemma}
\begin{proof} For this we apply the next \Cref{lemma:conditionalvariation2} with the lower bound $\delta_{\idset}[r^n]$ on the success probability.
\end{proof}

\begin{lemma} \label{lemma:conditionalvariation2} Let $\mathcal{A}$ be a randomized algorithm that can have input $y$ and $y'$.
Denote $\mathcal{A}_y$ respectively $\mathcal{A}_{y'}$ for the algorithm using $y$ respectively $y'$ as input. Assume that both $\mathcal{A}_y$ and $\mathcal{A}_{y'}$ have success probability at least $p \in (0,1]$.

Now denote $\mathcal{B}_y$ for the algorithm running $\mathcal{A}_y$ until it outputs a successful output, and similarly denote $\mathcal{B}_{y'}$ for the algorithm running $\mathcal{A}_{y'}$ until success. By abuse of notation, denote $\mathcal{B}_y$, $\mathcal{B}_y$, $\mathcal{A}_y$, $\mathcal{A}_{y'}$ for the output distributions of these algorithms. Then the total variation distances are related by
\[ \tfrac{1}{2}\| \mathcal{B}_y - \mathcal{B}_{y'} \|_1 \leq \frac{\tfrac{1}{2}\| \mathcal{A}_y - \mathcal{A}_{y'} \|_1 }{p} . \]

\end{lemma}

\begin{proof} Denote $X \cup \{ \perp \}$ for the output space of $\mathcal{A}$, where $\perp$ denotes failure (unsuccessful output) and where $x \in X$ are all successful outputs.

The output distribution $\mathcal{B}_x$ is the output distribution of $\mathcal{A}_x$ \emph{conditioned} on a successful output. Hence, denoting $\mathcal{A}_x^{\mathrm{success}}$ (and similarly for $x'$) for this conditional distribution, we have
\[ \| \mathcal{B}_x - \mathcal{B}_{x'} \|_1 = \| \mathcal{A}^{\mathrm{success}}_x - \mathcal{A}^{\mathrm{success}}_{x'} \|_1 = \sum_{x \in X} |\mathcal{A}^{\mathrm{success}}_y[x] - \mathcal{A}^{\mathrm{success}}_{y'}[x]|,\]
where $X$ is the probability space of the algorithm $\mathcal{A}$. The conditional probability of $\mathcal{A}$ outputting $x$ \emph{conditioned} on a successful output (which are all of $x \in X$) then must equal $\frac{\mathcal{A}_y[x]}{p_0}$
with  $p_0 = \mathbb{P}[ \mathcal{A}_y \mbox{ gives successful output} ]$; and similarly for input $y'$. Note that $|p_0 - p_1| \leq \|\mathcal{A}_y - \mathcal{A}_{y'}\|_1$. Hence, writing $p = \min(p_0,p_1)$ and noting that $p_1 = \sum_{x \in X} \mathcal{A}_{y'}[x]$,
\begin{align*} &\sum_{x \in X} |\mathcal{A}^{\mathrm{success}}_y[x] - \mathcal{A}^{\mathrm{success}}_{y'}[x]|  = \sum_{x \in X} \left |\frac{\mathcal{A}_y[x]}{p_0} - \frac{\mathcal{A}_{y'}[x]}{p_1} \right| \\
\leq &  \sum_{x \in X} \left |\frac{\mathcal{A}_y[x]}{p_0} - \frac{\mathcal{A}_{y'}[x]}{p_0} \right|  + \sum_{x \in X} \left |\frac{\mathcal{A}_{y'}[x]}{p_0} - \frac{\mathcal{A}_{y'}[x]}{p_1} \right|  \\
\leq & \frac{1}{p_0} \sum_{x \in X} \left |\mathcal{A}_y[x]- \mathcal{A}_{y'}[x] \right|  + \frac{1}{p_0} |p_1 - p_0| = \frac{1}{p_0} \sum_{x \in X \cup \{ \perp \}} \left |\mathcal{A}_{y'}[x] - \mathcal{A}_{y'}[x] \right|  \\
 \leq & \frac{\| \mathcal{A}_y - \mathcal{A}_{y'}\|_1 }{p}.   \end{align*}
\end{proof}

%% file: ideal_sampling/A08-norm-modulus-bound.tex
\section{Estimation of the factor \texorpdfstring{$\norm(\moduz)/\phi(\moduz)$}{N(m0)/phi(m0)}}
\label{section:estimationmodu}
\begin{lemma} \label{lemma:modulusnorm} Let $K$ be a number field and let $\moduz \subseteq \OK$ be an integral ideal. Then we have
\[  \frac{\norm(\moduz)}{\phi(\moduz)} = \prod_{\mp  \mid \moduz}  \frac{1}{1 - \norm(\mp)^{-1}} . \]
 where $\phi(\moduz) = \defphi$ and where $\mp$ ranges over the prime ideals of $K$ dividing $\moduz$.
\end{lemma}
\begin{proof}
If $\moduz = \OK$, we have $\phi(\moduz) = 1$. Otherwise $\phi(\moduz) = |(\OK/\moduz)^\times|$.
Writing $\moduz = \prod_{\mp \mid \moduz} \mp^{v_\mp(\moduz)}$, we have, by multiplicativity of the norm $\norm$ and the (generalized) Euler totient function $\phi$ (which follows from the Chinese remainder theorem), 
\[ \norm(\moduz) = \prod_{\mp \mid \moduz} \norm(\mp)^{v_\mp(\moduz)} ~\mbox{ and }~ \phi(\moduz) = \prod_{\mp \mid \moduz} \norm(\mp)^{v_\mp(\moduz)-1}(\norm(\mp) - 1) \] 
The quotient $\norm(\moduz)/\phi(\moduz)$ then equals
$ \prod_{\mp \mid \moduz} \frac{\norm(\mp)}{\norm(\mp)-1} = \prod_{\mp \mid \moduz} \frac{1}{1 - \norm(\mp)^{-1}}$.
\end{proof}

\begin{proposition}[ERH]\label{prop:boundrhonormphi} Let $K$ be a number field and let $\moduz = \prod_{\norm(\mp) < x} \mp$ be the product of all prime ideals with norm below $x \geq 10$.
 Then there exists $c_0 \in [-8,8]$ and $c_1 \in [0,2]$ such that
\[  \frac{\norm(\moduz)}{\phi(\moduz) \cdot \dedres} =  \log(x) \cdot \exp \left (c_1 +  c_0 \cdot \frac{\log |\dcrk| + n \log x}{\sqrt{x}} \right). \]
\end{proposition}

\begin{proof}
Write, as in Bach's paper \cite[\textsection 6]{Bach95}, %
\begin{equation} \label{def:Ax}  A(x) := \prod_{p < x} \frac{1-p^{-1}}{\prod_{\substack{\mp | p \\ \norm(\mp) < x}} (1 - \norm(\mp)^{-1})}, \end{equation}
for which we have the bound\footnote{In his theorem, Bach writes $\frac{\zeta_K}{\zeta}(1)$ for the residue $\rho_K$ at $s = 1$ of the Dedekind zeta function $\zeta_K$ \cite[Thm.~6.2]{Bach95}.} \cite[Thm.~6.2 \& Table~2]{Bach95}
\begin{equation} \label{eq:bachboundAx} | \log(\rho_K) - \log A(x) | \leq 8 \cdot \left ( \frac{\log |\Delta| + n \log x}{\sqrt{x}} \right) , \end{equation}
assuming the Riemann Hypothesis for the Dedekind zeta function of $K$.

By \Cref{lemma:modulusnorm} and \Cref{def:Ax}, we have
\begin{equation} \label{eq:modulusAx} \frac{\norm(\moduz)}{\phi(\moduz)} =  \prod_{\mp \mid \moduz} \frac{1}{1 - \norm(\mp)^{-1}} = \prod_{p < x} \frac{1}{1 - p^{-1}} \cdot A(x)   \end{equation}
                                                                                                                                                            
By taking logarithms of \Cref{eq:modulusAx} and subsequently applying  Bach's bound (\Cref{eq:bachboundAx}), we obtain
\begin{align}  \log \big( \tfrac{\norm(\moduz)}{\phi(\moduz)} \big)  - \log(\dedres) 
 = &    \log( \prod_{p < x} \frac{1}{1 - p^{-1}}) + \log A(x) - \log(\dedres)
\\ = &  \log( \prod_{p < x} \frac{1}{1 - p^{-1}})   +  c_0 \cdot  \left ( \frac{\log |\Delta| + n \log x}{\sqrt{x}} \right) , \label{eq:lastlogofprimes}
\end{align}
for some $c_0 \in [-8,8]$, by \Cref{eq:bachboundAx}. Apply \Cref{lemma:mertensbound} for the left-hand side of \Cref{eq:lastlogofprimes}, to obtain $\prod_{p < x} \frac{1}{1-p^{-1}} \in [e^0,e^2] \cdot \log x$. Exponentiating the expression yields the final result. %
\end{proof}

\begin{lemma} \label{lemma:mertensbound} For $x \geq 2$, we have
\[  \log x \leq  \prod_{p < x} \frac{1}{1 - p^{-1}} \leq 6 \cdot \log x, \]
where the product ranges over all prime numbers up to $x$.
\end{lemma}
\begin{proof} By the Euler product formula, we certainly have
\[  \prod_{p < x} \frac{1}{1 - p^{-1}}  \geq \sum_{n < x} \frac{1}{n} \geq \log x.  \]
For the upper bound, we invoke an explicit version of Mertens' third theorem of Rosser and Schoenfeld
\cite[Corollary 1, Equation 3.30]{rosser1962approximate},
\[  \prod_{p < x} \frac{1}{1 - p^{-1}}  < \log(x) \cdot e^{\gamma} \cdot (1 + \frac{1}{\log(x)^2}).\]
Here, $\gamma \leq 0.578$ is the Euler-Mascheroni constant. Instantiating with $x = 2$ and explicitly computing $e^\gamma$
yields a bound of $6 \cdot \log(x)$.
\end{proof}

The following result is a simplified version of a result by Greni\'e--Molteni \cite[Cor. 1.4]{Greni__2015}.
\begin{lemma}[ERH, Greni\'e--Molteni] \label{lemma:boundnormmoduz}
Let $K$ be a number field, with degree $n$, discriminant $\Delta_K$, and ring of integers $\OK$.
Let $\moduz = \prod_{\norm(\mp) < x} \mp$ be the product of all prime ideals with norm below $x > 100$.
We have
\begin{align*}
\log \norm(\moduz) & \leq x + \sqrt{x}\left(\left(\frac{\log x}{2\pi} + 2\right)\log|\Delta_K| + \left(\frac{(\log x)^2}{8\pi} + 2\right)n\right)\\
&= x + O\left(\sqrt{x}\log (x)\left(\log|\Delta_K| + n\log x\right)\right).
\end{align*}
\end{lemma}

%% file: rigorous_cgc/B01-introduction.tex
\section{Introduction}
\noindent
In this second part of the article, we present the first algorithm for computing class groups and unit groups of arbitrary number fields that provably runs in probabilistic subexponential 
time, assuming the Extended Riemann Hypothesis.
Let $K$ be a number field of degree $n$ and discriminant $\Delta_K$. The determination of the structure of its class group $\Cl(K)$, together with a system of fundamental units, is one of the main problems of computational number theory~\cite[p.~217]{Cohen1993}.
Previous subexponential algorithms were either restricted to imaginary quadratic fields, or 
relied on several heuristic assumptions that have long resisted rigorous analysis.

\input{rigorous_cgc/01c-logtunits}

\input{rigorous_cgc/01a-history}

\input{rigorous_cgc/01b-overview}

\input{rigorous_cgc/01d-morealgorithmic}

\subsection{Road map}
We start in \Cref{section:specializedsampling} by specializing the ideal sampling theorem of Part \partref{part1} to our needs. %
Then, a lower bound is given on the density of smooth ideals, by means of a combinatorial technique, in \Cref{sec:smooth}.
After that, some preliminaries on $\mS$-units are treated in \Cref{section:backgroundsunit}, and useful properties of $\mS$-unit lattices are analyzed in \Cref{app:proof-covering-radius}. 

Applying both the specialized sampling theorem and the density of the smooth ideals, we obtain an algorithm that samples a \emph{single} $\mS$-unit in \Cref{sec:BF-compute-1-rel}.

In order to have an algorithm that outputs sufficiently `random' $\mS$-units, as to generate the entire $\mS$-unit group, we combine two results. One is a bound on the `generating radius' $\rr(\logsunits)$ of the log-$\mS$-unit lattice that is polynomial in $\log|\dcrk|$, which is proven in \Cref{app:proof-covering-radius}. This is combined with the properties of the sampling algorithm (see \Cref{section:propertiessampling} in  Part \partref{part1}), to build an algorithm that indeed outputs `sufficiently random' $\mS$-units, as proved in \Cref{sec:BF-compute-many-rel}.

After gathering many such $\mS$-units, we would like to compute a set of \emph{fundamental $\mS$-units}, which can be seen as the multiplicative analogue of a basis of the log-$\mS$-unit lattice. This is done in \Cref{section:postproc} by `post-processing' the many $\mS$-units, which essentially consists in applying the \bkp algorithm.

Part \partref{part2} is then concluded in \Cref{sec:full-algorithm} with the final theorem on the complexity of the $\mS$-unit algorithm.

%% file: rigorous_cgc/01c-logtunits.tex
\subsection{The \texorpdfstring{\logtunit}{log-\letterSunits-unit}~lattice}
While we stated our main result as an algorithm for computing units and class groups in \Cref{thm:main-thm-intro}, our algorithm actually does slightly more than that: it computes the so-called Log-$\mT$-unit lattice for any finite set $\mT$ of prime ideals. For such a set $\mT$ of prime ideals of $K$, we consider the group $\Div_{K,\mS} = \prod_{\mp \in \mS} \Z \times \prod_{\nu} \R$, where $\nu$ ranges over all infinite places of $K$. The Log-$\mT$-unit lattice $\logsunits$ is a lattice living in $\Div_{K,\mS}$ generated by all vectors of the form $(\vec v, -\Log(\alpha))$ with $\Log: K^\times \rightarrow \prod_{\nu} \R$ the logarithmic embedding of $K$, where $\alpha \in K$ and $\vec v = (v_\fp)_{\fp \in \mT} \in \Z^{\mS}$ are such that $\alpha \cdot \OK = \prod_{\fp \in \mT} \fp^{v_\fp}$. Such elements $\alpha$ are called $\mT$-units, hence the name of the lattice. See \Cref{section:backgroundsunit} for more details.
The main result of this part is the following theorem.

\begin{theorem}[ERH]
\label{thm:S-units-intro}
There is a probabilistic algorithm which, on input a number field $K$ of degree $n$ and discriminant $\Delta_K$,
an LLL-reduced basis of its ring of integers $\OK$ 
and a finite set of prime ideals $\mT$, computes the Log-$\mT$-unit lattice and runs in expected time polynomial in the length of the input, in $L_{|\Delta_K|}(1/2)$, in $L_{n^n}(2/3)$, and in $\min(\dedres, L_{|\dcrk|}( 2/3 + o(1)))$, where $\dedres$ is the residue at $1$ of the Dedekind zeta function $\zeta_K$.
\end{theorem}

The Log-$\mT$-unit lattice computed by the algorithm is represented exactly by %
elements $\alpha_i \in K$ and vectors $\vec v_i \in \Z^{\mS}$ such that the vectors $((\vec v_i, -\Log \alpha_i))_i$ form a basis of $\logsunits$. These elements $\alpha_i \in K$ are represented by a so-called compact representation, i.e., as a product of smaller elements $\beta_j \in K$ (see \Cref{sec:representation}). Even though this is called a compact representation (and it is indeed more compact than representing the $\alpha_i$ in a basis of $\OK$), the bit-size of this representation of the $\alpha_i$ might be about as large as the running time of the algorithm.

%% file: rigorous_cgc/01a-history.tex
\subsection{Previous work}\label{subsec:historyclassgroup}
Shanks proposed in 1968 the first algorithm to compute the structure of the class group. His method, described for quadratic number fields, had an exponential running time $O_\varepsilon(|\Delta_K|^{1/4 + \varepsilon})$ (or $O_\varepsilon(|\Delta_K|^{1/5 + \varepsilon})$ assuming ERH). In 1989, Hafner and McCurley~\cite{HM89} proposed the first subexponential time probabilistic algorithm to compute the class group of imaginary quadratic number fields. Assuming ERH, they prove that the expected running time is $L_{|\Delta_K|}(1/2,\sqrt{2})$. The case of imaginary quadratic number fields distinguishes itself by the finiteness of the unit group, and the existence of a \emph{reduced} ideal in each class, two properties that proved helpful for a rigorous analysis of the algorithm.

In 1990, Buchmann~\cite{buchmann1988subexponential} generalized this algorithm to fields of arbitrary degree. The analysis proved much more delicate, requiring several heuristic assumptions, notably on the distribution of smooth ideals. From these assumptions, Buchmann argued that the expected running time is $L_{|\Delta_K|}(1/2,1.7)$ when the degree of the field is constant. Practical improvements ensued, notably by Cohen, Diaz Y Diaz and Olivier~\cite{HDM97}. In 2014, Biasse and Fieker~\cite{ANTS:BiasseFiecker14} proposed an algorithm capable of dealing with fields of varying degree: under heuristic assumptions, they argue that the algorithm computes class groups in subexponential time $L_{|\Delta_K|}(2/3 + \varepsilon)$ for arbitrary families of number fields, even with growing degree (and $L_{|\Delta_K|}(1/2)$ whenever $n \leq (\log |\Delta_K|)^{3/4 - \varepsilon}$, matching Buchmann's claimed complexity for fixed-degree fields). In the meantime, no progress was made towards a rigorous analysis. 

On the quantum computing side, a non-heuristic (assuming ERH) quantum polynomial time algorithm for computing $\mS$-units and class groups was developed in 2016 by Biasse and Song~\cite{BiasseSong14}, for any number field. The structure of this algorithm is different from the classical algorithms mentioned above: it relies on a quantum polynomial time algorithm solving the Continuous Hidden Subgroup Problem over~$\R^m$~\cite{eisentrager2014quantum, BDF20}.
In the present article, we only consider classical algorithms.

%% file: rigorous_cgc/01b-overview.tex
\subsection{The blueprint of a `classical' \texorpdfstring{\logtunit}{log-\letterSunits-unit}~lattice algorithm}
For the sake of simplicity we limit ourselves in this 
overview to the computation of the class group of a number field.
The situation for the \logtunit~lattice is, at this high level,
very similar. 

Classical algorithms computing the class group proceed by
generating many class group relations of the shape $[\prod_{\mp} \mp^{\ha_\mp}] = 1$,
where the prime factors $\mp$ in the ideal $\prod_{\mp} \mp^{\ha_\mp}$ have norm bounded by $B$, i.e., the product is a $B$-smooth ideal. Since there are only
finitely many generators $\mp$, enough `sufficiently independent' relations
of this shape then yield a description of the (finite) class group.

Such a relation $[\prod_{\mp} \mp^{\ha_\mp}] = 1$ is generally found by 
taking a random smooth ideal $\fs = \prod_{\mp} \mp^{b_\mp}$, and repeatedly 
 sampling random elements $\alpha \in \fs$ until $(\alpha) \fs^{-1} = \prod_{\mp} \mp^{c_\mp}$ is
\emph{also} a smooth ideal. This gives the relation $1 = [\alpha] = [\prod_{\mp} \mp^{\ha_\mp}]$ with $\ha_\mp = b_\mp + c_\mp$.

\subsubsection*{Heuristics in previous work.}
The usual heuristics for class group algorithms (and \logtunit~lattice algorithms) are twofold. 
The first kind of heuristic assumption concerns a lower bound on the probability of the random ideal $(\alpha) \fs^{-1}$ to be smooth, %
see for instance \cite[Heuristic 1, 2]{Biasse14}.
This allows one to bound the expected running time of sampling a single relation of the shape $[\prod_{\mp} \mp^{\ha_\mp}] = 1$.
This smoothness probability is typically assumed to match the \emph{density} of smooth ideals in the set of all ideals, a somewhat well-understood quantity related to the Dickman function.

The second kind of heuristic assumption concerns the number of such relations required in order to get the class group (rather than a larger group, of which the class group would be a quotient), see for instance \cite[Heuristic 3]{Biasse14}. The heuristic running time is then obtained as the product of the running time per relation and the required number of relations. %

\subsection{Why our algorithm is rigorous} \label{subsec:intuition} %
We would like to stress that the following overview of our proof makes substantial simplifications for expository purposes (in particular, we ignore complex embeddings, which in the full proof play a similar role to prime ideals).

To circumvent the first heuristic about the cost of finding one relation, we use a special \emph{sampling algorithm}
(\Cref{alg:samplerandom}) from \Cref{part1} of this paper.
This algorithm samples an element $\alpha \in \ma$ in such a way that the success probability of $(\alpha) \ma^{-1}$ being smooth is \emph{actually provably lower bounded} by the density of smooth ideals in the set of all ideals (up to a multiplicative constant). This yields an upper bound on the expected cost of finding one class group relation. The sampling algorithm is the object of
\Cref{part1}; this algorithm, together with its properties is used here as a black-box.

The remaining challenge addressed in the current part is circumventing the \emph{second heuristic}, concerning 
the number of sampled group relations $[\prod_{\mp} \mp^{\ha_\mp}] = 1$ required to
compute the class group (rather than an extension). To achieve that goal, we investigate the lattice of class group relations, which consists of those vectors $(\ha_\mp)_{\mp}$ for which $[\prod_{\mp} \mp^{\ha_\mp}] = 1$.

\subsubsection*{The class group lattice.}
Denoting $\mT$ for the set of prime ideals with norm bounded by $B$,
we can construct the class group lattice 
$L_{\Cl} = \{ (\ha_\mp)_{\mp \in \mT} \in \Z^\mS ~|~  [ \prod_{\mp \in \mT} \mp^{\ha_\mp}] = 1 \} \subseteq \Z^{\mS}$. In this paper we %
will construct (here treated quite informally) a randomized algorithm $\mathcal{A}: \Z^{\mS} \rightarrow \Z^{\mS}$
which has the following two properties.
\begin{enumerate}[(i)]
 \item For $\vec{u} \from \mathcal{A}(\vec{z})$, we have $\vec{u} + \vec{z}  \in L_{\Cl} $ and $\|\vec{v} \| < E$ for some bound~$E$.
 \item The distribution of $\mathcal{A}(\vec{z})$ only depends on the coset
 $\vec{z} + L_{\Cl}$, i.e., $\mathcal{A}(\vec{z'})$ and $\mathcal{A}(\vec{z})$ have the same distribution for any $\vec{z}' \in  \vec{z} + L_{\Cl}$.
\end{enumerate}
This algorithm $\mathcal{A}$ uses the sampling algorithm from \Cref{part1} mentioned above. We will use algorithm $\mathcal{A}$ to sample vectors of $L_{\Cl}$ in such a way that we can prove that they will generate the full lattice with sufficiently good probability.

\subsubsection*{Gaussians over the class group lattice.} 
Let $\Gaussian_{X,s}$ denote the discrete Gaussian distribution with parameter $s$ over the discrete set~$X \subseteq \R^{\mS}$.
We will make use of the following observation: when the deviation $s$ is large enough, sampling from $\Gaussian_{\Z^{\mS},s}$ is indistinguishable from sampling first a uniform $\vec{c} \in \Z^{\mS}/L_{\Cl}$, then sampling from $\vec{c} + \Gaussian_{L_{\Cl},s}$. %

For such $\vec{z} \from \vec{c} + \Gaussian_{L_{\Cl},s}$ for $\vec{c} \in \Z^{\mS}/L_{\Cl}$ uniformly distributed, we see by property (ii) of $\mathcal{A}$ that
\[ \mathcal{A}(\vec{z}) + \vec{z}  \sim  \mathcal{A}(\vec{c}) +  (\vec{c} + \Gaussian_{L_{\Cl},s})   = (\mathcal{A}(\vec{c}) + \vec{c})  +  \Gaussian_{L_{\Cl},s} \]
By property (i) of $\mathcal{A}$, the distribution $\mathcal{A}(\vec{c}) + \vec{c}$ is \emph{bounded}, i.e., for all $\vec{z} \from \mathcal{A}(\vec{c}) + \vec{c}$ we have $\|\vec{z}\| < E'$ with overwhelming probability for some bounded~$E'$ (which depends on $E$ and on the covering radius of the lattice $L_{\Cl}$). Then the variable $\mathcal{A}(\vec{z}) + \vec{z}$ is distributed as a Gaussian over $L_{\Cl}$ with some \emph{independent `noise'} of size $E'$ caused by $\mathcal{A}(\vec{c}) + \vec{c}$. By taking the standard deviation $s$ (of all discrete Gaussians involved) much larger than this `noise' of size~$E'$, 
one can deduce that, for
$\vec{z} \from \Gaussian_{\Z^{\mS},s}$,
\[ \mathcal{A}(\vec{z}) + \vec{z}  \approx  \Gaussian_{L_{\Cl},s}. \]

Summarizing, sampling $\vec{z} \from \Gaussian_{\Z^{\mS},s}$ with large enough $s$
and computing $\mathcal{A}(\vec{z}) + \vec{z}$ allows to compute \emph{close to Gaussian samples}
from the class group lattice $L_{\Cl}$. There are well-known bounds for the number of samples required to generate the entire lattice $L_{\Cl}$ for such (close to) Gaussian samples. 

\subsubsection*{The algorithm \texorpdfstring{$\mathcal{A}$}{A}.}
The algorithm $\mathcal{A}: \Z^{\mS} \rightarrow \Z^{\mS}$
satisfying the properties (i) and (ii), is defined as follows.
For $\vec{z} \in \Z^{\mS}$, compute $\ma = \prod_{\mp \in \mT} \mp^{\vec{z}_\mp}$.
Sample $\alpha \in \ma$ according to a not-too-wide discrete Gaussian and
put $\mb = (\alpha) \cdot \ma^{-1}$. Then, use the sampling algorithm
from \Cref{part1}
to find a representative $\prod_{\mp \in \mT} \mp^{\vec{u}_\mp} \in [\mb] = [\ma]^{-1} = [\prod_{\mp \in \mT} \mp^{-\vec{z}_\mp}]$, and output $\vec{u} \in \Z^{\mS}$. We have that $\vec{u} + \vec{z} \in L_{\Cl}$ and that $\vec{u}$ is
bounded (by properties of the sampling algorithm), so property (i) of algorithm $\mathcal{A}$ is true.

Property (ii) follows from the fact that the sample $\mb = (\alpha) \cdot \ma^{-1}$ only depends on the 
class $[\ma]$ of $\ma$, and not of the representation thereof%
\footnote{This is actually false when considering the class group alone. To apply this `independence of representation' technique, the infinite places (i.e., complex embeddings) must be included in the logarithmic embedding. The \logtunit~lattice accounts for the infinite places, so this technique applies.}%
. 
Therefore, the output $\vec{u}$ only depends on the class, i.e., the coset $\vec{z} + L_{\Cl}$ of the input $\vec{z}$. 

\subsubsection*{Technicalities involving the infinite places.}
In the general case of the \logtunit~lattice, there is also a \emph{continuous}
part involving the logarithmic unit lattice. Essentially the same proof structure applies, with additional technicalities introduced by the numerical approximation of the real numbers involved.

%% file: rigorous_cgc/01d-morealgorithmic.tex
\subsection{More algorithmic problems.}
Once we have an algorithm computing the Log-$\mT$-unit lattice for any set of prime ideals $\mT$, it is well known that one can use it to obtain other quantities, such as the units of the number field, or the structure of the class group. Combined with \Cref{thm:compute-1-rel}, which allows to decompose any integral ideal as an equivalent product of prime ideals in a sufficiently large set $\mT$, this can also be used to solve other algorithmic problems, such as the principal ideal problem, or the class group discrete logarithm problem.

\subsubsection*{Units.} To obtain the units of a number field $K$, one computes the Log-$\mT$-unit lattice with an empty set~$\mT$. The output elements $\alpha_i$ of the algorithm is then a fundamental system of units of the number field. Note that these elements $\alpha_i$, even when they are represented in compact representation (see \Cref{sec:sunitrepresentation}), might have a large bit-size (as large as the running time of the algorithm). Note also that computing a fundamental system of units of $K$ only describes the units modulo the roots of unity of $K$. Computing the roots of unity of $K$ can be done efficiently using, e.g., \cite[Section 4.9.4]{Cohen1993}, and is not related to the algorithm we describe here.

\subsubsection*{Class group.} In order to obtain the structure of the class group, one chooses a set $\mT$ generating the class group. It is known that under the extended Riemann hypothesis, there are such sets $\mT$ of size polynomial in $\log |\Delta_K|$~\cite{Bach90}. Given a basis $(\alpha_i, \vec v_i)$ of the Log-$\mT$-unit lattice, one then only keeps the rank $|\mT|$ lattice $L$ generated by the vectors $\vec v_i$. The class group of $K$ is then isomorphic to $\Z^{\mS}/L$. Note that $L$ is an integral lattice whose determinant is the class-number $h_k$ of $K$. Hence, the HNF basis of $L$ has bit-size polynomial in $|\mT|$ and $\log h_k$, which is polynomial in $\log |\Delta_K|$. This means that the output size of the algorithm computing the class group is polynomially bounded in $\log |\Delta_K|$.

\subsubsection*{Class group discrete logarithm.} The class group discrete logarithm problem asks, given as input any integral ideal $\ma$ and a set of prime ideals $\mT$, to compute $\alpha \in K$ and $\vec v \in \Z^{\mS}$ such that $\ma = \alpha \cdot \OK \cdot \prod_{\fp \in \mT} \fp^{-v_\fp}$.
Solving the class group discrete logarithm problem is exactly what is done by \Cref{algo:compute-1-rel}, provided that the set $\mT$ contains sufficiently many ideals of small norm.
In order to solve the class group discrete logarithm problem for smaller sets $\mT$, one can set $\mT'$ to be the minimal set containing $\mT$ and sufficiently many small prime ideals so that \Cref{thm:compute-1-rel} applies. One can then solve the class group discrete logarithm problem in $\mT'$ and compute the log-$\mT'$-unit lattice. Solving the class group discrete logarithm in the smaller set $\mT$ can then be performed efficiently by linear algebra, using the knowledge of the log-$\mT'$-unit lattice to replace primes of $\mT' \setminus \mT$ by a product of primes of $\mT$.

\subsubsection*{Principal ideal problem.} The principal ideal problem is the problem of computing a generator $\alpha$ of a principal integral ideal $\ma$ (represented by an HNF basis). Again, this can be solved by choosing a set $\mT'$ containing sufficiently many ideals of small norm and solving the class group discrete logarithm problem with respect to $\ma$ and this set $\mT'$ using \Cref{algo:compute-1-rel}. Then, computing the log-$\mT'$-unit lattice allows one to solve the principal ideal problem by linear algebra. The size of the computed generator $\alpha$, even in compact representation (see \Cref{sec:sunitrepresentation}), might be as large as the running time of the algorithm.

%% file: rigorous_cgc/B02-prelim.tex
\section{Specialized ideal sampling algorithm} \label{section:specializedsampling}
\noindent In this section, we
specialize the main result
from \Cref{part1} 
to $\subpic = \rayPic^0$.

\begin{definition} We denote by $\idset_B$ the set of $B$-smooth ideals, i.e.,
\[ \idset_B = \{ \ma \mbox{ integral ideal of } \OK ~\big|~ \mbox{ for any prime factor } \mp \text{ of } \ma \mbox{ we have } \norm(\mp) \leq B \} \]
\end{definition}

We also extend the notion of smoothness to sets $\mT$ of prime ideals: we say that $\fa \subseteq \OK$ is $\mT$-smooth if all its prime factors belong to $\mT$: $\idset_{\mT}$ is the set of $\mT$-smooth ideals.
Recall that the local density of a set of ideals $\idset$ is defined (\Cref{def:idealdensity}) as
\[ \delta_\idset[x] = \min_{t \in [x/e^n,x]} \frac{\countfun{\idset}{t}}{\dederesidue \cdot t} = \min_{t \in [x/e^n,x]} \frac{ |\{ \mb \in \idset ~| ~ \norm(\mb) \leq t \}| }{\dederesidue \cdot t} ,\]
where, $\dederesidue = \lim_{s \rightarrow 1} (s-1)\zeta_K(s)$ (see \Cref{eq:classnumberformula}).

\Cref{thm:sampling-simplified-2} below is a specialization of \Cref{theorem:ISmain} from \Cref{part1}, together with the properties proved in \Cref{section:propertiessampling}. We will only use this specialization in this part.
\begin{theorem}[ERH]
\label{thm:sampling-simplified-2}
\label{thm:sampling-for-single-relation}
There is a randomized algorithm $\Sample$ such that the following holds.
Let $K$ be a number field, with degree $n$, discriminant $\dcrk$, and ring of integers $\OK$ of which an LLL-reduced basis is given.
 Let $\moduz \subseteq \OK$ be an ideal modulus of which the prime ideal factorization is known, 
 let $\fa \in \ideals$
 be an ideal coprime to $\moduz$, and let $\ky \in \nfrstar$ be represented by rational numbers.
 Let $\blocksize \geq 2$ be an integer, let $\radpar \in \Q_{\geq 1}$, and let $r = \radiusformula$. \\
 Given the above data, the algorithm $\Sample(\fa,\ky,\moduz,\blocksize,\radpar)$ outputs $\beta \in \ma$ coprime to $\moduz$ such that for any set $\idsetmoduG$ of ideals in $\OK$ coprime to $\moduz$, we have
 $(\beta) \cdot {\ma}^{-1} \in \idealset\cdot \idealset_{\rwB}$
with probability at least
\[ \frac{\norm(\moduz)}{\phi(\moduz)}\frac{\delta_{\idset}[r^n]}{6} , \] for some smoothness bound
$\rwB = \poly(\log |\dcrk|, \log \norm(\moduz) )$, 
and where $\phi(\moduz) = |(\OK/\moduz)^\times|$.
Furthermore, the algorithm $\Sample$ runs in time
\[ \poly(\allowbreak  \size(\fa), \allowbreak \size(y),\allowbreak
\log(\norm(\moduz)), \allowbreak \hkztime, \allowbreak \log |\dcrk|,\allowbreak \size(\radpar)). \]
Additionally, the output distribution $\distr_{\ky ,\ma}$ of $\Sample$ on input $(\ky ,\ma) \in (\nfrstar, \ideals)$ (with $\moduz, \blocksize$ and $\radpar$ implicit) satisfies the following properties. 
\begin{enumerate}
 \item \label{item:simplipschitz}(Almost Lipschitz-continuous).  For all $\ky,\ky' \in \nfrstar$,  
 \[ \|\distr_{\ky,\ma} - \distr_{\ky', \ma}\|_1 \leq \frac{n^2}{2} \|\Log(\ky) - \Log(\ky')\| + \frac{1}{1200 \cdot |\dcrk| \cdot r^n } .\]
 \item \label{item:simpschifting}(Shifting property). For all $\beta \in K$ and $\alpha \in \Kmoduz$, we have the following identity of distributions:
 \[ \distr_{\ky,\ma}[\placeholder] = \distr_{\ky \cdot (|\sigma(\alpha)^{-1}|)_\sigma, \ma \cdot (\alpha)}[\alpha \cdot \placeholder] \]
 \item \label{item:simpbounded}(Bounded). Writing $\ldb \beta \rdb \in \Div_K^0$ and $\divnorm{\cdot}$ for the norm on $\Div_K$, we have
 \[  \divnorm{\ldb \beta \rdb } \leq 
  \divnorm{ d(\ma) +\Log(\ky) } + O(\log^2 |\dcrk| + \log \norm(\moduz) + n\log(\radpar))
 \]
 \item \label{item:simpindependent}(Independent of $|\norm(\ky)|$). $\distr_{\ky , \ma} = \distr_{\ky^0 , \ma}$,
where $\ky^0 = \ky/|\norm(\ky)|^{1/n}$.
\end{enumerate}

\end{theorem}

\begin{remark} \label{remark:moduz-is-OK-2} In the case where $\moduz = \OK$, we get $\norm(\moduz) = \phi(\moduz) = 1$, yielding a lower bound $\delta_{\mS}[r^n]/6$ on the success probability.
\end{remark}

\begin{proof} The algorithm $\Sample$ is \Cref{alg:samplerandom} from Part \partref{part1} instantiated with modulus $\modu := \moduz$, subgroup $\subpic := \rayPic^0$ (yielding $[\rayPic^0:\subpic]=1$), ideal $\mbb := \ma$,  $\ky := \ky \in \nfrstar$, %
block size $\blocksize := \blocksize$ and element %
$\tau \in (\OK/\moduz)^\times$ uniformly random\footnote{Uniform sampling of $\tau \in (\OK/\moduz)^\times$ can be done in time $\poly(\log \norm(\moduz),\log |\dcrk|)$ by sampling uniformly in $(\OK/\fq)^\times$ for all prime power divisors $\fq \mid \moduz$ and using the Chinese Remainder theorem. This uses that the factorization of $\moduz$ is assumed to be known.} if $\moduz \neq \OK$ and $\tau = 1$ if $\moduz = \OK$ (note that $\tau\Kmodu = \Kmodu = K^*$ if $\modu = \moduz = \OK$). 
We put the error parameter $\eps  :=  \frac{1}{1200\cdot |\dcrk| r^n} \leq  \frac{1}{6 \cdot |\dcrk| \cdot r^n} \leq \frac{1}{6 \cdot \dedres \cdot r^n}$ (see \Cref{eq:bounddedres}). We will use \Cref{theorem:ISmain}.

Note that we may assume that $\delta_{\idset}[r^n] = \min_{t \in [r^n/e^n,r^n]} \frac{ |\{ \mb \in \idset ~| ~ \norm(\mb) \leq t \}| }{\dederesidue \cdot t} > 0$, since otherwise the statement of the theorem is inherently true, as any probability is at least zero. Hence, the set $\{ \mb \in \idset ~| ~ \norm(\mb) \leq r^n/e^n \}$ is nonempty, and we can immediately deduce
\begin{equation} \delta_{\idset}[r^n] = \min_{t \in [r^n/e^n,r^n]} \frac{ |\{ \mb \in \idset ~| ~ \norm(\mb) \leq t \}| }{\dederesidue \cdot t} \geq \frac{1}{\dedres \cdot r^n} \geq 6 \cdot  \eps. \label{eq:smalleps} \end{equation}

By \Cref{theorem:ISmain}, the success probability of \Cref{alg:samplerandom} with these parameters is at least
\[ \frac{\norm(\moduz)}{\phi(\moduz)}\frac{\delta_{\idset}[r^n]}{3} - \eps \geq \frac{\norm(\moduz)}{\phi(\moduz)}\frac{\delta_{\idset}[r^n]}{3} - \frac{\delta_{\idset}[r^n]}{6} \geq \frac{\norm(\moduz)}{\phi(\moduz)}\frac{\delta_{\idset}[r^n]}{6}.\]
where the first inequality holds by \Cref{eq:smalleps} and the second one by the fact that $\norm(\moduz) \geq \phi(\moduz)$. This proves the lower bound on the success probability of algorithm $\mathcal{A}$ in the theorem.

It remains to show the bound on $\rwB$ and the running time. Note that $[\rayPic^0:G] = 1$,
$\size(\tau) = 1$, %
and that 
\[\log(1/\eps) = O(\log(|\dcrk| r^n)) = O( \log^2 |\dcrk|+\log \norm(\moduz)+ \log \radpar)\]
by direct computation.
Hence 
\[\rwB = \poly(\log |\dcrk| , \log \norm(\moduz), \allowbreak \log \radpar)\]
and we have a running time of 
\[T = \poly(\hkztime, \allowbreak  \size(\fa), \allowbreak \size(y),\allowbreak
\log(\norm(\moduz)), \allowbreak \log |\dcrk|, \allowbreak \size(\radpar)),\]
with no queries to any oracle, since $G = \rayPic^0$.

The properties (\itemref{item:simplipschitz})-(\itemref{item:simpbounded}), directly follow from the almost-Lipschitz continuity (\Cref{lemma:idealsamplinglipschitz}), the shifting property (\Cref{lemma:shifting}) and the boundedness property (\Cref{lemma:boundeddivisor}) of \Cref{alg:samplerandom}, with the following extra remarks. For the almost-Lipschitz continuity, 
we use that $\|\Log(y^0) - \Log(z^0)\| \leq \| \Log(y) - \Log(z)\|$ for $y^0 = \frac{\ky}{|\norm(\ky)|^{1/n}}$ and $z^0 = \frac{z}{|\norm(z)|^{1/n}} \in \nfr^0$. For the Lipschitz error,
we just substitute $\eps = \frac{1}{1200 \cdot |\dcrk| r^n}$. Here, $\Log(\ky),\Log(z) \in \Log(\nfrstar) \hookrightarrow \Div_K$ inherit their `natural' norm from the norm on $\Div_K$ (see \Cref{def:normondiv}).

For the boundedness property we use that 
\[\log(B^N r^n) = O(\log^2 |\dcrk| + \log \norm(\moduz) + n\log(\radpar))\] 
and $\sd = 1/n^2$ for the parameters used in \Cref{alg:samplerandom}. Additionally, with the instantiation $\eps = \frac{1}{1200 |\dcrk| r^n}$ we have
\[ \sd \sqrt{n \log(8n^2/\eps)} \leq O(\log|\dcrk|+\log(r)) \leq O( \log^2 |\dcrk|+\log \norm(\moduz)+ \log \radpar) . \]
With these simplifications, \Cref{lemma:boundeddivisor} tells us that $\divnorm{\ldb \beta \rdb } \leq \divnorm{ d^0(\ma) +\Log(\ky) } + O(\log^2 |\dcrk| + \log \norm(\moduz) + n\log(\radpar))$. In order to obtain the desired bound (i.e., with $d(\ma)$ instead of $d^0(\ma)$), we use property (\itemref{item:simpindependent}) and we replace the input $(\ma,\ky)$ by $(\mb,z) := (\ma,\ky \cdot \norm(\ma)^{1/n})$, so that $d^0(\mb)+\Log(z) = d(\ma)+\Log(\ky)$.

For the shifting property, we use the fact that $d(\ba) + \Log(y) + \ldb \alpha \rdb = d(\fa \cdot (\alpha)) + \Log(y \cdot (|\sigma(\alpha)|^{-1})_\sigma) \in \Div_K$. 
For the last property, item (\itemref{item:simpindependent}), note that $\Sample$ does not depend on $|\norm(\ky)|$, as follows from \Cref{lemma:alg2independentnorm}. 
\end{proof}

\begin{remark} \label{remark:part2OK2} 
In the $\mS$-unit computation algorithm in this paper, we assume that we are given an LLL-reduced basis of the ring of integers $\OK$. In other words, we take such a basis of $\OK$ as part of the input. We refer the reader to the discussion in \Cref{sec:representation}.
\end{remark}

\begin{remark}
Looking into the details of \Cref{part1},
the part $\hkztime$ in the running time of \Cref{theorem:ISmain} is caused by a $\blocksize$-BKZ reduction on an ideal lattice $\ma$ (see \Cref{section:sampling}).
One may wonder whether, in the context of this section, using an ideal-SVP algorithm specifically designed for ideals (e.g.,~\cite{CDW21, PHS19}), instead of the BKZ algorithm, might lead to better results.
This is unfortunately not the case at the moment, since all current algorithms
exploiting the special structure of ideal lattices are heuristic. Additionally, they all use an algorithm computing the class group and the unit group of $K$ as a building block, which would be a circular requirement here.
\end{remark}

\begin{remark} \Cref{thm:sampling-simplified-2} can also be made to work for arbitrary $\tau \in \Kmodumodu$ (with an additional $\size(\tau)$ in the running time). Even a random $\tau$ is possible (which might be useful in certain use-cases), though then the factorization of $\moduz$ must be known. Indeed, using this factorization, uniform sampling of $\tau \in (\OK/\moduz)^\times$ can be done in time $\poly(\log \norm(\moduz),\log |\dcrk|)$ by sampling uniformly in $(\OK/\fq)^\times$ for all prime power divisors $\fq \mid \moduz$ and using the Chinese Remainder theorem.
\end{remark}

%% file: rigorous_cgc/B03-density-smooth-ideals.tex
\section{Density of smooth ideals}
\label{sec:smooth}

\noindent
From the early work of Hafner and McCurley~\cite{HM89} for imaginary quadratic fields, to the generalizations for larger degree fields~\cite{buchmann1988subexponential,ANTS:BiasseFiecker14}, the core of the strategy to compute class groups as been to find random relations between classes involving smooth ideals.
The probability of a somewhat random ideal to be smooth thus plays a key role in the analysis of such algorithms.

For any $0 < B \leq x$, let $\Psi_K(x,B)$ be the number of $B$-smooth integral ideals of $K$ of norm at most $x$.
It is well known~\cite{Krause90} that whenever $(\log \log x)^{5/3 + \varepsilon} \leq \log B \leq \log x$, we have
\begin{equation}\label{eq:useless-estimate}
\Psi_K(x,B) = x \cdot \dedres \cdot \rho(u)\left(1 + O_{\varepsilon,K}\left(\frac{\log(u+1)}{\log B}\right)\right),
\end{equation}
where $\rho$ is the Dickman function, and $u = \frac{\log x}{\log B}$.
One then needs to estimate the dependence in $K$ hidden in the big O error.
Unfortunately, scrutinizing the proof of~\cite{Krause90}, and using the best known uniform estimates of the ideal counting function~\cite{sunley1972class}, one finds that the hidden constant is exponential in the degree $n$ of $K$. This means that for the estimate of \Cref{eq:useless-estimate} to be non-vacuous, one needs the smoothness bound $B$ to be doubly exponential in $n$.

It is not clear at the time whether this issue is unavoidable (i.e., the estimate $\Psi_K(x,B) \sim x \cdot \dedres \cdot \rho(u)$ is truly only valid when $B$ is large compared to the degree), or is an artifact of the techniques not being precise enough. Therefore, instead of this estimate, we rely on the following weaker but simpler combinatorial bound.

\newcommand{\varB}{B}
\begin{lemma}[ERH]
\label{lemma:proba-smooth}
For any $\varepsilon>0$, there exists a constant $C$ such that the following holds.
Let $\varB \geq \smoothdensityB := C(n+\log|\Delta_K|)^{2+\varepsilon}$, and let $A \leq \varB/(4\log \varB)$. Let $\idset_{A,\varB}$ be the set of integral ideals of $\OK$ whose prime divisors have norm in $(A,\varB]$. Let $x \geq \varB \cdot e^n$, and $u = \frac{\log x}{\log \varB}$. Then the density at $x$ of $\idset_{A,\varB}$ satisfies
$$ \delta_{\idset_{A,\varB}}[x] \geq \frac{(4\log \varB)^{1-u}}{\rho_K \varB}u^{-u}.$$
\end{lemma}
\begin{proof}
Let $y \in [x/e^n,x]$. Note that, by the assumption $x \geq \varB \cdot e^n$, we have $y \geq \varB$.
Let $(\mathfrak p_i)_{i = 1}^N$ be the list of all prime ideals in $\OK$ of norm in $(A,\varB]$. We have
\begin{align*}
|\idset_{A,\varB}(y)| &= \left|\left\{\prod_{i=1}^N \mathfrak p_i^{e_i} \middle| \sum_{i=1}^N e_i\log(\norm(\mathfrak p_i)) \leq \log y\right\}\right| \geq \left|\left\{(e_i) \in \Z_{\geq 0}^N \middle| \sum_{i=1}^N e_i \leq \frac{\log y}{\log \varB}\right\}\right|.
\end{align*}
Writing $u_y = \frac{\log y}{\log \varB}$ and $v = \lfloor u_y \rfloor \geq 1$ (since $y \geq B$), we deduce $$|\idset_{A,\varB}(y)| \geq \binom{v+N-1}{v} \geq \left(\frac{v+N-1}{v}\right)^v \geq N^vv^{-v} \geq N^vu_y^{-u_y}.$$
There exists $D > 0$ depending only on $\varepsilon$ such that for any $X > D$, we have $X^{1/(2+\varepsilon/2)} < X^{1/2}/(\log X)^2$.
Then, by a result of Lagarias and Odlyzko~\cite{LO77}, there is a constant%
\footnote{The original statement \cite{LO77} (putting $L= K$ and $G = \{1\}$) reads $|\pi_K(X) - \mbox{Li}(X)| = O\left(X^{1/2} \log(|\dcrk| X^{n}) + \log |\dcrk|\right)$, where the constants in the big-$O$ are absolute, and $\mbox{Li}$ is the logarithmic integral.}%
$E$ (absolute) such that for any $X \geq \max(D,E(n+\log|\Delta_K|)^{2+\varepsilon/2}) =: F$, we have $\frac{X}{2\log X} \leq \pi_K(X) \leq X $. Let $A' = \max(A,F)$. 
Recall that $A \leq \varB/(4\log \varB)$. If furthermore $F \leq \varB/(4\log \varB)$, 
then, $A' \leq \varB/(4\log \varB)$, and we get
$$N = \pi_K(\varB) - \pi_K(A) \geq \frac{\varB}{2\log \varB} - \pi_K(A')
\geq \frac{\varB}{2\log \varB} - A'
 \geq \frac{\varB}{4\log \varB}.$$

We now show that we can choose the constant $C$ such that indeed $F \leq \varB/(4\log \varB)$ (hence the conditions of the lemma do imply $N \geq \frac{\varB}{4\log \varB}$). There exists $G > 0$ depending only on $\varepsilon$ such that for any $X > G$, we have $X^{\frac{2+\varepsilon/2}{2+\varepsilon}} < X/(4 \log X)$.
Therefore, for any $\delta>0$ and $\varB \geq \max(G,F^\delta)$, we have 
$$\varB/(4 \log \varB) \geq \varB^{\frac{2+\varepsilon/2}{2+\varepsilon}}
\geq F^{\delta\frac{2+\varepsilon/2}{2+\varepsilon}}.
$$
Choosing $\delta = \frac{2+\varepsilon}{2+\varepsilon/2}$, we get that $\varB/(4 \log \varB) \geq F$ (as desired).
Now, we have $F^\delta = \max(D^\delta,E^\delta(n+\log|\Delta_K|)^{2+\varepsilon})$. One can therefore choose the constant $C = \max(G,D^\delta,E^\delta)$,
to enforce $\varB/(4 \log \varB) \leq F$.

We have just proved that $N \geq \frac{\varB}{4\log \varB}$. Then, we have
$$N^v \geq N^{u_y-1} \geq (y/\varB)(4\log \varB)^{1-u_y}. %
 $$
We obtain
$|\idset_{A,\varB}(y)| \geq (y/\varB)(4\log \varB)^{1-u_y}u_y^{-u_y},
$
hence, writing $u = u_x$, we have
$$\frac{|\idset_{A,\varB}(y)|}{\rho_K \cdot y} \geq \frac{(4\log \varB)^{1-u_y}}{\rho_K \varB}u_y^{-u_y} \geq \frac{(4\log \varB)^{1-u}}{\rho_K \varB}u^{-u}.$$
\end{proof}

%% file: rigorous_cgc/B04a-s-unit-background.tex
\section{Background on \texorpdfstring{$\mS$}{\letterSunits}-unit lattices}
\label{section:backgroundsunit}

\subsection{\texorpdfstring{$\mS$}{\letterSunits}-units}
In this section, we formally define $\mS$-units and the Log-$\mS$-unit lattice, and recall some of its properties.
\begin{definition}[$\mT$-units]
\label{def:S-units}
Let $\mS$ be a set of prime ideals of $\OK$.
An element $\alpha \in K$ is an $\mS$-unit if and only if $\alpha \cdot \OK = \prod_{\fp \in \mS} \fp^{v_\fp}$ for some $(v_\fp)_{\fp \in \mS} \in \Z^{\mS}$. We write $\sunits \subset K$ for the set of $\mS$-units.
\end{definition}

We will use the language of divisors introduced in \Cref{section:divisors}.

\begin{definition}[$\mS$-Divisor group]
 \label{def:S-divisor}
 Let $\mS$ be a (finite) set of prime ideals of $\OK$. We define the $\mS$-divisor group $\Div_{K,\mS} \subseteq \Div_K$ 
 \[ \Div_{K,\mS} := \bigoplus_{\mp \in \mS} \Z \times \bigoplus_{\nu} \R,\]
 where  $\nu$ ranges over the set of all infinite places  (embeddings into the complex numbers up to possible conjugation). 
Additionally, we define the degree-zero $\mS$-divisor group $\Div_{K,\mS}^0 = \Div_{K,\mS} \cap \Div_K^0$.
\end{definition}
In the present part of the article, 
we often write elements of $\Div_{K,\mS}$ as a vector $(  (\ha_\mp)_{\mp \in \mS} , (a_\nu)_\nu)$ to emphasize that it has finite dimension and that these elements are computationally represented as lists. %
We have the inclusion $H \subseteq \Log(\nfrstar) \subseteq \Div_{K,\mS}$.
This degree-zero $\mS$-divisor group serves as an ambient space for the Log-$\mS$-unit lattice,
which is defined as follows.

\begin{definition}[Log-$\mS$-unit lattice] \label{def:logsunits}
Let $\mS$ be a set of prime ideals of $\OK$. The map $\logs$ is defined over $\sunits$ by
\begin{align*}
\logs:  \sunits &\longrightarrow \Div_{K,\mS}^0 \\
		\alpha & \longmapsto ((-v_\fp)_{\fp \in \mT}, \Log(\alpha) )  = - \ldb \alpha \rdb.
\end{align*}
where the $v_\fp$ are such that $\alpha \cdot \OK = \prod_{\fp \in \mT} \fp^{v_\fp}$, and
where $\Log(\alpha) = (\npl \log|\embn(\alpha)|)_\nu \in \bigoplus_{\nu} \R$, with $\npl = 2$ if $\pl$ is complex and $1$ otherwise, see \Cref{section:logembedding}.

\noindent The Log-$\mT$-unit lattice is defined as
\[ \logs(\sunits) = \{\logs(\alpha)\,|\, \alpha \in \sunits\} \subset \Div_{K,\mS}^0. \]
\end{definition}

\begin{remark}
Note that the image of $\logs$ on $\sunits$ is indeed in $\Div_{K,\mS}^0$, since this map $\logs$ 
is just the \emph{negative} of the `principal divisor' map $\ldb \cdot \rdb:K \rightarrow \Div_K$ restricted to $\sunits$, and 
principal divisors have degree zero (see \Cref{section:arakelovclassgroup}).
\end{remark}

The Log-$\mT$-unit lattice is a lattice in the real vector space
\begin{equation} \label{eq:Rdiv} \Div_{K,\mS}^0(\R) := \Big\{   ( (x_\mp)_{\mp \in \mS}, (x_\nu)_\nu) \in  \bigoplus_{\mp \in \mS} \R \times \bigoplus_{\nu} \R ~\Big|~ \sum_{\mp \in \mS} x_\mp \cdot \log \norm(\mp) + \sum_{\nu } x_\nu  = 0 \Big\}. \end{equation}
 Here the sum formula in above definition is a generalization of the degree function 
to real $x_\mp \in \R$. This lattice $\logs(\sunits)$ has
rank $\dimh + |\mS|$ where $\dimh = \rem + \cem -1$, i.e., it is full rank in $\Div_{K,\mS}^0(\R)$.

Note that $\logsunits \subseteq \Div_{K,\mS}^0 \subseteq \Div_K$ inherits a Euclidean length notion via that of $\Div_K$ (see \Cref{def:normondiv}).
When $\mT = \emptyset$, the lattice $\logsunits \subset \R^n$ coincides with the Log-unit lattice $\logunits$.

\begin{lemma}  %
\label{lemma:det-log-S-unit}
If $\mS$ is a set of prime ideals generating the class group, then it holds that
\footnote{We understand $\covol(\logsunits)$ as the volume of the quotient $\Div_{K,\mS}^0(\R)/\logsunits$.} %
 $\covol(\logsunits) = h_K \cdot \covol(\logunits)= h_K \cdot \reg \cdot \sqrt{\rem + \cem}$, where $h_K$ is the class number of $K$, $\reg$ is the regulator, $\rem$ is the number of real embedding and $\cem$ is the number of complex pairs of embeddings. In particular, $\covol(\logsunits)$ does not depend on the choice of $\mS$ (as long as it generates the class group). Moreover, we have
\[
\log (\covol(\logsunits)) \leq \log |\dcrk|.
\]
\end{lemma}

\begin{proof}
The lemma follows from \Cref{lemma:boundvolraypicappendix} and the observation that the group $\Pic_K^0$ in~\cite{dBDPW} is isomorphic to $\Div_{K,\mS}^0/ \logsunits$, whenever $\mS$ generates the class group.
\end{proof}

\noindent A lower bound on the first minimum of the lattice $\logsunits$ can be derived from Kessler's lower bound on the first minimum of the Log-unit lattice (\Cref{le:lower-bound-first-minimum-log-unit}).
\begin{lemma}
\label{le:lower-bound-first-minimum-log-S-unit}
For any set of prime ideals $\mS$, it holds that $\lambda_1(\logsunits) \geq \kesslerformula$.
\end{lemma}

\begin{proof}
Let $\alpha \in \sunits$ be such that $\logs(\alpha)$ reaches the first minimum of $\logsunits$, i.e., $\|\logs(\alpha)\| = \lambda_1(\logsunits)$. Let $v_\fp \in \Z$ be such that $\alpha \cdot \OK = \prod_{\fp \in \mT} \fp^{v_\fp}$ (we know that they exist since $\alpha$ is an $\mT$-unit).
Assume first that one of the $v_\fp$ is non-zero (say $v_{\fp_0}$), then $\|\logs(\alpha)\| \geq |v_{\fp_0}| \geq 1$ since $v_{\fp_0}$ is an integer. This gives us $\lambda_1(\logsunits) \geq 1 \geq \kesslerformula$ as desired.

Suppose now that all $v_\fp$'s are zero. Then $\alpha \in \OK^\times$ is a unit and $\|\logs(\alpha)\| = \|\Log(\alpha)\|$, and the bound follows from~\Cref{le:lower-bound-first-minimum-log-unit}.
\end{proof}

\subsubsection*{Representation of the log-$\mT$-unit lattice.} 
In the present part of this article, 
we will compute a generating set of the Log-$\mT$-unit lattice by a collection of pairs $((v^{(j)}_\fp)_{\fp \in \mT},\alpha^{(j)}) \in \Z^{\mS} \times K$, which generate $\logsunits$ as a $\Z$-module. How this is done is explained more precisely in \Cref{sec:BF-compute-many-rel}. Up to and including that section, the elements $\alpha^{(j)}$ in these pairs are represented by their coordinates with respect to the basis of $\OK$ (as explained in \Cref{sec:representation}).

In the sections after that, \Cref{section:postproc} and \Cref{sec:full-algorithm}, instead of a generating set, a \emph{basis} of $\logsunits$ is computed. This will again lead to a collection of pairs $((w_\fp)_{\fp \in \mT},\beta) \in \logsunits$ but this time the $\beta \in \logsunits$ are written in a so-called \emph{compact representation}, see \Cref{sec:sunitrepresentation}. Indeed, the elements $\beta$ will be of the form $\beta = \prod_{j} (\alpha^{(j)})^{N_{j}}$, where $\alpha^{(j)} \in K$ come from the generating set, and $N_{j} \in \Z$. The product is \emph{not} evaluated but rather stored as $( \alpha^{(j)}, N_{j})_j$. Even in this compact representation, the bit size of the final $N_{j}$ can be as large as the running time of the entire algorithm. For more details, see \Cref{sec:sunitrepresentation} and \Cref{sec:full-algorithm}.

%% file: rigorous_cgc/B04b-covering-radius-s-unit-lattice.tex
\section{The generating and covering radius of the \texorpdfstring{log-$\mS$-unit}{log-S-unit} lattice}
\label{app:proof-covering-radius}

\noindent
In this section, we give an upper bound on the generating radius $\rr(\logsunits)$, which is related to the covering radius of the Log-$\mT$-unit lattice (see \Cref{definition:rr} and \Cref{lemma:bl-leq-2cov}). Recall that $\rr(\logsunits)$ is the smallest real number $r > 0$ such that $\logsunits$ can be generated by vectors of Euclidean norm $\leq r$. Having an upper bound on this quantity $r$ will be useful in \Cref{sec:BF-compute-many-rel}. Indeed, we will first see in \Cref{sec:BF-compute-1-rel} how to compute a single vector in $\logsunits$, and we will then want, in \Cref{sec:BF-compute-many-rel}, to compute a generating set of~$\logsunits$. In order to do so, we will need to know how short we can expect the vectors of this generating set to be, which is exactly the generating radius $\rr(\logsunits)$.

The first part of this section (until and including \Cref{lemma:cov-small-set-primes}) is devoted to computing an upper bound on the covering radius\footnote{This covering radius makes sense when seeing $\logsunits[\mT']$ as a lattice in the real vector space $\Div_{K,\mS}(\R)$, see \Cref{eq:Rdiv}.}
$\logsunits[\mT']$ of the Log-$\mT'$-unit lattice for a specific set $\mT'$ (of relatively small cardinality). This upper bound is obtained due to the random walk result from~\cite[Theorem 3.3]{dBDPW}. When rephrased with our formalism, and with a bit of work, the random walk theorem from~\cite{dBDPW} states that there exists a distribution $\Walk^T$ outputting somewhat short vectors of $\Div_{K,\mS'}^0$ (the ambient space of $\logsunits[\mT']$) and such that $\Walk^T \bmod \logsunits[\mT']$ is close to uniform in $\Div_{K,\mS'}^0 \bmod \logsunits[\mT']$. The shortness bound on the output vectors of $\Walk^T$ is proven in~\Cref{lemma:truncated-small-vector}, and the close to uniformity modulo $\logsunits[\mT']$ is obtained by combining \Cref{thm:rw} and \Cref{lemma:truncated-small-stat-distance}.
Intuitively, this means that any vector from $\Div_{K,\mS'}^0 \bmod \logsunits[\mT']$ is somewhat close to a vector in the support of the distribution $\Walk^T$, which are all somewhat short. Hence, the covering radius of $\logsunits[\mT']$ is somewhat small (see \Cref{lemma:cov-small-set-primes}). In order to formalize this intuition, and because our distribution $\Walk^T$ is not perfectly uniform but only statistically close to uniform, we need to show that the volume of a small ball in $\Div_{K,\mS'}^0 \bmod \logsunits[\mT']$ is not too small (it could happen that the ball folded modulo $\logsunits[\mT']$ becomes much smaller than the original unfolded ball). This is what is done in \Cref{lemma:proba-not-too-small}.

Once we have an upper bound on $\cov(\logsunits[\mT'])$, we also have an upper bound on $\rr(\logsunits[\mT'])$ by \Cref{lemma:bl-leq-2cov}. We then need to extend this to larger sets $\mT \supseteq \mS'$, which is done in \Cref{lemma:bl-decreases}. We could also extend directly the bound on $\cov(\logsunits)$ to larger sets $\mT$, but this would grow as $\sqrt{|\mT|}$, whereas our sharper bound on the generating radius $\rr(\logsunits)$ only grows as $\max_{\fp \in \mT} \log \norm(\fp)$. Using the larger bound on $\cov(\logsunits)$ would have no impact on the asymptotic complexity of our algorithm, but we prefer to keep the intermediate results as tight as possible, to facilitate reusability in other works.
Overall, our upper bound on $\rr(\logsunits)$, which is the main result of this section, is stated in \Cref{lemma:cov-radius}.

We start by rephrasing the random walk theorem from~\cite[Theorem 3.3]{dBDPW} (see also \Cref{def:rwdiv}) in $\mS$-unit terminology.
As already mentioned in the preliminaries, the quotient group $\Div_{K,\mS}^0 / \logsunits$ is isomorphic to the Arakelov class group $\Pic_K^0$,
provided that $\mT$ generates the class group.

\begin{definition}[Random walk distribution, rephrased from~{\cite[Definition 3.1]{dBDPW}}]
\label{def:walk}
Let $\mT$ be a set of prime ideals of $K$, $N \in \Z_{>0}$ and $s > 0$. The distribution $\Walk(\mT, N,s)$ is a probability distribution over $\Div_{K,\mS}^0$
obtained by the following procedure.
\begin{enumerate}
\item Set $\vec y = \vec 0 \in \Div_{K,\mS}$. The first $|\mT|$ coordinates of $\vec y$ are denoted by $y_\mp$ (with $\mp \in \mS)$ and the last $\rem + \cem$ coordinates of $\vec y$ are denoted by $y_\nu$ (with $\nu$ the places associated with the (pairs of) embeddings $\sigma:K \rightarrow \C$).
\item Sample $\vec x \in H$ from a continuous centered Gaussian distribution with standard deviation $s$, where $H = \Span_\R(\logunits)$
(i.e., $H = \Log(\nfr^0) = \{ (x_\pl)_\pl \in \bigoplus_{\nu} \R ~|~ \sum_{\pl} x_\pl = 0 \}$). 
Set $y_\nu = x_{\nu}$ for all places $\nu$.
\item For j from $1$ to $N$, sample $\fp$ uniformly at random in $\mT$ and update $y_{\fp} = y_{\fp} +1$ and $y_\pl = y_\pl - \frac{\npl \log \norm(\fp)}{n}$ for all places $\pl$. Here $\npl = 2$ if $\pl$ is complex and $1$ otherwise.
\item Return $y$
\end{enumerate}
\end{definition}

Observe that the distribution $\Walk(\mT, N,s)$ produces vectors in $\Div_{K,\mS}^0$, 
(i.e., vectors $\vec y \in \bigoplus_{\fp \in \mS} \Z \times \bigoplus_{\nu} \R$ such that %
$\sum_{\nu} y_\nu + \sum_{\fp \in \mT} y_{\fp} \cdot \log \norm(\fp) = 0$.

\begin{theorem}[{\cite[Theorem 3.3]{dBDPW}}, ERH]
\label{thm:rw}
Let $\varepsilon > 0$ and $\sd > 0$ be positive real numbers and $k \in \Z_{>0}$.
Let $s' = \min(\sqrt{2} \cdot \sd,1/\eta_1(\logunits^\vee))$, where $\eta_1(\logunits^\vee)$ is the smoothing parameter of the dual lattice of the log-unit lattice $\logunits$.

\noindent Then there exists a bound 
\[B = \widetilde O(n^{2k} [n^2 (\log \log(1/\varepsilon))^2 + n^2 (\log(1/s'))^2 +  (\log |\disc|)^2])\]
 such that for any 
 \[N \geq \frac{ \frac{\rem+\cem-1}{2} \cdot \log(1/s') + \frac{1}{2} \log(\covol(\logsunits)) + \log(1/\varepsilon) + 1}{k\log n},\] 
the random walk distribution $\Walk(\mT,N,\sd) \bmod \logsunits$ is within statistical distance at most $\varepsilon/2$ from uniform in $\Div_{K,\mS}^0/ \logsunits$,
where $\mT = \{\fp \text{ prime ideal }\,|\allowbreak\, \norm(\fp) \leq B\}$.
\end{theorem}

\begin{proof}
This theorem is a direct translation of~\cite[Theorem 3.3]{dBDPW} (or \Cref{thm:random-walk-weak}) with this section's $\mS$-unit terminology: recall that if $\mT$ generates the class group, then $\Pic_K^0$ is isomorphic to $\Div_{K,\mS}^0 / \logsunits$. By a result of Bach \cite{Bach90}, $\mS$ indeed generates the class group for $B \geq 12 \log^2|\dcrk|$, hence indeed the bound in the theorem statement suffices.
\end{proof}

In order to prove \Cref{lemma:cov-radius}, we will use a tail-cut random walk distribution.
\begin{definition}[Tail-cut random walk distribution]
\label{def:walk-tail-cut}
Let $\mT$ be a set of prime ideals of $K$, $N \in \Z_{>0}$, $\sd > 0$ and $\varepsilon \in (0,1]$. The distribution $\WalkT(\mT, N,s, \eps)$ is a probability distribution over $\Div_{K,\mS}^0$ obtained as in \Cref{def:walk}, except that the Gaussian distribution used to sample $\vec x$ is tail-cut at $t := \sd \cdot \sqrt{2 \cdot n \cdot \log\big(\frac{4 \cdot n}{\eps}\big)}$. In other words, $\vec x$ is sampled from $\Gaussian_{H,\sd}$ conditioned on $\|\vec x \| \leq t$.
\end{definition}

\begin{lemma}
\label{lemma:truncated-small-stat-distance}
For any set of prime ideals $\mT$, $N \in \Z_{>0}$, $\sd > 0$ and $\eps \in (0,1]$, the statistical distance between $\Walk(\mT,N,s)$ and $\WalkT(\mT,N,s,\eps)$ is upper bounded by $\eps/2$. 
\end{lemma}

\begin{proof}
By the data processing inequality (\Cref{theorem:dataprocessinginequality}), it suffices to prove that the statistical distance between $\Gaussian_{H,\sd}$ and its tail-cut variant is bounded by $\eps/2$. This directly follows from \Cref{lemma:bound-gaussian}.
\end{proof}

\begin{lemma}
\label{lemma:truncated-small-vector}
Let $\mT$ be a set of prime ideals, $N \in \Z_{>0}$, $\sd > 0$ and $\eps \in (0,1]$. Then the support of the distribution $\WalkT(\mT,N,\sd,\eps)$ is included in $\{\vec x \in \Div_{K,\mS}^0 \,|\, \|\vec x \| \leq R\}$, where
\[ R = \sd \cdot \sqrt{2 \cdot n \cdot \log\left(\frac{4 \cdot n}{\eps}\right)} + N \cdot \left( 1 + \max_{\fp \in \mT} \log \norm(\fp) \right).\]
\end{lemma}

\begin{proof}
Let $\vec y$ be output by $\WalkT(\mT,N,\sd,\eps)$. By the definition of this distribution, we have $\vec y \in \Div_{K,\mS}^0$. 
We want to show that $\|\vec y \| \leq R$.

Let us consider the components of $\vec y = ((\vec{y}_\mp)_{\mp \in \mS},(\vec{y}_\pl)_\pl)$ in $\bigoplus_{\fp \in \mS} \Z$ and $\bigoplus_{\nu} \R$ separately. The vector $(\vec{y}_\mp)_{\mp \in \mS}$ satisfies $\| (\vec{y}_\mp)_{\mp \in \mS} \| \leq \| (\vec{y}_\mp)_{\mp \in \mS} \|_1 \leq N$, by definition.
From the tail-cut definition of $\WalkT$, we also know that $(\vec{y}_\pl)_\pl$ is the sum of a vector of norm $\leq t$ (where $t$ is defined in \Cref{def:walk-tail-cut}) and of $N$ vectors of norm at most $\frac{\sqrt{2n} \cdot \max_{\fp \in \mT} \log \norm(\fp)}{n} \leq \max_{\fp \in \mT} \log \norm(\fp)$. Summing all these terms provides the desired upper bound on $\|\vec y\|$.
\end{proof}

We will also need the following auxiliary lemma, which lower bounds the probability that a uniform element modulo $\logsunits$ falls in a small ball.
We denote 
\[
t \Ballinf := \{ ((x_\mp)_{\mp},(x_\nu)_\nu) \in \Div_{K,\mS}^0(\R) ~|~ \| ((x_\mp)_{\mp},(x_\nu)_\nu)\|_\infty \leq t \},
\]
where $\Div_{K,\mS}^0(\R)$ was defined in \Cref{eq:Rdiv} (note that this is the $\R$-span, even for the coordinates in $\mS$).
Recall that 
$\unif(\Div_{K,\mS}^0/\logsunits)$ is the uniform distribution over the compact group $\Div_{K,\mS}^0/\logsunits$.
\begin{lemma}
\label{lemma:proba-not-too-small}
Let $\mT$ be any set of prime ideals generating the class group and let $\vec{x} \in \Div_{K,\mS}^0(\R)$ be fixed.
Then, for $t = 1 + \frac{1}{n} \cdot \sum_{\fp \in \mT} \log \norm(\fp)$,
\begin{align*}
&\Pr_{\vec y \from \unif(\Div_{K,\mS}^0/\logsunits)}\Big( \exists \vec z \in t \Ballinf \mbox{ so that } \vec y = \vec{x} + \vec z \bmod \logsunits \Big) \\ & \geq \frac{1}{(1000 \cdot n \cdot \log(n)^3)^n \cdot \classnumber \cdot \covol(\logunits)}. \label{eq:proba-not-too-small}
\end{align*}
\end{lemma}

\begin{remark}Note that in this lemma, the element $\vec y$ belongs to $\Div_{K,\mS}^0$, hence has its first coordinates in $\bigoplus_{\fp \in \mS}\Z$, but $\vec{x}$ and $\vec z$ are only required to live in the larger space $\Span_\R(\logsunits) = \Div_{K,\mS}^0(\R)$, which means that their coordinates $\vec{x}_\mp, \vec{z}_\mp$ might be real.
\end{remark}

\begin{proof}
We will start with writing the probability in terms of volumes in the quotient $X_\mS = \Div_{K,\mS}^0/\lsu$. For convenience, we will use the notation
\[\mathrm{Fold}_\mS : \Div_{K,\mS}^0 \longrightarrow X_\mS.\]
We have
\begin{align*}
 & \Pr_{\vec y \from \unif(X)}\Big( \exists \vec z \in t \Ballinf \mbox{ so that } \vec y = \vec{x} + \vec z \bmod \logsunits \Big)
 \\ 
 & = \frac{\vol(\mathrm{Fold}_\mS((t \Ballinf + x) \cap \Div_{K,\mS}^0))}{\vol(X_\mS)}
\end{align*}

To address the semi-continuous nature of $\Div_{K,\mS}^0$ (as opposed to $\Div_{K,\mS}^0(\R)$),
we define $\bar{x} \in \Div_{K,\mS}^0$ by the rule $\bar{x}_\mp := \lceil x_\mp \rfloor$ and $\bar{x}_\nu = x_\nu - \frac{\npl}{n} \sum_{\mp} (\bar{x}_\mp - x_\mp) \log \norm(\mp)$ (for all $\mp \in \mS$ and places $\nu$; here $\npl = 2$ if $\pl$ is complex and $1$ otherwise). Then $\bar{x} \in \Div_{K,\mS}^0$ and
\begin{equation} \| \bar{x} - x\|_\infty  \leq \frac{1}{2} + \frac{1}{n} \sum_{\mp \in \mS} \log \norm(\mp). \label{eq:diffxbarx} \end{equation}

In particular, $\tfrac{1}{2} \Ballinf + \bar{x} \subset t \Ballinf + x$.
Since volumes in $X_\mS$ are invariant by translation (by elements of $\Div_{K,\mS}^0$), we get
\begin{align*}
  \vol(\mathrm{Fold}_\mS((t \Ballinf + x) \cap \Div_{K,\mS}^0))
 &\geq \vol(\mathrm{Fold}_\mS((\tfrac{1}{2} \Ballinf + \bar{x}) \cap \Div_{K,\mS}^0))\\
 & = \vol(\mathrm{Fold}_\mS(\tfrac{1}{2} \Ballinf \cap \Div_{K,\mS}^0))
\end{align*}
We have \(\tfrac{1}{2} \Ballinf \cap \Div_{K,\mS}^0 = \{0\}^\mS \times (\tfrac{1}{2} \mathcal{B}_\infty^\emptyset),\)
where $\mathcal{B}_\infty^\emptyset$ is the unit $\infty$-ball in the $\R$-vector space $\Div_{K,\emptyset}^0(\R)$. We deduce
\begin{align*}
 \vol(\mathrm{Fold}_\mS(((\tfrac{1}{2} \Ballinf) \cap \Div_{K,\mS}^0))
  = \vol(\mathrm{Fold}_\emptyset(\tfrac{1}{2} \mathcal{B}_\infty^\emptyset))
  \geq \vol(\mathrm{Fold}_\emptyset(r \mathcal{B}_\infty^\emptyset))
\end{align*}
for any $0 < r \leq 1/2$. If furthermore $n^{1/2} r \leq \tfrac 1 2 \lambda_1(\logunits)$, the restriction of $\mathrm{Fold}_\emptyset$ to $r \mathcal{B}_\infty^\emptyset \subset n^{1/2} r \mathcal{B}_2^\emptyset$ is injective, and we get
\[\vol(\mathrm{Fold}_\emptyset(r \mathcal{B}_\infty^\emptyset)) = \vol(r \mathcal{B}_\infty^\emptyset) = (2r)^{\rem+\cem-1} \geq (2r)^n.\]
Composing all our volume (in)equalities, for $r \leq \tfrac 1 2 \min (1, \lambda_1(\logunits)/\sqrt{n})$ we have
\[\vol(\mathrm{Fold}_\mS((t \Ballinf + x) \cap \Div_{K,\mS}^0)) \geq (2r)^n.\]
From \Cref{le:lower-bound-first-minimum-log-unit}, we know that $\lambda_1(\logunits) \geq \kesslerformulainline$. Hence, taking $r = \tfrac 1 2 \cdot \tfrac{1}{1000n (\log n)^3}$ satisfies the above condition.
Together with the identity $\vol(X_\mS) = \classnumber \cdot \covol(\logunits)$ (see \Cref{lemma:det-log-S-unit}), we obtain the desired bound.
\end{proof}

From this, we can prove the following upper bound on the covering radius of $\logsunits$ when $|\mT|$ is sufficiently large.

\begin{lemma}[ERH]
\label{lemma:cov-small-set-primes}
There exist $B_1 = \poly(\log |\Delta_K|)$ such that if $\mT'$ is the set of all prime ideals of norm $\leq B_1$, then $\mT'$ generates the class group of $K$ and we have
\[\cov(\logsunits[\mT']) \leq \poly(\log |\Delta_K|).\]
\end{lemma}

\begin{proof}
Let us take $k = 1$, $\eps = \frac{1}{2\cdot (1000 \cdot n \cdot \log(n)^3)^n \cdot h_K \cdot \covol(\logunits)}$ and $\sd = 1/n$. Let $B$ and $N$ be as in \Cref{thm:rw}, for these choices of $k$, $\eps$ and $\sd$, with $N$ minimal.

We will prove that the lemma holds for $B_1 = B$. We have already seen that the $B$ from \Cref{thm:rw} is such that all prime ideals of norm $\leq B$ generates the class group, which proves the first part of the lemma.

Let us now prove the second part of the lemma. Let $\vec x \in \Span_\R(\logsunits[\mT']) = \Div_{K,\mS'}^0(\R)$. We would like to show that there exists a vector $\vec y$ in the support of the tail-cut random walk distribution $\WalkT(\mT',N,\sd,\eps)$ that is not too far away from $\vec x \bmod \logsunits[\mT']$.

From \Cref{lemma:proba-not-too-small} and by definition of $\eps$, we know that
\[\Pr_{\vec y \from \unif( \Div_{K,\mS'}^0/\logsunits[\mT'])}\Big(\exists z \in t\Ballinf \mbox{ such that }\vec y = \vec x + \vec z \bmod \logsunits[\mT'] \Big) \geq 2\eps,\]
where 
\(\Ballinf := \{ ((x_\mp)_{\mp},(x_\nu)_\nu) \in \Div_{K,\mS}^0(\R) ~|~ \| ((x_\mp)_{\mp},(x_\nu)_\nu)\|_\infty \leq t \}\)
and $t = 1 + \sum_{\fp \in \mT'} \log \norm(\fp)$.

Moreover, combining \Cref{lemma:truncated-small-stat-distance} and \Cref{thm:rw}, we get that the statistical distance between $\WalkT(\mT',N,\sd,\eps)$ and the uniform distribution over $\Div_{K,\mS}^0/\logsunits[\mT']$ is at most $\eps$. Hence, we obtain that
\[\Pr_{\vec y \leftarrow \WalkT(\mT',N,\sd,\eps)}\Big(\exists z \in t\Ballinf \mbox{ such that }\vec y = \vec x + \vec z \bmod \logsunits[\mT'] \Big) \geq 2\eps - \eps > 0.\]
This proves the existence of some $\vec y$ in the support of $\WalkT(\mT',N,\sd,\eps)$ and some $\vec z \in \Div_{K,\mS}^0(\R) = \Span_\R(\logsunits)$ with $\|\vec z \|_\infty \leq 1 + \sum_{\fp \in \mT'} \log \norm(\fp)$ such that $\vec x + \vec z - \vec y \in \logsunits[\mT']$. This means in particular that 
\[\min_{\vec v \in \logsunits[\mT']}\|\vec x- \vec v\| \leq \|\vec z\| + \|\vec y\|.\]

We know that the number of prime ideals of norm bounded by $B$ is at most $B \cdot n$ (since there are at most $n$ prime ideals above any prime integer in $\Z$). Hence, $\sum_{\fp \in \mT'} \log \norm(\fp) \leq B \cdot n\cdot \log B$ and we have $\|\vec z\| \leq \sqrt{\dimh+1+|\mT'|} \cdot (1+nB \log B) \leq \poly(B, \log|\Delta_K|)$.

Since $\vec y$ belongs to the support of $\WalkT(\mT',N,\sd,\eps)$, we know from \Cref{lemma:truncated-small-vector} that $\|\vec y \| \leq \sd \cdot \sqrt{2 \cdot n \cdot \log\big(\frac{4 \cdot n}{\eps}\big)} + N \cdot \big( 1 + \log B \big)$.

Let us upper bound the terms $\log(1/\eps)$, $B$ and $N$ appearing in the two upper bounds.
Recall from \Cref{lemma:det-log-S-unit} that $\log (h_K \cdot \covol(\logunits)) \leq \log |\Delta_K|$. From this, we see that $\log(1/\eps) = \poly(\log |\Delta_K|)$.
It was shown in~\cite{dBDPW} (in the proof of Corollary 3.4) that $\eta_1(\logunits^\vee)) \leq \poly(n)$. Hence, by choice of $\sd$, we have $1/s' = \poly(n)$. From this, \Cref{lemma:det-log-S-unit}, and the choice of $k = 1$, we obtain that $B = \poly(\log |\Delta_K|)$ and $N = \poly(\log |\Delta_K|)$.

We conclude that for any $\vec x \in \Span_\R(\logsunits[\mT'])$, we have
\[\min_{\vec v \in \logsunits[\mT']}\|\vec x- \vec v\| \leq \poly(\log |\Delta_K|),\]
hence \(\cov(\logsunits[\mT']) \leq \poly(\log |\Delta_K|).\)
\end{proof}

Only one last lemma remains to be proven before being able to prove \Cref{lemma:cov-radius}. This lemma relates the generating radii $\rr(\logsunits)$ and $\rr(\logsunits[\mT'])$ when $\mT'$ is a subset of~$\mT$.

\begin{lemma}
\label{lemma:bl-decreases}
Let $\mT$ and $\mT'$ be finite sets of prime ideals in $\OK$ satisfying $\mT' \subseteq \mT$. Let $\mT'$ generate the class group. Then we have 
\[ \rr(\logsunits) \leq (n+|\mT'|) \cdot \rr(\logsunits[\mT']) + \max_{\fp \in \mT} \log \norm(\fp) +1. \]
\end{lemma}

\begin{proof}
In order to show an upper bound on the generating radius $\rr(\logsunits)$, we will construct short vectors that generate the lattice $\logsunits$, \emph{using} the short vectors that generate the lattice $\logsunits[\mS']$. In this way, we relate the generating radii of $\logsunits$ and $\logsunits[\mT']$. We will first show that these vectors that we construct span a lattice of rank $\rem + \cem + |\mT| -1$ (the same as the rank of $\logsunits$) included in $\logsunits$. We then conclude that both lattices are equal by a volumetric argument.

Let us order the elements of $\mT = \{\fp_1, \ldots, \fp_{|\mT|}\}$ such that $\mT' = \{\fp_1, \ldots, \fp_{|\mT'|}\}$. By definition of $\logsunits$ and $\logsunits[\mT']$, we know that every vector $\vec v \in \logsunits[\mT']$ padded with $|\mT|-|\mT'|$ zeros (at the `prime places') belongs to $\logsunits$.

Let us fix some $\vec b_1, \cdots, \vec b_r$  (for some $r > 0$) generating $\logsunits[\mT']$ with $\|\vec b_i\| \leq \rr(\logsunits[\mT'])$ for all $i$'s, and let us consider the associated vectors $\vec{\bar{b}}_1, \cdots, \vec{\bar{b}}_r \in \logsunits$ obtained by padding them with zeros. Since padding with zeros does not change the euclidean norm, we still have $\|\vec{\bar{b}}_i\| \leq \rr(\logsunits[\mT'])$ for all $i$'s.
Moreover, the vectors $\vec{\bar{b}}_i$ span a lattice of rank $\rem + \cem + |\mT'|-1$, included in $\Div_{K,\mS}^0$.

Let us now construct $|\mT| - |\mT'|$ other vectors $\vec c_{|\mT'+1|},\cdots, \vec c_{|\mT|}$ in $\logsunits[\mT']$.
Let us fix some $i$ with $|\mT'| < i \leq |\mT|$ and consider the prime ideal $\fp_i \in \mT \setminus \mT'$. Since $\mT'$ generates the class group, we know that there exists $\alpha_i \in K$ and $\vec v^{(i)} \in \Z^{\mT'}$ such that $\alpha_i \cdot \OK = \fp_i \cdot \prod_{j \leq |\mT'|} \fp_j^{v^{(i)}_j}$. For $j > |\mT'|$, let us define $v^{(i)}_j = 1$ if $j = i$ and $0$ otherwise.
Then, the vector $\vec c_i := \big((v^{(i)}_j)_{j \leq |\mT|}, -\Log(\alpha_i)\big)$ belongs to $\logsunits$.

By definition, we know that we can decompose $\vec c_i = \vec d_i + \vec e_i$ where $\vec d_i$ belongs to $\Span_\R(\vec{\bar{b}}_1, \cdots, \vec{\bar{b}}_r)$ and $\vec e_i = ((0, \ldots, 1, \ldots, 0), (-\frac{\npl \log \norm(\fp_i)}{n})_\nu ) \in \Div_{K,\mS}$ (with the last coefficients (associated with the infinite places) equal to $\frac{\npl \log \norm(\fp_i)}{n}$ and the $1$ in position of the prime $\mp_i$). We can reduce the vector $\vec d_i$ by using the vectors $\vec{\bar{b}}_1, \cdots, \vec{\bar{b}}_r$ to ensure that $\|\vec d_i\| \leq (n+|\mT'|) \cdot \max_i \|\vec{\bar{b}}_i\|$ (for example, by taking $\rem+\cem+|\mT'|-1$ linearly independent vectors among the $\vec{\bar{b}}_i$'s and performing Babai's round-off algorithm).

Hence, without loss of generality, we can assume that 
\[\|\vec c_i\| \leq (n+|\mT'|) \cdot \rr(\logsunits[\mT']) + \max_{\fp \in \mT} \log \norm(\fp) +1.\]
So if we can show that the $\vec{\bar{b}}_i$'s and $\vec c_i$'s generate $\logsunits$, we have $\rr(\logsunits) \leq (n+|\mT'|) \cdot \rr(\logsunits[\mT']) + \max_{\fp \in \mT} \log \norm(\fp) +1$.

So, it remains to show that the vectors $\vec{\bar{b}}_i$'s and $\vec c_i$'s generate $\logsunits$.
Let $L$ be the lattice generated by the $\vec{\bar{b}}_i$'s and the $\vec c_i$'s. By construction, $L \subseteq \logsunits$. %
Moreover, from the structure of the $\vec c_i$'s, we know that $L$ has rank $\rem + \cem + |\mT'| - 1 + (|\mT| - |\mT'|)$, i.e., the same rank as $\logsunits$.
Finally, by choice of the $\vec c_i$'s, we know that $\covol(L) = \covol(\logsunits[\mT']) = \covol(\logsunits)$, where the last equality comes from \Cref{lemma:det-log-S-unit}. From this, we conclude that $L = \logsunits$, as desired.
\end{proof}

Combining everything, we can now prove \Cref{lemma:cov-radius}.

\begin{proposition}[ERH]
\label{lemma:cov-radius}
There exists some $B_1 = \poly(\log|\Delta_K|)$ such that for any set $\mT$ containing all prime ideals of $\OK$ of norm $\leq B_1$, it holds that 
\[ \rr(\logsunits) \leq \poly(\log |\Delta_K|,\allowbreak \max_{\fp \in \mT} \log \norm(\fp)). \]
Moreover, $B_1$ is such that the prime ideals of norm $\leq B_1$ generate the class group of $K$.
\end{proposition}

\begin{proof}[Proof of \Cref{lemma:cov-radius}]
We prove that the lemma holds for the same bound $B_1$ as in \Cref{lemma:cov-small-set-primes}. We already know from \Cref{lemma:cov-small-set-primes} that this bound $B_1$ is such that all prime ideals of norm $\leq B_1$ generates the class group.

Let us call $\mT'$ the set of prime ideals of norm $\leq B_1$. From \Cref{lemma:cov-small-set-primes} and \Cref{lemma:bl-leq-2cov}, we know that
$\rr(\logsunits[\mT']) \leq 2 \cdot \cov(\logsunits[\mT']) \leq \poly(\log |\Delta_K|)$.

From \Cref{lemma:bl-decreases}, we conclude that
\begin{align*}\rr(\logsunits) & \leq \sqrt{n+|\mT'|} \cdot \rr(\logsunits[\mT']) + \max_{\fp \in \mT} \log \norm(\fp) +1  \\
 &\leq  \poly\Big(\log|\Delta_K|,\max_{\fp \in \mT} \log \norm(\fp)\Big),\end{align*}
where we used the fact that $|\mT'| \leq n B_1 = \poly(B_1, \log |\Delta_K|)$ (see the proof of \Cref{lemma:cov-small-set-primes}).
\end{proof}

%% file: rigorous_cgc/B04-generating-a-single-s-unit.tex
\section{Generating a single \texorpdfstring{$\mS$}{\letterSunits}-unit}

\label{sec:BF-compute-1-rel}

\noindent
In this section, we describe an algorithm that computes in subexponential time a vector in the Log-$\mT$-unit lattice when the set $\mT$ is sufficiently large.
Our algorithm even does slightly more than that, namely given as input an integral ideal $\ma$, the algorithm computes a relation between the ideal $\ma$ and the ideals of $\mT$, that is, it outputs $\alpha \in K$ and $(v_\fp)_{\fp \in \mT}$ such that $\ma = \alpha \cdot \OK \cdot \prod_{\fp \in \mT} \fp^{-v_\fp}$.

Running the algorithm with $\ma = \OK$ provides a vector in the Log-$\mT$-unit lattice. However, in \Cref{sec:BF-compute-many-rel} we will use the fact that the algorithm can take as input any integral ideal $\ma$ and not just $\OK$. Indeed, in \Cref{sec:BF-compute-many-rel}, we want to generate many independent vectors in the Log-$\mT$-unit lattice (so that they generate the full lattice with good probability once we have enough of them). In order to ensure independence of the vectors, we will crucially rely on the fact that we can choose as input any ideal $\ma$.

Another place where we need to take as input an ideal $\ma$ is when we want to recover a generator of a principal ideal or compute a class group discrete logarithm (as explained in the introduction). In these cases, which correspond to the descent phase of the sieving algorithm, we need to find a relation between any input ideal $\ma$ and the prime ideals of $\mT$ for which we computed the Log-$\mT$-unit lattice.

Given as input an integral ideal $\ma$, \Cref{algo:compute-1-rel} uses \Cref{thm:sampling-simplified} to sample elements $\alpha \in \ma$ until the relative ideal $\alpha \cdot \ma^{-1}$ is $\mT$-smooth. Any such smooth relative ideal provides a relation between $\ma$ and the ideals of $\mT$. By \Cref{thm:sampling-simplified}, the success probability of this procedure is related to the local density of smooth ideals, which we computed in \Cref{sec:smooth}. As will be discussed later (see \Cref{subsec:complexity}) there are certain regimes, depending intricately on $\dedres, n$ and $\log |\dcrk|$ where the sampling algorithm's success probability increases when the first small primes are omitted from the computation, thus introducing a modulus $\moduz$ consisting of the product of small primes. This is the main reason why \Cref{algo:compute-1-rel} starts with computing an approximation of $\dedres$ and chooses different avenues depending on the value of $\dedres$, leading to either a modulus $\moduz$ consisting of small primes (if $\dedres$ is large) or an empty modulus $\moduz = \OK$ (if $\dedres$ is small).

\subsection{The algorithm that generates a single \texorpdfstring{$\mS$}{S}-unit relation}
The following definition of $\radpar$ will be used in \Cref{algo:compute-1-rel}.
\begin{definition}
	\label{def:radpar}
	For a given number field $K$ of degree $n$, a modulus $\moduz$, and real numbers $x,\blocksize \geq 1$, we define $\omega \in \Z_{>0}$ as the smallest positive integer such that $r = \radiusformula$ satisfies $r^n \geq e^n \max(\smoothdensityB, \rwB,10x^2)$, where $\smoothdensityB$ is defined in \Cref{lemma:proba-smooth} (with $\eps := 1$ loc.~cit.),
	and $\rwB$ is defined in \Cref{thm:sampling-for-single-relation}.
\end{definition}
Note that using the fact that $\smoothdensityB = \poly(\log |\dcrk|)$ and $\rwB := \poly(\log |\dcrk|,\allowbreak\log \norm(\moduz))$, we can see that $\radpar = O(1)$ is a constant (and so in particular $\size(\radpar) = O(1)$ since it is an integer).

We can now describe \Cref{algo:compute-1-rel}, which computes one relation.
\begin{algorithm}[ht]
    \caption{Computing one relation}
    \label{algo:compute-1-rel}
    \begin{algorithmic}[1]
    	\REQUIRE~\\
    	\begin{enumerate}[(i)]
          \item An LLL-reduced basis of $\OK$,
    	 \item an ideal $\fa \in \ideals$,
    	 \item a set of prime ideals $\mT$ of $K$,
    	 \item $y \in \nfrstar$.
    	\end{enumerate}
    	\ENSURE $\alpha \in \fa$ and $(v_\mathfrak{p})_{\mathfrak{p} \in \mT} \in \Z_{\geq 0}^{\mT}$ such that $\fa = \alpha \OK \cdot \prod_{\mathfrak{p} \in \mT} \mathfrak{p}^{-v_\mathfrak{p}}$
    		\STATE define $\blocksize = n^{\frac{2}{3}}$
    		\label{line:first}\label{line:defblock}
			\STATE compute an approximation $\tdedres$ of $\dedres$ such that $\tdedres \leq \rho_K < 2\tdedres$ (see \Cref{prop:approx-rho}) \label{line:comprho}
			\STATE define $x =\max\left(\frac{\log^{2/3}|\dcrk|}{\log^{4/3} \log |\dcrk|},\frac{n^{2/3}}{\log^{2/3}(n)} \right)$ \label{line:defx}
			\IF{$\tdedres \leq \exp(x \log^2(x))$} \label{line:case}
				\STATE define $\moduz = (1)$\label{line:Adefin2}
			\ELSE
				\STATE define $\moduz = \prod_{\norm(\mp) < x} \mp$\label{line:dedrescut2}\label{line:Adefin1} %
			\ENDIF \label{line:endifmzero}
			\STATE Define $\radpar$ as in \Cref{def:radpar}
    		\REPEAT \label{line:repeat}
			\STATE $\alpha \leftarrow \Sample(\fa, \ky , \moduz, \blocksize, \radpar)$
			(see  \Cref{thm:sampling-simplified-2})
			\label{line:sampleideal}
    		\UNTIL {$\alpha \OK \cdot \fa^{-1}$ is $\mT$-smooth} \label{line:until}
    		\STATE compute $(v_\mathfrak{p})_{\mathfrak{p} \in \mT} \in \Z_{\geq 0}^{\mT}$ such that $\alpha \OK \cdot \fa^{-1} = \prod_{\mathfrak{p} \in \mT} \mathfrak{p}^{v_\mathfrak{p}}$
    		\label{line:decomposition}
        \RETURN $(\alpha, (v_\mathfrak{p})_{\mathfrak{p} \in \mT})$. \label{line:return}
    \end{algorithmic}
\end{algorithm}

\begin{notation} \label{notation:variablesalg3} Throughout the following lemmas, we will use the following notation:
\begin{itemize}
\item We put $\onerelationB =
     		\max( \exp(\sqrt{\log(r^n) \log \log(r^n)}),
     		\smoothdensityB, \rwB, 10 x^2 )$. Note that this instantiation implies $x < \frac{10x^2}{4\log(10x^2)} < \onerelationB/(4 \log \onerelationB)$ (in the first inequality we used $x \geq 1$ and for the second one we used the fact that $y \mapsto y/(4\log(y))$ is increasing for $y > 3$).
 \item $u = \frac{\log(r^n)}{\log(\onerelationB)}$;
\end{itemize}
Note that, since $r = \radiusformula \geq 48$, we certainly have $r^n \geq e^n \cdot \exp(\sqrt{\log(r^n) \log \log(r^n)})$ and hence $r^n \geq e^n \onerelationB$. 
\end{notation}

In order to analyze the running time of \Cref{algo:compute-1-rel}, we need to get a lower bound on the success probability $\psucc$ that the \texttt{repeat...until} loop from lines~\lineref{line:repeat}-\lineref{line:until} terminate. This is what we compute in the next four lemmas. In order to improve readability, we decomposed the computation into small lemmas, the most interesting one being the last one (\Cref{lemma:lower-bound-psucc}).
\begin{lemma}[ERH] \label{lemma:computelowerboundprob} Put $r,\onerelationB,u \in \R_{>0}$ as in \Cref{notation:variablesalg3}. Assume that $\mS$ contains all prime ideals not dividing $\moduz$ of norm $\leq\onerelationB$, and that $\fa$ is coprime with $\moduz$. Then the success probability $\psucc$ of $\alpha \OK \fa^{-1}$ being $\mS$-smooth in the repeat-loop (lines \lineref{line:repeat}-\lineref{line:until}) of \Cref{algo:compute-1-rel} satisfies
\begin{equation} \psucc \geq \frac{1}{6} \frac{\norm(\moduz)}{\phi(\moduz) \cdot \dedres}\frac{(4\log(\onerelationB))^{1-u} \cdot  u^{-u}}{   \onerelationB}. \label{eq:lemmalowerboundpsucc} \end{equation}
\end{lemma}

\begin{proof} We denote $\idset_{\mS}$ for the $\mS$-smooth integral ideals. Since $\mS$ contains all prime ideals (coprime to $\moduz$) of norm $\leq \onerelationB$ per assumption, and $\onerelationB \geq \rwB$ by definition, we have $\idset_{\mS} \idset_{\rwB} = \idset_{\mS}$ (where we considered only prime ideals coprime to $\moduz$ in $\idset_{\rwB}$).

In line \lineref{line:sampleideal} of \Cref{algo:compute-1-rel}, we
use $\Sample$ from \Cref{thm:sampling-simplified-2}. According to that theorem,\footnote{We can apply the theorem because $\fa$ is assumed to be coprime with $\moduz$.}
the probability $\psucc$ of $\alpha  \OK \cdot \ma^{-1}$ lying in $\idset_{\mS} \cdot \idset_{\rwB} = \idset_{\mS}$ is at least
\begin{equation*} p_{\mathrm{success}} \geq \frac{\norm(\moduz)}{\phi(\moduz)}\frac{\delta_{\idset_\mS}[r^n]}{6} \geq \frac{1}{6} \frac{\norm(\moduz)}{\phi(\moduz) \cdot \dedres}\frac{(4\log(\onerelationB))^{1-u} \cdot  u^{-u}}{   \onerelationB}.
\end{equation*}
Here, the last inequality comes from computing the density of smooth numbers as in \Cref{lemma:proba-smooth}. We instantiate that lemma with $A := x$ if $\moduz \neq \OK$ and $A := 1$ otherwise (this is to take into account that the prime dividing $\moduz$ are omitted, see lines \lineref{line:Adefin1} and \lineref{line:Adefin2}), $\eps := 1$, and $B := \onerelationB$. These parameters indeed satisfy the constraints of \Cref{lemma:proba-smooth} since $\onerelationB \geq \smoothdensityB$, $r^n \geq e^n \onerelationB$, and $A \leq \onerelationB/(4 \log \onerelationB)$ (see \Cref{notation:variablesalg3}).
\end{proof}

\begin{lemma} \label{lemma:optimize1} Keeping the notations from \Cref{lemma:computelowerboundprob}, it holds that
	\begin{equation}  \frac{(4\log(\onerelationB))^{1-u} \cdot  u^{-u}}{   \onerelationB} \geq \frac{1}{(\onerelationB)^3}. \end{equation}
\end{lemma}

\begin{proof}
Taking negative logarithms, and writing $R = \log(r^n)$, $\beta = \log(\onerelationB)$ and $u = R/\beta$, we obtain
\begin{equation} -\log\left( \frac{(4\log(\onerelationB))^{1-u} \cdot  u^{-u}}{   \onerelationB} \right) = \beta + \left(\frac{R}{\beta}-1\right) \log( 4 \beta ) + \frac{R}{\beta} \log(R/\beta ).
 \label{eq:tooptimizething2} \end{equation}

It suffices to show that the two right-most summands of \Cref{eq:tooptimizething2} are upper bounded by $\beta$. Note that, by \Cref{notation:variablesalg3}, 
we have that $\beta \geq \sqrt{R \log R}$. Hence, we immediately see $R/\beta \cdot \log(R/\beta) \leq \sqrt{R \log R} \leq \beta$. 
For the middle summand, note that $(\tfrac{R}{\beta} - 1)\log(4\beta) \leq \beta$ is equivalent to $R \leq \beta^2/\log(4\beta) + \beta$. The 
latter inequality is true because $\beta \mapsto \beta^2/\log(4\beta) + \beta$ is an increasing function for $\beta \geq 1/2$ and, for $R \geq 3$ (which is clearly satisfied since $r \geq 16$),
\[ R \leq  \frac{(\sqrt{R \log R})^2}{\log(4 \sqrt{R\log R})} + \sqrt{R \log R}  \leq  \beta^2/\log(4\beta) + \beta, \]
where the first inequality follows by graphical inspection and the second by the monotonicity of the function $\beta \mapsto \beta^2/\log(4\beta) + \beta$.
\end{proof}

\begin{lemma} \label{lemma:lower-bound-Bmax} Keeping the notations from \Cref{lemma:computelowerboundprob}, it holds that
	\[ \onerelationB \in \poly(L_{|\Delta_K|}(\tfrac{1}{2}),L_{\norm(\moduz)}(\tfrac{1}{2}), L_{n^n}(\tfrac{2}{3})).\]
\end{lemma}

\begin{proof}
From the definition of $\onerelationB$ in \Cref{notation:variablesalg3}, we have that
\begin{align} \onerelationB &\in \poly(\exp(\sqrt{\log(r^n)\log \log (r^n)}), \log |\dcrk|, \log \norm(\moduz))  \\
	& = \poly(L_{|\Delta_K|}(\tfrac{1}{2}),L_{\norm(\moduz)}(\tfrac{1}{2}), L_{n^n}(\tfrac{2}{3})).
\end{align}
Indeed, by the definition of $r$ 
and the definition of $\blocksize$, and using that $\radpar = O(1)$
\[ \log(r^n) \leq O \Big( \frac{n^2 \log \blocksize}{\blocksize} + \log |\dcrk| + \log \norm(\moduz) \Big) = O( n^{4/3} \log n ,\log |\dcrk|,\log \norm(\moduz)). \]
And hence
\begin{equation} \exp(\sqrt{\log(r^n) \log \log(r^n)}) \leq \poly(L_{|\Delta_K|}(\tfrac{1}{2}),L_{\norm(\moduz)}(\tfrac{1}{2}), L_{n^n}(\tfrac{2}{3})). \label{eq:boundonexprn} \end{equation}
\end{proof}

\begin{lemma}
	\label{lemma:lower-bound-psucc} 
	Keep the notations from \Cref{lemma:computelowerboundprob} and define
	\[ \dedrescut =
	\min\Big(\dedres, \max(e^{\log^{\frac{2}{3}}|\dcrk| \cdot \log^{\frac{2}{3}}(\log |\dcrk|)},e^{n^{\frac{2}{3}} \log^{\frac{4}{3}}(n)} ) \Big). \]
	Then it holds that $\onerelationB$ and $\psucc^{-1}$ are both in
	$\poly(L_{|\Delta_K|}(\tfrac{1}{2}), L_{n^n}(\tfrac{2}{3}),\dedrescut)$.
\end{lemma}

\begin{proof}
Combining \Cref{lemma:computelowerboundprob}, \Cref{lemma:optimize1} we obtain that \[\psucc^{-1} \in \poly(\tfrac{\phi(\moduz) \cdot \dedres}{\norm(\moduz)}, \onerelationB).\] Moreover, \Cref{lemma:lower-bound-Bmax} gives us $\onerelationB \in \poly(L_{|\Delta_K|}(\tfrac{1}{2}), L_{n^n}(\tfrac{2}{3}),L_{\norm(\moduz)}(\tfrac{1}{2}))$.
To prove the result, it then suffices to show that
\[ \poly(L_{\norm(\moduz)}(\tfrac{1}{2}),  \tfrac{\phi(\moduz) \cdot \dedres}{\norm(\moduz)})  = \poly(\dedrescut). \]
We write $x =\max\left(\frac{\log^{2/3}|\dcrk|}{\log^{4/3} \log |\dcrk|},\frac{n^{2/3}}{\log^{2/3}(n)} \right)$ as in line \lineref{line:defx} of \Cref{algo:compute-1-rel}. 
Note that $\dedrescut =
\min\left(\dedres, \exp(\Theta(x \log^2(x))) \right).$
We distinguish two cases. ~\\
\textbf{Case 1: $\tdedres \leq \exp(x \log^2(x))$.} \\
By line \lineref{line:case} and \lineref{line:Adefin2}, we then have $\dedres < 2 \tdedres \leq 2 \cdot \exp(x \log^2(x))$ (in particular, $\dedres  = \poly(\dedrescut)$) and $\moduz = (1)$. Hence $\norm(\moduz) = \phi(\moduz) = 1$ and thus
\[ \poly(L_{\norm(\moduz)}(\tfrac{1}{2}),  \tfrac{\phi(\moduz) \cdot \dedres}{\norm(\moduz)}) = \poly(\dedres)  = \poly(\dedrescut). \]
\textbf{Case 2: $\tdedres > \exp(x \log^2(x))$.} \\
In this case, it holds that $x \log(x)^2 = O(\log \dedres)$ and have $\moduz = \prod_{\norm(\mp) < x} \mp$ (by lines \lineref{line:case} and \lineref{line:dedrescut2}).
On one hand, by \Cref{prop:boundrhonormphi} we have
\begin{align}
	-\log\left( \frac{\norm(\moduz)}{\phi(\moduz) \cdot \dedres} \right) & \leq
	\frac{8 \log |\dcrk|}{\sqrt{x}} + \frac{8 n \log x}{\sqrt{x}} \nonumber \\
	& \leq O( x \log^2(x)) = O(\log \dedrescut). \label{eq:boundfractionmoduzdedres}
\end{align}
This follows from the fact that $\log |\dcrk| \leq O(x^{\frac{3}{2}} \cdot \log^2(x))$ and $n \leq O(x^{\frac{3}{2}} \log x)$, since $x =\max\left( \left( \frac{\log|\dcrk|}{\log^2 \log |\dcrk|} \right)^{\frac{2}{3}}, \left(\frac{n}{\log n}\right)^{\frac{2}{3}} \right)$.
On the other hand, using the same bounds on $\log |\dcrk|$ and $n$, we have,  by \Cref{lemma:boundnormmoduz}
\begin{align}
	\log \norm(\moduz) & \leq O( x + \sqrt{x} \log(x) \log |\dcrk| + \sqrt{x} \log^2(x) \cdot n) \nonumber  \\
	& \leq O\left( x^2 \log^3(x) \right). \label{eq:boundlognormmoduz}
\end{align}
We deduce that
\begin{equation} \label{eq:boundnormmoduzdedres} \log L_{\norm(\moduz)}(\tfrac{1}{2})  \leq O(\sqrt{\log \norm(\moduz) \log \log  \norm(\moduz)}) \leq O( x \log^{2}(x)) = O(\log \dedrescut).\end{equation}
Combining \Cref{eq:boundfractionmoduzdedres,eq:boundnormmoduzdedres}, we thus obtain
\[ \poly(L_{\norm(\moduz)}(\tfrac{1}{2}),  \tfrac{\phi(\moduz) \cdot \dedres}{\norm(\moduz)}) = \poly(\dedrescut). \]
\end{proof}

\begin{theorem}[ERH]
\label{thm:compute-1-rel}
There exists some $\onerelationB = \poly(L_{|\dcrk|}(1/2), L_{n^n}(2/3),\dedrescut)$,
where \[ \dedrescut =
\min\Big(\dedres, \max(e^{\log^{\frac{2}{3}}|\dcrk| \cdot \log^{\frac{2}{3}}(\log |\dcrk|)},e^{n^{\frac{2}{3}} \log^{\frac{4}{3}}(n)} ) \Big), \] such that the following holds.
Assume that $\mS$ contains all prime ideals coprime to $\moduz$ of norm $\leq \onerelationB$.
Then, on input  an integral ideal~$\fa$ coprime with $\moduz$, the set $\mT$ and $y \in H = \Log(\nfr^0)$,
\Cref{algo:compute-1-rel} outputs $(\alpha, (v_\mathfrak{p})_{\mathfrak{p} \in \mT}) \in \fa \times \Z_{\geq 0}^{|\mT|}$ such that
\[ \fa = \alpha \OK \cdot \prod_{\mathfrak{p} \in \mT} \mathfrak{p}^{-v_\mathfrak{p}}.\]
Furthermore, \Cref{algo:compute-1-rel} runs in expected time
\[ \poly(L_{|\Delta_K|}(\tfrac{1}{2}), L_{n^n}(\tfrac{2}{3}), \allowbreak \size(\mT),
\allowbreak \dedrescut,
\allowbreak \size(\fa),\allowbreak \size(y), \log |\Delta_K|).\]
\end{theorem}
\begin{proof} 
Thanks to the \texttt{repeat...until} loop from lines \lineref{line:repeat} to \lineref{line:until}, the output of \Cref{algo:compute-1-rel} is correct whenever it terminates. 
The fact that $\onerelationB = \poly(L_{|\dcrk|}(1/2),\allowbreak L_{n^n}(2/3),\dedrescut)$ follows from \Cref{lemma:lower-bound-psucc}.
Hence, the main focus of this proof is the expected running time of the algorithm.

Since lines \lineref{line:first} up to \lineref{line:repeat} do not significantly contribute to the running time, it suffices to concentrate on the expected running time of lines \lineref{line:sampleideal} up to \lineref{line:return}.

For this time analysis we will use that  $\log \norm(\moduz) \leq \poly(\log |\dcrk|)$, by \Cref{lemma:boundnormmoduz} and the definition of $x$ on line \lineref{line:defx}. Also, we will use that $\hkztime \leq L_{n^n}(\tfrac{2}{3})$, from the definition of $\blocksize = n^{2/3}$.

\begin{itemize}

 \item (Line \lineref{line:sampleideal}) The sampling algorithm $\Sample$ takes (by \Cref{thm:sampling-simplified-2}) time $\poly(\log|\dcrk|, \allowbreak \size(\fa),\allowbreak \log(\norm(\moduz)),\allowbreak\size(y),\allowbreak\hkztime) = \poly(\log|\dcrk|, \allowbreak \size(\fa),\allowbreak\size(y),\allowbreak L_{n^n}(\tfrac{2}{3}))$. Here we use that $\size(\radpar) = O(1)$ (see \Cref{def:radpar}).
 \item (Line \lineref{line:until} and \lineref{line:decomposition})  Checking for $\alpha \OK \cdot \fa^{-1}$ being $\mS$-smooth (line \lineref{line:until}), as well as writing down the decomposition (line \lineref{line:decomposition}) if the ideal indeed is smooth can be done with $\poly(\size(\mS))$ division of ideals.
 \item (Line \lineref{line:return}) This is the return statement.
\end{itemize}
Hence, denoting $\psucc$ for the probability of $\alpha \OK \cdot \fa^{-1}$ being $\mS$-smooth in line \lineref{line:until}, we obtain an expected running time of
\begin{equation} \label{eq:intermediatecomp}   \psucc^{-1} \cdot \poly(L_{n^n}(\tfrac{2}{3}), \size(\mS), \size(\fa), \size(y),\log|\dcrk| ). \end{equation}
Using \Cref{lemma:lower-bound-psucc} provides an upper bound on $\psucc^{-1}$ and concludes the proof.
\end{proof}

\subsection{Properties of the output distribution of \texorpdfstring{\Cref{algo:compute-1-rel}}{the smooth sampling algorithm}.}\label{subset Properties of the output distribution}

From now on, we assume that the input ideal $\fa$ is $\mS$-smooth, since this will be the case when we will call \Cref{algo:compute-1-rel} in the next section (this ensures in particular that $\fa$ is coprime with $\moduz$, which is needed to apply \Cref{thm:compute-1-rel}).
We describe the input of \Cref{algo:compute-1-rel} as a divisor $\ba =  d(\ma) + \Log(y) \in \Div_{K,\mS}$ (we recover $\ma = \finpart{\Exp}(\finpart{\ba})$ and $y = \infinitepart{\Exp}(\infinitepart{\ba}) \in \nfrstar$). 
The following properties of \Cref{algo:compute-1-rel} 
are required in the later proof
that \Cref{alg:randomrelation} (which is essentially \Cref{algo:compute-1-rel} with a \emph{Gaussian distributed input}) has an evenly distributed and concentrated output distribution (see \Cref{lemma:evenlydistributed,lemma:concentrated}). These two properties of this last output distribution are indispensable to show that sufficiently many samples from this algorithm eventually generate the log-$\mS$-unit lattice $\lsu$.

We denote the output distribution of $\alpha \in K$ in \Cref{algo:compute-1-rel} with input
$\ba \in \Div_{K,\mS}$
by\footnote{Note the difference in notation compared to $\distr$ in \Cref{section:propertiessampling}; the bar signifies that the distribution is from \Cref{algo:compute-1-rel}.} $\bdistr_\ba$.
By abuse of notation, we denote by the same symbol $\bdistr_\ba$ the distribution over
$\lsu$
defined by sampling $\alpha \from \bdistr_\ba$ and taking
the log-$\mS$-unit map%
\footnote{This
sends the roots of unity to the zero-vector, but that is not troublesome, as elements $\alpha, \zeta \alpha \in K^*$
for a root of unity $\zeta$ have equal probability to be sampled.}
$\logs(\alpha)$. 
Note that this map is well-defined, as \Cref{algo:compute-1-rel} only outputs $\alpha \in K^*$ whose prime divisors lie in $\mS$, i.e., $\mS$-units (by line \lineref{line:until} and because $\fa$ is $\mS$-smooth).
This yields a distribution $\bdistr_\ba \in L^1(\lsu)$ for each $\ba \in \DivKS$.
Note that \Cref{algo:compute-1-rel} and hence the distribution $\bdistr_\ba$ is independent on the degree of
$\ba$ (as in \Cref{eq:defdeg}), as $\Sample$ from \Cref{thm:sampling-simplified-2} is independent on the norm.

\begin{lemma} \label{lemma:outputdistrprop} For all $\ba \in \DivKS$,  the output distribution $\bdistr_\ba$ of \Cref{algo:compute-1-rel} (where we consider the divisor $\ba = d(\ma) + \Log(\ky)$ as input, which is equivalent to the input $(\ma,\ky)$)
satisfies the
following three properties.
\begin{enumerate}[(i)]
 \item \label{lemma:outputdistrpropi} We have
 $\bdistr_\ba[ \lsu ] = 1$.
 \item  \label{lemma:outputdistrpropii}For all $\alpha \in K^{\moduz,1}$ and all $x \in \lsu$
 \[ \bdistr_{\ba}[ x ] = \bdistr_{\ba - \logs( \alpha )}[  x  + \logs( \alpha) ]. \]
 \item \label{lemma:outputdistrpropiii} There exists some $R = O(\log^2 |\dcrk|)  + \|\ba\|$ such that
 \[  \bdistr_{\ba} \left[ \lsu \backslash (R \cdot \mathcal{B}_2) \right] = 0 \]
where
 $R \cdot \mathcal{B}_2 = \{ \ba \in \DivKS ~|~ \| \ba \| \leq R \}$.
\item \label{lemma:outputdistrpropiv} For $\ba,\ba' \in \DivKS$ with $\ba' = \ba + \sum_{\nu} \hb_\nu \ldb \nu \rdb$, we have 
\[ \|\bdistr_{\ba} - \bdistr_{\ba'}\|_1 \leq 3 \cdot |\dcrk| \cdot r^n \cdot n^2 \cdot \| \hb \| + 1/200. \]
with $r = \radiusformula$.
\end{enumerate}

\end{lemma}
\begin{proof} 

Statement (\itemref{lemma:outputdistrpropi}) follows from the fact that  \Cref{algo:compute-1-rel} outputs $\alpha$ only if $\alpha \OK \cdot \fa^{-1}$ is $\mS$-smooth. Together with the fact that $\fa$ itself is assumed to be $\mS$-smooth (it is made out of $\ba \in \DivKS$), we deduce that $\alpha$ is $\mS$-smooth, thus an $\mS$-unit. Therefore, $\bdistr_\ba$ only has support on $\lsu \subseteq \Div_{K,\mS}^0$.

Statement (\itemref{lemma:outputdistrpropii}) follows from property (\itemref{item:simpschifting}) of \Cref{thm:sampling-simplified-2}. Let $\ba = d(\ma)+\Log(\ky)$ for $\ma = \finitepart\Exp(\finitepart\ba)$ and $y = \infpart\Exp(\infpart\ba)$. By definition of $\logs(\alpha)$, we have that $\ba - \logs(\alpha) = d(\ma \cdot (\alpha)) + \Log(\ky \cdot \alpha^{-1})$. Property (\itemref{item:simpschifting}) of \Cref{thm:sampling-simplified-2} tells us that $\distr_{\ky,\ma}[\placeholder] = \distr_{\ky \cdot \alpha^{-1}, \ma \cdot (\alpha)}[\alpha \cdot \placeholder] $ (note the different distribution, this is the distribution from~\Cref{thm:sampling-simplified-2}). As the exit criterion of the repeat loop (line~\lineref{line:until} of~\Cref{algo:compute-1-rel}) happens at the same occurrences, the output distribution function $\bdistr_{\ba}$ satisfies $\bdistr_{\ba}[ x ] = \bdistr_{\ba - \logs(\alpha)}[  x + \logs( \alpha) ]$ for all $x \in \lsu$.

The statement (\itemref{lemma:outputdistrpropiii}) follows from property (\itemref{item:simpbounded}) of \Cref{thm:sampling-simplified-2}, $\radpar = O(1)$ (see \Cref{def:radpar}) and the fact that $\log \norm(\moduz) = O( x \cdot \log^2(x) \cdot \log |\dcrk|) = O(\log^2 |\dcrk|)$ (see \Cref{eq:boundlognormmoduz} in the proof of \Cref{thm:compute-1-rel}, and see the definition of $x$ in line \lineref{line:defx} of \Cref{algo:compute-1-rel}). Note that here the correspondence
between $(y,\ma) \in \nfrstar \times \ideals$ and $\ba = d(\ma) + \Log(y)$ is used.

Statement (\itemref{lemma:outputdistrpropiv}) follows from
\Cref{lemma:conditionalvariation2} and \Cref{thm:sampling-simplified-2} (property (1), Lipschitz-property), with the fact that the success probability of a single loop of \Cref{thm:compute-1-rel} is lower bounded by $\delta_{\idset}[r^n]/6 \geq  \frac{1}{6 \cdot |\dcrk| \cdot r^n}$ 
(see \Cref{eq:smalleps} in the proof of \Cref{thm:sampling-simplified-2}). Hence we have (applying \Cref{lemma:conditionalvariation2} with $p^{-1} \leq 6 \cdot |\dcrk| \cdot r^n$) 
\[ \|\bdistr_{\ba} - \bdistr_{\ba'}\|_1 \leq 3 \cdot |\dcrk| \cdot r^n \cdot n^2 \cdot \| \Log(y) \| + 1/200.    \]
\end{proof}

%% file: rigorous_cgc/B05-generating-many-independent-s-units-A.tex
\section{Obtaining a generating set of \texorpdfstring{$\mS$}{\letterSunits}-units} \label{sec:BF-compute-many-rel} %
\subsection{Introduction}
In \Cref{sec:BF-compute-1-rel} a probabilistic algorithm is described that computes a
\emph{single} $\mS$-unit (or a logarithmic $\mS$-unit lattice point).
But to actually obtain the
entire $\mS$-unit group $\sunits$, one needs a \emph{generating set}
of $\mS$-units. Or, equivalently, a $\Z$-generating set of
lattice points for the Log-$\mS$-unit lattice $\logsunits$.
\subsubsection*{A `good' output distribution}
One could hope that such a generating set will eventually
be formed by \emph{repeating} the probabilistic algorithm (\Cref{algo:compute-1-rel}) many
times and joining the outputs. The chance for this approach to succeed strongly
depends on the \emph{output distribution} of the mentioned
probabilistic algorithm. Indeed, one could imagine cases
where the output distribution is, for example,
only being supported on a strict \emph{sublattice} of
the logarithmic $\mS$-unit lattice. In such instances
one never obtains a generating set of the full logarithmic
$\mS$-unit lattice.

It turns out that, for the repeated sampling approach to succeed, it is sufficient
for the output distribution to have the following two properties:
`evenly distributed', meaning not having too much
weight on a strict sublattice, and `concentrated', meaning that
most of the weight is reasonably centered around the origin. These
two properties are made more precise in \Cref{subsec:generatinglattice},
as well as the consequences of these properties to the number of samples
needed to generate the whole lattice.

\subsubsection*{The output distribution of the sampling algorithm}
It remains to be shown that the output distribution of the
sampling algorithm (\Cref{algo:compute-1-rel}) is indeed actually `good', i.e., evenly distributed and concentrated.

It turns out that the output distribution of \Cref{algo:compute-1-rel} for a \emph{fixed} input $(\ma,y) \in \ideals \times \nfrstar$
is hard to analyze; and it seems very difficult to show that it satisfies these properties.
Fortunately, one can feed the algorithm
different inputs, causing to gain more
control of the output distribution.

More specifically, in our approach we randomly sample a
Gaussian distributed point from the space $\Div_{K,\mS}$, %
which
can be considered as an `ambient space' of the logarithmic $\mS$-units.
This randomly sampled point is then transformed into an element $(\ma,y) \in \ideals \times \nfrstar$ ($\ma$ from the finite places and $y$ from the infinite places), which is then given as an input to the $\mS$-unit
sampling algorithm (\Cref{algo:compute-1-rel})%
.

We show that this `compound' distribution, taking into account
both the randomness of the Gaussian
over $\Div_{K,\mS}$ %
and the randomness of the $\mS$-unit
sampling algorithm, indeed satisfies the `evenly distributed' property (see \Cref{lemma:evenlydistributed}) and `concentrated' property (see \Cref{lemma:concentrated}). %
Thus, repeating the $\mS$-unit sampling algorithm with Gaussian
distributed inputs lets us indeed obtain a generating set of the logarithmic $\mS$-unit lattice with high probability; this
final statement is the object of \Cref{theorem:numberofsamples}.

\subsubsection*{Intuition}
The intuition why this approach (of letting the input
be Gaussian distributed) works is already sketched in \Cref{subsec:intuition}.
Another way of looking at it is by seeing \Cref{algo:compute-1-rel}
as a specific CVP-algorithm for the logarithmic $\mS$-unit lattice $\logsunits$, given
an input vector $\ba$ in $\Div_{K,\mS}$.

As a CVP-algorithm, \Cref{algo:compute-1-rel} outputs
a $\logs(\alpha)$ `close' to $\ba \in \Div_{K,\mS}$, so one can imagine
that the output only differs slightly from the input.
If then the input element $\ba \in \Div_{K,\mS}$ is distributed according to a sufficiently wide Gaussian, 
and this specific
CVP output of \Cref{algo:compute-1-rel} is sufficiently close to $\ba$,
this CVP output distribution must be close to
a Gaussian as well, but instead its support is on $\lsu$. For such Gaussian distributions over lattices
one can show that sufficiently many samples thereof generate the entire lattice
(see \Cref{prop:Koen-lemma}).

\subsection{Obtaining a generating set of a lattice} \label{subsec:generatinglattice}

Given a certain distribution $\psi$ on a lattice $\Lambda$, this section is about
how many samples one has to draw from $\psi$ to obtain a generating set of the
lattice $\Lambda$. Of course, this depends on the distribution at hand: for example, the
distribution that always outputs $0 \in \Lambda$ would never yield a generating set
of the lattice $\Lambda$, no matter how many samples one takes.

For a distribution on a lattice $\Lambda$ it it is sufficient to consider two particular properties
to analyze the number of samples required to draw in order to obtain a generating set of $\Lambda$. Those two
properties are \emph{evenly distributedness}, which measures the maximum weight of the distribution $\psi$
on strict sublattices of $\Lambda$, and \emph{concentratedness}, which measures how much weight
of the distribution $\psi$ is (relatively) close to the origin. This is formalized in the following definition and proposition.

\begin{definition}
Let $\Lambda \subset \R^n$ be a full-rank lattice, and $\psi$ be a distribution on $\Lambda$. The distribution $\psi$ is called \emph{$p$-evenly distributed} if $\Pr_{x \gets \psi}[x \in \Lambda'] \leq p$ for any proper sublattice $\Lambda' \subsetneq \Lambda$. It is called \emph{$(R,q)$-concentrated} if $\Pr_{x \gets \psi}[\|v\|\geq R]\leq q$.
\end{definition}

\begin{proposition}[{\cite[Lemma 5]{BDF20}}]
\label{prop:Koen-lemma}
Let $\Lambda \subset \R^n$ be a full-rank lattice, and $\psi$ be a distribution on $\Lambda$. Suppose that $\psi$ is $p$-evenly distributed and $(R,q)$-concentrated for $R \geq \covol(\Lambda)^{1/n}$. Denote by $S$ the random variable defined as the number of samples drawn from $\psi$ until these samples generate $\Lambda$. Then, for any $\alpha>0$,
$$\Pr\left[S > (2+\alpha)\frac{t+n}{1-p-q}\right] \leq e^{-\alpha(t+n)/2},$$
where $t = n\log_2(R) - \log_2(\covol \Lambda) \geq 0$.
\end{proposition}

\subsection{The discrete Gaussian over a lattice is evenly distributed} \label{subsec:gaussianevenly}
In the following two lemmas, we show that a Gaussian distribution over a lattice with an arbitrary center point is $p$-evenly distributed. This fact is used to show that the output of the sampling algorithm with a (continuous) Gaussian input is $p$-evenly distributed as well (see \Cref{lemma:evenlydistributed}).

Recall that $\gaussian_\latsd(x) := e^{-\pi \normx{x}^2/\latsd^2}$ is the Gaussian function, whereas $\Gaussian_{X,\latsd}$ is the Gaussian distribution over $X$ (where the definition depends on whether $X$ is discrete or continuous, see \Cref{subsec:gaussiandistr}).
\begin{lemma} \label{lemma:regevtechnique}
For any lattice $\Lambda \subseteq V$ (where $V$ is a Euclidean space) any $\latsd>0$ and any $t,w \in V$,
we have
$$\gaussian_\latsd(\Lambda + t + w) + \gaussian_\latsd(\Lambda + t - w) \geq 2\gaussian_\latsd(w)\gaussian_\latsd(\Lambda + t).$$
\end{lemma}
\begin{proof} This lemma is a simple generalization of \cite[Claim 2.10]{havivregev2014}, and we follow the same strategy:
\begin{align*}
\gaussian_\latsd(\Lambda+t + w) + \gaussian_\latsd(\Lambda+t - w) 
&=  \sum_{x \in \Lambda+t}\left(e^{-\pi\|x+w\|^2/\latsd^2} + e^{-\pi\|x-w\|^2/\latsd^2}\right)\\
&= 2e^{-\pi\|w\|^2/\latsd^2}\sum_{x \in \Lambda+t}\left(e^{-\pi\|x\|^2/\latsd^2}\cosh(2\pi\langle x,w\rangle/\latsd^2)\right)\\
&\geq 2\gaussian_\latsd(w)\gaussian_\latsd(\Lambda + t),
\end{align*}
where the last inequality follows from $\cosh(\alpha) \geq 1$ for any real $\alpha$.
\end{proof}

\begin{lemma}
\label{lemma:Gaussian-evenly-distributed}
Let $\Lambda \subseteq V$ be a lattice and $V$ an Euclidean space, $t \in V$ and $\latsd \geq c\cdot \rr(\Lambda)$ for some $c>0$ (see \Cref{definition:rr}). Then the discrete Gaussian distribution on $\Lambda$ with parameter $\latsd$ and centered at $t$ is $p$-evenly distributed with $p = \frac{1}{1+e^{-\pi c^{-2}}}$. Concretely, for any strict sublattice $\Lambda' \subsetneq \Lambda$
\[ \frac{ \gaussian_\latsd(\Lambda' + t)}{\gaussian_\latsd(\Lambda + t)} \leq  \frac{1}{1+e^{-\pi c^{-2}}}.\]
\end{lemma}
\begin{proof}

Let $\Lambda' \subsetneq \Lambda$ be a sub-lattice of $\Lambda$ and
let $w \in \Lambda \setminus \Lambda'$. Then, by \Cref{lemma:regevtechnique},
\begin{align*}\gaussian_\latsd(\Lambda+t) &\geq  \gaussian_\latsd(\Lambda'+t) + \frac{\gaussian_\latsd(\Lambda'+t+w) + \gaussian_\latsd(\Lambda'+t-w)}{2} \\
&\geq (1 + \gaussian_\latsd(w))\gaussian_\latsd(\Lambda' + t).
\end{align*}
Writing $\Gaussian_{\Lambda+t, \latsd}$ for the Gaussian distribution on $\Lambda +t$ with parameter $\latsd$, we have,
\begin{align*}
\Pr_{\substack{x \gets \Gaussian_{\Lambda+t, \latsd}}}[x \in \Lambda' + t] = \frac{\gaussian_\latsd(\Lambda' + t)}{\gaussian_\latsd(\Lambda + t)}
& \leq \frac{1}{1 + \gaussian_\latsd(w)}.
\end{align*}
Since $\Lambda' \subsetneq \Lambda$, we know from the definition of $\rr(\Lambda)$ that there exists $w \in \Lambda \setminus \Lambda'$ such that $\|w\| \leq \rr(\Lambda) \leq \latsd/c$, hence
$\gaussian_\latsd(w) \geq \exp\left(-\pi c^{-2}\right)$, proving the lemma.
\end{proof}

\subsection{The (semi-)discrete Gaussian over $\DivKS$}

We introduce the following notation for the discretization of the $\mS$-divisor group as in \Cref{def:S-divisor}.
It is needed in \Cref{alg:randomrelation}, as this algorithm samples a Gaussian over $\Div_{K,\mS}$ and feeds it to \Cref{algo:compute-1-rel}. 
As $\Div_{K,\mS}$ is partially continuous (for which a finite computer cannot sample), we thus need this discretized $\dDivKS$ version of $\DivKS$.

\begin{notation}[Discretized $\mS$-divisor group] For a set of prime ideals $\mS$, we denote by $\dDivKS \subseteq \DivKS$ the \emph{discrete} subgroup of the restricted divisor group, for which
$x_\nu \in \frac{1}{\varN} \Z$ for all places $\nu$. In other words, any element of $\dDivKS$ can be written as 
\[ \ba = \sum_{\mp \in \mS} \ha_\mp \ldb \mp \rdb +  \sum_{\nu} x_{\nu} \ldb \nu \rdb  ~~~\mbox{ (with $x_\nu \in \frac{1}{\varN} \Z$)} \]
 
\end{notation}

\begin{remark} \label{remark:basis} In \Cref{alg:randomrelation}, we need to compute an approximation of the discrete Gaussian over the group $\dDivKS$ (see \Cref{lemma:gpv}). For this we need a $\Z$-basis of $\dDivKS$, which is naturally given by $\{ \frac{\mathbf{e}_\nu}{N} ~|~ \nu \mbox{  places of } K \} \cup \{ \mathbf{e}_{\mp} ~|~ \mp \in \mS\}$, where $\mathbf{e}_{x}$ is one on the coordinate $x$ and zero elsewhere (where $x$ can be a place or a prime). We call this basis $\mB(\dDivKS)$.
\end{remark}

In order to define a Gaussian distribution over 
$\DivKS$, we recall the distance notion on this space (see \Cref{def:normondiv}).
Since the group $\DivKS$ is partially discrete, namely, at the finite places,
the Gaussian over this group is semi-discrete as well.
\begin{definition}[The semi-discrete Gaussian over $\DivKS$] \label{def:gaussianoverlogsunits}
The semi-discrete Gaussian distribution $\Gaussian_{\DivKS,\divsd} \in L_1(\DivKS)$
is the distribution defined by
\[ \Gaussian_{\DivKS,\divsd}[\ba] :=  \gaussian_\divsd(\ba) \cdot \left( \int_{\bb \in \DivKS} \gaussian_\divsd(\bb) d \bb  \right)^{-1} \]
where $\gaussian_\divsd(\ba) = \exp(-\pi \|\ba\|^2/\divsd^2)$ (with the distance notion from \Cref{def:normondiv}) and where we use the shorthand notation
\[ \int_{\bb \in \DivKS} \gaussian_\divsd(\bb) d \bb  :=  \sum_{(\ha_\mp) \in \Z^\mS} \int_{ (\hb_\nu)_\nu \in \Log(\nfrstar)} \gaussian_\divsd \Big( \sum_{\mp \in \mS} \ha_\mp \ldb \mp \rdb +  \sum_{\nu} \hb_{\nu} \ldb \nu \rdb  \Big) d \hb   \]
\end{definition}
\begin{definition}[The discrete Gaussian over $\dDivKS$] \label{def:gaussianoverlogsunitsdiscrete}
For $\varN \in \N_{>0}$ we define the discrete Gaussian over $\dDivKS$ by the rule
\[ \Gaussian_{\dDivKS, \divsd} :=  \gaussian_\divsd(\ba) \cdot \left( \sum_{\bb \in \dDivKS} \gaussian_\divsd(\bb) \right)^{-1} \]
where $\gaussian_\divsd(\ba) = \exp(-\pi \|\ba\|^2/\divsd^2)$ (with the distance notion from \Cref{def:normondiv}).
\end{definition}

\subsection{Sampling a set of generators for the \texorpdfstring{Log-$\mS$-unit lattice}{Log-\letterSunits-unit lattice}}

\subsubsection{\Cref{alg:randomrelation}: \Cref{algo:compute-1-rel} with Gaussian input}

\begin{algorithm}[ht]
    \caption{Computing a random relation of $\mS$-units}
    \label{alg:randomrelation}
    \begin{algorithmic}[1]
    	\REQUIRE A number field $K$
    	, an LLL-reduced basis of $\OK$, 
    	and a set $\mT$  of prime ideals of $K$.
    	\ENSURE An random element of $\lsu$.
    		\STATE Let $\radpar \in \Z_{\geq 1}$ be as in \Cref{def:radpar} (with $\moduz$, $\blocksize$, and $x$ as in \Cref{algo:compute-1-rel}).
    		\STATE Let $\tilde \rr = \poly(\log |\Delta_K|,\allowbreak \max_{\fp \in \mT} \log \norm(\fp))$ the bound from \Cref{lemma:cov-radius}, such that $\rr(\logsunits) \leq \tilde \rr$. 
    		\STATE Put $\divsd  = 3 \cdot \max(\sqrt{\log(\rem+\cem+|\mS|)} , \tilde \rr)$,
    		and put
    		$\varN := \varNformula$.
    		\label{line:alggenerating:definesigma}
    		\STATE Sample, using \Cref{lemma:gpv} with $\geps := 1/\gaussepsinv$, and the basis $\mB(\dDivKS)$ described in \Cref{remark:basis},
    		\begin{equation} \label{eq:gaussiansamplerandomrelation} \ba = \sum_{\mp \in \mS} \ha_\mp \ldb \mp \rdb + \sum_{\nu} \hb_\nu \ldb \nu \rdb  \from \Klein_{\mB(\dDivKS),\geps,\divsd,0}, \end{equation} 
    		\label{line:alggenerating:gaussian}%
    		put $z = \infpart{\Exp}(\infpart{\ba})$
    		and compute a rational $y \in \nfrstar$ such that $\| y -  z\|_\infty \leq \min_\pl z_\pl/\varN$. 
    		\STATE Put $\ma = \prod_{\mp \in \mS} \mp^{\ha_\mp} = \finpart\Exp(\finpart{\ba})$.  \label{line:alggenerating:prodprimes}
    		\STATE Apply \Cref{algo:compute-1-rel} with $\ma \in \ideals$, $y \in \nfrstar$, and $\mS$, yielding $(\alpha, (v_\mp)_{\mp \in \mS})$. \label{line:alggenerating:call}
        \RETURN $(-(v_\mp + \ha_\mp)_{\mp \in \mS}, \Log(\alpha))$ where the $\ha_\mp$ are from line \lineref{line:alggenerating:gaussian}.
    \end{algorithmic}
\end{algorithm}
In this section, we prove that the output distribution of \Cref{alg:randomrelation} 
(which is essentially 
\Cref{algo:compute-1-rel} with input $\ba \in \dDivKS$ following a (discrete) Gaussian distribution),
is evenly distributed and concentrated on the Log-$\mS$-unit lattice (for certain adequate parameters).
Hence, by repeating \Cref{alg:randomrelation} sufficiently many times, one obtains a generating set of the
Log-$\mS$-unit lattice, which is made precise in the later \Cref{theorem:numberofsamples}.

\begin{lemma} \label{lemma:algrandomrelation} 
	Assume that $\mS$ contains all prime ideals coprime to $\moduz$ of norm $\leq \onerelationB$ (where $\onerelationB$ is defined as in \Cref{thm:compute-1-rel}).
	Then, on input $\mS$, \Cref{alg:randomrelation} runs in expected time 
\[ \poly(\log|\dcrk|,\divsd,\size(\mS)) + T, \] where $T$ is an upper bound on the expected time of \Cref{algo:compute-1-rel} on input $\ma$ and $\ky$ of size $\size(\ma) \leq \poly(\divsd, \size(\mS), \log |\dcrk|)$ and $\size(\ky) \leq  O(\divsd n \cdot (\size(\mS)+ \log |\dcrk|))$.
\end{lemma}
\begin{proof} Apart from the line \lineref{line:alggenerating:call}, in which \Cref{algo:compute-1-rel} is called, \Cref{alg:randomrelation} only requires computation power during the sampling of the approximate discrete Gaussian and the computation of $y \in \nfrstar$. By \Cref{lemma:gpv}, and denoting $n_0 = |\mS| + \dimh + 1 = |\mS| + \rem + \cem = \dim(\dDivKS)$,
the sampling of the approximate discrete Gaussian takes time (since $\geps := 1/\gaussepsinv$) polynomial in the size of the input, which can be given by the size of the lattice basis of $\dDivKS$, which is $\poly(n_0)$. Note that $\max_{\mathbf{b} \in \mB(\dDivKS)}\|\mathbf{b}\| = 1$, so, by choice of $\divsd$, we have $\divsd \geq \sqrt{\frac{\log (1/\geps) + 2 \log(n_0) + 3}{\pi}} \cdot \max_{\mathbf{b} \in \mB(\dDivKS)}\|\mathbf{b}\|$ as required to apply \Cref{lemma:gpv}.
The run time of the computation of $y \in \nfrstar$ depends on the minimum of $z_\nu/N$ and the maximum of $z_\nu$. We have that $|\hb_\pl| \leq \divsd \sqrt{n_0 \log(2n_0^2/\geps)} $ by \Cref{lemma:gpv} (with $\geps := 1/\gaussepsinv$).
Hence, in the worst case, we must approximate $y \approx z = \infpart{\Exp}(\infpart{\ba}) = (e^{b_\nu})_\nu$ up to precision $\exp(- \divsd \cdot  \sqrt{n_0 \cdot\log(2n_0^2/\geps)})/N$.
For the maximum of $z_\nu$, a similar reasoning tells us that $\max_\nu z_\nu \leq \exp(\divsd \cdot  \sqrt{n_0 \cdot\log(2n_0^2/\geps)})$.
 
Since we have $\log \varN = O(\log |\dcrk|)$ and $ \sqrt{n_0\log(2n_0^2/\geps)} = O(\size(\mS)+ \log |\dcrk|)$, the approximation $y \approx z$ required in line \lineref{line:alggenerating:gaussian} can be computed within time $\poly( \divsd,\allowbreak \size(\mS),\allowbreak \log |\dcrk|)$. This reasoning also shows that $\size(y) \leq O(\divsd n \cdot (\size(\mS)+ \log |\dcrk|))$.  Additionally, using the same bound of \Cref{lemma:gpv}, the size of $\ma$ must also be bounded by $\poly(\divsd, \size(\mS), \log|\dcrk|)$.
\end{proof}

\newcommand{\dba}{\ddot{\ba}}
\subsubsection{The output distribution of \Cref{alg:randomrelation} is concentrated.}

In the next lemma, we prove that the output distribution of \Cref{alg:randomrelation} is concentrated.
\begin{lemma} \label{lemma:concentrated} 
The output distribution $\bdistr$ of \Cref{alg:randomrelation} is $(R_0,0)$-concentrated for some 
\[ R_0 = O(\log^2 |\dcrk|) + 3 \cdot \divsd \cdot (|\mS| + \rem + \cem). \]
\end{lemma}
\begin{proof}
The output distribution of \Cref{alg:randomrelation} depends on an approximate discrete Gaussian from line \lineref{line:alggenerating:gaussian}, which is computed using \Cref{lemma:gpv}. From that lemma, denoting $n_0 = |\mS| + \rem + \cem = \dim(\dDivKS)$ follows that 
\[ \|\ba\| \leq \divsd \cdot \sqrt{n_0 \log(2n_0^2/\geps))} \leq 3 \divsd n_0 \] 
with $\divsd$ from line \lineref{line:alggenerating:definesigma} and $\geps = 1/\gaussepsinv$ (see line \lineref{line:alggenerating:gaussian} of \Cref{alg:randomrelation}). Here we use that $\sqrt{x \log(2x^2/\geps)} \leq 3x$ for $x \geq 1$ and $\geps = 1/\gaussepsinv$.

But since \Cref{algo:compute-1-rel} satisfies property (\itemref{lemma:outputdistrpropiii}) from \Cref{lemma:outputdistrprop}, we know that 
\[ \| \logs(\alpha) \| \leq O(\log^2 |\dcrk|) +  \| \ba \| \leq  O(\log^2 |\dcrk|) + 3 \cdot \divsd \cdot n_0. \]
Hence, there exists a $R_0 = O(\log^2 |\dcrk|) + 3 \cdot \divsd \cdot (|\mS| + \rem + \cem)$ for which no $\alpha$ with $\|\logs(\alpha) \| > R_0$ can be the output of \Cref{alg:randomrelation}.
\end{proof}

\subsubsection{The output distribution of \Cref{alg:randomrelation} is evenly distributed.}

We now prove that the output distribution of \Cref{alg:randomrelation} 
is evenly distributed on the logarithmic $\mS$-unit lattice.

The strategy to show evenly distributedness is more convoluted than was concentratedness. We start with defining the `continuous version of \Cref{alg:randomrelation}', in which in line \lineref{line:alggenerating:gaussian} is used a \emph{semi-discrete Gaussian} (as in \Cref{def:gaussianoverlogsunits}) instead of a fully discrete Gaussian (as in \Cref{def:gaussianoverlogsunitsdiscrete}). 
More precisely, the `continuous version of \Cref{alg:randomrelation}' is \Cref{alg:randomrelation} in which line \lineref{line:alggenerating:gaussian} is replaced with
 \[ \ba = \sum_{\mp \in \mS} \ha_\mp \ldb \mp \rdb + \sum_{\nu} b_\nu \ldb \nu \rdb \from \Gaussian_{\DivKS,\divsd}, \mbox{ and put } y = \infpart{\Exp}(\infpart{\ba}). \]
We will show in \Cref{lemma:closerandomrelation} that the output distribution of the `continuous version' and the ordinary version of \Cref{alg:randomrelation} are close in statistical distance.

In the present section, we focus on the following step: proving that the output distribution of `continuous version of \Cref{alg:randomrelation}' is evenly distributed for certain parameters. By the closeness of the `continuous version' and the ordinary version of \Cref{alg:randomrelation}, the ordinary must then also be evenly distributed.

\begin{lemma} \label{lemma:evenlydistributed} Assume that $\mS$ generates the ray class group $\rayclassgroupz$ of $K$ and does not contain any prime dividing $\moduz$. Then the output distribution of \Cref{alg:randomrelation} is $2/3$-evenly distributed on $\lsu$.
\end{lemma}
\begin{proof} We will later show, in \Cref{lemma:closerandomrelation}, that \Cref{alg:randomrelation} is $1/50$-close in statistical distance to the `continuous version of \Cref{alg:randomrelation}'. We show in the current proof that this `continuous version' is $6/10$-evenly distributed on $\lsu$, and hence the ordinary version is then $2/3$-evenly distributed on $\lsu$, since $6/10 + 1/50 < 2/3$.

The remainder of this proof is then devoted to showing that the `continuous version of \Cref{alg:randomrelation}' is $6/10$-evenly distributed on $\lsu$.

The output distribution $\bdistr$ of the `continuous version of \Cref{alg:randomrelation}' is supported on $\logsunits$ and given by the rule, for $\alpha \in \sunits$,
\[ \bdistr[ \logs(\alpha)] =
\int_{\ba \in \DivKS} \bdistr_\ba[ \logs(\alpha) ] \cdot \gaussemis(\ba) d \ba  . \]
To prove that this distribution is $p$-evenly distributed (with $p = 6/10$), we need to show that $\bdistr[M] = \sum_{m \in M} \bdistr[m] \leq p$
for every strict sublattice $M \subsetneq \logsunits$.

We first prove this for the simpler case where the modulus is trivial, i.e., $\moduz = \OK$; in which case we have $\Kmoduz = K^*$, and the shifting property (Property (\itemref{lemma:outputdistrpropii}) of \Cref{lemma:outputdistrprop}) holds for any $\alpha \in \K^*$, so in particular for $\alpha \in \sunits$.

So, pick any strict sublattice $M \subsetneq \lsu$. Then, by Property (\itemref{lemma:outputdistrpropii}) of \Cref{lemma:outputdistrprop}, and, subsequently, by change of variables, we have
\begin{align} \sum_{\logs(\mu) \in M}  \bdistr[\logs(\mu)] & = \int_{\ba \in \DivKS} \sum_{\logs(\mu) \in M}  \bdistr_\ba[ \logs(\mu) ] \cdot \gaussemis(\ba) d \ba \nonumber \\
& = \int_{\ba \in \DivKS} \sum_{\logs(\mu) \in M}  \bdistr_{\ba + \logs(\mu)}[0 ] \cdot \gaussemis(\ba) d \ba \nonumber  \\
& = \int_{\ba \in \DivKS} \sum_{\logs(\mu) \in M}  \bdistr_{\ba }[ 0 ] \cdot \gaussemis(\ba -  \logs(\mu)) d \ba \nonumber  \\
& = \int_{\ba \in \DivKS} \bdistr_{\ba }[ 0 ] \cdot \gaussemis(\ba +  M) d \ba
\label{eq:lastequationdivisorthing}
\end{align}
Using the assumption that $\mS$ generates the full class group, we take a fundamental domain $F$ of $\lsu$ in $\Div_{K,\mS}^0$, in order to decompose $\DivKS = \R \cdot \mathbf{1} + \lsu + F$, where $\mathbf{1} = \sum_{\nu} \npl \cdot \ldb \nu \rdb \in \Log(\nfrstar)$ is the vector consisting $\npl$, which equals $2$ if $\pl$ is complex and $1$ otherwise. This vector corresponds to the freedom of the norm (or rather, degree) in the space $\DivKS$, compared to $\Div_{K,\mS}^0$.

We can then rewrite any integral over $\DivKS$ with integrand $h(\ba)$   as
\[ \int_{\ba \in \DivKS} h(\ba) d\ba = \int_{t \in \R} \int_{\ba \in F }  \sum_{\logs(\beta) \in \lsu} h(\ba + \logs(\beta) + t\cdot \mathbf{1}) d\ba dt \]  We apply this to \Cref{eq:lastequationdivisorthing}, we use again \Cref{lemma:outputdistrprop}(\itemref{lemma:outputdistrpropii}) to write $\bdistr_{\ba + \logs(\beta)}[0] = \bdistr_{\ba}[\logs(\beta)]$ and we use the fact that $\bdistr_{\ba + t \mathbf{1}} = \bdistr_{\ba}$, since \Cref{algo:compute-1-rel} does not depend on the norm of the input (since $\Sample$ does not, see \Cref{thm:sampling-simplified-2}). Then, \Cref{eq:lastequationdivisorthing} equals
\begin{align}
& = \int_{t \in \R} \!\int_{\ba \in F} \!  \sum_{\logs(\beta) \in \lsu} \!\!\!\!\!\! \!\!\!\!  \bdistr_{\ba}[ \logs(\beta) ] \cdot  \gaussemis(\ba + t \mathbf{1}  + \logs(\beta) + M ) d \ba  dt
\label{eq:intthing1} \end{align} 
By H\"older's inequality, the positivity of all arguments and the fact that (by \Cref{lemma:outputdistrprop}(\itemref{lemma:outputdistrpropi}))
\[ \sum_{\logs(\beta) \in \lsu} \!\!\!\!\!\!\!\!\!\!  \bdistr_{\ba}[ \logs(\beta) ]  = 1 ,\]  we can bound \Cref{eq:intthing1} by 
\begin{align}
& \leq \int_{t \in \R} \! \int_{\ba \in F}  \max_{\logs(\beta) \in \lsu}  \gaussemis(\ba  + t \mathbf{1} +  \logs(\beta)+ M)   d \ba dt \\
& \leq \int_{t \in \R} \! \int_{\ba \in F} \frac{\gaussemis(\ba + t \mathbf{1} + \lsu)}{1 + e^{-\pi/c^2}}  d \ba dt \leq \frac{1}{1 + e^{-\pi/c^2}}.
\end{align}
whenever $\divsd \geq c \cdot \rr(\logsunits)$. Here we apply \Cref{lemma:Gaussian-evenly-distributed}, which gives us \[\gaussemis(\ba  + t \mathbf{1} +  \logs(\beta)+ M) \leq \tfrac{\gaussemis(\ba  + t \mathbf{1} +  \logs(\beta)+ \lsu)}{1 + e^{-\pi/c^2}} = \tfrac{\gaussemis(\ba  + t \mathbf{1} +  \lsu)}{1 + e^{-\pi/c^2}}.\] By picking $c = 3$, it holds that $\divsd \geq c \cdot \rr(\logsunits)$ (see line~\lineref{line:alggenerating:definesigma} of \Cref{alg:randomrelation}) and that $(1 + e^{-\pi/c^2})^{-1} \leq 6/10$.

We now explain how to amend the proof for the case where $\moduz \neq \OK$, i.e., a non-trivial modulus. In this case, the shifting property (Property (\itemref{lemma:outputdistrpropii}) of \Cref{lemma:outputdistrprop}) holds for $\alpha \in \Kmoduz$ only.

Write $M_{\moduz} = \logs(\Kmoduz) \cap M$ and
put $\{\logs(\mu_1),\ldots,\logs(\mu_w)\} \subseteq M$ for a set of representatives of $M/M_{\moduz}$ (which is finite by the fact that $K/\Kmoduz$ is finite). By similar computations as in \Cref{eq:lastequationdivisorthing}, we obtain
\begin{align}  \sum_{\logs(\mu) \in M}  \bdistr[\logs(\mu)] & = \int_{\ba \in \DivKS} \sum_{\logs(\mu) \in M_{\moduz}} \sum_{i = 1}^w  \bdistr_\ba[\logs(\mu_i)+ \logs(\mu) ] \cdot \gaussemis(\ba) d \ba \nonumber \\
& = \int_{\ba \in \DivKS} \sum_{\logs(\mu) \in M_{\moduz}} \sum_{i =1}^w \bdistr_{\ba + \logs(\mu)}[\logs(\mu_i) ] \cdot \gaussemis(\ba) d \ba \nonumber  \\
& = \int_{\ba \in \DivKS} \sum_{\logs(\mu) \in M_{\moduz}}   \sum_{i =1}^w \bdistr_{\ba }[ \logs(\mu_i) ] \cdot \gaussemis(\ba -  \logs(\mu)) d \ba \nonumber  \\
& = \int_{\ba \in \DivKS} \sum_{i = 1}^w \bdistr_{\ba }[ \logs(\mu_i) ] \cdot \gaussemis(\ba +  M_{\moduz}) d \ba \label{eq:modudivisorthing} \end{align}
Now we use that $\mS$ and $\moduz$ do not share primes, and that $\mS$ generates the ray class group $\rayclassgroupz$, so that we can split up the space 
\begin{equation} \label{eq:splitupspace} \DivKS = \R \cdot \mathbf{1} + \lsu \cap \logs(\Kmoduz) + F_{\moduz}, \end{equation} 
with a compact $F_{\moduz}$. The set $\mS$ and $\moduz$ not sharing primes is essential for the fundamental domain $F_{\moduz}$ to be compact (which is required for the current proof); indeed, if $\mS$ and $\moduz$ \emph{would} share a prime, $\lsu \cap \logs(\Kmoduz)$ would be of rank stictly less than that of $\DivKS$, resulting in a $F_{\moduz}$ in \Cref{eq:splitupspace} of infinite volume.

For $\mS$ and $\moduz$ not sharing a prime divisor, one could construct $F_{\moduz}$ as follows. Pick a set of ideals $A = \{\ma_1,\ldots,\ma_w\}$ (with all prime factors in $\mS$) that are representatives of the ray class group $\rayclassgroupz$. Next, define $V$ to be the Voronoi fundamental domain of $\Log(\Kmoduz \cap \OKstar) \subseteq H$. Then $F_{\moduz} = \{ \ldb \ma \rdb +  x ~|~ \ma \in A \mbox{ and } x \in V \}$ is an example of a fundamental domain satisfying \Cref{eq:splitupspace}.

Similar computations as in \Cref{eq:intthing1} then show that \Cref{eq:modudivisorthing} equals (where $\logs(\beta)$ sums over $\logs(\sunits \cap \Kmoduz)$)
\begin{align}
& = \int_{t \in \R} \!\int_{\ba \in F_{\moduz}} \!  \sum_{\logs(\beta)} \!  \sum_{i = 1}^w \bdistr_{\ba}[ \logs(\beta) + \logs(\mu_i) ] \cdot  \gaussemis(\ba + t \mathbf{1}  + \logs(\beta) + M_{\moduz} ) d \ba  dt
\label{eq:intthing1modu} \end{align}
Since $\{\logs(\mu_1),\ldots,\logs(\mu_w)\}$ are different representatives in $\lsu/\logs(\sunits \cap \Kmoduz)$ (indeed, if $\logs(\mu_1) - \logs(\mu_2) \in \logs(\Kmoduz)$, we have $\logs(\mu_1) - \logs(\mu_2) \in \logs(\Kmoduz) \cap M$, contradiction), we can apply H\"olders inequality again, yielding an upper bound of 
\[ \int_{t \in \R} \! \int_{\ba \in F}  \max_{\logs(\beta) \in \logs(\sunits \cap \Kmoduz)}  \gaussemis(\ba  + t \mathbf{1} +  \logs(\beta)+ M_{\moduz})   d \ba dt \leq \frac{1}{1 + e^{\pi/c^2}}\]
whenever $\divsd \geq c \cdot \rr(\logsunits)$, since $M_{\moduz} + \logs(\sunits \cap \Kmoduz)$ is a lattice in $\lsu$. 
\end{proof}

\subsubsection{About the closeness of the `continuous version of \Cref{alg:randomrelation}'}
In this part, we show that the `continuous version of \Cref{alg:randomrelation}' and the ordinary one 
are close in statistical distance. Before we do that, we first need a 
lemma that shows that $\varN$, the discretization parameter in $\dDivKS$ in \Cref{alg:randomrelation}
satisfies some bound.

\begin{lemma} \label{lemma:helplemmacomputation} Let $r = \radiusformula$ as used in \Cref{algo:compute-1-rel}. Then we have 
\[ 40000 \cdot |\dcrk| \cdot r^n \cdot n^{5/2} \leq \varNformula = \varN,
\] 
 where $\varN$ is as in line \lineref{line:alggenerating:definesigma} in \Cref{alg:randomrelation}.
\end{lemma}
\begin{proof} Let $r = \radiusformula$, where $\moduz$ and $\blocksize$ are as chosen in \Cref{algo:compute-1-rel}. The largest value $\norm(\moduz)$ can have is whenever $\moduz$ is the product of the prime ideals of norm smaller than $x \leq \tfrac{3}{2} \log |\dcrk|$ (by line \lineref{line:defx} of \Cref{algo:compute-1-rel} and Minkowski's theorem, which tells that $\log|\dcrk| \geq \log(\pi/2) n \geq 0.4n$). Hence, using \Cref{lemma:boundnormmoduz} with $x \leq \tfrac{3}{2} \log |\dcrk|$, using that $\sqrt{x}(\log(x)/(2\pi) + 2) \leq 3x$ and $\sqrt{x}(\log(x)^2/(8\pi) + 2 \leq 3x$ (for $x > 0.5$),
\begin{align} \log \norm(\moduz) & \leq x + \sqrt{x}\left(\left(\frac{\log x}{2\pi} + 2\right)\log|\Delta_K| + \left(\frac{(\log x)^2}{8\pi} + 2\right)n\right) \\ & \leq \underbrace{x}_{\leq \tfrac{3}{2} \log |\dcrk|} + \underbrace{3x \log |\dcrk|}_{\leq \tfrac{9}{2} \log^2 |\dcrk|}+ \underbrace{ 3x n  }_{\leq \tfrac{9}{2} \log |\dcrk| \cdot \tfrac{3}{2} \log |\dcrk|}
\leq 13 \log^2|\dcrk|.
\end{align}

By taking logarithms and using that $\log \norm(\moduz) \leq 13 \log^2 |\dcrk|$, that $\log(\blocksize^{2n^2/\blocksize}) \leq n^2$, that $n^{7n/2} \leq e^{3n^2}$, and that $\radiusformulaconstant^n \leq e^{4n^2}$, we have that 
\begin{align*} r^n  & = (\radiusformula)^n  \\ &\leq \radpar^n \cdot  e^{4n^2 + n^2 + 3n^2} \cdot e^{\frac{3}{2} \log|\dcrk| + 13 \log^2 |\dcrk|}
\leq \radpar^n \cdot e^{8n^2} e^{15 \log^2 |\dcrk|}. \end{align*}
Hence, using $40000 \leq e^{11}$ and $n^{5/2} \leq e^{n^2}$, we obtain 
\[ 40000 \cdot |\dcrk| \cdot r^n \cdot n^{5/2} \leq \varNformula = \varN . \]
\end{proof}

Recall that the `continuous version of \Cref{alg:randomrelation}' means: \Cref{alg:randomrelation} in which line \lineref{line:alggenerating:gaussian} is replaced by
  \[ \ba = \sum_{\mp \in \mS} \ha_\mp \ldb \mp \rdb + \sum_{\nu} b_\nu \ldb \nu \rdb \from \Gaussian_{\DivKS,\divsd}, \mbox{ and put } y = \infpart{\Exp}(\infpart{\ba} ). \]
\begin{lemma} \label{lemma:closerandomrelation} The statistical distance of the output distributions of \Cref{alg:randomrelation} and the `continuous version of \Cref{alg:randomrelation}' is at most $1/50$. 
\end{lemma}
\begin{proof} 
The output distribution of the `continuous version of \Cref{alg:randomrelation}' can be written as 
\begin{equation} \int_{\ba \in \DivKS} \bdistr_{\ba}[\placeholder] \Gaussian_{\DivKS,\divsd}(\ba) d\ba \end{equation}
 whereas the output distribution of the ordinary one can be written as
 \[\sum_{\dba \in \dDivKS} \bdistr_{\tilde{\dba}}[\placeholder] \hat{\Gaussian}_{\dDivKS,\divsd}(\dba) ,\]
 which is $\geps$-close (see \Cref{lemma:gpv}, with $\geps = 1/\gaussepsinv$) to 
 \begin{equation} \sum_{\dba \in \dDivKS} \bdistr_{\tilde{\dba}}[\placeholder] \Gaussian_{\dDivKS,\divsd}(\dba) ,\end{equation}
 where the approximation $\tilde{\dba} \approx \dba$ is caused by the computation of $\ky$ in line \lineref{line:alggenerating:gaussian}. 
Writing $F$ for a fundamental domain of $\dDivKS \subseteq \DivKS$ (i.e., $F + \dDivKS = \DivKS$ is a tiling), and taking into account the error of $\geps = 1/\gaussepsinv$ caused by the Gaussian approximation (\Cref{lemma:gpv}), the statistical distance between the continuous and ordinary variant of the algorithm can be estimated by (writing $|F|$ for the volume of $F$)
\[ \geps + \frac{1}{2} \cdot \left \|  \int_{\ba \in F} \sum_{\dba \in \dDivKS} \bdistr_{\dba + \ba}[\placeholder] \Gaussian_{\DivKS,\divsd}(\dba + \ba) -  |F|^{-1}  \bdistr_{\tilde{\dba}}[\placeholder] \Gaussian_{\dDivKS,\divsd}(\dba) d \ba \right \|_1 \]
By the trick $ab - a'b' = b(a - a') - a'(b' - b)$, the triangle inequality, and the fact that the $\ell_1$ norm of a distribution is equal to $1$, the quantity above can be bounded by 
\begin{align} \geps & + \frac{1}{2} \cdot \underbrace{\int_{\ba \in F} \sum_{\dba \in \dDivKS} \| \bdistr_{\dba + \ba}  -   \bdistr_{\tilde{\dba}}\|_1 \cdot \Gaussian_{\DivKS,\divsd}(\dba+\ba) d \ba}_{(A)} \label{eq:lipschitzbound} \\ & +  \frac{1}{2}\cdot \underbrace{\int_{\ba \in F} \sum_{\dba \in \dDivKS}  |\Gaussian_{\DivKS,\divsd}(\dba + \ba) - |F|^{-1} \Gaussian_{\dDivKS,\divsd}(\dba) | d \ba}_{(B)}  \label{eq:gaussianboundthing} 
\end{align}
By property \Cref{lemma:outputdistrprop}(\ref{lemma:outputdistrpropiv}), we have that $\bdistr_{\ba}$ is almost Lipschitz-continuous in $\ba$, 
\[ \|\bdistr_{\dba + \ba}  -   \bdistr_{\tilde{\dba}}\|_1 \leq 3 \cdot |\dcrk| \cdot r^n \cdot n^2 \cdot \| \dba+\ba - \tilde{\dba} \| + \tfrac{1}{200} , \]
and $\|\dba + \ba  - \tilde{\dba}\| \leq \|\tilde{\dba} - \dba \| + \| \ba\| \leq 2\sqrt{n}/\varN$ by construction\footnote{%
We have $\|\Log(y \cdot z^{-1})\| \leq \sqrt{n} \max_\pl \log(y_\pl z^{-1}_\pl) \leq \sqrt{n} \max_\pl |y_\pl z_\pl^{-1} - 1| \leq \sqrt{n} (\min_{\pl} z_\pl)^{-1} \cdot \max_\pl |y_\pl - z_\pl| \leq \sqrt{n}/N$, by line \lineref{line:alggenerating:gaussian} of \Cref{alg:randomrelation}.
}
of $\tilde{\dba} = \Log(y/z) + \dba$ on line \lineref{line:alggenerating:gaussian} and since $\max_{\ba \in F} \|\ba\| = \sqrt{n}/\varN$ for $F = \{ \sum_{\nu} x_\nu \ldb \nu \rdb \in \DivKS ~|~ x_\nu \in [-1/2,1/2)   \} $.  
Therefore, we can bound part (A) of \Cref{eq:lipschitzbound} by (using \Cref{lemma:helplemmacomputation})
\[ (A) \leq \underbrace{\frac{6 \cdot |\dcrk| \cdot r^n \cdot n^{5/2}}{\varN}}_{\leq 1/200} + 1/200 < 1/100.  \]

For the bound on \Cref{eq:gaussianboundthing}, note that the sole continuity is on the $\Log(\nfrstar)$-part of $\DivKS$. Hence, part (B) in \Cref{eq:gaussianboundthing} is equal to 
\begin{equation} \label{eq:Bwith2} (B) = \int_{\ba \in F} \sum_{\dba \in\frac{1}{\varN} \prod_{\nu} \Z}  |\Gaussian_{\Log(\nfrstar),\divsd}(\dba + \ba) - |F|^{-1} \Gaussian_{\frac{1}{\varN} \prod_{\nu} \Z,\divsd}(\dba) | d \ba , \end{equation}
where we understand the inclusion $\frac{1}{\varN} \prod_{\nu} \Z \hookrightarrow \Log(\nfr)$. 
Instantiating \Cref{lemma:closegaussians} with $\eps = 1/700$ (loc.~cit.), we can bound (B) via \Cref{eq:Bwith2} by $7 \eps = 1/100$. 
Here we use that $\divsd > 1$, that $\lambda_n(\frac{1}{\varN} \prod_{\nu} \Z) = \frac{1}{\varN}$, and that (by \Cref{lemma:helplemmacomputation}) $\varN \geq 40000  n^2 \geq 5 \pi \sqrt{n \log(4n/\eps)} \eps^{-1} n$, hence the assumptions of \Cref{lemma:closegaussians} are satisfied.

Concluding, the statistical distance between the continuous and ordinary variant of \Cref{alg:randomrelation}, by using that $\geps = 1/\gaussepsinv$, is at most $1/100 + 1/200 + 1/200 = 1/50$, since the respective summands in \Cref{eq:lipschitzbound,eq:gaussianboundthing} are bounded by $1/100,1/200,1/200$ respectively.
\end{proof}

\subsection{Conclusion: the number of samples that generate the full $\mS$-unit lattice}
The concluding theorem of this section quantifies how many samples
of the output distribution $\bdistr$ of \Cref{alg:randomrelation}
one needs to draw in order to get a generating set of the logarithmic unit lattice with high probability.
The proof of the theorem uses the fact that this output distribution is evenly distributed and concentrated
and applies \Cref{prop:Koen-lemma} on distributions over lattices.

\begin{theorem} \label{theorem:numberofsamples} 
There exists an absolute constant $C > 0$ and some $\onerelationB = \poly(L_{|\dcrk|}(1/2), L_{n^n}(2/3),\dedrescut)$,
where \[ \dedrescut =
\min\Big(\dedres, \max(e^{\log^{\frac{2}{3}}|\dcrk| \cdot \log^{\frac{2}{3}}(\log |\dcrk|)},e^{n^{\frac{2}{3}} \log^{\frac{4}{3}}(n)} ) \Big), \] such that the following holds.
For every number field $K$, for every $k \in \Z_{>0}$, and for every set of primes $\mS$ that generates the ray class group $\rayclassgroupz$ of $K$, does not contain primes dividing $\moduz$ and contains all primes (coprime with $\moduz$) of norm $\leq \onerelationB$, 
the probability that 
\[ 6 \cdot k + 6 \cdot (|\mS| + \dimh) \cdot \Big[\log((|\mS| + \dimh) \cdot \divsd) + C \log \log |\dcrk| \Big] \] 
samples from \Cref{alg:randomrelation} generate the entire logarithmic $\mS$-unit lattice $\lsu$
is at least $1 - e^{-k}$. Here, $|\mS|+ \dimh = \dim(\lsu)$ (with $\dimh = \dim(\logunits)$) and $\divsd$ is from line \lineref{line:alggenerating:definesigma} of \Cref{alg:randomrelation}.
\end{theorem}
\begin{proof} By \Cref{lemma:evenlydistributed} and \Cref{lemma:concentrated} the output distribution of \Cref{alg:randomrelation} on $\lsu$ is $2/3$-evenly distributed and $(R_0,0)$-concentrated for
$R_0 = O(\log^2|\dcrk|) + 3 \divsd \cdot n_0$,
where $\divsd$ is the deviation of the Gaussian distribution as in \Cref{alg:randomrelation}, and where $n_0 = |\mS| + \rem + \cem = \dim(\lsu) + 1$.

So, one can apply \Cref{prop:Koen-lemma} with $p = \tfrac{2}{3}$ and $q = 0$ (using $(1-p-q)^{-1} = 3$) to deduce that the number of samples $S$ required from the output distribution of \Cref{alg:randomrelation} to generate the entire log-$\mS$-unit lattice satisfies, for any $\alpha>0$,
\begin{equation} \label{eq:probsamplinglattice} \Pr\left[S > 3 \cdot (2+\alpha)(t + n_1) \right] \leq e^{-\alpha(t+n_1)/2},\end{equation}
where $n_1 = |\mS| + \dimh = \dim(\lsu) = n_0 - 1$ and\footnote{We use here that $\covol(\logsunits) = \classnumber \cdot \reg \cdot \sqrt{n_\R+n_\C} \geq \reg \geq 0.206$, since $\reg \geq 0.206$ uniformly for all number fields \cite[Theorem B]{Friedman1989AnalyticFF}.}
\begin{align*} t & = n_1 \cdot \log_2(R_0) - \log_2(\covol( \logsunits)) \\
  & \leq n_1 \cdot \log_2(R_0) + 2 \\
  &\leq n_1 [ \log(n_0 \divsd) + O(\log \log |\dcrk|)].
  \end{align*}
Taking $\alpha = 2k/(t + n_1)$, we have $\alpha(t + n_1)/2 = k$, and
\[  3 \cdot (2+\alpha)(t + n_1) = 6(t + n_1) + 3 \alpha (t + n_1) = 6(t + n_1) + 6k.  \]
Hence, replacing the above formula into \Cref{eq:probsamplinglattice}, we obtain
\[ \Pr\left[S > 6(t + n_1) + 6k \right] \leq e^{-k}. \]
Therefore, the probability that $6(t + n_1) + 6k$ samples from the output distribution of \Cref{alg:randomrelation} generate the entire logarithmic $\mS$-unit lattice $\lsu$
is at least $1 - e^{-k}$. Replacing $t + n_1$ by the larger $n_1 [ \log( n_1 \divsd) + C(\log \log |\dcrk|)]$ (for some absolute constant, and hiding $n_1$ under this constant by increasing the constant slightly if needed),
yields the final claim.
\end{proof}

This theorem roughly states that the number of samples drawn from $\bdistr$ (the output distribution of \Cref{alg:randomrelation})
for the sampled vectors to generate the log-$\mS$-unit lattice $\logsunits$ (with high probability) only needs to exceed a small multiple of the dimension of $\lsu$. In other words, if one assembles sufficiently many of such samples (quasi-linearly in the dimension), they generate the log-$\mS$-unit lattice except for an exponentially small probability.

In almost all applications involving $\mS$-units, a generating set with high probability is not sufficient. Instead, often a \emph{basis} of the logarithmic $\mS$-units is demanded (or, equivalently, a \emph{fundamental set of $\mS$-units}, see \Cref{postproc:intro}); and, additionally, no probabilistic error is allowed (i.e., one wants to be sure that the basis at hand truly is a basis of the full log-$\mS$-unit lattice).

To resolve these issues, one needs to post-process the set of generators by lattice reduction techniques. This is the subject of \Cref{section:postproc}.

%% file: rigorous_cgc/B06-postprocessing.tex
\section{Post-processing phase} \label{section:postproc}

\subsection{Introduction} \label{postproc:intro}
Up to now, the algorithm of this paper allows to provably sample $\mS$-units
in $\sunits$ with a certain probability. Furthermore, a sufficient number of these
$\mS$-units samples will generate the
entire $\sunits$ with overwhelming probability. To complete
the entire algorithm, two tasks remain to be done.
\begin{enumerate}[(I)]
 \item An algorithm to \emph{verify} that the sampled elements truly generate all of $\sunits$ and not a subgroup thereof\footnote{One could say that this algorithm changes the probabilistic algorithm of this paper from a Monte Carlo algorithm (a randomized algorithm whose output may be incorrect with a small probability, but with bounded running time) to a Las Vegas algorithm (a randomized algorithm that always gives correct results, but whose running time is a random variable)}.
 \item An algorithm to compute %
 a \emph{fundamental system of $\mS$-units} of $\sunits$ out of these generators, i.e., a sequence of elements $(\eta_j)_{j} \in \sunits$ with $j \in \{1,\ldots,\rank(\sunits) \}$ %
 such that every element $\eta \in \sunits$ is a power product of these elements (and a possible root of unity): $\eta = \tau_K \prod_{j=1}^{\rank(\sunits)} \eta_j^{m_j}$ for some $m_j \in \Z$; where $\tau_K \in \OK$ is some root of unity\footnote{As roots of unity vanish under the Logarithmic embedding into $\hyper$ and are easily computed, we omit them often in discussions.}.
\end{enumerate}

\subsection{Solving both (I) and (II) by computing bases}
Let $G \subseteq \sunits$ be a finite set of $\mS$-units. Then its associated lattice
\[ \lat_G := \sum_{\eta \in G} \logs(\eta) \cdot \Z =  \{ \logs(\eta') ~|~ \eta' \in \langle G \rangle \} \]
satisfies $\lat_G  \subseteq \logsunits$, i.e.,
it is a \emph{sublattice} of $\logsunits$.

An immediate observation is now that both the tasks (I) and (II) can be solved\footnote{Note that we assume that the primes in $\mS$ generate the ideal class group.}
if one is able to compute a \emph{basis} of the lattice $\lat_G$. Indeed, such a
basis allows to retrieve the rank and the determinant of $\lat_G$. If $\rank(\lat_G) = \rank(\logsunits) = \cem + \rem + |\mS| - 1$
and $\covol(\lat_G) = \covol(\logsunits) = \detlogsunits$ (see \Cref{eq:unittorusvolume}), we deduce that
$\lat_G = \logsunits$ (solving (I)) and that the basis of $\lat_G$ is a basis for $\logsunits$ (solving (II)).

So, the naive approach would be to create the matrix $\mG$ whose rows
consist of $\logs(\eta)$ with $\eta \in G$; the rows of this matrix
then generate $\logs(\lat_G)$. Apply LLL-reduction to $\mG$
to obtain a basis of $\logs(\lat_G)$ and compute its rank and determinant.
Unfortunately, this approach is not directly applicable.

\subsection{The challenge of approximate matrices}
The challenge lies in the fact that the infinite places
of $\logs(\eta)$ for $\eta \in \sunits$ can only be \emph{approximated}. Indeed,
these infinite places consists of logarithms of algebraic numbers,
whose can
only be computed with a certain precision.

This has as a consequence that the row-oriented matrix $\mG$ consisting
of the (exact) elements $\logs(\eta)$ for $\eta \in G$ is computationally
out of reach. Instead one is forced to work with $\mGa$, an \emph{approximation}
of $\mG$. This rational matrix $\mGa$ %
consists of elements $\logs(\eta) + \epsilon_\eta$ for $\eta \in G$ and $\|\epsilon_\eta\|_1 < \epsilon$.
Given the set $G$, we can compute $\mGa$ in time $|G| \cdot \poly(\log(\epsilon))$.
Of course, we have
\[ \|\mG - \mGa\|_\infty \leq \epsilon .\]
where $\| \cdot \|_\infty$ is the induced $\infty$-norm on matrices\footnote{We have $\| \mA \|_\infty = \max_j \|\av_j\|_1$ where $\av_j$ are the rows of $\mA$.}.

\subsection{The \bkp algorithm} \label{subsection:bpk}
For sufficiently (but still feasibly) small $\epsilon \in (0,1)$, an algorithm by Buchmann, Pohst and Kessler \cite{buchmann96,buchmann87} allows to compute an integer matrix $\mM$ such that $\mBa = \mM \mGa$ is an approximate basis that satisfies
\[ \|\mBa - \mB\|_\infty \leq C \cdot \epsilon ~\mbox{  and  }~  \| \mM \|_\infty \leq C \]
for some very large, but sufficiently bounded $C$ (depending on $\|\mG\|_\infty$ and invariants of the lattice $\logsunits$) and where $\mB = \mM \mG$ is a well-conditioned basis for the lattice $\lat_G$ generated by $\mG$.

Though this algorithm is analyzed well \cite{buchmann96}, its analysis can only be applied whenever the rank of the lattice $\lat_G$ is known.
A slight variation of this analysis, in which the rank of $\lat_G$ is not required to be known beforehand, is described in \Cref{section:buchmannkesslerpohst}.
Note that this analysis proves that the \bkp algorithm actually \emph{computes} the rank of $\lat_G$.

\subsection{Computing the determinant} \label{subsec:computedet}
So, the \bkp algorithm allows to compute a $\mBa = \mB + \epsilon' X$ close to a well-conditioned basis $\mB$; say,
$\|X\|_\infty \leq 1$ and $\epsilon' < 1$ is small.

Then, by a bound of Ipsen and Rehman \cite[Corollary 2.14]{ipsenrehman08}, for sufficiently (but feasibly) small $\epsilon' < 1$, one
has
\[ \det(\mBa^{\top} \mBa) \in [\tfrac{7}{8},\tfrac{9}{8}] \cdot \det(\mB^{\top} \mB). \]
The determinant $\covol(\logsunits)$ of the \emph{full} log-$\mS$-unit lattice can be (multiplicatively) approximated by means of Euler products (e.g., \cite{Bach95} or \Cref{prop:approx-rho}). In other words,
we can efficiently compute a $D \in \R_{>0}$ such that
\[ D \in [\tfrac{3}{4},\tfrac{5}{4}] \cdot \covol(\logsunits). \]
The index $[\logsunits:\lat_G]$ can then be approximated by computing $\sqrt{\det(\mBa^{\top} \mBa)}/D$, and satisfies%
\footnote{We have $\sqrt{\det(\mB^{\top} \mB)} \cdot c_1 = \sqrt{\det(\mBa^{\top} \mBa)}$ for $c_1 \in [\sqrt{\tfrac{7}{8}},\sqrt{\tfrac{9}{8}}]$ and $D = c_2\cdot \covol(\logsunits)$  with  $c_2 \in [\tfrac{3}{4},\tfrac{5}{4}]$. Therefore, $ \sqrt{\det(\mBa^{\top} \mBa)}/D = (c_1/c_2) \cdot\sqrt{\det(\mB^{\top} \mB)}/\covol(\logsunits)$, with $(c_1/c_2) \in [\sqrt{\tfrac{7}{8}} \cdot \tfrac{4}{5}, \sqrt{\tfrac{9}{8}} \cdot \tfrac{4}{3}] \subseteq (0.74,1.42)$.}%
\[  \sqrt{\det(\mBa^{\top} \mBa)}/D \in (0.74,1.42) \cdot [\logsunits:\lat_G]. \]
One deduces that if $[\logsunits:\lat_G] \geq 2$, then $\sqrt{\det(\mBa^{\top} \mBa)}/D > 1.48$. Contrarily, if $[\logsunits:\lat_G] = 1$, then $\sqrt{\det(\mBa^{\top} \mBa)}/D < 1.42$, which allows us to distinguish the two cases.
Additionally, in the second case $\mBa$ is an approximate
basis of $\logsunits$.

\subsection{Assembling a fundamental set of \texorpdfstring{$\mS$}{\letterSunits}-units}

Recall that $G = \{ \eta_1,\ldots,\eta_k \} \subseteq \sunits$ is a set of $\mS$-units
and $\mG = (\logs(\eta_j))_{j \in \{1,\ldots,k\}}$ its associated matrix of logarithmic
images. Assume that (by means of the algorithm in the text above) one has deduced that $G$ indeed
generates $\sunits$, i.e., $\langle G \rangle = \sunits$.

By the algorithm of \bkp in \Cref{subsection:bpk}, we can find an $\mM = (m_{ij})_{ij}$
for which $\mB = \mM \mG$ is a basis of $\logsunits$. In other words, for $j = 1,\ldots,\rank(\sunits)$,
\[  \bbv_j = \sum_{i = 1}^k m_{ji} \logs(\eta_i) = \logs(\beta_j) ~~~\mbox{  form a basis of $\logsunits$}. \]

As a consequence, by taking exponentials, for $j = 1,\ldots,\rank(\sunits)$, %
\begin{equation} \label{eq:prodfu} \beta_j = \prod_{i = 1}^k  \eta_i^{m_{ji}} ~~~\mbox{  form a system of fundamental units of $\sunits$}. \end{equation}

Note that the expression in \Cref{eq:prodfu} is in exact arithmetic (i.e., no precision issues). Namely, the elements $\eta_i \in G$ are
written in the basis of the number ring and the exponents $m_{ji}$ are integers from the integer matrix $\mM$.

\subsection{A fundamental set of \texorpdfstring{$\mS$}{\letterSunits}-units by means of compact representation}
Note that actually representing $\beta_j$ in terms of basis elements of $\OK$ or $K$
is generally computationally out of reach. Indeed, since the coefficients $m_{ij}$ are rather large,
the computation of $\beta_j$ (i.e., expanding the product in \Cref{eq:prodfu}) suffers from coefficient explosion.

This is the reason why the computation of a system of fundamental ($\mS$)-units is
generally given in `compact representation' (e.g., \cite{ANTS:BiasseFiecker14}, \cite[\textsection 5.8.3]{cohen2008computational}). That just means that,
instead of giving the $\beta_j$ in \Cref{eq:prodfu} by expanding its product and writing
$\beta_j$ in the basis of the number ring, one represents $\beta_j$ by the formal product consisting of $m_{ji}$-th powers of $\eta_i$.

In other words, a compact representation of $(\beta_j)_{j = 1,\ldots,\rank(\sunits)}$
is just given by the pair $(\mM,G)$ with $\mM = (m_{ij})_{ij} \in \Z^{k \times (\rank(\sunits))}$ and $G = \{ \eta_1,\ldots,\eta_k \}$,
by which we represent, for $j \in \{1,\ldots,\rank(\sunits)\}$,
\[ \beta_j = \prod_{i = 1}^k  \eta_i^{m_{ij}}. \]
The final output of the log-$\mS$-unit algorithm of this paper is thus the pair $(\mM,G)$.

\subsection{Final post-processing theorem}
A rigorous treatment of the post-processing phase can be found in \Cref{section:buchmannkesslerpohst,section:approximatedeterminant,section:postprocessingappendix},
as well as the proof of the following theorem.
The techniques used in that treatment generally involve 
well-known lattice reductions and computations, though for their rigorousness there is a heavy emphasis on numerical stability. For now, to stay close to the current subject manner, those numerical stability computations are thus deferred to Part \partref{part2b}.

\begin{restatable}[]{theorem}{postprocessingtheorem} \label{theorem:postprocessing}
 Let $K$ be a number field with degree $n$, %
and let $G = \{ \eta_1,\ldots,\eta_k\} \subseteq \sunits$ a finite set of $\mS$-units.
There exists an algorithm that takes as input the set~$G$\footnote{The elements $\eta_i \in K$ are represented as vectors of rational coordinates in a fixed $\OK$-basis, as explained in~\Cref{sec:representation}.} and computes a system of fundamental $\mS$-units for the subgroup
$\langle G \rangle \subseteq \sunits$
generated by $G$ in compact representation, i.e., a pair $(\mN,G)$ with $\mN \in \Z^{r \times k}$ such
that $\upsilon_j = \prod_{i = 1}^k \eta_i^{\mN_{ji}}$ for $j \in \{1,\ldots,r\}$ form
a system of fundamental $\mS$-units for $\langle G \rangle$.

Additionally, this algorithm decides whether $\langle G \rangle = \sunits$ or not,
and runs in time polynomial in %
$k$, $\log|\dcrk|$, $|\mT|$, $\log\big(\max_j \normx{\logs(\eta_j)} \big)$ and $\max_j(\size(\eta_j))$.
\end{restatable}

%% file: rigorous_cgc/B05b-generating-exceptional-unit.tex
\section{Generating the exceptional \texorpdfstring{$\mS$}{S}-units} \label{sec:exceptional-s-units} %
\subsection{Introduction}
In \Cref{theorem:numberofsamples}, the main result 
of \Cref{sec:BF-compute-many-rel}, the set of primes $\mS$
is assumed to contain no primes dividing the modulus $\moduz$.
This modulus $\moduz$ is defined in lines \lineref{line:comprho} -
\lineref{line:endifmzero}  of \Cref{algo:compute-1-rel}, and
only depends on the Dedekind residue $\rho_K$, the degree $n = [K:\Q]$
and the absolute discriminant $|\Delta_K|$.

In our end result, we would like to be able to compute $\mS'$-units 
for \emph{any} set of primes $\mS'$, without any restrictions. Indeed,
we do want to allow $\mS'$ to contain primes dividing $\moduz$.
Hence, in this section we aim, for any $\mp \mid \moduz$, 
for computing an $(\mS \cup \{ \mp \})$-unit
$\eta_\mp \in \Xunits{\mS \cup \{\mp \}}$ 
that additionally satisfies $\ord_{\mp}(\eta_\mp) = 1$.
Such an $(\mS \cup \{\mp \})$-unit we will then call an 
\emph{exceptional $\mS$-unit}.

By adding these exceptional $\mS$-units $\eta_\mp$ (for $\mp \mid \moduz$)
to a fundamental system of $\mS$-units, one gets a fundamental system
of $(\mS \cup \{ \mp ~|~ \mp \text{ divides } \moduz\})$-units. Hence computing 
these $\eta_\mp$ allows for computing a $\mS'$ unit group without 
any restrictions on the prime ideals.
Note that, if $\moduz = (1)$ is trivial, no exceptional $\mS$-units
exist, and the algorithm of this section can be omitted. Hence, throughout 
this section we will assume that $\moduz \neq (1)$.

\subsection{Algorithm for generating exceptional \texorpdfstring{$\mS$}{S}-units}
In \Cref{alg:exceptional}, we describe how to compute an exceptional
$\mS$-unit in the case of $\moduz \neq (1)$. The following lemma 
shows that the running time of this algorithm is the same as 
the running time of sampling a single $\mS$-unit as in \Cref{thm:compute-1-rel}.

\begin{algorithm}[ht]
    \caption{Computing an exceptional $\mS$-unit}
    \label{alg:exceptional}
    \begin{algorithmic}[1]
    	\REQUIRE~\\
    	\begin{enumerate}[(i)]
          \item An LLL-reduced basis of $\OK$,
    	 \item a prime ideal $\fq \mid \moduz$, (we assume $\moduz \neq (1)$)
    	 \item a set of prime ideals $\mT$ of $K$, that does not contain the primes dividing $\moduz$.
    	\end{enumerate}
    	\ENSURE $\alpha \in \fq$  and $(v_\mathfrak{p})_{\mathfrak{p} \in \mT} \in \Z^{|\mT|}_{\geq 0}$ such that $\alpha \OK = \fq \cdot \prod_{\mathfrak{p} \in \mT} \mathfrak{p}^{v_\mathfrak{p}}$
    		\STATE define $\blocksize, \radpar$ as in \Cref{algo:compute-1-rel}.
    		\REPEAT \label{line:repeat-except}
			\STATE $\alpha \leftarrow \Sample(\fq, \ky = 1 , \moduz \cdot \fq^{-1}, \blocksize, \radpar)$
			(see  \Cref{thm:sampling-simplified-2})
			\label{line:sampleideal-except}
    		\UNTIL {$\alpha \OK \cdot \fq^{-1}$ is $\mT$-smooth} \label{line:until-except}
    		\STATE compute $(v_\mathfrak{p})_{\mathfrak{p} \in \mT} \in \Z_{\geq 0}^{\mT}$ such that $\alpha \OK \cdot \fq^{-1} = \prod_{\mathfrak{p} \in \mT} \mathfrak{p}^{v_\mathfrak{p}}$
    		\label{line:decomposition-except}
        \RETURN $(\alpha, (v_\mathfrak{p})_{\mathfrak{p} \in \mT})$. \label{line:return-except}
    \end{algorithmic}
\end{algorithm}

\begin{proposition} \label{prop:exceptionalunits} There exists some $\onerelationB = \poly(L_{|\dcrk|}(1/2), L_{n^n}(2/3),\dedrescut)$,
where \[ \dedrescut =
\min\Big(\dedres, \max(e^{\log^{\frac{2}{3}}|\dcrk| \cdot \log^{\frac{2}{3}}(\log |\dcrk|)},e^{n^{\frac{2}{3}} \log^{\frac{4}{3}}(n)} ) \Big), \] such that the following holds.
Assume that $\mS$ contains all prime ideals coprime to $\moduz$ of norm $\leq \onerelationB$ (where $\moduz$ is defined as in lines \lineref{line:comprho} -
\lineref{line:endifmzero}  of \Cref{algo:compute-1-rel}).

Then, on input  $\fq \mid \moduz$ and the set $\mT$, 
\Cref{alg:exceptional} outputs $(\alpha, (v_\mathfrak{p})_{\mathfrak{p} \in \mT}) \in \fq \times \Z_{\geq 0}^{|\mT|}$ such that
\[ \alpha \OK =  \fq  \prod_{\mathfrak{p} \in \mT} \mathfrak{p}^{v_\mathfrak{p}}.\]
Furthermore, \Cref{algo:compute-1-rel} runs in expected time
\[ \poly(L_{|\Delta_K|}(\tfrac{1}{2}), L_{n^n}(\tfrac{2}{3}), \allowbreak \size(\mT),
\allowbreak \dedrescut, \log |\Delta_K|).\]
\end{proposition}
\begin{proof} The proof can be copied from that of \Cref{thm:compute-1-rel}. The sole differences are the instantiation $\moduz \cdot \fq^{-1}$ instead of $\moduz$, and the instantiations $\fa = \fq$ and $\ky = 1$. The latter instantiations do not impact the proof, and since $\size(\fq) \leq \poly(\log|\Delta_K|) \cdot x \leq \poly(\log |\Delta_K|)$ by the definition of $x$ in line \lineref{line:defx} of \Cref{algo:compute-1-rel}, the size of $\fq$ can be omitted in the running time.

Note that in \Cref{thm:compute-1-rel}, the input ideal $\fa$ is required to be coprime to $\moduz$. In the context of the current theorem, $\fq$ is coprime to $\moduz \cdot \fq^{-1}$ by construction, and hence the same reasoning as in \Cref{thm:compute-1-rel} applies.

The instantiation $\moduz \cdot \fq^{-1}$ instead of $\moduz$ does impact the value of $N(\moduz)/\phi(\moduz)$, but only by a factor $\frac{N(\fq)}{N(\fq) - 1} \in [1,2]$, which only impacts the success probability by a factor in $[1,2]$ and hence does not impact the final running time. Therefore, the running time of \Cref{prop:exceptionalunits} equals that of \Cref{thm:compute-1-rel}, with $\size(\fa)$ and $\size(\ky)$ deleted.
\end{proof}

%% file: rigorous_cgc/B07-finaltheorem.tex
\section{Full algorithm, final theorem and discussion} \label{sec:full-algorithm}
\begin{algorithm}[ht]
    \caption{Computing a fundamental system of $\mS$-units}
    \label{alg:fullalgorithm}
    \begin{algorithmic}[1]
    	\REQUIRE A number field $K$ 
    	, an LLL-reduced basis of $\OK$, 
    	and a set $\mT$  of prime ideals of $K$.
    	\ENSURE A pair $(\mM,G)$ with $\mM = (m_{ij})_{ij} \in \Z^{k \times \rank(\sunits)}$ and $G = \{ \eta_1,\ldots \eta_k\}\subseteq \sunits$ such that $\{\beta_1,\ldots,\beta_{\rank(\sunits)}\}$ is a fundamental set of $\mS$-units,
    	where
\[ \beta_j = \prod_{i = 1}^k  \eta_i^{m_{ji}}. \]
    		\STATE Put $G = \emptyset$
    		\STATE Define $\moduz$ as in lines \lineref{line:comprho} -
\lineref{line:endifmzero}  of \Cref{algo:compute-1-rel}.
			\STATE Define $\mathbb{T} = \mS \backslash \{ \mp ~|~  \mp \mid \moduz\}$.
    		\REPEAT
    		\STATE Apply \Cref{alg:randomrelation} with input $\mathbb{T}$, \label{line:alg4:applyrandomrelation} yielding $(\alpha, (\ha_\mp)_{\mp \in \mathbb{T}})$.
    		\STATE Add the output $\alpha$ (and its factorization) to $G'$. \label{line:alg4:alpha}
            \UNTIL{the \bkp algorithm as in \Cref{section:postproc} finds that $G'$ generates $\Xunits{\mathbb{T}}$. \label{line:final:until} }
		\STATE Let $(\mM',G')$ be the output of the \bkp algorithm, as in \Cref{theorem:postprocessing}. Write $G' = \{\eta_1,\ldots,\eta_{k'}\}$ and $\mM' =  (m'_{ij})_{ij} \in \Z^{k' \times \rank(\sunits)}$.
		\STATE Compute for any $\fq \mid \moduz$, an exceptional $\mathbb{T}$-unit $\eta_\fq \in \Xunits{\mathbb{T} \cup \{ \fq\}}$, using \Cref{alg:exceptional}. 
		\STATE Order $\{ \eta_{\fq} ~|~ \fq \mid \moduz\} = \{\eta_{k'+1}, \ldots, \eta_{k}\}$ and extend the set $G'$ to $G = \{ \eta_1,\ldots,\eta_{k'}, \eta_{k'+1} , \ldots, \eta_k\}$. Put $m_{ji} = m'_{ji}$ for $i \in \{1,\ldots,k'\}$ and $j \in \{1,\ldots,\rank(\Xunits{\mathbb{T}})\}$, put $m_{ji} = \delta_{ji}$ (Kronecker delta) for $i \in \{k'+1,\ldots,k\}$. Set $\mM = (m_{ij})_{ij} \in  \Z^{k \times \rank(\sunits)}$
        \RETURN $(\mM, G)$.
    \end{algorithmic}
\end{algorithm}

\begin{proposition}[ERH] \label{theorem:computationlogsunits}
There is a randomized algorithm $\mathcal{A}$ (\Cref{alg:fullalgorithm}) for which the following holds.
Let $K$ be a number field and let
\[ \dedrescut = \min(\dedres, \max(e^{\log^{\frac{2}{3}}|\dcrk| \cdot \log^{\frac{2}{3}}(\log |\dcrk|) }, e^{ n^{\frac{2}{3}} \cdot \log^{\frac{4}{3}} (n)})). \]
Note that $\dedrescut \leq L_{|\dcrk|}(2/3 + o(1))$. \\
Then there exists a bound
\[ B =  \poly(L_{|\dcrk|}(1/2), L_{n^n}(2/3),\dedrescut) .\]
such that if the algorithm $\mathcal{A}$ is given a set of primes $\mS$ of $K$ containing all primes with norm bounded by $B$ and generating the ray class group $\rayclassgroupz$ of $K$,
the algorithm $\mathcal{A}$
computes a fundamental system of $\mS$-units of $\sunits$ in compact representation, in expected time
\[ \poly(L_{|\dcrk|}(1/2), L_{n^n}(2/3),\dedrescut , \size(\mS)). \]
\end{proposition}
\begin{proof} The correctness of the algorithm follows from the following two arguments. By the `until' criterion in step \lineref{line:final:until},  that requires $G'$ to generate the $\mathbb{T}$-unit group, where $\mathbb{T} = \mS \backslash \{\mp ~|~ \mp \mid \moduz \}$.
Since the exceptional $\mathbb{T}$-units $\eta_\fq$ satisfy $\ord_{\fq}(\eta_\fq) = 1$ and $\eta_\fq \in \Xunits{\mS}$, adding these indeed result in a fundamental set of $\mS$-units. 

The running time of the algorithm is dominated by the repeat loop and the computation of the exceptional units. We will use that $\size(\mathbb{T}) \leq \size(\mS)$.

In the repeat loop, \Cref{alg:randomrelation} and the `post-processing' \bkp algorithm take expected time $ \poly(L_{|\dcrk|}(1/2),\allowbreak L_{n^n}(2/3),\allowbreak \dedrescut ,\allowbreak \size(\mS))$ (see \Cref{lemma:algrandomrelation}, \Cref{thm:compute-1-rel},  \Cref{theorem:postprocessing}, and the bound 
\[ \divsd \leq 3\max(\sqrt{\log(\rem + \cem + |\mS|)},\allowbreak \tilde{\rr}) \allowbreak \leq \poly(\log |\Delta_K|,\allowbreak \size(\mS)) \]
from \Cref{lemma:cov-radius}). The number of repetitions is expected to be within $\softO(|\mS|\log \log |\dcrk|) = \poly(\size(\mS))$ (see \Cref{theorem:numberofsamples}). 

The run time of computing all exceptional units is, by \Cref{prop:exceptionalunits}, at most $\poly(L_{|\Delta_K|}(\tfrac{1}{2}), L_{n^n}(\tfrac{2}{3}), \allowbreak \size(\mT),
\allowbreak \dedrescut, \log |\Delta_K|)$, where we use that the number of prime ideals dividing $\moduz$ can be at most $O(\log(\norm(\moduz))) = \poly(\log|\Delta_K|)$ (see \Cref{eq:boundlognormmoduz}).

Therefore, the expected running time of the overall computation is $\poly(L_{|\dcrk|}(1/2),\allowbreak L_{n^n}(2/3),\allowbreak\dedrescut ,\allowbreak \size(\mS))$ as well.
\end{proof}

The following theorem is the same as \Cref{theorem:computationlogsunits} except that the set of primes $\mS$ can now be arbitrary.

\begin{theorem}[ERH] \label{theorem:final_all_S} There is a probabilistic algorithm which, on input
a number field $K$ and a set of primes $\mS$ of $K$,
computes a fundamental system of $\mS$-units of $\sunits$ in compact representation, in expected time
\[ \poly(L_{|\dcrk|}(1/2), L_{n^n}(2/3),\dedrescut , \size(\mS)). \]
where $\dedrescut = \min(\dedres, \max(e^{\log^{\frac{2}{3}}|\dcrk| \cdot \log^{\frac{2}{3}}(\log |\dcrk|) }, e^{ n^{\frac{2}{3}} \cdot \log^{\frac{4}{3}} (n)}))$. Note that $\dedrescut \leq L_{|\dcrk|}(2/3 + o(1))$.
\end{theorem}
\begin{proof} Extend $\mS$ to a set $\mS' \supset \mS$ of primes, containing all primes with norm below $B =  \poly(\max(L_{|\dcrk|}(1/2,1), L_{n^n}(2/3,1),\dedrescut))$ and generating the ray class group $\rayclassgroupz$ of $K$, as in \Cref{theorem:computationlogsunits}. Note that, since $\log(\norm(\moduz)) = O(\log^2 |\Delta_K|)$ (see \Cref{eq:boundlognormmoduz}) and by Bach's bound $O(\log(|\Delta_K|^2 N(\moduz)))$ for ray class groups \cite{Bach90} this is asymptotically clearly satisfied.

Then, apply \Cref{alg:fullalgorithm} to get the pair $(\mM,G)$ representing a fundamental system of $\mS'$-units. This costs expected time $ \poly(L_{|\dcrk|}(1/2),\allowbreak L_{n^n}(2/3),\allowbreak \dedrescut ,\allowbreak \size(\mS))$, since $\size(\mS') = \poly(B, \size(\mS))$.

 As a preparation for the Hermite normal form algorithm, we put the infinite primes and the primes of $\mS$ to the very left in the representation of $\Log_{\mS'}(\gamma)$, i.e.,
\[ \Log_{\mS'}(\gamma) := \Big(  \underbrace{ \Log(\gamma) }_{\mbox{ \footnotesize{infinite places} }}, \underbrace{(-\ord_{\mp}(\gamma))_{\mp \in \mS}}_{\mbox{ \footnotesize{places in $\mS$} }},  (-\ord_{\mp}(\gamma))_{\mp \in \mS' \backslash \mS} \Big) \]
Put $\tilde{\mG}$ for the row-matrix consisting of the approximations of $\Log_{\mS'}(\gamma_j)$ for $j = 1, \ldots,|G|$, where $G =\{\gamma_1,\ldots,\gamma_{|G|}\}$. We then can write $\tmB = \mM \tilde{\mG}$ for an approximate basis for the log-$\mS'$-unit lattice (with the infinite places and the places of $\mS$ at the very left).

We now apply a Hermite normal form algorithm to $\tmB$, leading to $\tmH = \mathbf{U} \tmB$ that is in a lower-triangular shape. Hence, the first $|\mS| + \rem + \cem -1$ rows must then represent the Log-$\mS$-units, as it is the dimension of $\mS$ plus the dimension of the logarithmic unit lattice $\Log(\units)$.
Hence, the $\mS$-units can be given in compact representation $(\mN, G)$ with $\mN$ being the first $|\mS| + \rem + \cem -1$ rows of $\mathbf{U} \mM$. I.e., a fundamental set of $\mS$-units is
\[ \eta_k = \prod_{j = 1}^{|G|} \gamma_j^{\mN_{kj}} ~~~\mbox{ for } k \in \{1,\ldots, |\mS| + \rem + \cem -1 \}. \]
Since the running time of the Hermite normal form is polynomial in the dimensions of the input matrix and the maximum $\size(c_{ij})$ of the coefficients\footnote{We might have to scale up the coefficients of the infinite places to integers (column-wise), but that does not impact significantly the size of the coefficients.} \cite{storjohann}, we simply deduce that computing this Hermite normal form also takes time $ \poly(L_{|\dcrk|}(1/2),\allowbreak L_{n^n}(2/3),\allowbreak \dedrescut ,\allowbreak \size(\mS))$. This concludes the proof.
\end{proof}

\begin{figure}
\centering
\includegraphics[trim={40mm 180mm 40mm 40mm}, clip, width=.95\textwidth]{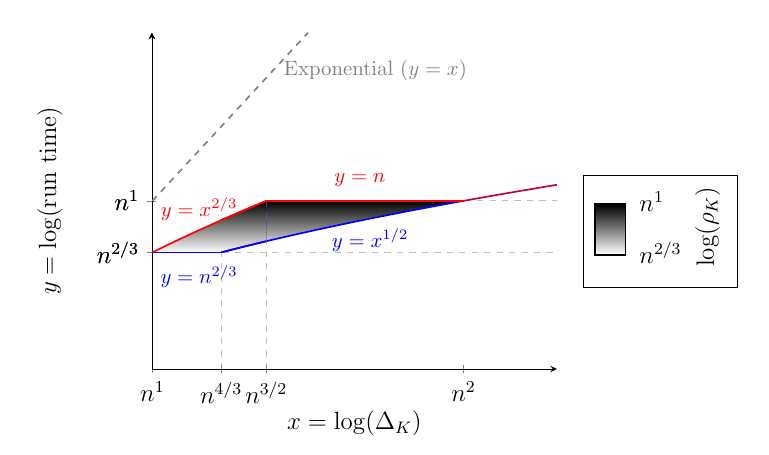}
\caption{A depiction of the (dominant part of the) provable running time compared to the (dominant part of the) heuristic running time for number fields $K$, depending on $\log(\rho_K)$ and the value of $\log(|\dcrk|)$ compared to the degree $n$.
\\
The blue line corresponds to the heuristic running time claimed in~\cite{ANTS:BiasseFiecker14}.
\\
The provable running time of the $\mS$-unit algorithm of the present work varies between the blue and red lines, depending on $\log(\dedres)$. The running time can be found in this graph by first searching the point $\log|\dcrk|$ (in terms of powers of the degree $n$) on the $x$-axis, and then read off, depending on the size of $\log(\dedres)$, where between the blue and red line the (logarithm of the) run time must be.
For number fields with $\log(\dedres) \leq n^{2/3}$, the running time is dictated by the blue line, and when $\log(\dedres) \geq n$, the complexity is dictated by the red line.
}
\label{fig:runningtime}
\end{figure}

\subsection{On the complexity of the $\mS$-unit computation algorithm} \label{subsec:complexity}
Now that we have proved the main result of this part of the article, we turn to its meta-analysis, comparing it with former heuristic claims.
Both the provable and the heuristic running time of the $\mS$-unit computation algorithm 
depend heavily and quite intricately on both the parameters $\log(|\dcrk|)$ and the degree $n$ of 
the number field $K$ at hand (see \Cref{fig:runningtime}). We will explain here the origin of the various changes in the running times, starting with that of the heuristic algorithm in \Cref{subsec:heuristicruntime} and ending with the running time of the provable algorithm of this paper in \Cref{subsec:provableruntime}.

In this explanation of the running times, we will only discuss the the running time in terms of
their `dominant part', which are of the shape $\exp(\softO(\log^{\delta}|\dcrk|))$ or $\exp(\softO(n^\delta))$
for some $\delta \in [0,1]$. Also, in \Cref{fig:runningtime} only this `dominant part' is shown. It is useful
to think of, for example, $\exp(\softO(\log^{\delta}|\dcrk|))$ of being of `approximately the same magnitude' as $L_{|\dcrk|}(\delta,1)$, while keeping in mind that $L_{|\dcrk|}(\delta,1)$ is a much more precise way of estimating complexities.
For the the more intricate analysis, we refer to the proof of \Cref{thm:compute-1-rel}.

\subsection{A general formula for the complexity of the $\mS$-unit computation algorithm}
At the very core, most $\mS$-unit computation algorithms consist of the following two steps (in which $\mS = \{ \mp \mbox{ prime ideal} ~|~ \norm(\mp) \leq B\}$):
\begin{enumerate}[(I)]
 \item Assemble $R$ `relations' in the shape of $B$-smooth elements of norm $\leq M$. Those are just elements of $\sunits$.
 \item Post-process these $R$ `relations' to get a system of fundamental $\mS$-units using lattice reduction techniques. Such a system of fundamental $\mS$-units can be seen as a (multiplicative) `basis' of $\sunits$.
\end{enumerate}
\subsubsection{Part (I)} The $B$-smooth elements are sampled
in a probabilistic way, by `picking an element of norm $\leq M$ at random and hoping it will be smooth'. One such attempt costs $1$ sampling of a norm $\leq M$ element (denoted $S_M$) and $1$ check for being $B$-smooth (denoted $C_B$). Let us denote $p_{M,B}$ for the probability of success, namely
that such sampling indeed yields a $B$-smooth element. Then, the running time for sampling a
single $B$-smooth element (a `relation') with constant success probability takes time
$(S_M + C_B) \cdot p_{M,B}^{-1}$. To obtain $R$ such relations, one obtains for part (I) an expected complexity of
\[ \poly\big( R (S_M + C_B) \cdot p_{M,B}^{-1} \big) .\]
\subsubsection{Part (II)}
The post-processing step consists of lattice reduction on a sort-of `valuation matrix'
of the gathered relations (elements in $\sunits$). Namely, each of these $R$ elements $\eta \in \sunits$
are decomposed as
\[ \big[ \Log(\eta) ,  (v_\mp(\eta))_{\mp \in \mS} \big] \in \R^{\rem + \cem -1} \times \Z^{\mS} ,\]
where each vector entry is bounded by\footnote{This is not the case per se, a fixed norm element can have arbitrarily large entries of the logarithmic embedding. Though, here, in this quick explanation, we assume that this is the case.} $\poly(\log(M))$, where $M$ is the maximum norm of the elements.
Reducing an $R \times (\mS + \rem + \cem -1)$-matrix like this costs, roughly said, time at most $\poly(R \cdot B \cdot \log(M))$ (where we use $B \approx |\mS|$). Here we assume that the decomposition of the elements into prime ideals has already be done in the `check for being smooth' step in Part (I).

\subsubsection{The general formula for the complexity}
Adding the complexities, using that $R \geq B$ and assuming that $\log(M)$ is negligible
compared to $B$, we obtain a general formula for the complexity, by effectively ignoring the costs of Part (II):
\begin{equation} \label{eq:generalformula} \poly\big( R \cdot (S_M + C_B) \cdot p_{M,B}^{-1} \big), \end{equation}
where $S_M$ is the cost of sampling a random element in $\OK$ of norm $\leq M$, where $C_B$ is the
costs of checking whether an element is $B$-smooth, where $p_{M,B}$ is the probability that a randomly sampled norm $\leq M$ element in $\OK$ is $B$-smooth, and where $R$ is the (average)
number of relations (elements in $\sunits$) needed to generate $\sunits$.

\subsection{Heuristic running time of Biasse \& Fieker \texorpdfstring{\cite{ANTS:BiasseFiecker14}}{}}
\label{subsec:heuristicruntime}

For the heuristic running time \cite{ANTS:BiasseFiecker14}, there are essentially two `regimes' of number fields. Namely, `regime A', the fields for which $n \leq \log|\dcrk| \leq n^{4/3}$ and `regime B' the fields for which $n^{4/3} \leq \log|\dcrk| < \infty$ (see the blue `Heuristic' line in \Cref{fig:runningtime}). We will explain, with the general formula of \Cref{eq:generalformula} and the heuristic assumptions of \cite{ANTS:BiasseFiecker14}, why
these regimes pop up.

Heuristic 1 (and 2) of \cite{ANTS:BiasseFiecker14} roughly state that they assume that
the probability $p_{M,B}$ of a random element of norm $\leq M$ being $B$-smooth equals
\[ p_{M,B} = \exp\Big(-\softO\Big(\frac{\log(M)}{\log(B)}\Big)\Big)  ~~\mbox{  (heuristic)}\]
Additionally, Heuristic 3 of \cite{ANTS:BiasseFiecker14}, roughly states that the number
of elements (`relations') $R$ required to generate the full $\mS$-unit group $\sunits$,
is only slightly larger than $B$, i.e.,
\[ R = \softO(B)  ~~  \mbox{   (heuristic)} \]
Using trivial trial division for $B$-smooth checking yields $C_B = O(B)$. One of the main
improvements of \cite{ANTS:BiasseFiecker14} is the use of stronger lattice reduction techniques (BKZ-$\blocksize$)
to sample small norm elements. More precisely, in their paper, they roughly take $M = e^{n^2/\blocksize} \cdot |\dcrk|$ and $S_M = \softO(\hkztime)$; where $\blocksize$ is the block size
parameter of BKZ.

Filling in these parameters into the formula of \Cref{eq:generalformula} and simplifying adequately,  we obtain
\[ \poly\big( \hkztime \cdot B \cdot p_{M,B} \big) %
=  \exp \Big( \softO \Big(\blocksize + \log(B) + \frac{n^2/\blocksize + \log|\dcrk|}{\log(B)} \Big) \Big)   \]
Optimizing the parameters $B$ and $\blocksize$, we quickly arrive at a complexity%
\footnote{
A lower bound of this complexity is $\exp( \softO(\log^{1/2}|\dcrk|))$
, which is
achieved when $\log B = \blocksize = \log^{1/2}|\dcrk|$ and $n^2/\blocksize \leq \log |\dcrk|$, i.e., $n^2 \leq \log^{3/2} |\dcrk|$. This explains the running time for `regime B', in which $n^{4/3} \leq \log|\dcrk|$.

In `regime A', we have $n^{4/3} \geq \log |\dcrk|$, for which a lower bound for the complexity is $\exp(\softO( n^{2/3}))$ which is achieved by taking $\log(B) = \blocksize = n^{2/3}$.
}
of $\exp( \softO(\log^{1/2}|\dcrk|))$
whenever $n^{4/3} \leq \log|\dcrk|$ and $\exp(\softO(n^{2/3}))$
whenever $n^{4/3} \geq \log|\dcrk|$. This explains the differences in complexity in the two regimes.

\begin{remark} Intuitively, these two regimes can be explained by the hardness
of sampling small-norm elements in $\OK$. Recall that for a sample of an element of maximum norm $M = e^{n^2/\blocksize} |\dcrk|$ one has to pay $\hkztime$.

For $|\dcrk|$ large compared to the degree (regime B), the extra factor $e^{n^2/\blocksize}$ in $M$ does not impact the asymptotic size of $M$. Contrarily, if the discriminant is small (regime A), the factor $e^{n^2/\blocksize}$ in $M$ becomes dominant. Therefore, the block size $\blocksize$ needs to be increased, resulting in a larger complexity.

To summarize, for large discriminants, it is the probability of sampling smooth elements that is the bottleneck of the algorithm, whereas for small discriminants it is the run time of BKZ.
\end{remark}

\subsection{Provable running time of this paper} \label{subsec:provableruntime}
We revisit the running time of the $\mS$-unit group computing algorithm,
but now without heuristics and only assuming the Generalized Riemann hypothesis.
The provable running time depends again on the parameters $n$ and $\log |\dcrk|$ and fall into three regimes. Regime A are the number fields for which $n \leq \log|\dcrk| \leq n^{3/2}$, regime B the number fields satisfying $n^{3/2} \leq \log|\dcrk| \leq n^2$ and regime C the number fields for which $n^2 \leq \log|\dcrk| < \infty$ (see the red `Provable' line in \Cref{fig:runningtime}). In the following we show how these different regimes arise.

In the algorithm of this paper the sampling probability for smooth elements has a provable lower bound from \Cref{part1} (see \Cref{theorem:ISmain})
\[ p_{M,B} \geq \frac{1}{\rho_K} \cdot  e^{ -\softO\left( \frac{\log M}{\log B} \right) }  ~~\mbox{ (provable)}\]
where $\rho_K = \lim_{s\rightarrow 1} (s-1) \zeta_K(s)$ is the residue of the
Dedekind zeta function at $s = 1$.
Additionally, in \Cref{part2}, we prove that taking
\[ R = \softO( B)   ~~\mbox{   (provable) } \]
is indeed sufficient to generate the full $\mS$-unit group. Instantiating the
rest of the parameters like in the heuristic version (\Cref{subsec:heuristicruntime}), we obtain the same running time,
with an extra factor $\rho_K$.
\[  \exp \Big( \softO \Big(\blocksize + \log(B) + \log(\rho_K) + \frac{n^2/\blocksize + \log|\dcrk|}{\log(B)} \Big) \Big)   \]

\subsubsection*{Regime C}
As we have $\log(\rho_K) \leq n \log(\log(|\dcrk|)/n)$ \cite{louboutin00}, we can deduce,
with the same reasoning as in \Cref{subsec:heuristicruntime} that the run time
is bounded by $\exp( \log^{1/2}|\dcrk|)$ whenever $\log|\dcrk| \geq n^2$. Then, namely, $\log \rho_K \leq \softO(\log^{1/2}|\dcrk|)$ and therefore $\rho_K$ has essentially no significant influence on the running time. This
explains the complexity of $\exp(\softO(\log^{1/2}|\dcrk|))$ for regime C in \Cref{fig:runningtime}.
\subsubsection*{Regime B}
When $\log|\dcrk| \leq n^2$, the running time of the provable algorithm possibly depends on $\rho_K$. As $\rho_K$ can be exponentially large in $n$, the worst-case running time for this regime is $\exp(O(n))$. Note, however, that for this regime, depending on the magnitude of $\rho_K$, the running time can lie anywhere in the gray area of \Cref{fig:runningtime}.

\subsubsection*{Regime A}
Whenever the discriminant is sufficiently small, namely, $\log|\dcrk| \leq n^{3/2}$, one can apply a trick to significantly diminish the influence of the residue $\rho_K$ on the running time.

Instead of sampling random elements in $\OK$, one samples random elements
\emph{coprime with some modulus ideal $\moduz \subseteq \OK$} that is the product of all primes below some bound $X$. This sampling is done in a specific way, namely, by first sampling $\tau \in (\OK/\moduz)^\times$ and then sampling a random element equivalent to $\tau$ modulo $\moduz$. This causes
the sampling probability of smooth elements (coprime with $\moduz$) to
\emph{increase} by a factor $\norm(\moduz)/\phi(\moduz)$,
but it also causes the maximum norm $M$ of the sampled elements to increase%
\footnote{Because, instead of sampling in $\OK$, we now have to sample in the shifted lattice $\tau + \moduz$, which has covolume $\norm(\moduz)$ times the covolume of $\OK$.}
by a factor $\norm(\moduz)$, i.e.,
\begin{equation} \label{eq:newprobbound} p_{M,B} \geq \frac{\norm(\moduz)}{\phi(\moduz) \cdot \rho_K} \cdot  e^{ - \softO \Big( \frac{\log M + \log \norm(\moduz)}{\log B} \Big) }  ~~\mbox{ (provable)}\end{equation}
with $M = 2^{n^2/\blocksize} |\dcrk|$. Using the fact that $\frac{\norm(\moduz)}{\phi(\moduz) \cdot \rho_K} \geq \exp( -\softO(\log|\dcrk|/\sqrt{X}))$ (see \Cref{prop:boundrhonormphi}), and $\norm(\moduz) \approx \exp(\softO(X + \log|\dcrk|\cdot \sqrt{X}))$ (see \Cref{lemma:boundnormmoduz}),
we can rewrite \Cref{eq:newprobbound}:
\begin{equation}
  p_{M,B} \geq  e^{- \softO\Big( \frac{\log|\dcrk|}{\sqrt{X}} + \frac{n^2/\blocksize + \log |\dcrk|   + X + \log|\dcrk| \cdot \sqrt{X}}{\log B} \Big)   }  ~~\mbox{ (provable)}.
\end{equation}
Plugging in this probability lower bound into the general formula, the logarithm
of the running time becomes, asymptotically,
\begin{equation}
 \softO\left(\blocksize + \log(B) + \frac{n^2/\blocksize + \log|\dcrk| \cdot \sqrt{X}}{\log(B)}  + \frac{\log|\dcrk|}{\sqrt{X}}\right).
\end{equation}
This has as asymptotic lower bound $\softO(\log^{2/3}|\dcrk|)$, attained
when $X = \log^{2/3}|\dcrk|$ and $\log(B) = \log^{2/3}|\dcrk|$.
Note that this trick is only useful whenever $\log^{2/3}|\dcrk| \leq \log \dedres \leq n \log(\log|\dcrk|/n)$,
otherwise the approach as in regime B (whose running time depends on $\dedres$) would be faster. So,
surely whenever $\log^{2/3}|\dcrk| \leq n$, i.e., $\log|\dcrk| \leq n^{3/2}$, this trick could outperform
regime B (if $\dedres$ is large), which explains regime A of the provable running time (see \Cref{fig:runningtime}).

\subsection{Discussion}

\subsubsection*{Is the dependency of the running time on \texorpdfstring{$\rho_K$}{rhoK} innate?}
The (dominant parts of the) provable expected running time of our $\mS$-unit computation
algorithm is depicted in \Cref{fig:runningtime} depending
on the parameters $\log(\rho_K)$ and $\log|\dcrk|$ compared to the degree $n$.
For a fixed number field with $\log|\dcrk| = n^\delta$ for $\delta \in [1,2]$
the provable expected running time can be anywhere where the line $x = n^\delta$ intersects the shaded area, depending on the size of $\rho_K$. For small $\rho_K$, this running time can be equal to the heuristic one (the bottom blue line), whereas for large $\rho_K$, it might be equal to the worst-case provable one (the top red line).
As this quantity $\rho_K$ is computable by Euler products.

\begin{mdframed}
The main question in this discussion is: \emph{does $\rho_K$ really have this influence on the running time} or is it merely an artifact of the lack of
tight bounds on smooth ideal densities?
\end{mdframed}

Namely, there are asymptotic formulae \cite{Krause90} (see also \Cref{eq:useless-estimate}) indicating that, for $u = \log(x)/\log(B)$,
\[ \delta_{\idset_B}[x] \geq u^{-u} \]
but, generally, nothing is said about what the smallest $x \in \R_{>0}$ is
for which this holds. For our purposes, we would like this smallest $x$ to be at most $O(|\dcrk|)$ (and $B = \exp(O(\log^{1/2}(|\dcrk|)))$). This is not something we could deduce from
asymptotic results. Therefore, we chose to use a combinatorial result (\Cref{lemma:proba-smooth}), which shows that, for every $x > \log(B) \geq 2 \log \log |\dcrk|$,
\[ \delta_{\idset_B}[x] \geq \frac{1}{\rho_K \cdot B \cdot (4\log(B))^u} u^{-u}. \]
This lower bound has as an advantage that it is explicit and holds for small values of $x$ (in particular the range around $|\dcrk|$ we need). The disadvantage is the extra factor $\frac{1}{\rho_K}$, which increases, for large $\rho_K$, the provable running time of our algorithm drastically.

We could ask the main question of this discussion differently:
\begin{mdframed}
Are there bounds $B = \exp(O(\log^{1/2}(|\dcrk|)))$, $X = O(|\dcrk|)$, and $C^{-1} = \poly(B,\exp(\log^{1/2}|\dcrk|))$ such that
for all $x \geq X$, and for $u = \log(x)/\log(B)$, the inequality
\[ \delta_{\idset_B}[x] = \frac{|\{ \fa \in \idset_B~|~ \norm(\fa) \leq x\}|}{\rho_K \cdot x} \geq  C \cdot u^{-u} \]
holds?
\end{mdframed}
If the answer is \emph{yes}, our algorithm would reach the heuristic running time proposed by \cite{ANTS:BiasseFiecker14}.
If the answer is \emph{no}, it is likely that the heuristic algorithm of \cite{ANTS:BiasseFiecker14} does not achieve its claimed running time, and instead has the same dependence on $\rho_K$ as our algorithm.

\subsubsection*{The use of variants of the Riemann Hypothesis}
An interesting follow-up question would be how the algorithms in this paper 
would perform without assuming variants of the Riemann Hypothesis. To aid such research we now identify which type of Riemann Hypothesis is used for which task and discuss the possible impact of letting go of these hypotheses.

\begin{enumerate}
 \item \textbf{Random Walk Theorem.} In \Cref{thm:random-walk-weak}, the Riemann Hypothesis for all Hecke L-functions associated with a Hecke-character over a number field $K$ (with fixed modulus $\modu$) is assumed. The Riemann Hypothesis in this form here seems indispensable for this proof, due to its influence on the bound on the Hecke eigenvalues (see \Cref{eq:fulleigenvaluebound}). 
 
 Without the Riemann Hypothesis, one might have to resort to an alternative, weaker bound, optimistically maybe of the following form, inspired by bounds \cite{Trudgian2016,FORD_2002} for the prime counting function:
 \begin{equation}  \lambda_\chi = O\left( c_{K,\modu} \cdot e^{-\sqrt{\log(B)}} \right). \label{eq:trudgianbound} \end{equation}
 Here, $c_{K,\modu}$ is some constant depending on $K$ and $\modu$, and the big-O is, as it is throughout the entire present work, absolute. A bound as \Cref{eq:trudgianbound} for all eigenvalues of Hecke characters under Hecke operators, of which we do not know a proof 
 and what the constant $c_{K,\modu}$ then might be, implies then that we are required to choose
 \[ B = O( e^{\log^2(c_{K,\modu})}) ,\]
 which might or might not be useful, depending on the value of $c_{K,\modu}$.
 \item \textbf{Estimation of the factor $\norm(\moduz)/\phi(\moduz)$.} In \Cref{section:estimationmodu}, the Riemann hypothesis for the Dedekind zeta function of $K$ is assumed to estimate the factor $\norm(\moduz)/\phi(\moduz)$. This estimation is used to diminish the impact of $\dedres$ on the running time of the $\mS$-unit computation algorithm. To be more precise, this estimate on this factor allows for making the running time depend on $\dedrescut$ instead of $\dedres$ (see \Cref{theorem:final_all_S}, also the left-hand down-sloping part of the red graph in \Cref{fig:runningtime}).
 
 Not assuming the Riemann hypothesis for the Dedekind zeta function of $K$ might impact the bound in \Cref{prop:boundrhonormphi} as well as the bound in \Cref{lemma:boundnormmoduz}. Hence, the precise implications of letting go of the Riemann hypothesis is not totally clear, but might heavily impact the dependency of the provable running time on $\dedres$.
 \item \textbf{Density of smooth ideals.} In \Cref{sec:smooth}, a combinatorial bound is given for the density of smooth ideals, in which bounds on the number of prime ideals in a number field are used. Letting go of the Riemann Hypothesis here might have as a consequence that heavier restrictions on the numbers $\smoothdensityB$ and $A$ are required.
 
 \item \textbf{Computation of an approximation of $\dedres$.} In multiple parts of the $\mS$-unit computation algorithm an approximation of the Dedekind zeta function residue $\dedres$ is required. In \Cref{algo:compute-1-rel}, one uses this approximation whenever $\dedres$ is large, to apply the `modulus trick' in order to diminish the influence of $\dedres$ on the running time.  In the full algorithm, \Cref{alg:fullalgorithm}, an approximation of $\dedres$ (or equivalently, of $\reg \cdot \classnumber$) is needed to decide whether the algorithm indeed found the full logarithmic $\mS$-unit lattice (see \Cref{subsec:computedet}). The knowledge of an approximation of $\dedres$ thus allows to change the full algorithm from a Las Vegas algorithm (i.e., the run time is a random variable but the answer of the algorithm is always correct) into a Monte Carlo algorithm (i.e., the run time is bounded but the answer might be wrong or inconclusive with a certain fixed small probability).
 
 If the Riemann hypothesis for the Dedekind zeta function of $K$ is not assumed, longer Euler products are required to approximate $\dedres$. Depending on the length of these Euler products compared to the running time of the rest of the entire algorithm (which is already at least $L_{|\dcrk|}(1/2)$), letting go of the Riemann hypothesis might or might not influence the running time of the algorithm.
\end{enumerate}

%% file: discretization/introduction.tex
\section{Introduction}

\noindent
Roughly, this last part of the current paper
can be divided into two main subjects: The post-processing part, which is about
processing the generators of the log-$\mS$-unit lattice into a basis thereof
(\Cref{section:buchmannkesslerpohst,section:approximatedeterminant,section:postprocessingappendix});
 and the BKZ-algorithm part on approximate bases, which is needed to find short
vectors in ideal lattices (\Cref{section:latprelim,section:BKZinteger,section:BKZapproximate}).

\subsection{Post-processing}
The first subject is the post-processing of the generating
matrix of the log-$\mS$-unit lattice in order to get a basis
(and hence a fundamental set of $\mS$-units) of this log-$\mS$-unit
lattice, see also \Cref{section:postproc}.

In this post-processing stage, one consecutively computes the \emph{rank},
the \emph{determinant} and a well-conditioned \emph{basis} of the lattice generated
by the (approximate) generating matrix. The rank and determinant are
computed first in order to check if the generating matrix indeed generates
the entire log-$\mS$-unit lattice (and not a sublattice thereof).

This post-processing makes use of the \bkp algorithm, an algorithm
that, roughly said, LLL-reduces a real-valued basis of one which only
knows a sufficiently good approximation of. Though, in the work of Buchmann, Kessler
and Pohst \cite{buchmann96}, it is assumed that the rank and the determinant
of the generated lattice is \emph{known}. In our application, though, these
quantities are not known but instead need to be computed.

Hence, a slight modification of the analysis of Buchmann, Kessler and Pohst
is treated in \Cref{section:buchmannkesslerpohst} in order to show that
one can indeed \emph{compute} the rank and a `relatively well conditioned' approximate basis
of the generated
lattice with use of the \bkp algorithm. This relatively well conditioned approximate
basis can then be used to compute an approximation of the determinant (see \Cref{section:approximatedeterminant})

The entire compound result with the exact precision needed, applied to our
$\mS$-unit computation algorithm of \Cref{part2}, is then treated in \Cref{section:postprocessingappendix}.

\subsection{The BKZ-algorithm on ideal lattices}
As we aim for a provable running time in this paper,
we use a variant of the BKZ-algorithm \cite{HPS11} with a proven upper
bound on the run time. The challenge with the lattices occurring
in this paper, is that only an approximate basis of them are known. Hence,
to make use of this variant of the BKZ-algorithm, we have to prove
that this variant behaves well on small changes of the input.

In order to do so, a preliminary section about lattices is required (\Cref{section:latprelim}).
After that, the running time of the variant of the BKZ-algorithm \cite{HPS11} on \emph{integer} bases
is analyzed, following a sketch of \cite[Section 3, Cost of BKZ']{HPS11} (\Cref{section:BKZinteger}).
Lastly, this BKZ-variant
is applied on an approximate basis, with use of the \bkp algorithm (\Cref{section:BKZapproximate}).

%% file: discretization/C-A-buchmannkessler.tex
\section{\bkp algorithm} \label{section:buchmannkesslerpohst}
\subsection{Introduction}
In this section, we will treat a slightly modified analysis
of the algorithm of Buchmann, Pohst and Kessler \cite{buchmann96,buchmann87}.
The original algorithm allows to compute an approximate basis
of a lattice $\Lambda$ from an approximate generating matrix of a lattice (i.e., a matrix
whose rows generate the lattice), provided that the
approximation of this generating matrix is good enough. In this original
algorithm the rank and the determinant of the lattice generated by the generating matrix are
assumed to be known.

In our case, computing (a fundamental system of) $\mS$-units,
the rank and the determinant of the lattice of the exact generating matrix
are \emph{not} known. Rather, we would like to approximately compute these quantities,
in order to deduce whether the generating matrix generates the entire
logarithmic $\mS$-unit lattice or not.

Exactly this lack of knowing the rank and determinant of the lattice (spanned by the exact generating matrix)
requires us to slightly modify the analysis of Buchmann, Pohst and Kessler.
In this modified analysis, the algorithm of Buchmann, Pohst and Kessler is actually applied
twice (this was already suggested in \cite[p.~9, remark after Theorem~4.2]{buchmann96}). These two
applications of the same algorithm serve two different goals.

During the
first application of the \bkp algorithm, one computes an (approximate) \emph{basis} of the lattice at hand.
From this approximate basis one can directly obtain the \emph{rank} of the lattice. This first (approximate) basis, though, might have bad
conditioning properties and too long vectors. It is exactly the purpose of the second application of the
\bkp algorithm to transform this approximate basis in a \emph{well-conditioned} approximate basis,
with reasonably short vectors (close to what LLL can achieve).

This well-conditioning of the (approximate) basis allows to approximate the \emph{determinant}
of this basis within reasonable precision. This will be the subject of \Cref{section:approximatedeterminant}.
Knowing the determinant and the rank of the resulting basis allows to verify whether
the basis is an actual basis of the logarithmic $\mS$-unit lattice or not (as there are
formulas to compute the rank and the covolume of the logarithmic $\mS$-unit lattice).

\subsection{Preliminaries}
\subsubsection{On the quasi-norm $\normt{\mA}$}
In order to phrase the results of \cite{buchmann96} more
succinctly, we introduce the following norm-like function on
matrices. It is just the maximum over the $2$-norms of the rows of the matrix.
\begin{definition} For any $k \times m$-matrix $\mA$, we denote
\[ \normt{\mA} :=  \max_{1 \leq j \leq k} \| \av_j \|, \]
where $\av_j$ are the \emph{rows} of $\mA$, and where $\| \cdot \|$ is the Euclidean norm on the row vectors.
\end{definition}
This is not a matrix norm as it is not submultiplicative, but it is almost submultiplicative:
\begin{lemma} \label{lemma:halfnorm} For a real $m \times \ell$-matrix $\mM'$ and a real $\ell \times k$-matrix $\mM$, we have
\[ \normt{\mM' \mM} \leq \sqrt{\ell} \cdot \normt{\mM'} \normt{\mM}  \]
\end{lemma}
\begin{proof} Take the $2$-norm of the $t$-th row in $\mM' \mM$, which equals
\[ \| \sum_{j = 1}^\ell  m'_{tj} \mM_j\| \leq \sum_{j = 1}^\ell  |m'_{tj}| \| \mM_j\| \leq \| \mM'_t \| \cdot \Big( \sum_{j = 1}^\ell \| \mM_j\|^2 \Big)^{1/2} \leq \| \mM'_t \| \cdot \sqrt{\ell} \cdot \normt{\mM} \]
where $m'_{tj}$ is the $tj$-th entry in $\mM'$, $\mM_j$ is the $j$-th row of $\mM$ and $\mM'_t$ is the $t$-th row of $\mM'$. The first inequality is the triangle inequality, the second inequality follows from the Cauchy-Schwarz inequality and the third from the inequality between $2$-norms and infinity norms. As this holds for any row in $\mM'\mM$, we conclude that
\[ \normt{\mM' \mM} \leq \sqrt{\ell} \cdot \normt{\mM'} \normt{\mM}  .\]
\end{proof}

\subsubsection{Rounding in the \bkp algorithm} \label{subsubsec:roundinginBKP}
In the original analysis of the \bkp algorithm \cite{buchmann96},
the input approximate generating matrix $\mA \in  \R^{k \times n}$ (consisting of rows $\av_j \in \R^n$) is scaled up by $2^q$ and rounded entry-wise
to the nearest integer, thus obtaining an integer matrix $\hat{\mA} \in \Z^{k \times n}$ (consisting of rows $\hat{\av}_j \in \Z^n$) \cite[Section 2]{buchmann96}. This is done this way because then the LLL-algorithm can be applied to $\hat{\mA}$ (the LLL-algorithm is generally only applied on integer (or rational) matrices). As a result, such an integer matrix $\hat{\mA}$ satisfies $\normt{\hat{\mA} - 2^q \mA} \leq \sqrt{n}/2$ \cite[Equation (1)]{buchmann96}.

In our analysis, we deviate from this. Going over the result of Buchmann and Kessler \cite{buchmann96}, one
can verify that the only requirements on $\hat{\mA}$ for their algorithm to work is:
\begin{itemize}
 \item Closeness to $2^q \mA$, i.e., $\normt{\hat{\mA} - 2^q \mA} \leq \sqrt{n}/2$ for some sufficiently large $q \in \N$, see \cite[Equation (1)]{buchmann96}.
 \item Being able to call LLL on $\hat{\mA}$; so, in fact $\hat{\mA}$ is allowed be a rational matrix, as long as the denominators
 not too large.
\end{itemize}
By choosing, for example, $\hat{\mA} \in \tfrac{1}{2} \Z^{k \times n}$, the inequality $\normt{\hat{\mA} - 2^q \mA} \leq \sqrt{n}/2$ can be easily achieved, using a sufficiently good approximation of $\mA$ to construct $\hat{\mA}$.

Summarizing, for our purposes it is not required that $\hat{\mA}$ is integral (rationals with small common denominators will suffice, too) nor that $\hat{\mA}$ is obtained by rounding $2^q \mA$. For this reason,
the following notation turns out to be useful.

\begin{notation}[Binary approximation]
\label{notation:approx-matrix}
A matrix $\mAa \in \Q^{k \times n}$ is called a binary approximation of $\mA \in \R^{k \times n}$
with precision $\eps_0$ if
$2^{q} \cdot \mAa \in \Z^{k \times n}$ for some $q \in \N$ and
$\normt{\mA - \mAa} \leq \eps_0$.
\end{notation}
Note that we can choose $q = \log_2(\eps_0 \cdot n)$ in above notation, as
that precision is sufficient to obtain the approximation $\normt{\mA - \mAa} \leq \eps_0$.
As a result, there always exists such a matrix $\mAa \in \Q^{k \times n}$ whose bit-size is bounded by $q^2 \cdot k \cdot n \cdot \log(\normt{\mA}) = \log_2(\eps_0 n)^2\cdot k\cdot n \cdot \log(\normt{\mA})$ (each coefficient of $\mAa$ is a rational number with numerator bounded by $2^q \cdot \normt{\mA}$ and denominator bounded by $2^q$).
Moreover, from any rational matrix $\tilde{\mA}'$ satisfying $\normt{\mA - \tilde{\mA}'} \leq \eps_0/2$, one can compute in polynomial time a matrix $\mAa$ with bit-size as above (polynomial in $n$, $k$, $\log(1/\eps_0)$ and such that $\log \normt{\mA}$) and $\normt{\mA - \mAa} \leq \eps_0$.

\subsection{The algorithm of Buchmann, Pohst and Kessler}

In the following lemma, we start by updating slightly the analysis of~\cite[Theorem 4.1]{buchmann96}, to make it compatible with our matrix norm, and also take into account the fact that the rank and determinant of our lattice are not known. In the later theorem (\Cref{theorem:buchmannpohstkessler}) we will apply this lemma twice, first to obtain just any basis of the lattice, and the second time to obtain a well-conditioned basis consisting of reasonably short vectors.

\begin{lemma} \label{lemma:buchmannpohst} Let $\Lambda \subseteq \Z^{n_1} \times \R^{n_2}$ be a lattice
with $\lambda_1(\Lambda) \geq \mu$ and $\rank(\Lambda) \leq r_0$.
Let $\mA$ be a $k \times (n_1 + n_2)$-matrix whose rows generate the lattice $\Lambda$ with $k \geq n_2 \geq 1$, let %
\[ C := 2^{4k} \cdot \left( \frac{r_0 \cdot \normt{\mA}}{\mu} \right)^{r_0+1}, \]
and let $\mAa$ be a binary approximation of $\mA$ satisfying $\normt{ \mAa - \mA }< \tfrac{1}{4} \mu \cdot C^{-1}$.

Then there exists an algorithm that, on input $\mAa,\mu,r_0$, computes a matrix $\mM \in \Z^{k \times r}$ with $r = \mbox{rank}(\Lambda)$
in time polynomial in $k$, $n_1$, $n_2$, $\log \normt{\mA}, \log(1/\mu)$ and the bit-size of its input,
and such that
\begin{enumerate}[(i)]
 \item $\mM \mA = \mB$ is a basis of $\Lambda$.
\item $\normt{\mM} \leq C/\sqrt{k} \leq C$,
\item $\normt{\mB}
\leq 4^k \cdot \normt{\mA}$.
\item Writing $\mBa = \mM \mAa$, we have $\normt{\mBa - \mB} \leq C \cdot \normt{\mAa - \mA}$.

\end{enumerate}
Additionally, if $\mA$ happened itself to be a basis of $\Lambda$, i.e., $k = r = \rank(\Lambda)$, we have, for all $j \in \{1,\ldots,r\}$,
\[ \normx{\bbv_j} \leq (\sqrt{rn_2} + 2) 2^{\frac{r-1}{2}} \cdot \lambda_j(\Lambda), \]
where $\bbv_j$ are the rows of $\mB = (\bbv_1,\ldots,\bbv_r)$.
\end{lemma}

\begin{proof}

\textbf{Constructing the matrix $\mAh$ from $\mAa$.}
Using the discussion after \Cref{notation:approx-matrix}, we can also replace the matrix $\mAa$ by a matrix with bounded bit-size (in time polynomial in the bit-size of $\mAa$). Hence, in the rest of the proof, we will assume without loss of generality that the bit-size of $\mAa$ is polynomial in $k$, $n_1$, $n_2$, $\log \normt{\mA}$, $\log(1/\mu)$, and $\mAa$ satisfies $\normt{\mA - \mAa} \leq 1/4 \cdot \mu \cdot C^{-1}$.

The algorithm first creates the matrix $\mAh$ defined by $\mAh_{ij} = \tfrac{1}{2} \lfloor 2^{q+1} \mAa_{ij} \rceil$ with $q = \lfloor \log_2(T) \rfloor$, where $\lfloor \cdot \rceil$ denotes rounding to the nearest integer, $\lfloor \cdot \rfloor$ denotes rounding down, and
\begin{equation} \label{eq:putT} T = \frac{2^{3k}}{\mu} \cdot \left( \frac{r_0 \cdot \normt{\mA}}{\mu} \right)^{r_0} = \frac{2^{-k}}{\mu} \left( \frac{r_0 \cdot \normt{\mA}}{\mu} \right)^{-1} \cdot C \leq C/(2\mu). \end{equation}
Note that computing $\mAh$ can be done in time polynomial in $q$ and in the bit-size of $\mAa$.
The entries of $\mAh$ lie in $\tfrac{1}{2} \Z$, and the matrix $\mAh$ satisfies
\[ \normt{\mAh - 2^q \mA} \leq \underbrace{\normt{\mAh - 2^q \mAa}}_{\leq \sqrt{n_2}/4} + \underbrace{\normt{2^q \mAa - 2^q \mA}}_{\leq 2^q \cdot \tfrac{1}{2} \cdot \mu/C \leq 1/4}  \leq \sqrt{n_2}/2. \] According to the discussion in \Cref{subsubsec:roundinginBKP} this is sufficient to apply the \bkp algorithm. %

\noindent \textbf{Applying \bkp to $[\mI \mid \mAh]$.}
This \bkp algorithm \cite[Theorem 4.1]{buchmann96}  essentially consists of applying LLL-reduction to the matrix $[ \mI \mid \mAh ]$ consisting of the
horizontal concatenation of $\mAh$ and $\mI$, the $k \times k$ identity matrix.

Put,\footnote{We replaced $\alpha$ in \cite[Proposition 3.2]{buchmann96} here by $\normt{\mA}$. As $\alpha \in \R_{>0}$ is required to be an upper bound on the $2$-norms of the rows of $\mA$, $\normt{\mA} = \max_{j} \|\gv_j\|$ clearly suffices.}
in order to satisfy the prerequisites of \cite[Proposition 3.2]{buchmann96},
\begin{equation} \label{eq:putlambda} \lambda :=
2^k \cdot \left( \frac{r_0 \normt{\mA}}{\mu} \right)^{r_0}
\geq ( \tfrac{1}{2} k\sqrt{n_2} + \sqrt{k}) \cdot \frac{\normt{\mA}^r}{\covol(\Lambda)}. \end{equation}
where $r = \rank(\Lambda) \leq r_0$; in the inequality we made use of Minkowski's inequality%
\footnote{Minkowski's inequality states that $1/\covol(\Lambda) \leq (r/\lambda_1(\Lambda))^r \leq (r/\mu)^r$ with $r = \rank(\Lambda)$. By the fact that $\normt{\mA}/\mu \geq 1$ and $r_0 \geq r$, we can conclude that
\[ \left( \frac{r_0 \normt{\mA}}{\mu} \right)^{r_0} \geq  \left( \frac{r \normt{\mA}}{\mu} \right)^{r} \geq \frac{\normt{\mA}^r}{\covol(\Lambda)} \]
} %
and the fact that $2^k \geq \tfrac{1}{2} k\sqrt{n_2} + \sqrt{k}$ (since we assumed $k \geq n_2$).

Note that $q = \lfloor \log_2(T) \rfloor$ satisfies  \cite[Equation~7]{buchmann96}, since we have
\begin{align*} 2^q  &\geq  
\tfrac{1}{2} \cdot  \frac{2^{3k}}{\mu} \cdot \left( \frac{r_0 \cdot \normt{\mA}}{\mu} \right)^{r_0} = \tfrac{1}{2} \cdot  \frac{2^{2k}}{\mu} \cdot \lambda  \\
& >  ( \sqrt{n_2 k} + 2) \cdot  2^{\frac{k-3}{2}} \cdot \frac{\lambda}{\mu},  \end{align*}
by \Cref{eq:putlambda} and because $\tfrac{1}{2} \cdot 2^{2k} > (\sqrt{n_2 k} + 2) 2^{\frac{k-3}{2}}$ for all $k \geq n_2 \geq 1$.

With these parameters $\lambda$ and $q$, one can apply
\cite[Theorem 4.1]{buchmann96}: LLL-reducing $[\mI ~~|~~ \mAh]$ yields an integral matrix $(\mR^\top, \mM^\top)^\top$ such that
\begin{equation}
\begin{bmatrix} \mR  \\ \mM \end{bmatrix} \cdot \begin{bmatrix} \mI  &  \mAh \end{bmatrix}  = \begin{bmatrix} \mR  &  \mR \mAh \\ \mM  & \mM \mAh   \end{bmatrix}  =  \begin{bmatrix} \mR &  \hat{\mE} \\ \mM & \hat{\mB}   \end{bmatrix} \label{eq:outputmatrix}   \end{equation}
for which holds $\mR\mA = \mathbf{0}$, and $\mM \mA = \mB$, a basis of $\Lambda$. Furthermore, every row $\hat{\mathbf{e}}$ of $ \hat{\mE}$ satisfies (\cite[Theorem 4.1, first equation on page 8]{buchmann96})
\[  \| \hat{\mathbf{e}} \| \leq 2^{\frac{k-1}{2}} \lambda \]
and every row $\hat{\mathbf{b}}$ of $\hat{\mB}$ satisfies  (\cite[Proposition 3.1]{buchmann96})
\[  \|  \hat{\mathbf{b}} \| > 2^{\frac{k-1}{2}} \lambda. \]
\noindent \textbf{Identifying $\mM$ and $\hat{\mB}$, proving (i).}
This gives means to distinguish where in the output matrix of \Cref{eq:outputmatrix} the approximate basis elements $\hat{\mB}$ are. Also, this gives us a way to identify $\mM$, as it shares the same rows as $\hat{\mB}$. Therefore,
we can compute an $\mM$ that satisfies part (i) of the lemma.

\noindent \textbf{Bounds on $\mM$, $\mB$ and $\normt{\mB - \tilde{\mB}}$.}
By \cite[Theorem 4.1, Equation 9, bottom inequality]{buchmann96}, we have,%
\footnote{In \cite[Eq. 10]{buchmann96}, the lattice successive minima $\lambda_j(L_r)$ of the lattice $L_r$ are used; this lattice $L_r$ is any lattice spanned by $r$ rows of $\mA$. So these successive minima $\lambda_j(L_r)$ are trivially bounded by $\alpha$, the uniform bound on the lengths of the rows of $\mA$. Additionally, we replaced $\alpha$ by $\normt{\mA}$ again.} since $2^q \leq T$ (see \Cref{eq:putT}),
\begin{align} \| \mv \| & \leq 2^{\frac{k-1}{2} + q + 1} \cdot \normt{\mA} =  2^{\frac{k+1}{2}} \cdot \normt{\mA} \cdot T  \\
& \leq 2^{\frac{k+1}{2}} \cdot \normt{\mA} \cdot  \frac{2^{3k}}{\mu} \cdot \left( \frac{r_0 \cdot \normt{\mA}}{\mu} \right)^{r_0}
 \\
& \leq \frac{2^{4k}}{\sqrt{k}} \cdot \left( \frac{r_0 \cdot \normt{\mA}}{\mu} \right)^{r_0+1} \\
& \leq \frac{C}{\sqrt{k}} \leq C ~~~\mbox{ for all rows $\mv$ of $\mM$,} \label{eq:boundonM}
\end{align}
where we used $2^{\frac{k+1}{2}} \leq 2^k/\sqrt{k}$ for all $k \in \Z_{>0}$. Additionally, by \cite[Theorem 4.1, Equation 10]{buchmann96}, we have,
\begin{align} \normx{\bbv} &\leq (\sqrt{kn_2} + 2) 2^{\frac{k-1}{2}} \cdot \normt{\mA} \\
& \leq 2^{2k} \cdot  \normt{\mA} ~~~\mbox{ for all rows $\bbv$ of $\mB$,}  %
\end{align}
which prove (ii) and (iii), as $\normt{\mM}$ and $\normt{\mB}$ are just the maximum $2$-norm of the rows of the respective matrices.

For part (iv) use the `almost submultiplicativity' of $\normt{\cdot}$ (\Cref{lemma:halfnorm}) and \Cref{eq:boundonM}, to obtain
\[ \normt{\mBa - \mB} =  \normt{\mM(\mAa - \mA)} \leq \sqrt{k} \cdot \normt{\mM} \cdot \normt{\mAa - \mA} \leq C \cdot\normt{\mAa - \mA} . \]
\noindent \textbf{LLL-alike bounds for $\mB$ when $\mA$ is itself a basis of $\Lambda$.}
In the case that $\mA$ itself happens to be a basis, sharper bounds on $\mB$ can be established, as can be deduced%
\footnote{We substitute $k = r$ (as $\mB$ consists of $r$ rows) and $L_r = \Lambda$. Here $L_r$ is defined as any lattice defined by any $r$ row vectors from $\mB$. As $\mB$ consists of only $r$ row vectors and is a basis of $\Lambda$, we have $L_r = \lambda$.} %
from \cite[Theorem 4.1, Equation 10]{buchmann96}.
Applying this bound yields, for all $j \in \{1,\ldots,r\}$,
\[ \normx{\bbv_j} \leq (\sqrt{kn_2} + 2) 2^{\frac{k-1}{2}} \cdot \lambda_j(\Lambda) \]
where $\bbv_j$ is the $j$-th row of $\mB$.

\noindent \textbf{Running time.}
For the running time, note that the most costly part of the algorithm is to run the LLL algorithm on the matrix $[\mI ~~|~~ \mAh]$. This can be done in polynomial time in the bit-size of the matrix, which is polynomial in $k$, $n_1$, $n_2$, $\log \normt{\mA}$ and $\log(1/\mu)$, as desired.
\end{proof}

\begin{theorem}[\bkp] \label{theorem:buchmannpohstkessler}
Let $\Lambda \subseteq \Z^{n_1} \times \R^{n_2}$ be a lattice
with $\lambda_1(\Lambda) \geq \mu$ and $\rank(\Lambda) \leq r_0$.
Let $\mA$ be a $k \times (n_1 + n_2)$-matrix whose rows generate the lattice $\Lambda$ with $k \geq n_2 \geq 1$, let
\[ C_0 = 2^{8k} \cdot \left(\frac{ r_0 \cdot 4^k \cdot \normt{\mA}}{ \mu} \right)^{2(r_0 + 1)} , \]
and let $\mAa$ be a binary approximation of $\mA$ that satisfies $\normt{\mAa - \mA} < \tfrac{1}{4} \cdot \mu \cdot C_0^{-1}$.

Then, there exists an algorithm that, on input $\mAa, \mu, r_0$,  computes in time polynomial in $k$, $n_1$, $n_2$, $\log \normt{\mA}$, $\log(1/\mu)$ and the bit-size of its input,
a matrix $\mN \in \Z^{r \times k}$ with $r = \rank(\Lambda)$ such that
\begin{enumerate}[(i)]
 \item $\normt{\mN} \leq C_0$
 \item $\mB = \mN \mA$ is a basis of $\Lambda$, whose rows satisfy, for $j \in \{1,\ldots,r\}$,
 \[ \| \bbv_j \| \leq  (\sqrt{rn_2} + 2) \cdot 2^{\frac{r-1}{2}} \cdot \lambda_j(\Lambda)  \]
 \item The approximate basis $\mBa = \mN \mAa$ satisfies $\normt{ \mBa - \mB} \leq  C_0 \cdot  \normt{\mAa - \mA}.$
\end{enumerate}
\end{theorem}

\begin{proof} We will apply the \bkp algorithm as in \Cref{lemma:buchmannpohst} twice: the first time
we apply it to $\mAa$ to obtain an approximate basis $\mBza$ of $\Lambda$, and the second time we apply it to $\mBza$  to obtain an approximate \emph{good} basis $\mBa$ of $\Lambda$. This basis $\mBa$ is called `good' because it is an approximation of a basis $\mB$ with a guarantee on the shortness of its basis vectors.

\vspace{2mm}
\noindent \textbf{First application of \Cref{lemma:buchmannpohst}.}

\noindent We have that $C_0 >  2^{4k} \cdot \left( \frac{r_0 \cdot \normt{\mA}}{\mu} \right)^{r_0+1} = C$ from \Cref{lemma:buchmannpohst}. Therefore, the approximation $\mAa$ of $\mA$ is sufficiently good to apply \Cref{lemma:buchmannpohst}, and compute $\mM \in \Z^{k \times r}$ such that
$\mM \mA = \mBz$ is a basis of $\Lambda$. Additionally, this $\mM$ and $\mBz$ satisfy
\begin{equation} \label{eq:boundonmandb} \normt{\mM} \leq C~~ \mbox{ and } ~~ \normt{\mBz}
\leq 4^k \cdot \normt{\mA}, \end{equation}
and $\mBza := \mM \mAa$ satisfies
\begin{equation} \label{eq:diffofb} \normt{\mBza - \mBz} \leq C \cdot \normt{\mAa - \mA} \leq C \cdot \mu \cdot C_0^{-1} \end{equation}
Moreover, the bit-size of $\mBza$ is polynomially bounded by the bit-size of $\mAa$, and the bit-size of $\mM$, the later being polynomially bounded by $k$, $n_1$, $n_2$, $\log \normt{\mA}$ and $\log(1/\mu)$.

\vspace{2mm}
\noindent \textbf{Second application of \Cref{lemma:buchmannpohst}.}

\noindent By \Cref{eq:diffofb} we deduce that $\normt{\mBza - \mBz} < \mu (C_0/C)^{-1}$; since $\mBza = \mM \mAa$, this matrix $\mBza$ is a binary approximation of $\mBz$.
In order to apply \Cref{lemma:buchmannpohst}, we need to show that
$C_0/C \geq 2^{4r} \cdot \left(\frac{ r \cdot \normt{\mBz}}{ \mu} \right)^{r + 1} =: C'$, i.e., that the precision of $\mBza$ is good enough . By the choice of $C_0$, the bound $4^k \cdot \normt{\mA} \geq \normt{\mBz}$ (\Cref{eq:boundonmandb}), $r_0 \geq r$ and $k \geq r$, we have
\begin{align} \frac{C_0}{C}  &= \frac{2^{8k} \cdot \left(\frac{ r_0 \cdot 4^k \cdot \normt{\mA}}{ \mu} \right)^{2(r_0 + 1)} }{2^{4k} \cdot \left(\frac{ r_0  \cdot \normt{\mA}}{ \mu} \right)^{r_0 + 1} } \geq 2^{4k} \cdot \left(\frac{ r_0 \cdot 4^k \cdot \normt{\mA}}{ \mu} \right)^{r_0 + 1} \\ & \geq 2^{4k} \cdot \left(\frac{ r_0 \cdot \normt{\mBz}}{ \mu} \right)^{r_0 + 1} \geq 2^{4r} \cdot \left(\frac{ r \cdot \normt{\mBz}}{ \mu} \right)^{r+ 1} =: C' \label{eq:cprime} \end{align}

Therefore, we can apply \Cref{lemma:buchmannpohst} again, but now with $k = r$ and $C'$ as in \Cref{eq:cprime}, and with input $(\mBza, \mu, r)$ (note that $r$ is now known, and we still have $\mu \leq \lambda_1(\Lambda)$).
This yields $\mM'$ such that $\mM' \mBz = \mB$ is a basis of $\Lambda$,
\begin{equation} \label{eq:boundonmprime} \normt{\mM'} \leq  C'/\sqrt{r} \end{equation}
and for all $j \in \{1,\ldots,r\}$ we have
 \[ \| \bbv'_j \| \leq  (\sqrt{rn_2} + 2) \cdot 2^{\frac{r-1}{2}} \cdot \lambda_j(\Lambda),  \]
 which proves (ii), putting $\mN = \mM' \mM$.

 For (i), we apply the `almost submultiplicativity' (\Cref{lemma:halfnorm}) to $\mN = \mM' \mM$, which yields, by \Cref{eq:boundonmprime} and \Cref{eq:cprime} ($C_0/C \geq C'$),
 \[ \normt{\mN} \leq \sqrt{r} \normt{\mM'} \normt{\mM} \leq C' \cdot C \leq C_0. \]
 For (iii), we note that, by the results of \Cref{lemma:buchmannpohst},
 \[ \normt{\mBa - \mB} \leq C' \cdot \normt{\mBza - \mBz} \leq C' \cdot C \cdot \normt{\mAa - \mA} \leq C_0 \cdot \normt{\mAa - \mA}\]
For the running time of the algorithm, note that we applied the algorithm of \Cref{lemma:buchmannpohst} twice, once with $(\mA, \mu, r_0)$ and once with $(\mBz, \mu, r)$.
Recall that the bit-size of $\mBz$ is polynomially bounded by $k$, $n_1$, $n_2$, $\log \normt{\mA}$ and $\log(1/\mu)$ and by the bit-size of $\mAa$. Using \Cref{lemma:buchmannpohst} gives the desired upper bound on the total running time of the algorithm.
\end{proof}

%% file: discretization/C-A-approxdeterminant.tex
\section{Computing the determinant of an approximate basis} \label{section:approximatedeterminant}
\subsection{Preliminaries}
In this section we will use the induced $2$-norm
on matrices, which is defined (for $m \times n$-matrices $\mA$) by the rule,
\[ \normmattwo{\mA} := \sup \{  \normx{\mA \mathbf{x}} ~|~  \mathbf{x} \in \C^n \mbox{ with }  \|\mathbf{x}\| = 1 \}.\]
It satisfies the following inequalities
\[ \normmattwo{\mA} \leq \normfrob{\mA} := (\sum_{i,j} |A_{ij}|^2)^{1/2},\]
where $\normfrob{\mA}$ is the Frobenius norm. As a result,
a $m \times n$ matrix $\mE$ with $E_{ij} \leq \eps$ for all $i,j$,
satisfies
\[ \normmattwo{ \mE} \leq \normfrob{\mE} \leq \sqrt{nm} \cdot  \eps. \]
We will use the fact that $\normmattwo{\mA}$ corresponds to the largest singular value of $\mA$, which gives us
\begin{align}
\label{eq:norm-2-trace}
\normmattwo{\mA^{\top}} = \normmattwo{\mA}.
\end{align}
\subsection{The main result of this section}
This section is about the following problem. Suppose we have
a rank-$n$ lattice $\Lambda \subseteq \R^m$ of which we can approximate
the determinant, say $D \approx \covol(\Lambda)$. Suppose furthermore
that we have an \emph{approximate basis} $\mBa' \approx \mB'$ of a \emph{sub}lattice
$\Lambda' \subseteq \Lambda$ (with the same rank $n$). The task is to decide whether $\Lambda' = \Lambda$, using
only the approximations $\mBa'$ and $D$.

A straightforward way is to compute $\sqrt{\det((\mBa')^{\top} \mBa')} \approx \covol(\Lambda')$ and compare
it with $D \approx \covol(\Lambda)$. If the approximations are sufficiently good and
$D$ and $\sqrt{\det((\mBa')^{\top} \mBa')}$ are sufficiently close, we can conclude $\Lambda' = \Lambda$.
If they are not close, we can conclude $\Lambda' \neq \Lambda$.

This section treats how well $\mBa'$ needs to approximate $\mB'$ (the exact basis of $\Lambda'$)
and how well $D$ needs to approximate $\covol(\Lambda)$ in order to make this reasoning correct.
The main result of this section is the following statement.

\begin{theorem} \label{theorem:determinantapproximation} Let $\Lambda \subset \R^m$ be a rank-$n$ lattice (for some $m \geq n \geq 1$). Let $\mB' \in \R^{m \times n}$ be a basis of a rank-$n$ sublattice $\Lambda' \subseteq \Lambda \subseteq \R^m$ %
and let $\mBa' \in \Q^{m \times n}$ be an approximation of $\mB'$ where the entry-wise error is less than
\[ \eps = 2^{-6} \cdot n^{-(n + 4)} \cdot m^{-1} \cdot \left(\prod_{j} \frac{\|\bbv_j'\|}{\lambda_j(\Lambda')}\right)^{-2} \cdot \lambda_1(\Lambda')^2 / \normmattwo{\mB'}, \]
and let $D \in [\tfrac{3}{4},\tfrac{5}{4}] \covol(\Lambda) \cap \Q$.

Then
we can decide whether $\Lambda' = \Lambda$ or not, using only $\mBa'$ and $D$,
in time polynomial in %
the bit-sizes of $\mBa'$ and $D$.
\end{theorem}

\subsection{Some preliminary lemmas.} Before proving \Cref{theorem:determinantapproximation}, we will prove some intermediary lemmas. The first lemma gives an upper bound on the size of the inverse of a square matrix $\mA$ that is the basis of a lattice (the upper bound depends on some quantities related to $\mA$, which quantify how well-conditioned $\mA$ is). This will be used in the second lemma to prove that if two rectangular matrices $\mB$ and $\mBa$ are close coefficient-wise, then the determinants of the square matrices $\mB^{\top} \mB$ and $\mBa^{\top} \mBa$ are close too. This second lemma will be central in the proof of \Cref{theorem:determinantapproximation}.

\begin{lemma} \label{lemma:wellconditioned} Suppose $\mA = (\mathbf{a}_1,\ldots,\mathbf{a}_n) \in \R^{n \times n}$ is a square real matrix and a basis of a lattice $\Lambda \subset \R^n$.
Then
\[ \normmattwo{ \mA^{-1} } \leq n^{n/2+1} \cdot \lambda_1(\Lambda)^{-1} \cdot \left(\prod_{j =1}^n \frac{\|\mathbf{a}_j\|}{\lambda_j(\Lambda)}\right).\]
\end{lemma}

\begin{proof} 
For $j \in \{1,\dots,n\}$, define $C_j = \frac{\|\mathbf{a}_j\|}{\lambda_j(\Lambda)}$.
We have that $\mA^{-1} = \frac{1}{\det \mA} \mbox{adj}(\mA)$. We have that $\mbox{adj}(\mA)_{ij}$ is
defined by the determinant of the minor of $\mA$ where the $i$-th row and $j$-th column are deleted. By the
Hadamard bound and subsequently by Minkowski's second theorem (see, e.g.,~\cite[Theorem 1.5]{MGbook}),
\begin{align*} |\mbox{adj}(\mA)_{ij}| &\leq \prod_{k \neq i} \|\mathbf{a}_k\| = \prod_{k \neq i} C_k \cdot \lambda_k(\Lambda) \\
 & \leq n^{n/2}  \cdot (\prod_{k \neq i} C_k) \cdot  \covol(\Lambda)/\lambda_i(\Lambda)  \leq  n^{n/2} (\prod_{k=1}^n C_k) \cdot \covol(\Lambda)/\lambda_1(\Lambda).
\end{align*}
Therefore, by observing that $\det(\mA) = \covol(\Lambda)$, we obtain
\[ \normmattwo{ \mA^{-1} } \leq \normfrob{\mA^{-1}} \leq \frac{1}{\det(\mA)} \cdot n \cdot \max_{ij} |\mbox{adj}(\mA)_{ij}| \leq n^{n/2+1} \cdot (\prod_{k =1}^n C_k)/\lambda_1(\Lambda) \]

\end{proof}

\begin{lemma} \label{lemma:detapprox} Let $\mB \in \R^{m \times n}$ be a basis of a rank-$n$ lattice $\Lambda \subset \R^m$ (for some $m \geq n \geq 1$),
and let $\mBa$ be an approximation of $\mB$ where the entry-wise error satisfies s
$|\mB_{ij} - \mBa_{ij}| \leq \eps$ for every $i \in \{1,\ldots,m\}, j \in \{1,\ldots,n\}$, where
\begin{equation} 
\label{eq:epsiloninst} 
\eps = 2^{-6} \cdot n^{-(n + 4)} \cdot m^{-1} \cdot \left(\prod_{j} \frac{\|\bbv_j\|}{\lambda_j(\Lambda)}\right)^{-2} \cdot \lambda_1(\Lambda)^2 / \normmattwo{\mB}. \end{equation}
Then
\[ \det(\mBa^{\top} \mBa) \in [\tfrac{7}{8},\tfrac{9}{8}] \cdot \det(\mB^{\top} \mB). \]
\end{lemma}

\begin{proof} We have, by  \cite[Corollary 2.14]{ipsenrehman08}, for $n \times n$ complex matrices $\mA,\mE$, where
$\mA$ is non-singular,
\begin{align} \nonumber  \frac{|\det(\mA + \mE) - \det(\mA) |}{|\det(\mA)|}  &\leq  (\normmattwo{\mA^{-1}} \normmattwo{\mE } + 1)^n - 1 \\ & \leq 2n\normmattwo{\mA^{-1}} \normmattwo{\mE} ~~~\mbox{ if $\normmattwo{\mA^{-1}} \normmattwo{\mE} \leq 1/n$} \label{eq:determinantclose} . \end{align}
The last inequality here easily follows from $(x + 1)^n - 1 \leq (e^{x})^n -1 \leq 2nx$ for $nx \leq 1$.

We instantiate this inequality with $\mA = \mB^{\top} \mB$ and $\mE = \mBa^{\top} \mBa - \mB^{\top} \mB$, so that $\mA + \mE = \mBa^{\top} \mBa$ (note that $\mA$ is non-singular since $\mB$ is of rank $n$, so the instantiation is valid). In order to obtain a meaningful upper bound, let us compute bounds on the two quantities $\normmattwo{\mA^{-1}}$ and $\normmattwo{\mE}$.

We can rewrite $\mE$ as
\begin{align*}
\mE &= \big(\mB + (\mBa - \mB) \big)^{\top} \cdot \big(\mB + (\mBa - \mB) \big) - \mB^{\top} \mB \\
&= \mB^{\top} \cdot (\mBa - \mB) + (\mBa - \mB)^{\top} \cdot \mB + (\mBa - \mB)^{\top} \cdot (\mBa - \mB).
\end{align*}
Using the sub-multiplicativity of the induced $2$-norm and~\Cref{eq:norm-2-trace}, we obtain
\[ \normmattwo{\mE} \leq 2 \cdot \normmattwo{\mB} \cdot \normmattwo{\mBa - \mB} + \normmattwo{\mBa - \mB}^2.\]
By assumption on the entry-wise error, we know that $\normmattwo{\mBa - \mB} \leq \normfrob{\mBa - \mB} \leq \sqrt{nm}  \cdot \eps$. Moreover, it holds that $\normmattwo{ \mB} \geq \normx{\bbv_1} \geq \lambda_1(\Lambda) \geq \lambda_1(\Lambda)^2/\normmattwo{\mB} \geq \eps$ (where $\bbv_1$ is the first column vector of $\mB$, and see \Cref{eq:epsiloninst}). This yields
\begin{align}
\label{eq:norm-E}
\normmattwo{\mE}  \leq 3 \cdot n \cdot m \cdot \eps \cdot \normmattwo{\mB}.
\end{align}

Let us now compute an upper bound on $\normmattwo{\mA^{-1}} = \normmattwo{(\mB^{\top} \mB)^{-1}}$. Let $(\mQ, \mathbf{R})$ be the QR-decomposition of $\mB$, with $\mQ \in \R^{m \times m}$ orthogonal and $\mathbf R \in \R^{m \times n}$ upper triangular, i.e., $\mB = \mQ \mathbf{R}$. Since $m \geq n$, we have $\mathbf R = \begin{pmatrix}
\mC \\ 0 \end{pmatrix}$ for some invertible matrix $\mC \in \R^{n \times n}$. It then holds, by orthogonality of $\mQ$, that $\mC^{\top} \mC = \mB^{\top} \mB$; that $\normx{\bbv_i} = \normx{\mathbf{c}_i}$ for all $1 \leq i \leq n$; and that $\lambda_i(\Lambda) = \lambda_i(\Lambda_\mC)$ for all $i$'s, where $\Lambda_\mC$ is the lattice spanned by the columns of $\mC$.
This gives us
\begin{align}
\normmattwo{\mA^{-1}} = \normmattwo{(\mB^{\top} \mB)^{-1} } &= \normmattwo{ (\mC^{\top} \mC)^{-1} } \nonumber \\
&\leq \normmattwo{(\mC^{-1})^{\top}} \cdot \normmattwo{\mC^{-1}} \nonumber \\
&= \normmattwo{\mC^{-1}}^2 \nonumber \\
& \leq \left(n^{n/2+1} \cdot \lambda_1(\Lambda_\mC)^{-1} \cdot \left(\prod_{j =1}^n \frac{\|\mathbf c_j\|}{\lambda_j(\Lambda_\mC)}\right) \right)^2 \nonumber \\
& = n^{n+2} \cdot \lambda_1(\Lambda)^{-2} \cdot \left(\prod_{j =1}^n \frac{\|\bbv_j\|}{\lambda_j(\Lambda)}\right)^2, \label{eq:boundmA}
\end{align}
where we used~\Cref{eq:norm-2-trace} for the third line and~\Cref{lemma:wellconditioned} on $\mC$ (which is $n$ by~$n$ as required) for the fourth line.

Combining the bounds \Cref{eq:determinantclose,eq:norm-E,eq:boundmA} and observing that $\normmattwo{\mA^{-1}} \cdot \normmattwo{\mE} \leq 1/n$ by choice of~$\eps$, we obtain
\begin{align*}
\frac{|\det(\mBa^{\top} \mBa) - \det(\mB^{\top} \mB) |}{|\det(\mB^{\top} \mB)|}  &\leq  2n \cdot \normmattwo{\mA^{-1}} \cdot \normmattwo{\mE}
\\ & \leq  2n \cdot \underbrace{ n^{n+2} \cdot \lambda_1(\Lambda)^{-2} \cdot \left(\prod_{j =1}^n \frac{\|\bbv_j\|}{\lambda_j(\Lambda)}\right)^2}_{\normmattwo{\mA^{-1}}} \cdot \underbrace{ 3 \cdot n \cdot m \cdot \eps \cdot \normmattwo{\mB} }_{\normmattwo{\mE}} \\ & \leq 1/8,
\end{align*}
which precisely means that $\det(\mBa^{\top} \mBa) \in [\tfrac{7}{8},\tfrac{9}{8}] \cdot \det(\mB^{\top} \mB)$.
\end{proof}

\subsection{Proving the main result.} We are now equipped to prove \Cref{theorem:determinantapproximation}. The proof is quite short once we have \Cref{lemma:detapprox}.
\begin{proof}[Proof of \Cref{theorem:determinantapproximation}]
Let $\ell = [\Lambda:\Lambda']$ be the index of $\Lambda'$ in $\Lambda$. Then $\ell \geq 1$ is an integer, and we know that $\det\big((\mB')^{\top} \mB'\big) / \covol(\Lambda)^2 = \ell^2$. Our objective is to determine whether $\ell^2 = 1$ (i.e., $\Lambda' = \Lambda$) or $\ell^2 \geq 4$ (i.e., $\Lambda' \subsetneq \Lambda$).

By the choice of $\eps$, we can apply~\Cref{lemma:detapprox} to $\mBa'$, which proves that $\det\big((\mBa')^{\top} \mBa'\big) \in [\tfrac{7}{8},\tfrac{9}{8}] \cdot \det\big((\mB')^{\top} \mB'\big)$. Recall that by assumption, $D \in [\tfrac{3}{4},\tfrac{5}{4}] \cdot \covol(\Lambda)$. From these two inequalities, we obtain that

\[ \frac{\det\big((\mBa')^{\top} \mBa'\big)}{D^2} \in [\tfrac{16}{25} \cdot \tfrac{7}{8},\tfrac{16}{9}\cdot \tfrac{9}{8}] \cdot \frac{\det\big((\mB')^{\top} \mB'\big) }{ \covol(\Lambda)^2} = [\tfrac{14}{25},2] \cdot \ell^2.\]

If $\Lambda = \Lambda'$, then $\ell = 1$ and we have $\frac{\det\big((\mBa')^{\top} \mBa'\big)}{D^2} \leq 2$. Otherwise, $\ell \geq 2$ and we have $\frac{\det\big((\mBa')^{\top} \mBa'\big)}{D^2} > 2.2$. We can then distinguish the two cases by computing the quantity $\det\big((\mBa')^{\top} \mBa'\big)$ and checking whether is it $\leq 2 \cdot D^2$ or $> 2 \cdot D^2$.

For the running time, we only need to compute the determinant of the rational matrix $(\mBa')^{\top} \mBa'$, and then compare it with the rational number $2 \cdot D^2$. This can be done in polynomial time in the bit-sizes of the $\mBa'$ and $D$ (see, e.g.,~\cite{KaltofenVillard} for an algorithm that computes the determinant of integer matrices in polynomial time).
\end{proof}

%% file: discretization/C-A-postprocessing.tex
\section{Post-processing in the \texorpdfstring{$\mS$}{\letterSunits}-unit computation} \label{section:postprocessingappendix}

\postprocessingtheorem*

\begin{proof}
\textbf{Part I: Setting up the log-matrix and assembling properties of the log-$\letterSunits$-unit lattice.}
Define $A = \max_j \normx{\logs(\eta_j)}$ and $\mu = \kesslerformulainline$. We know from \Cref{le:lower-bound-first-minimum-log-S-unit} that $\mu \leq \lambda_1(\logsunits)$. Note that this lower bound is automatically also a lower bound for the first minimum of any sublattice of $\logsunits$.

Consider the exact matrix $\mG$ where the rows consist of $\logs(\eta_j)$ with $j \in \{1, \ldots,k\}$,
and compute an approximation $\mGa \in \Q^{k \times (n+|\mT|)}$, with $\normx{\mGa_j - \mG_j} \leq \varepsilon$
for each row, where we put %
\begin{align} 
 \varepsilon &=  \mu \cdot 
 \underbrace{2^{-8k} \cdot \left(\frac{ k \cdot 4^k \cdot A}{ \mu} \right)^{-2(k + 1)}}_{\substack{ \mbox{\normalsize{$\leq C_0^{-1}$}} \mbox{ as in \Cref{theorem:buchmannpohstkessler}} \\ \mbox{(\bkp)} }} 
 \cdot \underbrace{2^{-6} \cdot (n+|\mT|)^{-(k + 6)} \cdot 2^{-(2k+1)(k+2)} \cdot A^{-1} \cdot \mu}_{\substack{\mbox{Determinant} \\ \mbox{ approximation}}}   \nonumber
 \\ & \leq \tfrac{1}{4} \cdot \mu \cdot C_0^{-1}  \label{eq:precisioninst} 
\end{align}
Note that from the exact representation of the $\eta_i$'s, one can compute approximations of the $\logs(\eta_i)$'s with arbitrary precision $\eps$ in time polynomial in $\size(\eta_i)$, in $\log( \|\logs(\eta_i)\|)$ and in $\log(1/\eps)$. From our choice of $\eps$, this is polynomial in $k$, $\log|\dcrk|$, $\log\big(\max_j \normx{\logs(\eta_j)} \big) = \log(A)$, $|\mT|$, and $\max_i(\size(\eta_i))$ as expected. Moreover, the bit-size of the resulting matrix $\mGa$ is also polynomial in these quantities.

The quantity $\eps$ is chosen such that the left part of the product in \Cref{eq:precisioninst} counteracts the loss in precision due to the \bkp algorithm (\Cref{theorem:buchmannpohstkessler}), whereas the right part is the minimum precision required to make the determinant approximation work (as in \Cref{theorem:determinantapproximation}).

\medskip
\noindent\textbf{Part II: Applying the \bkp algorithm.}
We apply \Cref{theorem:buchmannpohstkessler} to $\mGa$ with $r_0 = k$, $\mu \leq \lambda_1(\logs(\langle G \rangle))$ (as defined above)
and we note that, by definition, $\normt{\mG} \leq A$ which implies $\varepsilon < \tfrac{1}{4} \cdot \mu \cdot C_0^{-1}$ (for $C_0$ as in \Cref{theorem:buchmannpohstkessler}), so the prerequisites of the theorem are satisfied.

So, the algorithm of \Cref{theorem:buchmannpohstkessler} outputs an $\mN$ such that $\mB := \mN \mG$ is
a basis of $\logs(\langle G \rangle)$ and $\mBa := \mN \mGa$ satisfies
\begin{equation} \label{eq:basisbound} \normt{\mB - \mBa} \leq C_0 \cdot \eps \leq 2^{-6} \cdot (n+|\mT|)^{-(k + 6)} \cdot 2^{-(2k+1)(k+2)} \cdot A^{-1} \cdot \mu^2  \end{equation}

Additionally, the rows $\bbv_j$ of $\mB$ satisfy
\begin{equation} \label{eq:basisllllike}  \| \bbv_j \| \leq  (\sqrt{rn} + 2) \cdot 2^{\frac{r-1}{2}} \cdot \lambda_j(\logs(\langle G \rangle)) ,  \end{equation}
where $r$ is the rank of the lattice $\logs(\langle G \rangle)$.

Note that the pair $(\mN,G)$ encodes a fundamental system of units of $\logs(\langle G \rangle)$ in compact representation.

The \bkp  algorithm computes the rank $r$ of $\logs(\langle G \rangle)$; we proceed
only to Part III if this rank equals the rank of $\logsunits$, namely $\rem + \cem + |\mS| -1$. It thus remains to decide whether $\logs(\langle G \rangle) = \logsunits$ or not. This is done by approximating determinants.

\medskip
\noindent
\textbf{Part III: Computing the determinant.}
In order to apply \Cref{theorem:determinantapproximation} to the matrix $\mBa$ obtained, we need upper bounds on the quantities $\prod_{j = 1}^r  \frac{\| \bbv_j \|}{\lambda_j(\logs(\langle G \rangle))}$ and $\normmattwo{\mB}$, and a lower bound on the quantity $\lambda_1(\logs(\langle G \rangle))$.
Since $\logs(\langle G \rangle)$ is a sublattice of $\logsunits$, we have already seen that $\lambda_1(\logs(\langle G \rangle)) \geq \mu$.

By \Cref{eq:basisllllike}, the output basis of the \bkp algorithm satisfies
\[  \frac{\| \bbv_j \|}{\lambda_j(\logs(\langle G \rangle))} \leq (\sqrt{rn} + 2) \cdot 2^{\frac{r-1}{2}} \leq 2^{k+2},\]
where the last inequality holds because $k \geq r$ and $ r = n_\R+n_\C-1 \geq n/2 -1$ (since we proceeded to Part III).
So, $\prod_{j = 1}^r  \frac{\| \bbv_j \|}{\lambda_j(\logs(\langle G \rangle))} \leq 2^{k(k+2)}$.

Finally, we observe that
\[ \normmattwo{\mB} \leq r \cdot \max_j \normx{\bbv_j} \leq r \cdot \max_j\Big(\frac{\| \bbv_j \|}{\lambda_j(\logs(\langle G \rangle))}\Big) \cdot \lambda_r(\logs(\langle G \rangle))
 \leq r \cdot 2^{k+2} \cdot A,\]
where  for the last inequality we used the bound on $\frac{\| \bbv_j \|}{\lambda_j(\logs(\langle G \rangle))}$ computed above, and the fact that $A = \max_j \normx{\logs(\eta_j)}$ is an upper bound on the last minimum of the lattice generated by the $\eta_j$'s.

Combining these bounds with~\Cref{eq:basisbound} (and using the fact that $k \geq r$ and $n+|\mT| \geq r$), we obtain, using $(n + |\mS|)^{-(k+6)} \cdot r  \leq  (n + |\mS|)^{-1} \cdot (n + |\mS|)^{-(k+5)} \cdot r \leq (n + |\mS|)^{-1} \cdot r^{-(r+4)}$,
\begin{align*}
&\normt{\mB - \mBa} \\
&\ \ \leq 2^{-6} \cdot (n+|\mT|)^{-(k + 6)} \cdot \textcolor{blue}{2^{-(2k+1)(k+2)}} \cdot A^{-1} \cdot \mu^2 \\
&\ \ \leq 2^{-6} \cdot (n+|\mT|)^{-(k + 6)} \cdot \underbrace{\textcolor{blue}{2^{-2k(k+2)}}}_{\leq \left( \prod_{j} \|\mathbf{b}_j\|/\lambda_j \right)^{-2}}  \cdot \textcolor{blue}{r} \cdot \underbrace{\textcolor{blue}{r^{-1}  \cdot 2^{-(k+2)}} \cdot A^{-1}}_{\leq \normmattwo{\mB}^{-1}} \cdot \underbrace{\mu^2}_{\leq \lambda_1^2} \\
&\ \ \leq 2^{-6} \cdot r^{-(r + 4)} \cdot (n+|\mT|)^{-1} \cdot \left(\prod_{j} \frac{\|\bbv_j\|}{\lambda_j(\logs(\langle G \rangle))}\right)^{-2} \cdot \lambda_1(\logs(\langle G \rangle))^2 / \normmattwo{\mB},
\end{align*}
as needed to apply~\Cref{theorem:determinantapproximation} to the matrix $\mBa$. We conclude that from the knowledge of $\mBa$ and a rational number $D \in [\tfrac{3}{4},\tfrac{5}{4}] \cdot \covol(\logsunits)$, we can decide whether $\logs(\langle G \rangle) = \logsunits$ or not.

\medskip
\noindent\textbf{Part IV: The running time.}
For the running time, we go through each part. For Part I, we have already seen that $\mGa$ can be computed in time polynomial in the desired quantities, and that its bit-size is also polynomial.
For Part II, note that we applied the \bkp algorithm as in \Cref{theorem:buchmannpohstkessler},
which takes time
polynomial in $k$, $(n+|\mS|)$, $\log(\normt{\mG})$, $\log(1/\mu)$ and the bit-size of $\mGa$. By definition of $\mu$, this is polynomial in the desired quantities. Moreover, the bit-size of the output matrix $\mBa$ is also polynomial in those same quantities.

For Part III, we first compute a rational $D \in [\tfrac{3}{4}, \tfrac{5}{4}] \cdot \covol(\logsunits)$, which can be done in time polynomial in $\log |\dcrk|$ from \Cref{prop:approx-rho} and \Cref{lemma:det-log-S-unit}. We then apply \Cref{theorem:determinantapproximation} to the matrix $\mBa$ and the rational number $D$. This takes time polynomial in the bit-sizes of $\mBa$ and $D$, which is polynomial in all the desired quantities.
This concludes the proof.
\end{proof}

%% file: discretization/C-A-BKZ.tex
\newcommand{\mU}{\mathbf{U}}
\newcommand{\der}{dually exponentially reduced\xspace}

\section{Lattice preliminaries}
\label{section:latprelim}
In this section, we succinctly treat lattice preliminaries
required for understanding the next two sections (\Cref{section:BKZinteger,section:BKZapproximate})
about a BKZ-variant by Hanrot, Stehl\'e and Pujol. This includes
the Gram-Schmidt orthogonalization of a basis, the potential
of a basis, HKZ-reduced bases, Banaszczyk's transference theorem and
the definition of a dually exponentially reduced basis.

\subsection{Gram-Schmidt orthogonalization in \texorpdfstring{$\Z$}{Z}-bases}

\begin{notation} For a basis $\mB \in \R^{n \times n}$ we denote its Gram-Schmidt orthogonalization by $\gmB = (\gbb_1,\ldots,\gbb_n)$. We denote the dual basis of $\mB$ by $\mD = (\mB^{-1})^\top =: \mB^{-\top} = (\bd_1,\ldots,\bd_n) = (\dbb_1,\ldots,\dbb_n)$. If $\mB$ is a basis of $\Lambda$, then $\mD$ is a basis of the dual lattice $\dlat$ of $\Lambda$.
\end{notation}

\begin{notation} By a tilde, we generally denote an \emph{approximation}. In this section, often a real-valued $\mB$ is approximated by a basis $\tmB$ with rational coefficients. Likewise, an approximation of the dual basis $\mD$ is denoted by $\tmD$.
\end{notation}

\begin{notation} For a basis $\mB = (\bb_1,\ldots,\bb_n) \in \R^{m \times n}$ we denote by $\mB_{[j:k]}$ the basis $(\pi_{j-1}(\bb_j),\ldots,\pi_{j-1}(\bb_k)) \in \R^{m \times (k+1-j)}$, where $\pi_{j-1} = \pi_{(\bb_1,\ldots,\bb_{j-1})^\perp}$ is the projection orthogonal to the linear subspace generated by $(\bb_1,\ldots,\bb_{j-1})$.
Note that $\mB_{[1:k]}$ is just $(\bb_1,\ldots,\bb_k)$.

\end{notation}

\begin{definition}[Potential of a basis]
For a basis $\mB \in \R^{m \times n}$ we denote the potential by
\[ P(\mB) = \prod_{j = 1}^n \|\gbb_j\|^{n+1-j} = \prod_{j = 1}^n \covol(\lat(\bb_1,\ldots,\bb_j)), \]
where $\lat(\bb_1,\ldots,\bb_j)$ is the lattice generated by $(\bb_1,\ldots,\bb_j)$.
\end{definition}

\begin{lemma} \label{lemma:gsfraction} Let $\mB = (\bb_1,\ldots,\bb_n) \in \Z^{m \times n}$ be a basis of the lattice $\Lambda$,
then $\covol(\lat(\bb_1,\ldots,\bb_{j-1}))^2 \cdot \mB_{[j:n]} \in \Z^m$.
In particular,
$\covol(\lat(\bb_1,\ldots,\bb_{j-1}))^2 \cdot \gbb_j \in \Z^m$ and hence $P(\mB)^2 \cdot \gmB \in \Z^{m \times n}$.
\end{lemma}
\begin{proof} By the determinant formula for orthonormal projection \cite[Vol.~1, Chapter IX, \textsection 4, Equation (21)]{Gantmacher59}
we have that, for $D_j := \det(\mB_{[1:j-1]}^\top \mB_{[1:j-1]})$, $D_j \mB_{[j:n]} \in \Z^m$.
In particular $D_j \gbb_j \in \Z^m$. Note that $D_j = \covol(\lat(\bb_1,\ldots,\bb_{j-1}))^2$ by definition.

As a result, we certainly deduce that for $P(\mB)^2 = \left(\prod_{j = 2}^n D_j\right) \cdot \covol(\lat(\mB))^2$ we have $P(\mB)^2 \cdot \gbb_k \in \Z^m$ for any $k \in \{1,\ldots,n\}$ (here, we use that $D_j \in \Z$ since $\mB \in \Z^{m \times n}$). That is, $P(\mB)^2 \cdot \gmB \in \Z^{m \times n}$.
\end{proof}

\begin{lemma} \label{lemma:gsupper} Let $\mB = (\bb_1,\ldots,\bb_n) \in \Z^{m \times n}$ be a basis of the lattice $\Lambda$,
then $\| \gbb_j \| \leq \covol(\lat(\bb_1,\ldots,\bb_{j})) \leq P(\mB)$.

\end{lemma}
\begin{proof} 
We know that $\covol(\lat(\bb_1,\ldots,\bb_{j})) = \covol(\lat(\bb_1,\ldots,\bb_{j-1})) \cdot \|\gbb_j\|$. Moreover, since $\mB$ is integral, then $\covol(\lat(\bb_1,\ldots,\bb_{j-1}))^2 \in \Z$, and so it is $\geq 1$ (because $\mB$ is a basis, so its vectors are linearly independent). The inequality $\| \gbb_j \| \leq \covol(\lat(\bb_1,\ldots,\bb_{j}))$ follows.
\end{proof}

\begin{definition} We call a basis $\mB = (\bb_1,\ldots,\bb_n) \in \R^{m \times n}$
size-reduced if for all $0 < i<j \leq n$
\[ | \langle \gbb_i, \bb_j \rangle | \leq \tfrac{1}{2} \| \gbb_i \|^2. \]
\end{definition}

\begin{lemma} \label{lemma:sizereducedbound} Let $\mB \in \Z^{m \times n}$ be a size reduced basis of a lattice $\Lambda$. Then, for all $j \in \{1,\ldots,n\}$,
\[\|\bb_j \| \leq \tfrac{\sqrt{n}}{2} \max_i \|\gbb_i\|. \]
\end{lemma}
\begin{proof} Size reduced means that any basis vector $\bb_j$ can be written as $\sum_{i = 1}^n c_{ij} \gbb_i$ for $|c_{ij}| \leq 1/2$. Hence, by the Pythagorean theorem, we have
\[ \| \bb_j \|^2 = \sum_{i = 1}^n c_{ij}^2 \|\gbb_i\|^2 \leq \tfrac{n}{4} \max_{j} \|\gbb_i\|^2,\]
 from which the claim follows.
\end{proof}

\begin{corollary} \label{cor:boundedsizepotential} Let $\mB \in \Z^{m \times n}$ be a size reduced basis of a lattice $\Lambda$. Then each coefficient in $\mB$, $\gmB$ and $\mB_{[j:k]}$ (for $k>j \in \N$) can be represented in $5\log_2 (n \cdot P(\mB))+3$ bits.
\end{corollary}
\begin{proof} Each vector in $\mB$ and $\gmB$ is bounded in Euclidean norm by $\sqrt{n} \cdot P(\mB)/2$, which is also an upper bound on the absolute value of each coefficient. Moreover, $\mB$ is integral and the denominators occurring in $\gmB$ can be maximally $P(\mB)^2$. Every coefficient of $\mB$ and $\gmB$ is then a rational number which can be represented by a fraction $a/b$ with $a$, $b$ integers, $|a| \leq \sqrt{n} \cdot P(\mB)^3/2$ and $0 < b \leq P(\mB)^2$. They can be represented by respectively by $\lceil\log_2(\sqrt{n} \cdot P(\mB)^3+1)\rceil$ and $\lceil\log_2(P(\mB)^2)\rceil$ bits.

Precisely the same reasoning can be used for $\mB_{[i:j]}$, since $\covol(\lat(\bb_1,\ldots,\bb_{j-1}))^2 \cdot \mB_{[j:n]} \in \Z^m$ by \Cref{lemma:gsfraction} (and $\mB_{[j:k]}$ consists of just the first $k-j+1$ vectors of $\mB_{[j:n]}$). 
\end{proof}

\subsection{HKZ reduction}

\begin{definition}[HKZ-reduced basis] A basis $\mB \in \R^{m \times n}$ of a lattice $\Lambda$ is called HKZ-reduced if it is size reduced and if
\[ \| \gbb_j \| = \lambda_1( \mB_{[j:n]} ) \mbox{ for all } j \in \{1,\ldots,n\}, \]
 i.e., if the Gram-Schmidt vector $\gbb_j$ attains the minimum of the projected lattice $\pi_{(\bb_1,\ldots,\bb_{j-1})^\perp} (\Lambda)$ for every $j \in \{1,\ldots,n\}$.
\end{definition}

\begin{lemma} \label{lemma:HKZpotential} Let $\mB$ be a HKZ-reduced basis of $\Lambda$, then
\[ P(\mB) \leq n^{n^2} \min_{\mC} P(\mC), \]
where the minimum over the potential $P( \cdot )$ is over all bases $\mC$ of $\Lambda$.
\end{lemma}
\begin{proof}
The HKZ reduced basis $\mB$ satisfies $\| \gbb_i \| \leq \lambda_i(\Lambda)^2$ (see, e.g.,~\cite[p.38]{LLLAlgorithm2009}), hence
\begin{align} P(\mB) &= \prod_{j = 1}^n (\| \gbb_1\| \cdots \|\gbb_j\| )  \leq \prod_{j = 1}^n (\lambda_1(\Lambda) \cdots \lambda_j(\Lambda))
\nonumber \\
&  \leq \prod_{j = 1}^n j^j \min_{\Lambda_j} \covol(\Lambda_j) \leq  n^{n^2} \prod_{j = 1}^n \min_{\Lambda_j} \covol(\Lambda_j)  \leq n^{n^2} \min_{\mC} P(\mC). \label{eq:minimalpotential}
\end{align}
where the minimum $\min_{\Lambda_j}$ is over all $j$-dimensional sublattices $\Lambda_j \subseteq \Lambda$ and where the minimum $\min_{\mC}$ is over all bases of $\Lambda$.
The second inequality of \Cref{eq:minimalpotential} might require some explanation. Any $i$-dimensional sublattice $\Lambda_i \subseteq \Lambda$ 
satisfies $\lambda_j(\Lambda) \leq \lambda_j(\Lambda_i)$ for all $j \in \{1,\ldots,i\}$. Therefore, by Minkowski's second inequality,
\[ \lambda_1(\Lambda) \cdots \lambda_i(\Lambda) \leq\lambda_1(\Lambda_i) \cdots \lambda_i(\Lambda_i) \leq i^i \covol(\Lambda_i). \]
\end{proof}
This lemma has as a consequence that applying the HKZ algorithm can only increase the potential of a basis by a factor $n^{n^2}$.

\begin{corollary} \label{cor:HKZincreasepotential} Let $\mB \in \R^{m \times n}$ be a basis of the $n$-dimensional lattice $\Lambda$ and let $\mB' = \mathrm{HKZ}(\mB)$ be a HKZ-reduced basis of $\Lambda$. Then
\[ P(\mB') \leq n^{n^2} P(\mB). \]
\end{corollary}
\begin{proof} We have $P(\mB') \leq n^{n^2} \cdot \min_{\mC} P(\mC) \leq n^{n^2} P(\mB)$.
\end{proof}

The following lemma is for when a \emph{block} of $\mB$ is being HKZ-reduced, 
which happens in the BKZ-algorithm.
\begin{corollary} \label{cor:HKZincreasepotential2} Let $\mB \in \R^{m \times n}$ be a basis of a lattice $\Lambda$, let $\blocksize \in \{2,\ldots,n\}$, 
and let $\mC =  \mB_{[k:k+\blocksize-1]}$ be a projected sub-block of $\mB$. Suppose $\mU \in GL_{\blocksize}(\Z)$ is such that $\mC' = \mC \mU$ is HKZ-reduced 
and write $\mB' = \mB \bar{\mU}$, where $\bar{\mU}$ acts as $\mU$ on the basis elements $(\bb_k,\ldots,\bb_{k+\blocksize-1})$ and leaves the 
rest intact. Then 
\[ P(\mB') \leq \blocksize^{\blocksize^2}  P(\mB), \]
\end{corollary}
\begin{proof} The action of $\mB' = \mB\bar{\mU}$ only changes the vectors $(\bb_k,\ldots,\bb_{k+\blocksize-1})$, but does \emph{not} change the \emph{space} spanned by $(\bb_k,\ldots,\bb_{k+\blocksize-1})$. Hence, the Gram-Schmidt vectors $\gbb_j$ for $j \notin \{k , \ldots, k + \blocksize - 1\}$ also remain intact, as well as the volume $\covol(\lat(\bb_k,\ldots,\bb_{k+\blocksize-1})) = \prod_{j = k}^{k + \blocksize -1 } \|\gbb_j\| = \prod_{j = k}^{k + \blocksize -1 }\|(\bb'_j)^\star\|$. So, 
\begin{align*} \frac{P(\mB')}{P(\mB)}  &=  \frac{ \prod_{j = 1}^n \|(\bb'_j)^\star\|^{n+1-j}  }{ \prod_{j = 1}^n \|\gbb_j\|^{n+1-j}} = \frac{ \prod_{j = k}^{k + \blocksize-1} \|(\bb'_j)^\star\|^{n+1-j}  }{ \prod_{j = k}^{k + \blocksize-1} \|\gbb_j\|^{n+1-j}} = \frac{ \prod_{j = 0}^{\blocksize-1} \|(\bb'_j)^\star\|^{\blocksize-j}  }{ \prod_{j = 0}^{\blocksize-1} \|\gbb_j\|^{\blocksize-j}}  
\\ & = \frac{P(\mC')}{P(\mC)} \leq \blocksize^{\blocksize^2}.
\end{align*}
where the last inequality follows from \Cref{cor:HKZincreasepotential} and where the third equality follows from dividing by $\prod_{j = k}^{k + \blocksize -1 } \|\gbb_j\|^{(n+1) - (\blocksize + k)}$ so that 
\[ \prod_{j = k}^{k + \blocksize-1} \|\gbb_j\|^{n+1-j} = \prod_{j = k}^{k + \blocksize-1} \|\gbb_j\|^{(\blocksize+k)-j} =  \prod_{j = 0}^{ \blocksize-1} \|\gbb_j\|^{j},\]
and similarly for $(\bb'_j)^\star$, where we use that $\prod_{j = k}^{k + \blocksize -1 } \|\gbb_j\| = \prod_{j = k}^{k + \blocksize -1 } \|(\bb'_j)^\star\|$.
\end{proof}

\subsection{Reduction of dual bases}
\begin{definition}[A \der basis] \label{def:der} 
A square basis $\mB \in \R^{n \times n}$ of a lattice $\Lambda$ is $T$-\der (for $T \geq 1$) if the dual basis $\mD = \mB^{-\top}$ satisfies $\| \bd_j \| \leq 2^{Tn} \lambda_j(\Lambda^\vee)$, where $\Lambda^\vee$ denotes the dual lattice of $\Lambda$.
\end{definition}

\subsection{Transference theorems}
Banaszczyk proved the following transference theorem, relating the successive minima of a lattice $\Lambda$ with the ones of its dual lattice~$\Lambda^\vee$.

\begin{theorem}[{\cite[Theorem 2.1]{Banaszczyk1993}}]
\label{thm:transference}
Let $\Lambda$ be an arbitrary lattice of rank $n$ in $\R^m$, for some $m \geq n \geq 1$. Then, for all $1 \leq i \leq n$ it holds that
\[ 1 \leq \lambda_i(\Lambda) \cdot \lambda_{n-i+1}(\Lambda^\vee) \leq n.\]
\end{theorem}

\begin{proof}
The upper bound is exactly \cite[Theorem 2.1]{Banaszczyk1993}. The lower bound comes from the following standard argument. Let $\vec v_1, \dots, \vec v_i$ be $i$ linearly independent vectors of~$\Lambda$ with euclidean norm $\leq \lambda_i(\Lambda)$, and $\vec w_1, \dots, \vec w_{n-i+1}$ be $(n-i+1)$ linearly independent vectors of~$\Lambda^\vee$ with euclidean norm $\leq \lambda_{n-i+1}(\Lambda^\vee)$. Let $V$ be the vector space spanned by the $(\vec v_j)_j$ and $W$ be the vector space spanned by the $(\vec w_j)_j$. Both spaces live in $\Span(\Lambda)$, which has dimension $n$. Moreover, $\dim(V) + \dim(W) = n+1$. Hence, $V$ and $W$ cannot be orthogonal, which means that there should exist $\vec v_k$ and $\vec w_\ell$ such that $\langle \vec v_k, \vec w_\ell \rangle \neq 0$. By definition of the dual, this implies that $|\langle \vec v_k, \vec w_\ell \rangle| \geq 1$ (because the inner product must be an integer). And so, $1 \leq |\langle \vec v_k, \vec w_\ell \rangle| \leq \|\vec v_k \| \cdot \|\vec w_\ell\| \leq \lambda_i(\Lambda) \cdot \lambda_{n-i+1}(\Lambda^\vee)$.
\end{proof}

\section{The BKZ algorithm on integer bases}
\label{section:BKZinteger}
\subsection{Introduction}
In this section, we show that the variant of BKZ
from Hanrot, Pujol and Stehl\'e \cite[Algorithm 2]{HPS11} (which we call BKZ' in the rest of this section)
does not cause coefficient explosion in integer basis matrices.
Hence, BKZ' is suitable for integer lattice bases and one can obtain a bound on the bit-complexity of the algorithm, and not only the number of tours needed. This was already discussed in \cite[Section 3, Cost of BKZ']{HPS11}, where the authors explained the big lines of the reasoning. In this section, we make this discussion fully formal.

Concretely, applying the $\blocksize$-BKZ' algorithm of Hanrot, Pujol and Stehl\'e on an integer basis $\mB$ of a lattice $\Lambda$ yields a new basis $\mC = (\bc_1,\ldots,\bc_n)$ of $\Lambda$ satisfying $\|\bc_1\| \leq 2 \cdot \blocksize^{\frac{n-1}{2(\blocksize-1)} + \frac{3}{2}} \cdot \covol(\Lambda)^{1/n}$. Furthermore, this algorithm runs in time $\poly(n, \log \max_{j} \|\bb_j\|) \cdot \hkztime$.

\subsubsection*{Approach.} In~\cite{HPS11}, the authors give an upper bound on the \emph{number of tours} needed for the BKZ' algorithm, in order to have a provable upper bound on the short vector output by the algorithm.
To obtain a total and provable bit-complexity for the time of the $\blocksize$-BKZ' variant
for integer matrices, we need to show that no coefficient explosion occurs and that the $\blocksize$-HKZ algorithm on the sub-blocks runs in time about $\hkztime$.

To show that no coefficient explosion occurs, we rely on standard techniques using the potential $P(\mB)$ of a basis. Namely, it is a fact that the bit sizes of the coefficients occurring in a size-reduced basis $\mB$ and its Gram-Schmidt basis $\gmB$ are bounded by $O(\log_2(n P(\mB)))$ (see \Cref{cor:boundedsizepotential}). Hence, it is enough to sufficiently bound the potential.

In the BKZ' algorithm of Hanrot, Pujol and Stehl\'e, only size-reduction and HKZ-reduction in dimension $\blocksize$ occur (see \Cref{alg:BKZHPS}). As size-reduction does not change the potential, only the influence of HKZ-reduction on the potential needs to be examined. It can be shown that $\blocksize$-HKZ-reduction can only increase the potential of a basis by a factor $\blocksize^{\blocksize^2}$ (see \Cref{cor:HKZincreasepotential}). Hence, all coefficients remain sufficiently bounded, whenever the number of HKZ-reductions is polynomially bounded. But the latter is true by the fact that the number of `BKZ tours' (essentially $n$ times a $\blocksize$-HKZ-reduction) is bounded; this is what Hanrot, Pujol and Stehl\'e show in their work \cite[Theorem 1]{HPS11}.

For the HKZ-reduction algorithm in dimension $\blocksize$, that is used as a subroutine in the BKZ-algorithm, we use the provable Kannan-algorithm \cite{HS07_Improved} on integer matrices with a run time of about $\hkztime$ (disregarding the size of the basis matrix). Note that the projected sub-blocks of the matrix are rational, but can be scaled up to be integral. This does not significantly increase the bit-size, because the denominators can be shown to be bounded by the potential.

\subsection{The BKZ' algorithm of Hanrot, Pujol and Stehl\'e}
\label{subsec:bkzofhps}
We restate the algorithm of Hanrot, Pujol and Stehl\'e \cite[Algorithm 2]{HPS11} and their result on the upper bound
on the number of tours required to obtain a sufficiently short non-zero vector of the input lattice.

\begin{figure}[H]
\begin{algorithm} [H]
    \caption{The BKZ' algorithm of Hanrot, Pujol and Stehl\'e, for integer bases}%
    \label{alg:BKZHPS}%
    \begin{algorithmic}[1]
    	\REQUIRE ~\\
         \begin{itemize}
          \item A basis $\mB = (\bb_1,\ldots,\bb_n) \in \Z^{m \times n}$
          \item A block size $\blocksize \in \{1,\ldots,n\}$.
         \end{itemize}

    	\ENSURE  A basis of $\lat(\mB)$.

        \REPEAT
        \FOR{$k$ from $1$ to $n - \blocksize + 1$}
        \STATE Modify $(\bb_i)_{k \leq i < k + \blocksize}$ so that $(\pi_{k-1}(\bb_i))_{k \leq i < k + \blocksize}$ is HKZ-reduced; \label{line:hkz}
        \STATE  Size-reduce $(\bb_1,\ldots,\bb_n)$. \label{line:sizereduce}
        \ENDFOR
        \UNTIL{no changes occur or termination is requested.}
        \RETURN $(\bb_1,\ldots,\bb_n)$.
    \end{algorithmic}
\end{algorithm}
\end{figure}

\begin{theorem}[Theorem 1 of \cite{HPS11}] \label{theorem:mainHPS} There exists an absolute constant $C > 0$ such that the following holds for all $n$ and $\blocksize$. Let $\mB = (\bb_1,\ldots,\bb_n)$ be a basis of a lattice $\Lambda$, given as input to the modified BKZ algorithm (\Cref{alg:BKZHPS}) with block-size $\blocksize$. If terminated after
\[ \tau_{\mathrm{BKZ}} = C \frac{n^3}{\blocksize^2} (\log n + \log \log \max_i \frac{\|\bb_i\|}{\covol(\Lambda)^{1/n}}) \]
calls to an HKZ-reduction in dimension $\blocksize$, the output $\mC = (\bc_1,\ldots,\bc_n)$ is a basis of $\Lambda$ that satisfies
\[  \| \bc_1 \| \leq 2 \blocksize^{\frac{n-1}{2 (\blocksize-1)} + \frac{3}{2}} \cdot \covol(\Lambda)^{1/n}. \]

\end{theorem}
\begin{remark} In the original theorem statement \cite[Theorem 1]{HPS11}, the last line reads:
``If $\Lambda \subseteq \Q^n$, then the overall cost\footnote{Here, $\size(\mB) = \sum_{ij} \size(\mB_{ij})$ and $\size(a/b) = \log|a| + \log|b|$ for reduced fractions $a/b$. The number $\mbox{Cost}_{\mathrm{HKZ}}(\blocksize)$ is the cost of applying HKZ in dimension $\blocksize$.} is $\leq \poly(n, \size(\mB)) \cdot \mathrm{Cost}_{\mathrm{HKZ}}(\blocksize)$.''

It is precisely this statement in the last line that we prove in this section. This, because the authors \cite{HPS11} only gave a sketch of how to show this (\cite[Section 3, Cost of BKZ']{HPS11}). In this section we will follow their sketch and prove this statement rigorously.
\end{remark}

\subsection{Applying BKZ' to integral bases}

\begin{proposition} \label{proposition:HPS} Let $\mB \in \Z^{m \times n}$ be
a
basis of a lattice $\Lambda$, with $m \geq n \geq 2$. Let $\blocksize \in \{2, \cdots, n\}$ be a block-size parameter.
Then the BKZ' algorithm of Hanrot, Pujol and Stehl\'e on input $\mB$ and $\blocksize$ has a bit-complexity
\[ \poly(\size(\mB),m) \cdot \blocksize^{\blocksize} \]
and
outputs a basis $\mC = (\bc_1,\ldots,\bc_n)$ satisfying
\begin{equation} \label{eq:boundonbc1}   \| \bc_1 \| \leq 2 \cdot  \blocksize^{\frac{n-1}{2(\blocksize-1)} + \frac{3}{2}} \cdot \covol(\Lambda)^{1/n}  \leq 2 \cdot \blocksize^{2n/\blocksize} \cdot  \lambda_n(\Lambda). \end{equation}
\end{proposition}
\begin{proof} We instantiate \Cref{theorem:mainHPS} with the integral basis $\mB$, yielding the shortness bound\footnote{The additional shortness bound in terms of $\lambda_n(\Lambda)$ comes from the fact that $\lambda_n(\Lambda)^n \geq \prod_{j=1}^n \lambda_j(\Lambda) \geq \covol(\Lambda)$. For the simplification into $\blocksize^{n/\blocksize}$, one can observe that $(n-1)b + 3(\blocksize-1)b \leq 4n (\blocksize -1)$ for all $\blocksize$ satisfying $2 \leq \blocksize \leq n$. Dividing by $2(\blocksize-1)\blocksize$ yields the simplification.} in \Cref{eq:boundonbc1} and an upper bound on the number of HKZ-reductions $\tau_{\mathrm{BKZ}}$ in dimension $\blocksize$. Note that $\tau_{\mathrm{BKZ}}$ is $\poly(n,\size(\mB))$, according to \Cref{theorem:mainHPS}.

For the $\blocksize$-HKZ reduction algorithm we use the provable Kannan-algorithm \cite[Theorem 2]{HS07_Improved}: on input an integral basis $\mathbf{M}$ of a rank-$\blocksize$ lattice $L \subseteq \Z^m$, this algorithm runs in time $\poly(\size(\mathbf{M}), m) \cdot \blocksize^{\blocksize/(2e)+o(\blocksize)} = \poly(\size(\mathbf{M}), m) \cdot \blocksize^{\blocksize}$ and outputs an HKZ-reduced basis of $L$.
We will apply this algorithm to the projected bases $\mB_{[k:k+\blocksize-1]}$ occurring in the BKZ' algorithm. Even though these bases are rational, this is still possible just by multiplying out the denominators of the coefficients (and revert that process after the lattice reduction). For this to be feasible, we will of course need to show that these coefficients  are sufficiently small in size (i.e., have sufficiently small numerator and denominator).

Concluding, if we can show that all sizes of the coefficients occurring in the subblocks of all intermediately computed matrices in \Cref{alg:BKZHPS} remain sufficiently bounded, say $\leq \poly(\size(\mB).m)$, we can indeed conclude that \Cref{alg:BKZHPS} has bit-complexity $\poly(\size(\mB),m)$.

So, indeed, it remains to show that all (rational) coefficients of the occurring bases, their Gram-Schmidt bases and the sub-blocks $\mB_{[j:j+\blocksize-1]}$ remain bounded in bit-size by $\poly(\size(\mB),m)$. We distinguish some cases
\begin{itemize}
 \item The very first time that a HKZ-reduction is applied is just on the basis $\mB_{[1:\blocksize]}$, whose coefficient sizes are clearly bounded by $\poly(\log P(\mB))$. 
 \item \emph{After} this very first time, we can, by line \lineref{line:sizereduce} always assume that the basis at hand, which we call $\mB^{(0)}$ for now, is size-reduced right before line \lineref{line:hkz}. Hence, the coefficients of $(\pi_{k-1}((\bb^{(0)}_i))_{k \leq i < k + \blocksize}$ are bounded in size by $\poly(\log P(\mB^{(0)}))$, 
by \Cref{cor:boundedsizepotential}. Hence Kannan's HKZ algorithm can reasonably applied in time $\poly(\log(P(\mB^{(0)})), m) \cdot \blocksize^{\blocksize}$.
 \item Right after line \lineref{line:hkz} but \emph{before} line \lineref{line:sizereduce}, the new matrix $\mB^{(1)} = \mB^{(0)} \mU$ (where $\mU$ only changes $(\bb^{(0)}_i)_{k \leq i < k + \blocksize}$) is not size-reduced anymore. But we do know, by \Cref{cor:HKZincreasepotential}, that
 $P(\mB^{(1)}) \leq \blocksize^{\blocksize^2} P(\mB^{(0)})$. Since $(\pi_{k-1}(\bb^{(1)}_i))_{k \leq i < k + \blocksize}$ is HKZ-reduced (and thus size-reduced by definition), we know that the coefficients of 
 $\mB^{(1)}_{[k:k+\blocksize-1]}$ are bounded by $\poly(\log(P(\mB^{(1)})),m)$, by \Cref{cor:boundedsizepotential}. Since for $i \notin \{k,\ldots, k + \blocksize - 1\}$, $\bb^{(0)}_i$ are unaffected by $\mU$, we know that these still 
 satisfy $\size(\bb^{(0)}_i) \leq \poly(\log(P(\mB^{(0)})),m)$. Hence, we can certainly find a lift (by using Babai's nearest plane algorithm) of $\mB^{(1)}_{[k:k+\blocksize-1]}$ to $\mB^{(1)}$ 
 for which the coefficient sizes of $(\bb^{(1)}_i))_{k \leq i < k + \blocksize}$ are bounded by 
 \[ \poly(\log(P(\mB^{(0)})),\log(P(\mB^{(1)})),m) = \poly(\log(P(\mB^{(0)})) + \blocksize,m). \] 
 (Note that this lift does \emph{not} affect the potential of $\mB^{(1)}$.
\end{itemize}
Note that in the previous reasoning we see that in each HKZ-reduction, the potential of the basis 
the algorithm is working on is maximally increased by a factor $\blocksize^{\blocksize^2}$. But we know 
that the maximum number of HKZ reductions in the entire algorithm is at most $\tau_{\mathrm{BKZ}} \cdot n = \poly(\size(\mB),m)$ 
(where $\tau_{\mathrm{BKZ}}$ is the number of tours). Hence, the bases occurring in the algorithm can maximally have log potential of 
\[ \log(\mB_{\mathrm{occurring}}) \leq \log( \blocksize^{\blocksize^2}) \cdot \tau_{\mathrm{BKZ}} \cdot n + \log(P(\mB))  \leq \poly(\size(\mB),m).\]
Hence all occurring bases $\mB_{\mathrm{occurring}}$ have a logarithmic potential bounded by $\poly(\size(\mB),m)$, which has 
as an immediate consequence that all bases, Gram-Schmidt bases and projected bases have coefficient sizes bounded by $\poly(\size(\mB),\allowbreak m)$, which
finishes the proof.
\end{proof}

\begin{corollary} \label{cor:HPS} Let $\mB \in \Z^{m \times n}$ be
a %
basis of a lattice $\Lambda$, with $m \geq n \geq 2$ and let $\blocksize \in \{2, \cdots, n\}$.
Then there exists an algorithm that takes as input $\mB$ and $\blocksize$, has bit-complexity
\[ \poly(m,\size(\mB)) \cdot \blocksize^{\blocksize} \]
and outputs a basis $\mC = (\bc_1,\ldots,\bc_n)$ of $\Lambda$ satisfying
\begin{equation} \label{eq:boundonbcj}   \| \bc_j\| \leq n \cdot \blocksize^{2n/\blocksize} \cdot \lambda_n(\Lambda) ~\mbox{ for all } j \in \{1,\ldots, n\}. \end{equation}
\end{corollary}
\begin{proof} %
We apply the provable BKZ' algorithm of Hanrot, Pujol and Stehl\'e as in \Cref{proposition:HPS},
recursively to
\[ \mB,~\det(\mB_{[1:1]})^2\cdot  \mB_{2:n},~\ldots,\det(\mB_{[1:j]})^2 \cdot \mB_{j+1:n}~,\ldots,~\det(\mB_{[1:n-\blocksize]})^2 \cdot \mB_{n-\blocksize+1:n}. \]
Here, we keep using
the notation $\mB$ after an application of BKZ'. After these applications,
we backwards recursively lift this sequence into a size-reduced basis of $\lat(\mB)$.

More precisely we apply the following steps:
\begin{itemize}
 \item Apply BKZ' to $\mB$, call the result $\mB$.
 \item Construct the basis $\det(\mB_{[1:1]}) \mB_{[2:n]}$ (integral, by \Cref{lemma:gsfraction}) and apply BKZ' to it, call the result again $\det(\mB_{[1:1]}) \mB_{[2:n]}$.
 \item $\ldots$
 \item Construct the basis $\det(\mB_{[1:j]}) \mB_{[j+1:n]}$ (integral) and apply BKZ' to it, call the result again $\det(\mB_{[1:j]}) \mB_{[j+1:n]}$.
 \item $\ldots$
 \item Construct the basis $\det(\mB_{[1:n-\blocksize]})^2 \cdot \mB_{n-\blocksize+1:n}$ (integral) and apply BKZ' to it, call the result again $\det(\mB_{[1:j]}) \mB_{[j+1:n]}$.
 \item Going backwards, add $c_k \gbb_{n-\blocksize}$ with $c_k \in [-1/2,1/2)$ to each vector in $\mB_{n - \blocksize+1:n}$ to reconstruct a short basis $\mB_{n - \blocksize:n}$.
 \item \ldots
 \item Add $c_k \gbb_{j}$ with $c_k \in [-1/2,1/2)$ to each vector in $\mB_{j+1:n}$ to reconstruct a short basis $\mB_{j:n}$.
 \item \ldots
 \item Add $c_k \gbb_{j}$ with $c_k \in [-1/2,1/2)$ to each vector in $\mB_{2:n}$ to reconstruct a short basis $\mB$.
\end{itemize}
Since a BKZ'-reduction as in \Cref{proposition:HPS} consists of $\tau_{\mathrm{BKZ}}$ HKZ-reductions, each BKZ' reduction can maximally increase the potential by a factor $\blocksize^{\blocksize^2 \cdot \tau_{\mathrm{BKZ}}}$ (see \Cref{cor:HKZincreasepotential}). Hence the $n-\blocksize$ BKZ'-reductions in above procedure can at most result in intermediate potentials of size
\[  \blocksize^{\blocksize^2 \cdot \tau_{\mathrm{BKZ}} \cdot n} \cdot P(\mB_{\mathrm{init}}), \]
where $\mB_{\mathrm{init}}$ is the starting basis. Since $\det(\mB_{[1:k]})$ divide $P(\mB)$ (the potential of the same basis), and since $\mB_{[k+1:n]}$ has size bounded in $O(\log(n \cdot P(\mB)))$ (see \Cref{cor:boundedsizepotential}), we can immediately deduce that in above procedure, all bit sizes remain within (for the bound on $\tau_{\mathrm{BKZ}}$ see \Cref{theorem:mainHPS})
\[ \log\left( n  \blocksize^{\blocksize^2 \cdot \tau_{\mathrm{BKZ}} \cdot n} \cdot P(\mB_{\mathrm{init}}) \right) = \poly(m, \size(\mB_{\mathrm{init}})).  \]

In the resulting basis, we have, by construction,
\[ \| \gbb_{j}\| \leq  2 \cdot \blocksize^{2n/\blocksize} \cdot \lambda_n(\Lambda) ~\mbox{ for all } j \in \{1,\ldots, n\},  \]
where we use that $\lambda_{n-j}(\lat(\mB_{j+1:n})) \leq \lambda_n(\Lambda)$ and $\blocksize^{2(n-j)/\blocksize} \leq \blocksize^{2n/\blocksize}$. For the last $\blocksize$ vectors, we have this shortness guarantee from the $\blocksize$-HKZ reduction (see \cite{lagarias90:_korkin_zolot_bases_and_succes} or the proof of \Cref{lemma:HKZpotential}).

Hence, by the size-reduced-like way the basis $\mB$ is constructed, we have $\bb_{j} = \gbb_j + \sum_{i < j} c_i \gbb_i$ with $c_i \in [-1/2,1/2)$, and we obtain
\[ \|\bb_{j} \| \leq 2n \cdot \blocksize^{2n/\blocksize} \cdot \lambda_n(\Lambda) ~\mbox{ for all } j \in \{1,\ldots, n\}, \]
as required.
\end{proof}

\section{The BKZ algorithm on approximate bases} \label{section:BKZapproximate}

\subsection{Introduction}
In this section, we will show that there exists a provable algorithm that runs essentially in time $O(\blocksize^{\blocksize})$ that computes a basis of an ideal lattice $x \ma$ whose Gram-Schmidt lengths are at most $\blocksize^{O(n/\blocksize)} \lambda_n(x \ma)$, which is sufficient for the needs of this paper (see \Cref{subsec:latticereduction,subsec:samplingboxintersected}). We assume throughout this entire section that an LLL-reduced basis of $\OK$ is given; in the end result, \Cref{lemma:shortvectorsinideallattice}, this assumption is again explicitly mentioned.

In all this section, the ideal lattices $x \ma$ we considered will be represented exactly. To do so, we will restrict ourselves to elements $x \in \nfrstar$ that have all their complex coordinates of the form $a + ib$, with $a, b \in \Q$. In other words, we consider the intersection of $\KR$ with the subset $(\Q+i\Q)^n$ of $\C^n$. The ideal $\ma$ is a fractional ideal, and will be represented by a $\Z$-basis, consisting of $n$ elements $(\alpha_1, \dots, \alpha_n)$ of~$K$ (which can be exactly represented in a fixed basis of $\OK$).

Even though we have a way to represent the ideal lattice $x \ma$ exactly, we can only compute an \emph{approximate basis} of the corresponding lattice, due to the (irrational) Minkowski embedding of $x\ma$ into $\nfr$ (even whenever $x \in \nfrstar$ itself is rational). In this section
we settle these precision issues by applying the basis reduction algorithms (LLL, BKZ') on rational approximations of the bases of $x \ma$.

The algorithm and its analysis consists of the following steps. These steps (i) - (v) will correspond with \Cref{subsec:compdualbasis,subsec:applybpk,subsec:bkzofhps2,subsec:closenesslemma,subsec:BKZ:conclusion}
\begin{enumerate}[(i)]
 \item \textbf{Computing an (approximate) dual basis of the basis of $x \ma$.} For the algorithm to succeed, the basis that we start with is required to be \emph{dually} reduced. For the purpose
 of reducing it, we therefore first compute the dual basis of the input basis of $x \ma$ (which we may assume to be in some Hermite Normal form). %
 \item \textbf{Apply the the Buchmann-Pohst-Kessler algorithm to this approximate dual basis.} To make the dual basis reduced, we apply the Buchmann-Pohst-Kessler algorithm \cite{buchmann96,buchmann87} (see also \Cref{section:buchmannkesslerpohst}). We apply the inverse-transposed transformation to the original basis of $x\ma$, to obtain a \der basis (see \Cref{def:der}). This part is treated in \Cref{subsec:applybpk}.
 \item \textbf{Apply the variant of BKZ by Hanrot, Pujol and Stehl\'e \cite{HPS11}.} Applying a variant of BKZ with provable running time on this \der (approximate) basis of $x \ma$ yields an approximate basis of $x \ma$ with BKZ-like shortness guarantees on the basis vectors. %
 \item \textbf{Apply the closeness-lemma.} By computing an initial approximate basis with sufficiently good precision, the same unimodular transformation that makes this approximate basis have BKZ-like qualities, will also make the exact basis of $x \ma$ satisfy these shortness properties. Note that the unimodular transform can be actually exactly (that is, not approximately) applied to a $\Z$-basis of $x \ma$ which consists of elements $(x \cdot \alpha_1,\ldots , x \cdot \alpha_n) \in \nfr$, where $\alpha_j \in K$. %
 \item \textbf{Concluding.} Applying the compositions of the unimodular transformations in each step to the initial basis $(x \cdot \alpha_1,\ldots , x \cdot \alpha_n)$ of $x \ma$ yields a new basis of $x \ma$ that has BKZ-like shortness guarantees.
\end{enumerate}

\subsection{Computing an (approximate) dual basis of the basis of \texorpdfstring{$x \ma$}{xa}}
\label{subsec:compdualbasis}
\subsubsection*{Introduction}
Computing the dual basis of an approximate basis $\tmB_{x\ma} \approx \mB_{x \ma}$ of $x \ma$
is done by inverting and transposing the matrix: $\tilde{\mD} = \tmB_{x \ma}^{-\top}$.
The precision loss in numerical approximation of inversion is proportional to $\|\mB_{x \ma}^{-1} \|$, the matrix norm of the inverse of the basis. Hence, we would like to upper bound this matrix norm.

We do that by applying \Cref{lemma:wellconditioned}, which bounds this matrix norm in terms of how far the basis vectors $\mB_{x\ma}$ are from attaining the successive minima $\lambda_j(x \ma)$. Since we generally assume that ideals are given in a Hermite Normal Form (HNF) in coordinates of an LLL-reduced $\Z$-basis of $\OK$ (see \Cref{sec:representation}), a reasonable upper bound on the basis vectors $\mB_{x \ma}$ can be achieved. On the other hand, a lower bound on $\lambda_j(x \ma)$ can be achieved by using ideal lattice properties from \Cref{lemma:idlatfacts}.

To obtain these results, the following lemma will turn out useful.
\begin{lemma} \label{lemma:usefulsizeinequalities}
Let $\kx \in \nfrstar$ and $\ma \in \ideals$. Let $\mB_{x \ma} = (x \alpha_1,\ldots,x\alpha_n) \in \nfr^n$ a basis of $x \ma$, where the basis $(\alpha_1,\ldots,\alpha_n)$ is assumed to be in HNF form whenever written in terms of an LLL-reduced basis of $\OK$ (see \Cref{sec:representation}). Then
\begin{enumerate}[(i)]
 \item For all $j \in \{1,\ldots,n\}$, we have $\| x \alpha_j\| \leq 2^{\size(x \ma)} \cdot n^{3/2} \cdot 2^{n} \cdot |\dcrk|^{1/n}$. \label{item:boundonlength}
 \item $\norm(x \ma) \leq 2^{n \size(x\ma)}$.\label{item:boundonnorm}
 \item $\norm(x \ma)^{-1} \leq 2^{n \size(x \ma)}$. \label{item:boundoninversenorm}
 \item $\lambda_1(x\ma)^{-1} \leq 2^{\size(x\ma)}$. \label{item:boundoninverseminima}
\end{enumerate}
where $\size$ is defined in \Cref{sec:representation}. 
\end{lemma}
\begin{proof}  For (\itemref{item:boundonlength}), we make use of the assumption %
that $\ma$ is represented by a rational HNF basis in terms of an LLL-reduced $\Z$-basis $(\beta_1,\ldots,\beta_n)$ of $\OK$ (for which thus holds $\|\beta_j\| \leq 2^{n} \cdot \sqrt{n} \cdot |\dcrk|^{1/n}$ by \Cref{lemma:idlatfacts}). More precisely, $\alpha_j = \sum_{i = 1}^n a_{ij} \beta_i$,
where $(a_{ij})_{ij} \in \Q^{n \times n}$ is an HNF-reduced basis. By definition of the size of an ideal (see \Cref{sec:representation}), we have that $\size(\ma) \geq \size(a_{ij})$. Recall that the size of rational numbers $a/b$ with $\gcd(a,b) = 1$ is equal to $\log_2|a| + \log_2|b|$. Hence, using $\size(x \ma) = \size(x) + \size(\ma)$ 
\[ \|x \alpha_j\| \leq n \cdot \max_j |x_j| \cdot \max_{ij} 2^{\size(a_{ij})} \cdot \underbrace{\max_j \| \beta_j \|}_{\leq 2^n \cdot \sqrt{n} \cdot |\dcrk|^{1/n}} \leq  2^{\size(x\ma)} \cdot n^{3/2} \cdot 2^n \cdot |\dcrk|^{1/n}. \]
For (\itemref{item:boundonnorm}) note that $N(x \ma)$ is equal to $\det( (a_{ij})_{ij}) \cdot \prod_{j} x_j$.
By Hadamard's bound, $\det( (a_{ij})_{ij}) \leq \prod_j \| (a_{ij})_i \|_2 \leq \prod_{j} (\sum_{i} |a_{ij}|)$ (by the inequality between $1$ and $2$-norms)
and the latter can be bounded by $(\sum_{i} |a_{ij}|)^n$ 
and hence $\log_2 N(x\ma) \leq n\size(x) + n \log_2(\sum_{i} |a_{ij}|) \leq  n\size(x) + n \size(\ma) = n\size(x\ma)$.

Item (\itemref{item:boundoninversenorm}) can be proven as follows. Because $\norm(x\ma) = |\norm(x)| \cdot \norm(\ma)$ and $|\norm(x)|^{-1} = \prod_{j=1}^n |x_j|^{-1} \leq 2^{\size(x)}$, we only need to show that $\norm(\ma)^{-1} \leq 2^{n \size(\ma)}$. Write $\ma = \frac{1}{d} \cdot \mb$ with $d$ a minimal positive integer and $\mb$ an integral ideal. Since $d$ minimally clears the denominators of $\ma$, we certainly have $\size(d) \leq \size(\ma)$. Now $\norm(\ma)^{-1} = \norm(\ma^{-1}) = |\norm(d)| \cdot  \norm(\mb)^{-1} \leq |\norm(d)| = d^n \leq 2^{n \size(d)} \leq 2^{n \size(\ma)}$.

Item (\itemref{item:boundoninverseminima}) holds because $x \ma$ is an ideal lattice, which implies $\lambda_1(x\ma) \geq \norm(x\ma)^{1/n}$ due to item~(\itemref{item:lower-bound-lambda1}) from \Cref{lemma:idlatfacts}. Hence $\lambda_1(x\ma)^{-1} \leq \norm(x\ma)^{-1/n}$ and the result follows from item (\itemref{item:boundoninversenorm})
\end{proof}

\subsubsection*{Bounding the matrix norm of $\mB_{x \ma}^{-1}$}
We start by bounding the matrix norm of the exact inverse of $\mB_{x \ma}$.
\begin{lemma} \label{lemma:inverseideallatticebound} Let $\kx \in \nfrstar$ and $\ma \in \ideals$. Let $\mB_{x \ma} = (x \alpha_1,\ldots,x\alpha_n) \in \nfr^n$ a basis of $x \ma$, where the basis $(\alpha_1,\ldots,\alpha_n)$ is assumed to be in HNF form whenever written in terms of an LLL-reduced basis of $\OK$ (see \Cref{sec:representation}). Then
\[ \normmattwo{\mB_{x \ma}^{-1} }  \leq
n^{2n + 1} \cdot 2^{n^2} \cdot  |\dcrk|^{1/2} \cdot 2^{(2n+1)  \size(x\ma)  } . \]
\end{lemma}
\begin{proof}
We apply \Cref{lemma:wellconditioned} to obtain
\begin{align} \nonumber \normmattwo{\mB_{x \ma}^{-1} }  & \leq n^{n/2+1}  \cdot \lambda_1( x \ma)^{-1} \cdot \left( \prod_{j = 1}^n \frac{ \| x\alpha_j \|}{\lambda_j(x \ma)} \right) \\ 
& \leq n^{n/2 + 1} \cdot \lambda_1(x \ma)^{-1} \cdot |\dcrk|^{-1/2} \cdot \norm(x \ma)^{-1} \cdot \prod_{j = 1}^n \| x \alpha_j \|. 
\label{eq:inversebasisbound} 
\end{align}
where we used $\prod_{j = 1}^n \lambda_j(x \ma) \geq \covol(x \ma) =  \norm(x \ma) |\dcrk|^{1/2}$ in the last inequality.

Using that $\norm(x \ma)^{-1} \leq 2^{n\size(x\ma)}$, $\lambda_1(x\ma)^{-1} \leq 2^{\size(x\ma)}$ and $\| x \alpha_j\| \leq 2^{\size(x \ma)} \cdot n^{3/2} \cdot 2^{n} \cdot |\dcrk|^{1/n}$ for all $j \in \{1,\ldots,n\}$ (see \Cref{lemma:usefulsizeinequalities}(\itemref{item:boundoninversenorm}), (\itemref{item:boundoninverseminima}) and (\itemref{item:boundonlength})), and applying it to \Cref{eq:inversebasisbound}, we obtain
\begin{align*}  \normmattwo{\mB_{x \ma}^{-1} }   &\leq n^{n/2 + 1} \cdot \lambda_1(x \ma)^{-1} \cdot |\dcrk|^{-1/2} \cdot \norm(x \ma)^{-1} \cdot \prod_{j = 1}^n \| x \alpha_j \| 
\\ & \leq n^{n/2 + 1} \cdot 2^{\size(x\ma)} \cdot |\dcrk|^{-1/2} \cdot 2^{n\size(x \ma)}  \cdot \left( 2^{\size(x \ma)} \cdot n^{3/2} \cdot 2^{n} \cdot |\dcrk|^{1/n}  \right)^n
\\ & \leq n^{2n + 1} \cdot 2^{n^2} \cdot  |\dcrk|^{1/2} \cdot 2^{(2n+1) \size(x\ma)} .   \end{align*}
\end{proof}

We can therefore conclude that the loss in precision by
inverting an approximate basis of $x \ma$ can be reasonably
upper bounded.
\begin{lemma} \label{lemma:inverseprecisionloss} Let $\tmB_{x \ma} \in \Q^{n \times n}$ be an approximation of $\mB_{x \ma}$, then
\[ \log_2 \normmattwo{\tmB_{x \ma}^{-1} - \mB_{x \ma}^{-1}} \leq \log_2 \normmattwo{ \tmB_{x \ma} - \mB_{x \ma} }  + 5n^2 + \tfrac{1}{2}\log_2|\dcrk| + 3n \size(x \ma),  \]
under the assumption that $\log_2 \normmattwo{\mB_{x \ma} - \tmB_{x \ma} } \leq - (5n^2 + \frac{1}{2} \cdot \log_2|\dcrk| + 3n \size(x \ma))$.
\end{lemma}
\begin{proof} By \cite[Cor.~7.2, Eq.~7.46]{NLA17}, we have that
\[ \normmattwo{ \mB^{-1} - \tmB^{-1} } \leq 2  \normmattwo{ \mB^{-1} } ^2 \cdot \normmattwo{ \mB - \tmB}, \]
 as long as $2 \normmattwo{ \mB^{-1} } \cdot \normmattwo{ \mB - \tmB} \leq 1$.
 Instantiating with $\mB = \mB_{x \ma}$, $\tmB = \tmB_{x \ma}$ and the bound of \Cref{lemma:inverseideallatticebound}, we obtain the claim, 
 using the following computation:
 \begin{align*} \log_2 (2 \| \mB_{x \ma}^{-1}\|) & \leq 1 + (2n+1)\log_2(n) + n^2 + \tfrac{1}{2}\log|\dcrk| + (2n+1) \size(x \ma)
 \\ & \leq 2n + 2n^2 + n^2 + \tfrac{1}{2}\log|\dcrk| + (2n+1) \size(x \ma) \\ & \leq 5n^2 + \tfrac{1}{2}\log|\dcrk| + 3n \size(x \ma)
 \end{align*}
 using the fact that for any positive integer $n$ it holds that $2n \leq 2n^2$, $\log_2(n) \leq n$ and $\log_2(n)+1 \leq 2n$. The condition $2 \normmattwo{ \mB^{-1} } \cdot \normmattwo{ \mB - \tmB} \leq 1$ holds by the assumption in the lemma's statement.
\end{proof}
Hence, the inversion of the approximate basis $\tmB_{x \ma}$ only yields a polynomial (in the size of the input) loss of bit precision.

\begin{lemma} \label{lemma:bothder} Let $\tmB_{x\ma}$ be an approximation of a $T$-\der basis $\mB_{x\ma}$ with $\log \normmattwo{ \tmB_{x \ma} - \mB_{x\ma}} \leq - ( 9n^2 + \tfrac{5}{2}\log_2|\dcrk| + 5n \size(x \ma) + n(T+2))$. Then $\tmB_{x\ma}$ is $(T+3)$-\der.
\end{lemma}
\begin{proof} %
Let us write
$(\bd_1,\ldots,\bd_n) = \mathbf{D} = \mB_{x\ma}^{-\top}$ and $(\tilde{\bd}_1,\ldots,\tilde{\bd}_n) = \tilde{\mathbf{D}} = \tmB_{x\ma}^{-\top}$. Let us also define $\eps$ such that $\log_2(1/\eps) = 4n^2 + 2 \log_2 |\dcrk| + 2n \size(x\ma)$.
By \Cref{lemma:inverseprecisionloss} (and since transposing does not change the $2$-norm of matrices), we have

\begin{align*} \log_2 \normmattwo{ \mathbf{D} - \tilde{\mathbf{D}} } & \leq \log_2 \normmattwo{ \tmB_{x \ma} - \mB_{x \ma} }  + 5n^2 + \tfrac{1}{2}\log_2|\dcrk| + 3n \size(x \ma) 
\\ & \leq -(4n^2 + 2 \log_2 |\dcrk| + 2n \size(x\ma))
\leq  \log(\eps).
\end{align*}
Then, we have, for all $j$ that $\|\tilde{\bd_j} - \bd_j\| = \|(\mathbf{D} - \tilde{\mathbf{D}} ) \mathbf{e}_j\| \leq \normmattwo{ \mathbf{D} - \tilde{\mathbf{D}} }\leq \eps$, which implies that
\begin{equation} \| \tilde{\bd_j} \| = \| \bd_j \| + \eps \leq 2^{Tn} \lambda_j( (x \ma)^\vee ) + \eps \leq ( 2^{Tn}+1)\lambda_j( (x \ma)^\vee )\leq 2^{(T+1)n} \lambda_j((x \ma)^\vee), \label{eq:tildedual} \end{equation}
where, in the second inequality, we used the fact that $\eps \leq \lambda_1((x \ma)^\vee)$. This inequality comes from combining the transference bound (\Cref{thm:transference}) with items (\itemref{item:gap-bound}) and (\itemref{item:covering-bound}) from \Cref{lemma:idlatfacts} and item (\itemref{item:boundonnorm}) of \Cref{lemma:usefulsizeinequalities}
\begin{align*} \lambda_1((x \ma)^\vee) & \geq \lambda_n(x \ma)^{-1} \geq ( \sqrt{n} \cdot |\dcrk|^{3/2} \cdot \norm(x \ma)^{1/n})^{-1} 
\\ & \geq (\sqrt{n} \cdot |\dcrk|^{3/2} \cdot 2^{\size(x \ma)})^{-1} \geq \eps \end{align*}

Writing $\tilde{\Lambda}$ for the lattice generated by the approximation $\tmB_{x \ma}$, it suffices to show that $\lambda_j( (x \ma)^\vee ) \leq 2n \lambda_j(\tilde{\Lambda}^\vee)$ for all $j$, to show that $\tmB_{x \ma}$ is $(T+2)$-\der. Indeed, then, by \Cref{eq:tildedual}, we have $\| \tilde{\bd_j} \| \leq
2^{(T+1)n} \lambda_j((x \ma)^\vee) \leq 2^{(T+3)n} \lambda_j( \tilde{\Lambda}^\vee)$ for all $j$.

By transference (see \Cref{thm:transference}), we have
\[ \frac{1}{n} \leq \frac{\lambda_j(\Lambda^\vee) \cdot \lambda_{n-j}(\Lambda)}{\lambda_j(\tilde{\Lambda}^\vee) \cdot \lambda_{n-j}(\tilde{\Lambda})} \leq n\]
hence $\lambda_j(\Lambda^\vee) \leq n  \frac{\lambda_{n-j}(\tilde{\Lambda})}{ \lambda_{n-j}(\Lambda)} \cdot \lambda_j(\tilde{\Lambda}^\vee) \leq 2n \cdot \lambda_j(\tilde{\Lambda}^\vee) $, if one shows that $\frac{\lambda_{n-j}(\tilde{\Lambda})}{ \lambda_{n-j}(\Lambda)} \leq 2$, i.e., $\lambda_j(\tilde{\Lambda}) \leq 2 \lambda_j(\Lambda)$ for all $j$. This latter statement is what we finish the proof with.

Write $\bx_j \in x \ma = \Lambda$ for the vector attaining $\lambda_j( x \ma)$. Then $\bx_j = \mB_{x \ma} \bu_j$ for some $\bu_j \in \Z^n$. By the closeness lemma (\Cref{lemma:wessel}), we must have $\| \bu_j \| \leq 2^{(T+2)n} \cdot \lambda_j( x\ma)/\lambda_1(x\ma) \leq 2^{(T+2)n} \cdot |\dcrk|^{1/n} $ (by \Cref{lemma:idlatfacts}). Hence
$\tilde{\bx}_j = \tmB_{x \ma} \bu_j = \bx_j + (\mB_{x\ma} - \tmB_{x \ma}) \bu_j $ satisfy the property that $\{\tilde{\bx}_1,\ldots, \tilde{\bx}_j \}$ span a $j$-dimensional space,  and therefore surely
\begin{align*} \lambda_j(\tilde{\Lambda}) &\leq \lambda_j(x\ma) + \normmattwo{\mB_{x\ma} - \tmB_{x \ma} } \cdot 2^{(T+2)n} \cdot |\dcrk|^{1/n} \\ 
& \leq \lambda_j(x\ma) + \lambda_1(x\ma) \leq 2 \cdot \lambda_j(x \ma),\end{align*}
since
$\normmattwo{\mB_{x\ma} - \tmB_{x \ma}} \leq 2^{-n \size(x \ma) - \log |\dcrk| -  (T+2)n} \leq 
\norm(x \ma)^{1/n} \cdot |\dcrk|^{-1/n} \cdot 2^{-(T+2)n} \leq \lambda_1(x \ma)  \cdot 2^{-(T+2)n} \cdot |\dcrk|^{-1/n}$ (where we used item (\itemref{item:lower-bound-lambda1}) from \Cref{lemma:idlatfacts}).
\end{proof}

\subsection{Applying Buchmann-Pohst-Kessler to the approximate dual basis}
\label{subsec:applybpk}
We apply \Cref{theorem:buchmannpohstkessler} with $k = r_0 = n_2 = [K:\Q] =: n$, $n_1 = 0$, $\mA = \mB_{x \ma}^{-\top}$ and $\tilde{\mA} = \tmB_{x \ma}^{-\top}$. To satisfy the initial conditions, we must have $\log_2 \|\mB_{x \ma}^{-\top} - \tmB_{x \ma}^{-\top} \|_{2, \infty} < \log_2( 1/4 \cdot \mu \cdot C_0^{-1})$,
where%
\footnote{%
We have $\lambda_n(x \ma) \leq \sqrt{n} \cdot |\dcrk|^{\tfrac{3}{2n}} \cdot \norm(x \ma)^{1/n}$ by items~(\itemref{item:gap-bound}) and (\itemref{item:covering-bound}) of \Cref{lemma:idlatfacts}.
}
$\mu = \norm(x \ma)^{-1/n} \cdot n^{-1/2} \cdot |\dcrk|^{-\tfrac{3}{2n}} \leq \lambda_n(x \ma)^{-1} \leq \lambda_1( (x \ma)^\vee)$.
And
\[C_0 = 2^{8n} \cdot \left( \frac{n \cdot 4^n \cdot \| \mB_{x \ma}^{-\top} \|_{2,\infty} }{\mu} \right) ^{2(n+1)}  \]
By \Cref{lemma:inverseideallatticebound}, we have that $\| \mB_{x \ma}^{-\top} \|_{2,\infty} \leq  \| \mB_{x \ma}^{-\top} \|_{2} \leq n^{2n+1} 2^{n^2} |\dcrk|^{1/2} \cdot 2^{2n \size(x \ma)}$. Hence,
$-\log_2 (1/4 \cdot \mu \cdot C_0^{-1}) = O( n^3 + n \log |\dcrk| + n^2 \size(x \ma))$,
and a precision polynomial in $n, \log |\dcrk|$ and $\size(x \ma)$ is sufficient to apply
Buchmann-Pohst-Kessler. Hence we obtain a unimodular transformation $\mU$ such that
the basis elements of $\mB_{x \ma}^{-\top} \cdot \mU$ have an LLL-like shortness quality
compared to the successive minima of $(x \ma)^\vee$. More precisely, the basis elements 
$(\bc_1,\ldots,\bc_n) = \mB_{x \ma}^{-\top} \cdot \mU$ satisfy 
\[ \| \bc_j \| \leq (n + 2) 2^{\frac{n-1}{2}} \cdot \lambda_j( x\ma^\vee) \leq 2^{3n} \cdot \lambda_j( x\ma^\vee), \]
hence, the basis $\mC_{x \ma} := (\mB_{x\ma}^{-\top} \cdot \mU)^{-\top} =  \mB_{x\ma} \mU^{-\top}$
is $T$-\der with $T := 3$ (see \Cref{def:der}). For this last computation, note that both inverse and transpose switch positions, hence doing them both keeps the positions of the matrices.

We can conclude with the following proposition.
\begin{proposition} \label{proposition:efficientder} There exists an algorithm that computes a $3$-\der $\Z$-basis $(x \alpha_1,\ldots,x \alpha_n)$ of an ideal lattice $x \ma$ in time polynomial in $n, \log |\dcrk|$ and $\size(x \ma)$.
\end{proposition}
\begin{proof} Computing $\tmB_{x \ma}$ with sufficient precision of order $O( n^3 + n \log |\dcrk| + n^2 \size(x \ma))$ allows for computing the dual basis $\tmB_{x \ma}^{-\top} = (\tmB_{x \ma}^{-1})^\top$ and applying the \bkp algorithm (see \Cref{lemma:inverseprecisionloss} and the discussion above). Hence, by applying the inverse transpose $\mU^{-\top}$ of the unimodular matrix $\mU$ that is the output of the \bkp algorithm to the initial $\Z$-basis $(x \alpha_1,\ldots,x\alpha_n)$ of $ x \ma$ yields a $3$-\der basis of $ x \ma$.
\end{proof}

\subsection{Applying the variant of BKZ by Hanrot, Pujol and Stehl\'e} \label{subsec:bkzofhps2}
By clearing the denominators in the (new, \der) matrix $\tmB_{x \ma}$,
we obtain an integer basis on which we can apply  \Cref{alg:BKZHPS}, a variant of the BKZ algorithm from Hanrot, Pujol and Stehl\'e. By \Cref{cor:HPS}, this algorithm outputs in time\footnote{Clearing denominators in a basis $\mB \in \Q^{m \times n}$ at most increases the size of that basis by $2 \cdot m \cdot n$. Since $d \mB$ for $d = \prod_{ij} d_{ij}$ (where $d_{ij}$ is the denominator of $\mB_{ij}$) has size $mn \size(d) + \size(\mB) \leq (mn + 1) \size(\mB) \leq 2 mn \size(\mB)$ } $\poly(n,\size(\tmB_{x\ma})) \cdot \hkztime$ a basis $\tilde{\mC}_{x\ma} = (\bc_1,\ldots,\bc_n)$ of $x \ma$ that satisfies
\[ \| \bc_j \| \leq n \cdot  \blocksize^{2n/\blocksize} \cdot \lambda_n(x \ma). \]

\subsection{The closeness lemma}
\label{subsec:closenesslemma}

\begin{lemma} \label{lemma:wessel} Let $\mB \in \R^{m \times n}$ be basis of a rank-$n$ lattice $\Lambda$
that is $T$-\der as in \Cref{def:der} and
let $\bv \in \Lambda$ be any lattice vector, with $\bv = \mB \bu$ for some integral vector $\bu$.
Then \[ \|\bu\|^2  \leq n^3 \cdot 2^{nT} \cdot \| \bv \|^2/\lambda_1(\Lambda)^2 \leq  2^{n(T+2)} \cdot \| \bv \|^2/\lambda_1(\Lambda)^2.\]
\end{lemma}
\begin{proof}
Write $\mQ = \mB^\top \mB$ for the quadratic form associated to $\mB$; then $\mQ^{-1} = (  \mB^\top \mB)^{-1} = \mB^{-1} \mB^{-\top} = \mD^{\top} \mD$ is the quadratic form associated to the dual basis $\mD = \mB^{-\top}= (\bd_1,\ldots,\bd_n)$ of $\mB$.

Denote $\Lambda^\vee$ for the dual lattice of $\Lambda$.
By $\mB$ being $T$-\der, we have $(\mQ^{-1})_{jj} = \| \bd_j \|^2 \leq 2^{Tn} \cdot \lambda_j(\Lambda^\vee)^2 \leq n^2 \cdot 2^{T \cdot n}/ \lambda_{n-j+1}(\Lambda)^2$
by Banaszczyk's transference theorem (\Cref{thm:transference}). Therefore
\[ \Tr(\mQ^{-1}) = \sum_{i=1}^n (\mQ^{-1})_{ii} \leq n^2 \cdot 2^{Tn} \sum_{i = 1}^n \frac{1}{\lambda_i(\Lambda)^2} \leq n^3 \cdot 2^{Tn}/ \lambda_1(\Lambda)^2. \]
Write $\mu_1,\ldots,\mu_n > 0$ for the eigenvalues of the positive definite matrix $\mQ$. Then,  for all $i \in \{1, \dots, n\}$, $\frac{1}{\mu_i} \leq \Tr(\mQ^{-1}) \leq n^3 \cdot 2^{Tn}/ \lambda_1(\Lambda)^2$, since all eigenvalues of $\mQ$ are positive.
Therefore, the inverse of any eigenvalue of $\mQ$ is upper bounded by $n^3\cdot 2^{Tn}/ \lambda_1(\Lambda)^2$,

Write $\bv = \mB \bu$, then
\[ \| \bu \|^2 =  \bu^\top \bu \leq  \frac{\bu^\top \mQ \bu}{\min_i \mu_i} \leq \|\bv \|^2 \cdot \frac{n^3 \cdot 2^{Tn}}{\lambda_1(\Lambda)^2} \] %
which, together with the fact that $n^3 \leq 2^{2n}$, proves the claim.
\end{proof}

\begin{corollary}
\label{cor:ratio-lambda-i}
 Let $\tmB$ and $\mB \in \R^{m \times n}$ be $T$-\der bases of rank-$n$ lattices $\tilde{\Lambda}$ and $\Lambda$ respectively.

Then
\[ \left( 1 - \frac{ 2^{(T+2)n} \cdot \normmattwo{ \tmB - \mB }}{\min( \lambda_1(\Lambda),\lambda_1(\tilde{\Lambda}))}\right) \leq \frac{\lambda_j(\Lambda)}{\lambda_j(\tilde{\Lambda})} \leq \left( 1 + \frac{ 2^{(T+2)n} \cdot \normmattwo{ \tmB - \mB }}{\min( \lambda_1(\Lambda),\lambda_1(\tilde{\Lambda}))}\right) \]
and $ |\lambda_1(\tilde{\Lambda}) - \lambda_1(\Lambda)| \leq 2^{(T+2)n} \cdot  \normmattwo{ \tmB - \mB } $.
\end{corollary}
\begin{proof} Write $\mU \in \Z^{n \times n} \cap \mathrm{GL}_n(\R)$ for the (not necessarily unimodular) transformation such that $\mC = \mB \mU$ attains the successive minima, i.e., $\| \bc_j \| = \lambda_j(\Lambda)$, and the vectors $\bc_j$ are linearly independent. By \Cref{lemma:wessel}, we know that $\|\bu_j \| \leq 2^{(T+2)n} \cdot \lambda_j(\Lambda)/\lambda_1(\Lambda)$.

Since $\mU$ is non-singular, then
$\tilde{\mC} := \tmB \mU = \tmB (\bu_1,\ldots,\bu_n)$ consists of $n$ independent vectors. Hence $\lambda_j(\tilde{\Lambda}) \leq \max_{1 \leq k \leq j} \|\tilde{\bc}_k\|$, but we have
\[  \| \tilde{\bc}_k\| \leq \| \tilde{\bc}_k - \bc_k\| + \| \bc_k\| \leq \normmattwo{\tmB - \mB} \cdot  \| \bu_k \| + \lambda_k(\Lambda) \leq \lambda_k(\Lambda) \left( 1 + \frac{ 2^{(T+2)n} \cdot \normmattwo{\tmB - \mB}}{\lambda_1(\Lambda)}\right),  \]
hence, by applying the same reasoning but interchanging $\Lambda$ and $\tilde{\Lambda}$,
\[ \lambda_j(\tilde{\Lambda}) \leq \lambda_j(\Lambda)\left( 1 + \frac{ 2^{(T+2)n} \cdot \normmattwo{\tmB - \mB}}{\lambda_1(\Lambda)}\right) \mbox{ and } \lambda_j(\Lambda) \leq \lambda_j(\tilde{\Lambda})\left( 1 + \frac{ 2^{(T+2)n}\cdot \normmattwo{\tmB - \mB}}{\lambda_1(\tilde{\Lambda})}\right). \]
Instantiating this with $j = 1$, shows
$ |\lambda_1(\tilde{\Lambda}) - \lambda_1(\Lambda)| \leq 2^{(T+2)n} \cdot  \normmattwo{\tmB - \mB} $. Applying the inequality $1-x \leq 1/(1+x)$ (which holds for all $x \geq 0$) and replacing $\lambda_1(\Lambda)$ and $\lambda_1(\tilde{\Lambda})$ by their minimum, we obtain the final claim
\[ \left( 1 - \frac{ 2^{(T+2)n} \cdot \normmattwo{\tmB - \mB}}{\min( \lambda_1(\Lambda),\lambda_1(\tilde{\Lambda}))}\right) \leq \frac{\lambda_j(\Lambda)}{\lambda_j(\tilde{\Lambda})} \leq \left( 1 + \frac{ 2^{(T+2)n} \cdot \normmattwo{\tmB - \mB}}{\min( \lambda_1(\Lambda),\lambda_1(\tilde{\Lambda}))}\right). \]

\end{proof}

\begin{corollary} \label{cor:closenessbkz} Let $\tmB$ and $\mB \in \R^{m \times n}$ be $T$-\der bases of rank-$n$ lattices $\tilde{\Lambda}$ and $\Lambda$ respectively.
Let $\tilde{\mC} = \tmB \mU$ be a basis of $\tilde{\Lambda}$ satisfying
\[ \| \tilde{\bc}_j \| \leq q_j \cdot \lambda_n(\tilde{\Lambda})  ~\mbox{ for all $j \in \{1,\ldots,n \}$}, \]
for $q_j \in \R_{>0}$.
Suppose $\normmattwo{ \tmB - \mB} \leq \frac{1}{4} \cdot 2^{-(T+2)n} \cdot \min(\lambda_1(\Lambda),\lambda_1(\tilde{\Lambda}))$ and write $\mC = \mB \mU$. Then,
\[ \| \bc_j \| \leq 2 \cdot q_j \cdot \lambda_n(\Lambda)~\mbox{ for all $j \in \{1,\ldots,n \}$}. \]
\end{corollary}
\begin{proof}
We have $\|\bu_j\| \leq 2^{(T+2)n} \| \tilde{\bc}_j \|/\lambda_1(\tilde{\Lambda})$ by \Cref{lemma:wessel}, and hence
\begin{align*} 
\| \bc_j\| &\leq \| \bc_j -  \tilde{\bc}_j\| +  \| \tilde{\bc}_j\|  \leq \normmattwo{\tmB - \mB} \cdot \|\bu_j\| + \| \tilde{\bc}_j\| \\
& \leq  \| \tilde{\bc}_j \| \left( 1 + \frac{\normmattwo{\tmB - \mB} \cdot 2^{(T+2)n} }{\lambda_1(\tilde{\Lambda})} \right)
 \leq  q_j \cdot \lambda_n(\tilde{\Lambda}) \cdot \left( 1 + \frac{\normmattwo{\tmB - \mB} \cdot 2^{(T+2)n} }{\lambda_1(\tilde{\Lambda})} \right) \\
& \leq  q_j \cdot \lambda_n(\Lambda) \cdot \left( 1 + \frac{\normmattwo{\tmB - \mB} \cdot 2^{(T+2)n} }{\min(\lambda_1(\tilde{\Lambda}),\lambda_1(\Lambda))} \right)^2
\leq 2 \cdot q_j \cdot \lambda_n(\Lambda),
\end{align*}
where in the penultimate inequality, we used \Cref{cor:ratio-lambda-i} and in the last inequality we used the assumed upper bound on $\normmattwo{\tmB - \mB}$.
\end{proof}

\subsection{Conclusion}
\label{subsec:BKZ:conclusion}

\approxbkz*
\begin{proof} Using the HNF rational matrix $M_\ma$, we have a sequence of elements $(\alpha_1^{(0)},\ldots,\alpha_n^{(0)})$ generating $\ma$. The exact Minkowski basis is given by $\mB_{x \ma} = (\sigma(\alpha_j^{(0)}))_{\sigma,j}$, i.e., the columns of $\mB_{x \ma}$ are the Minkowski embeddings of $\alpha_j^{(0)}$ into $\nfr$.

Due to the fact that these numbers are irrational in general, the basis $\mB_{x \ma}$ can only be approximated; we call this approximation $\tmB_{x \ma}$. Note that the values of $\sigma(\alpha_j^{(0)})$ can reasonably be approximated by computing the roots of the defining polynomial of the number field $K$. This can be done efficiently by, for example, Newton iteration; this costs time polynomial in the desired bit precision. Hence, if the final required bit precision is polynomial in the size of the input (which we indeed will show to be the case), approximation is not going to have a significant influence on the running time.
\\ \noindent
\textbf{Computing the dual.}
So, suppose we have a rational approximation $\tmB_{x \ma}$ of $\mB_{x \ma}$ within polynomial bits of precision. By \Cref{lemma:inverseprecisionloss}, computing the dual basis of $\tmB_{x \ma}$ only looses $O(n^2 + \log |\dcrk| + (n+2) \size(x \ma))$ bits of precision. Therefore we can assume that the dual basis of $\tmB_{x \ma}$ approximates the dual basis of $\mB_{x \ma}$ with polynomial bits of precision.
\\ \noindent
\textbf{Applying \bkp on the approximate dual basis.}
Using the \bkp algorithm on the approximate dual basis, we have \Cref{proposition:efficientder}, which states we can compute
efficiently a $\Z$-basis $(x \alpha_1^{(1)},\ldots,x \alpha_n^{(1)})$
that is `$3$-\der' (see \Cref{def:der}). That means the dual basis
$(\bd_1,\ldots,\bd_n)$ of this new basis satisfies $\| \bd_j \| \leq 2^{3n} \cdot \lambda_j( (x \ma)^\vee)$, where $(x \ma)^\vee$ is the dual lattice of $x \ma$.

This new basis $(x \alpha_1^{(1)},\ldots,x \alpha_n^{(1)})$ is exact
and we can re-approximate the associated basis $\mB_{x \ma}^{(1)}$ by $\tmB_{x \ma}^{(1)}$. By \Cref{lemma:bothder}, a sufficient (polynomial) approximation $\tmB_{x \ma}^{(1)}$ can be shown to be $6$-\der as well.
\\ \noindent
\textbf{Applying Hanrot-Pujol-Stehl\'e.}
Without loss of generality, we can clear the denominators in $\tmB_{x \ma}^{(1)}$. Then, we apply the BKZ version of Hanrot-Pujol-Stehl\'e with provable bounds on the number of tours. Using \Cref{cor:HPS}, 
we conclude that there exists an algorithm using time $\poly(n, \size(\tmB_{x \ma}^{(1)}), \hkztime) = \poly(\log |\dcrk|, \allowbreak \size(x \ma), \hkztime)$
that outputs a
new basis $(\tilde{\bb}_1^{(2)},\ldots,\tilde{\bb}_n^{(2)}) = \tmB_{x \ma}^{(2)} = \tmB_{x \ma}^{(1)} \cdot \mU$ of the lattice $\mathcal{L}( \tmB_{x \ma}^{(1)})$ that satisfies
\[  \| \tilde{\bb}_j^{(2)} \| \leq n\cdot  \blocksize^{2n/\blocksize} \cdot \lambda_n( \mathcal{L}(\tmB_{x \ma}^{(1)})) \]
But we would like the new \emph{exact basis} $(x \alpha_1^{(2)},\ldots,x \alpha_n^{(2)})$, obtained by applying $\mU$ to $(x \alpha_1^{(1)},\ldots,x \alpha_n^{(1)})$ to satisfy above claim. For this we need the \emph{closeness lemma}.
\\ \noindent
\textbf{Applying the closeness lemma.}
Note that $\tmB_{x \ma}^{(2)} = \tmB_{x \ma}^{(1)} \mU$ is an approximation of $\mB_{x \ma}^{(2)}= \mB_{x \ma}^{(1)} \mU$, the exact Minkowski-basis of the $\Z$-basis $(x \alpha_1^{(2)},\ldots,x \alpha_n^{(2)})$. Since $\tmB_{x \ma}^{(1)}$ and $\mB_{x \ma}^{(1)}$ are both $T$-\der (with $T = 6$ and $3$ respectively), we can apply  \Cref{cor:closenessbkz}. A sufficient (polynomial bit precision) approximation $\tmB_{x \ma}^{(1)}$ of $\mB_{x \ma}^{(1)}$ then suffices to deduce that the vectors of $\mB_{x \ma}^{(2)}$ satisfy the same BKZ-like bounds (with a blow-up factor of $2$), hence,
\[ \| x \alpha_j^{(2)} \| \leq 2n \cdot \blocksize^{2n/\blocksize} \cdot \lambda_n( x \ma) \mbox{ for all } j \in \{1,\ldots,n\}. \]
\end{proof}

%% file: appendix/C-A2-RWproofs.tex
\section{Proof of~\texorpdfstring{\Cref{thm:random-walk-weak},}{} the random walk theorem} \label{sec:appendixRW}

\noindent This following proof of \Cref{thm:random-walk-weak} is put in the appendix because it is very similar to proofs in other work \cite[Section 3]{dBDPW} \cite[Chapter 4]{KoenThesis}, with the difference that it is here generalized to arbitrary moduli $\modu$ and finite-index subgroups of $\rayPic^0$, and a slight change of the definition of the log-space $H$.

\randomwalktheorem*

\noindent
We follow the proof of \cite[Section 3]{dBDPW}, with small adaptations. \\
\noindent
\textbf{The Hecke operator, characters and eigenvalues} \\
Recall that $\subpic \subseteq \rayPic^0$ is a finite-index subgroup of $\rayPic^0$ where we try to randomize over.
Putting $\primeset = \{ \mp \mbox{ prime ideal of } K ~|~  \norm(\mp) \leq B, [d^0(\mp)] \in \subpic \mbox{ and } \mp \nmid \modu \}$, we define the Hecke operator $\heckeop:L^2(\subpic) \rightarrow L^2(\subpic)$ by the following rule, for $f \in L^2(\subpic)$.
 \[ \heckeop(f) (x) := \frac{1}{|\primeset|} \sum_{\mp \in \primeset} f(x - [d^0(\mp)]). \]
This Hecke operator has the \emph{characters} $\chi \in \widehat{\subpic}$ as eigen functions, with eigenvalues $\lambda_\chi \in \C$ satisfying $|\lambda_\chi| \leq 1$ \cite[Section 3.2]{dBDPW}. The trivial character $\mathbf{1} \in \widehat{\subpic}$ can be shown to have eigenvalue $\lambda_{\mathbf{1}} = 1$.
\[ \heckeop(\chi) (\cdot)   := \frac{1}{|\primeset|} \sum_{\mp \in \primeset} \chi(\cdot  - [d^0(\mp)]) = \chi( \cdot) \cdot \underbrace{ \frac{1}{|\primeset|} \sum_{\mp \in \primeset} \chi(- [d^0(\mp)])}_{\lambda_\chi} \] 
For a fixed $\chi \in \widehat{\subpic}$, there are $[\rayPic^0:\subpic]$ extensions of this character 
to $\theta \in \widehat{\rayPic^0}$. These characters satisfy the following identity, which can be proved 
by using standard character orthogonality techniques.
\begin{equation} \label{eq:characterfilter} [\rayPic^0:\subpic]^{-1} \sum_{\substack{ \theta \in \widehat{\rayPic^0} \\ \theta|_{\subpic} = \chi }} \theta( -[d^0(\mp)]) = \left \{ \begin{matrix} \chi( -[d^0(\mp)]) & \mbox{ if } [d^0(\mp)] \in G \\ 0 & \mbox{ otherwise }  \end{matrix}   \right . \end{equation}
Putting $\overline{\primeset} = \{ \mp \mbox{ prime ideal of } K ~|~  \norm(\mp) \leq B  \mbox{ and } \mp \nmid \modu \}$ (note the absence of the class condition on $\mp$), we thus have, for any fixed $\chi \in \widehat{\subpic}$, using \Cref{eq:characterfilter},  
\begin{align} \label{eq:characterfiltering2} \lambda_\chi = \frac{1}{|\primeset|} \sum_{\mp \in \primeset} \chi(- [d^0(\mp)]) & = \frac{1}{|\primeset| \cdot [\rayPic^0:\subpic]} \sum_{\mp \in \overline{\primeset}} \sum_{\substack{ \theta \in \widehat{\rayPic^0} \\ \theta|_{\subpic} = \chi }} \theta( -[d^0(\mp)])
\\ & =  \frac{|\overline{\primeset}|}{|\primeset|} \cdot  \frac{1}{[\rayPic^0:\subpic]}  \sum_{\substack{ \theta \in \widehat{\rayPic^0} \\ \theta|_{\subpic} = \chi }} \frac{1}{|\overline{\primeset}|}  \sum_{\mp \in \overline{\primeset}} \theta( -[d^0(\mp)]) \label{eq:characterfiltering3} \end{align}

Assuming the Extended Riemann Hypothesis (as formulated in \cite[\textsection 5.7]{iwaniec2004analytic}), one can use results from analytic number theory \cite[Theorem 5.15]{iwaniec2004analytic} to derive the following asymptotic bound for non-constant characters $\theta \in \widehat{\rayPic^0}$ %
(see \cite[Section 3.3]{dBDPW}). %
\begin{equation} \label{eq:eigenvaluebound}  \frac{1}{|\overline{\primeset}|}  \sum_{\mp \in \overline{\primeset}} \theta( -[d^0(\mp)]) = O\left(\frac{ \log(B)  \log(B^n  \cdot |\dcrk| \cdot \norm(\modu) \cdot \infancond(\theta))}{B^{1/2}}\right),\end{equation}
where $\infancond(\chi)$ is the infinite part of the analytic conductor of the character $\chi$ (cf. \cite[Eq. (5.6)]{iwaniec2004analytic}), %
The proof of this result uses the Abel summation formula and is almost the same as in \cite[Section 3.3]{dBDPW} or \cite[Corollary 2.3.5]{wesolowski_phd} (see also \cite[Section 4.3.3]{KoenThesis}), with the sole difference that the analytic conductor gets an extra factor $\norm(\modu) = \norm(\moduz)\cdot 2^{|\modur|}$.

By applying \Cref{eq:characterfiltering2} for $\chi = 1 \in \widehat{G}$, the trivial character, one can get an upper bound on the fraction $\overline{|\primeset|}/|\primeset|$. By the fact that $\lambda_1 = 1$,
we obtain, from \Cref{eq:eigenvaluebound}
\begin{equation} \label{eq:boundPbarP}
 \frac{|\primeset| \cdot [\rayPic^0:\subpic]}{|\overline{\primeset}|} =  1 +  O\left(\frac{ [\rayPic^0:\subpic] \cdot \log(B)  \log(B^n  \cdot |\dcrk| \cdot \norm(\modu) \cdot \infancond(1))}{B^{1/2}}\right).
\end{equation}
Hence, choosing %
$B$ such that the value of the big-O of \Cref{eq:boundPbarP} is strictly bounded by a half%
\footnote{Notice that, in order the bound \Cref{eq:fulleigenvaluebound} to be non-trivial (i.e., smaller than $1$), 
we already need $B$ to satisfy this condition, so it does not add an extra requirement on $B$. In \Cref{eq:eigenvaluebound}, the analytic conductor satisfies $\infancond(\theta) \leq 4^n$
for extensions $\theta$ of the unit character on $\subpic$, as $\Tmodu \subseteq \subpic$. This follows from the definition of the analytic conductor \cite[Eq. (5.7)]{iwaniec2004analytic} (denoted $\mathfrak{q}(f) = \mathfrak{q}(f,0)$ there), and the fact that the local parameters (denoted $\kappa_j$ in \cite{iwaniec2004analytic}) are zero for characters trivial on $\Tmodu$.}%
, we obtain%
\begin{equation}  \frac{|\overline{\primeset}|}{|\primeset|} \cdot  \frac{1}{[\rayPic^0:\subpic]} \leq 2\label{eq:boundpp}   \end{equation}

Combining \Cref{eq:characterfiltering3,eq:eigenvaluebound,eq:boundpp}, we obtain that the eigen value $\lambda_\chi$ for non-constant characters $\chi \in \widehat{\subpic}$ under the Hecke operator $\heckeop$ satisfies
\begin{equation} \label{eq:fulleigenvaluebound}  \lambda_\chi = O\left(\frac{ [\rayPic^0:\subpic] \cdot \log(B)  \log(B^n  \cdot |\dcrk| \cdot \norm(\modu) \cdot \infancond(\chi))}{B^{1/2}}\right), \end{equation}
\\ \noindent
\textbf{The infinite analytic conductor $\infancond$} \\
The infinite analytic conductor $\infancond(\chi) \in \R_{>0}$ is a number 
quantifying the amount of oscillation of the character $\chi \in \widehat{\subpic}$. By restricting a character $\chi \in \widehat{\subpic}$ to the ray unit torus\footnote{Due to the fact that $[\rayPic^0:\subpic]$ is finite and the ray unit torus satisfies $\Tmodu \subseteq \rayPic^0$, it is also included in $\subpic$.} 
$\Tmodu = H/\Log(\units \cap \Kmodu) \subseteq \subpic$, %
 one gets a character on $\Tmodu$, which can be uniquely associated with a dual lattice point $\dell \in \lograyunits^\vee$, where $\lograyunits = \Log(\units \cap \Kmodu)$. The infinite part of the analytic conductor is in that case equal to\footnote{This definition of the infinite part of the analytic conductor is slightly different to that of Kowalski \& Iwaniec \cite[Equation (5.7)] {iwaniec2004analytic} (with $s = 1$) which is due to the difference in definitions of the real and the complex $L$-functions (see also \cite[Remark 4.13]{KoenThesis}). This difference is solved by putting $L_\R(s)L_\R(s+1) = L_\C(s)$, (see \cite[Chapter 7, Proposition 4.3(iv)]{neukirch2013algebraic}). } (see \cite[Equation (5.3), (5.4), (5.7)] {iwaniec2004analytic} and \cite[Equation (3.3.6), (3.3.12)]{miyake2006modular} or  \cite[Section 3.4]{dBDPW} or \cite[Section 4.3.4]{KoenThesis})
\begin{equation} \label{eq:infanconddef} \infancond(\chi) = \prod_{\nu \mbox{\footnotesize{ real }}} (3 + |2\pi \dell_{\nu}|) \prod_{\nu \mbox{\footnotesize{ complex }}} (3 + |2\pi\dell_{\nu}|)\big (3 + |i 2\pi\dell_{\nu}+1| \big) \end{equation}
Here, the components of the dual lattice point in the ray logarithmic unit lattice $\lograyunits^\vee$ are parametrized by the places $\pl$. As a result, by applying the geometric and arithmetic inequality for vector norms, we obtain the following bound on the infinite analytic conductor.
\[ \infancond(\chi) \leq \left( 4 + 4 \pi \vnorm{\dell}/\sqrt{n} \right)^n, \]
where $\dell \in \lograyunits^\vee$ is the unique dual logarithmic ray unit lattice point associated with the character $\chi|_{\Tmodu} \in \widehat{\Tmodu}$. \\
\noindent 
\textbf{Fourier analysis of the Gaussian} \\
Since the initial distribution of the random walk is a Gaussian $\gaussian_\sd:H \rightarrow \R, x \mapsto e^{-\pi \|x\|^2/\sd^2}$ over the hyperplane $H = \log \nfr^0$, 
we are particularly interested in the behavior of the distribution resulting from applying the Hecke operator $N$ times, i.e., $\heckeop^N(\gaussian_\sd)$. By standard Fourier computations, one can prove that periodized Gaussian function\footnote{The periodization is defined as follows: $\gaussian_{\sd}|^{\Tmodu}(x) :=  \sum_{\ell \in \lograyunits} \gaussian_{\sd}(x + \ell)$.} $\sd^{-\dimh}\gaussian_\sd|^{\Tmodu} \in L^2(\Tmodu)$ satisfies
\begin{equation} \label{eq:gaussiandecomp} \sd^{-\dimh}\gaussian_\sd|^{\Tmodu} = \sum_{\dell \in \logunitsmodu^\vee} a_{\dell} \chi_{\dell} \end{equation}
where $a_{\dell} = \frac{1}{\vol(\Tmodu)} \gaussian_{1/\sd}(\dell)$%
, where $\logunitsmodu^\vee$ is the dual lattice of the log ray unit lattice $\logunitsmodu$, and where $\chi_{\dell}(x) = e^{-2\pi i \inner{x,\dell}}$ is a function on $\Tmodu$, i.e., $\chi_{\dell} \in \widehat{\Tmodu}$.

By standard character arguments, one can simply write each $\chi_{\dell} \in \widehat{\subpic}$ (which is zero everywhere except on $\Tmodu \subseteq \subpic$) as an average of all characters in $\widehat{\subpic}$ that restrict to $\chi_{\dell}$ on $\Tmodu$.
This results in the following identity (see also \cite[\textsection 3.6]{dBDPW}), where $\chi' \in \widehat{\subpic}$ 
ranges over the characters of $\subpic\subseteq \rayPic^0$.
\begin{equation} \label{eq:fullgaussum}
\sd^{-\dimh} \gaussian_\sd|^{\Tmodu} = \frac{1}{\volabs{\subpic}}\sum_{\chi_{\dell} \in \widehat{\Tmodu}} \gaussian_{1/\sd}(\dell)  \sum_{\chi'|_{\Tmodu} = \chi_{\dell}}  \chi'.
\end{equation}
\\ 
\noindent 
\textbf{Splitting up the sum} \\
By splitting up the sum of \Cref{eq:fullgaussum} into a `unit part', a `low frequency part' where $\| \dell \| < r$ and a `high frequency part' where $\| \dell \| \geq r$, we obtain the following decomposition, for any $r > 0$, where $\chi' \in \widehat{\subpic}$
ranges over the characters of $\subpic\subseteq \rayPic^0$.
\begin{align*}
 \volabs{\subpic} \cdot \sd^{-\dimh} \cdot  \gaussian_\sd|^{\Tmodu} & = \underbrace{ \mathbf{1}_{\subpic} }_{\mbox{\footnotesize{Unit character}}} \\ &+  \underbrace{ \sum_{\substack{\chi_{\dell} \in \widehat{\Tmodu} \\  \|\dell\| < r  }} \gaussian_{1/\sd}(\dell)  \sum_{\substack{\chi'|_{\Tmodu} = \chi_{\dell} \\ \chi' \neq \mathbf{1} }}  \chi'}_{\mbox{\footnotesize{Low frequency characters}}} + \underbrace{ \sum_{\substack{\chi_{\dell} \in \widehat{\Tmodu} \\  \|\dell\| \geq r }} \gaussian_{1/\sd}(\dell)  \sum_{\chi'|_{\Tmodu} = \chi_{\dell}}  \chi'}_{\mbox{\footnotesize{High frequency characters}}},
\end{align*}
Now, applying the Hecke operator $\heckeop^N$ times to this equation, and taking the absolute value of $\volabs{\subpic} \cdot \heckeop^N(\gaussian_\sd|^{\Tmodu}) - \mathbf{1}_{\subpic}$, we obtain, by the Pythagorean theorem (using that $\| \chi' \overline{\chi'}\|_2^2 = |\subpic|$),
\begin{align} 
\label{eq:normsplitlow} 
\vnorm{ \volabs{\subpic} \cdot  \heckeop^N (\sd^{-\dimh} \gaussian_\sd|^{\Tmodu} ) - \mathbf{1}_{\subpic} }^2  \!\! & = \underbrace{|\subpic|  \sum_{\substack{\chi_{\dell} \in \widehat{\Tmodu} \\  \|\dell\| < r }} \gaussian_{1/\sd}^2(\dell)  \sum_{\substack{\chi'|_{\Tmodu} = \chi_{\dell} \\ \chi' \neq \mathbf{1} }} |\lambda_{\chi'}|^{2N} }_{\mbox{Low frequency}} \\
& +  \underbrace{ |\subpic| \sum_{\substack{\chi_{\dell} \in \widehat{\Tmodu} \\ \|\dell\| \geq r }} \gaussian_{1/\sd}^2(\dell)  \sum_{\chi'|_{\Tmodu} = \chi_{\dell}} |\lambda_{\chi'}|^{2N} }_{\mbox{High frequency}} .
\label{eq:normsplithigh} 
\end{align}
We will bound the parts \Cref{eq:normsplitlow} and \Cref{eq:normsplithigh} separately. For the latter, we have
\begin{align} \label{eq:bigellstar} |\subpic| \sum_{\substack{\chi_{\dell} \in \widehat{\Tmodu} \\ \|\dell\| \geq r }} \gaussian_{1/\sd}^2(\dell)  \sum_{\chi'|_{\Tmodu} = \chi_{\dell}} |\lambda_{\chi'}|^{2N}
& \leq |\subpic| \cdot |\subpic/\Tmodu| \sum_{\substack{ \dell \in \logunitsmodu^\vee \\ \|\dell\| \geq r }} \gaussian_{\frac{1}{\sqrt{2}\sd}}(\dell) \\ & \leq |\subpic| \cdot |\subpic/\Tmodu|  \cdot \banas{\sqrt{2} r \sd}^{(\dimh)} \cdot \gaussian_{\frac{1}{\sqrt{2}\sd}}(\logunitsmodu^\vee),
\end{align}
where the last inequality follows from Banaszczyk's tail bound \cite{Banaszczyk1993} and the assumption that $rs > \sqrt{\dimh/(4 \pi)}$,
where $\banas{z}^{(n)} :=  \left ( \frac{2 \pi e z^2}{n } \right)^{n/2} e^{-\pi z^2}$, for which holds $\banas{t}^{(n)} \leq e^{-t^2}$ for all $t \geq \sqrt{n}$ (see also  \cite[\textsection 3.6]{dBDPW}).

To bound the share of the low-frequency characters, 
we use the bound from \Cref{eq:fulleigenvaluebound}, $|\lambda_{\chi'}|\leq \allowbreak O\left(\frac{[\rayPic^0:\subpic] \cdot \log(B) \log(  B^n \cdot |\dcrk| \cdot \norm(\modu) \cdot \infancond(\chi') )}{B^{1/2}}\right)$. Since these characters have a `low frequency', their analytic conductor $\infancond(\chi')$ is bounded, namely, $\infancond(\chi') \leq (4 + 2 \pi r/\sqrt{n})^n$ for any $\chi' \in \widehat{\subpic}$ such that $\chi'|_{\Tmodu} = \chi_{\dell}$ for some $\dell \in \logunitsmodu^\vee$ with $\|\dell \| < r$. Therefore, 
\[ |\lambda_{\chi'}|\leq c \allowbreak =O\left(\frac{[\rayPic^0:\subpic] \cdot \log(B) \log(B^n \cdot |\dcrk| \cdot \norm(\modu) \cdot (4 + 4\pi r/\sqrt{n})^n )}{B^{1/2}}\right). \]  We then obtain
 \begin{equation} \label{eq:smallellstar} %
 |\subpic| \sum_{\vnorm{\dell} \leq r}  \gaussian_{1/\sd}^2(\dell)  \underbrace{\sum_{\chi'|_{\Tmodu} = \chi_{\dell}} |\lambda_{\chi'}|^{2N}}_{ \leq |\subpic/\Tmodu| \cdot c^{2N}} \leq  |\subpic | \cdot  |\subpic/\Tmodu| \cdot  c^{2N} \cdot  \gaussian_{\frac{1}{\sqrt{2}\sd}}(\logunitsmodu^\vee)  \end{equation}
 We obtain the following bound by combining \Cref{eq:bigellstar,eq:smallellstar}, dividing by $\volabs{\subpic}$ (which accounts to dividing by $\volabs{\subpic}^2$ under the square $2$-norm)
and using the identity $\volabs{\subpic} = |\subpic/\Tmodu| \cdot \vol(\Tmodu)$. Assuming the Extended Riemann Hypothesis for Hecke L-functions (e.g., \cite[\textsection 5.7]{iwaniec2004analytic}),  and
for all $r, s > 0$ with $rs > \sqrt{\frac{\dimh}{4\pi}}$, we have
\begin{equation} \label{eq:2normbound} \vnorm{ \heckeop^N(\sd^{-n} \gaussian_\sd |^{\Tmodu} ) - \frac{1}{\volabs{\subpic}} \mathbf{1}_{\subpic} }^2 \leq
\frac{\gaussian_{\frac{1}{\sqrt{2}\sd}}(\logunitsmodu^\vee)}{\vol(\Tmodu)}\left( c^{2N} + \beta_{\sqrt{2} r \sd}^{(\dimh)} \right)   \end{equation}
with $c =  O\left(\frac{[\rayPic^0:\subpic] \cdot \log(B) \log(B^n \cdot |\dcrk| \cdot \norm(\modu) \cdot (4 + 4\pi r/\sqrt{n})^n )}{B^{1/2}}\right)$. \\
\noindent 
\textbf{Tuning parameters} \\ 
Let $1 > \varepsilon > 0$, $\sd > 0$ and $k \in \R_{>0}$ be given. We have $ \gaussian_{\frac{1}{\sqrt{2}s}}(\logunitsmodu^\vee) \leq \gaussian_{1/\sprime}(\logunitsmodu^\vee) \leq 2 \covol(\logunitsmodu)/\sprime^\dimh$ by smoothing arguments (see \Cref{lemma:smoothing}).
Using that inequality and H\"older's inequality (i.e., $\| f \cdot 1 \|_1 \leq \normtwoexplicit{f} \normtwoexplicit{1}$), noting that $\normtwoexplicit{1_{\subpic}}^2 = \volabs{\subpic}$%
, we obtain, for each $r > \sqrt{\dimh}/(\sqrt{2} \sd)$, 
\begin{align} & ~~~~\left\| \Walk_{\subpic}(B,N,\sd) - \unif(\subpic)\right\|_1^2  \\ 
& \leq \volabs{\subpic} \cdot \left \normtwoexplicit{\heckeop^N(\sd^{-\dimh} \gaussian_{\sd}|^{\Tmodu}) - \frac{1}{|\subpic|} \mathbf{1}_{\subpic} \right }^2
 \\ 
& \leq  |\subpic/\Tmodu|   
\cdot \gaussian_{\frac{1}{\sqrt{2}\sd}}(\logunitsmodu^\vee) \cdot  (c^{2N} + \banas{\sqrt{2}rs}^{(\dimh)})  \\
  & \leq 2  \cdot |\subpic|    %
  \cdot  \sprime^{-\dimh}  \cdot ( c^{2N} + \banas{\sqrt{2}rs}^{(\dimh)}) \label{eq:boundL1}
 \end{align}
Here, $c = O\left(\frac{ [\rayPic^0:\subpic] \cdot  \log(B) \log(B^n \cdot |\dcrk| \cdot \norm(\modu) \cdot (4 + 2\pi r/\sqrt{n})^n )}{B^{1/2}}\right)$. We proceed by bounding the two summands in \Cref{eq:boundL1} separately.

\begin{itemize}
 \item By putting\footnote{We use the bound $\banas{\alpha}^{(\dimh)} \leq e^{-\alpha^2}$ for $\alpha \geq \sqrt{\dimh}$} $r$ equal to
\[\frac{1}{\sqrt{2} \sd} \cdot  \max \left ( \sqrt{\dimh}, \sqrt{2 + \dimh \log (1/\sprime) + 2 \log(1/\varepsilon)  
 +\log\volabs{\subpic} 
} \right ) \] 
we deduce that 
$2 \cdot \volabs{\subpic} \cdot\sprime^{-\dimh} \cdot \beta_{\sqrt{2}r \sd}^{(\dimh)}  \leq \varepsilon^2/2$. 
\item Subsequently, choose\footnote{In this bound on $B$ one would expect an additional $\log \log \volabs{\subpic} \leq \log \log\volabs{\rayPic^0}$. But as it is bounded by $\log(\log( |\dcrk| \norm(\modu)))$ (see \Cref{lemma:boundvolraypicappendix}), it can be put in the hidden polylogarithmic factors.}
a $B = \tilde{O}( [\rayPic^0:\subpic]^2 \cdot n^{2k} [\log(|\dcrk|\norm(\modu))^2 + n^2 \log(r)^2] )$, i.e., 
\begin{align*} B =  \tilde{O}\Big( [\rayPic^0:\subpic]^2 \cdot n^{2k} \cdot \big [ \log(|\dcrk| \norm(\modu))^2 & + n^2 \log(1/\sprime)^2 \\ & +  n^2 \log(\log(1/\varepsilon))^2  \big] \Big) \end{align*}
such that $c \leq 1/n^k$, 
where $c=O\left(\frac{[\rayPic^0:\subpic] \cdot  \log(B) \log(B^n \cdot |\dcrk| \cdot \norm(\modu) \cdot (4 + 4\pi r/\sqrt{n})^n )}{B^{1/2}}\right)$. 
Lastly, taking any integer $N \geq \frac{1}{2k\log n} \cdot ( \dimh \cdot \log(1/\sprime) + 2\log(1/\varepsilon) + 
 \log\volabs{\subpic}   
+ 2)$ and noting that $c^{\frac{1}{k \log n}} \leq 1/e$, we deduce that
$ 2 \cdot 
\volabs{\subpic} \cdot 
\sprime^{-\dimh} \cdot c^{2N} 
\leq  \varepsilon^2/2$. 
\end{itemize}
Combining, we can bound the right-hand side of \Cref{eq:boundL1} by $\varepsilon^2$.  
Taking square roots gives the final result.
\qed

\section{Sampling of random prime ideals} \label{A:primegsample}
\noindent In this section we show how to sample prime ideals $\mp$ that 
are coprime with $\modu$ and whose class $[d^0(\mp)]$ fall into
$\subpic \subseteq \rayPic^0$, using an \emph{oracle} that can check
whether an ideal $\ma$ satisfies $[d^0(\mp)] \in \subpic$ or not. 
The algorithm doing this is a slight generalization of the algorithm 
sampling prime ideals in {\cite[Lemma 2.2]{dBDPW}}. 
Before giving this algorithm (together with its properties), we first 
need the following result on the number of prime ideals that are coprime
with $\modu$.

\begin{lemma}[\normalfont{ERH}] \label{lemma:primecountingfunction} Let $\modu \subseteq \OK$ be an ideal modulus and denote
\[ \pi_K^\modu(x) = |\{ \mp \in \idealsmodu ~|~ \mp \mbox { prime and } \norm(\mp) \leq x \}| \]  for the number of prime ideals not dividing $\modu$ and having norm bounded by $x \in \R$. Let $\omega(\modu)$ denote the number of different prime ideal divisors of $\modu$.

Then, there exists $x_0 \in \softO(\log^2|\dcrk|+\omega(\modu))$ such that for all $x \geq x_0$
we have
\[ \pi_K^\modu(x) \geq \frac{x}{4 \log x}. \]
\end{lemma}
\begin{proof} Denote $\pi_K(x) = |\{ \mp \in \ideals ~|~ \mp \mbox { prime and } \norm(\mp) \leq x \}|$, i.e., whenever $\modu = \OK$. We will prove the statement for this specific case first. By simplifying an explicit result of Greni\'{e} and Molteni \cite[Corollary 1.4]{Greni__2015}, we obtain, under the Extended Riemann Hypothesis\footnote{In the paper of Greni\'{e} and Molteni \cite[Corollary 1.4]{Greni__2015}, only the Dedekind zeta function $\zeta_K(s) = \sum_{\ma} \norm(\ma)^{-s}$ needs to satisfy the condition that all of its non-trivial zeroes lie at the vertical line $\Re(s) = 1/2$.},
\[ \left |\pi_K(x) - \pi_K(3) - \int_3^x \frac{du}{\log u} \right|  \leq \sqrt{x} [ 6 \log |\dcrk| + 4 n \log x + 14].\]
 Therefore, we have
\begin{align*} \pi_K(x) & \geq \int_3^x \frac{du}{\log u} -  \sqrt{x} [ 6 \log |\dcrk| + 4 n \log x + 14] \\ & \geq \frac{x}{\log x} -  \sqrt{x} \log (x) [ 6 \log |\dcrk| + 4 n + 14] \\ & =  \frac{x}{\log x} \left(1 - \frac{\log(x)^2( 6 \log |\dcrk| + 4 n + 14)}{\sqrt{x}} \right) \geq \frac{x}{2 \log x}  \end{align*}
where the first inequality follows from omitting $\pi_K(3)$ and the second inequality from $\int_3^x \frac{du}{\log u} \geq \frac{x}{\log x}$ and from the assumption that $x > x_0$, where $x_0 \in \softO(\log^2|\dcrk|)$ such that $\frac{\log(x_0)^2( 6 \log |\dcrk| + 4 n + 14)}{\sqrt{x_0}} < 1/2$.

For the general case of $\modu \neq \OK$, we need to avoid $\modu$; so writing $\omega(\modu)$ for the
number of different prime ideals dividing $\modu$, we obtain
\[ \pi_K^\modu(x) \geq  \pi_K(x)- \omega(\modu) \geq \frac{x}{2 \log x} \left( 1 - \frac{2\cdot \omega(\modu) \cdot \log x}{x} \right) \geq \frac{x}{4 \log x}.
\]
Where the last inequality can be deduced from the fact that $x > x_0$ and where $x_0 = \softO(\log^2|\dcrk| + \omega(\modu))$ is chosen such that $\frac{2\cdot \omega(\modu) \cdot \log x_0}{x_0} < 1/2$.
\end{proof}

The following lemma is a slight generalization of {\cite[Lemma 2.2]{dBDPW}},
where we demand the Arakelov ray class of the sampled prime ideals
to be lying in $\subpic$, some subgroup of $\rayPic^0$.
\primegsampling*
\begin{proof} The algorithm can be described as follows. Sample a uniform integer from $[0,B]$ and check if it is prime (if not, output `failure'). If it is, name the prime $p$, factor $p \OK$ into prime ideals of $\OK$ and list the different prime factors $\{\mp_1,\ldots,\mp_k\}$ that have norm bounded by $B$, do not divide $\moduz$ and satisfy $[d^0(\mp_j)] \in \subpic$. If this set is empty, output `failure'; otherwise, choose one $\mp_j$ uniformly at random in $\{\mp_1,\ldots,\mp_k\}$ and output it with probability $k/n$. Otherwise, output `failure'.

Let $\mathfrak{q} \in \mathcal{P}_B$ be arbitrary, and let $p = \mathfrak{q} \cap \Z$ the prime `below' $\mathfrak{q}$. Then the probability of sampling this $\mathfrak{q}$ equals $\frac{1}{nB}$, namely, $\frac{1}{n}$ times the probability of sampling $p$. So, the probability of sampling successfully (that is, no failure) equals
\[ \frac{|\mathcal{P}_B|}{nB} \geq \frac{1}{8 \cdot [ \rayPic^0:\subpic] \cdot n \log B} \]
since, by \Cref{eq:boundPbarP}, denoting $\overline{\mathcal{P}}_B = \{ \mp \mbox{ prime ideal of K } ~|~ \norm(\mp) \leq B, \mp \nmid \moduz \}$,
\[ \frac{|\mathcal{P}_B|}{|\overline{\mathcal{P}}_B|} \geq \frac{1}{2 \cdot [\rayPic^0:\subpic]}, \]
and by \Cref{lemma:primecountingfunction}, we have $|\overline{\mathcal{P}}_B| \geq \frac{B}{4 \log B}$ (by adequately increasing $B_0$ so that $B_0 \geq x_0$ from \Cref{lemma:primecountingfunction}).

The most costly part of the algorithm is the factorization of a prime $p \leq B$ in $\OK$. This can be performed using Kummer-Dedekind algorithm, which essentially amounts to factoring a degree $n$ polynomial modulo $p$.
Using Shoup's algorithm~\cite{Shoup1995} (which has complexity $O(n^2 + n \log p)$ \cite[\textsection 4.1]{VONZURGATHEN20013}) and the fact that the algorithm needs to repeat $[ \rayPic^0:\subpic] \cdot n \log B$ times to get constant success probability, yields the complexity claim and the query complexity for $\mathbf{O}_\subpic$.
\end{proof}

%% file: appendix/C-A3-lambdan_upperbound.tex
\section{Upper bound on the \texorpdfstring{$n$}{n}th successive minimum of \texorpdfstring{$\OK$}{OK} by Bhargava et al.} \label{A:lambdan}
\noindent The following theorem and its proof is a copy of that of Bhargava et al.~\cite[Theorem 3.1]{bhargava2020}, with the difference that it is applied to the infinity norm
and has explicit constants everywhere. A similar result, but also without explicit constants, can be found in an article by Peikert and Rosen \cite[Lemma 5.4]{PeikertRosen06}.
\begin{theorem}[Bhargava et al. \cite{bhargava2020}]%
\label{theorem:nthminimum} Let $K$ be any number field of degree $n$ and let $\OK$ be its ring of integers. Let $\OK \subseteq \nfr$ 
have the structure of a lattice via the Minkowski embedding (see \Cref{sec:main-preliminaries}), and denote
$\lambda_j^\infty(\OK)$ for the $j$-th successive minimum with respect to the infinity norm in $\nfr$. Then
\[ \lambda_n^\infty(\OK) \leq |\dcrk|^{1/n}. \]
\end{theorem}

\begin{proof} Let $\alpha_j \in \OK$ attain the successive minima for the
infinity norm $\lambda_j^{\infty}(\OK)$ for $j \in \{1,\ldots,n\}$, with $\alpha_1 = 1$.
For any element $\beta \in \OK$, we write $\beta = \sum_{j = 1}^n [\beta]_j \alpha_j$,
i.e., $[\beta]_j$ are the coordinates of $\beta$ with respect to $(\alpha_1,\ldots,\alpha_n)$.

For $2 \leq k,\ell \leq n-1$ consider the $(n-2)\times(n-2)$-matrix $C = ([ \alpha_k \alpha_\ell]_{n})$, %
i.e., the matrix consisting of the 
coordinates of $\alpha_k \alpha_\ell$ with respect to $\alpha_n$. We will 
show at the end of this proof that this is a non-degenerate matrix, implying that there are no
zero rows or columns. In other words, there exists a permutation $\pi: \{2,\ldots,n-1\}
\rightarrow \{2,\ldots,n-1\}$ such that $[ \alpha_k \alpha_{\pi(k)}]_{n} \neq 0$ for all $k \in \{2,\ldots,n-1\}$.

So, the product $\alpha_k \alpha_{\pi(k)} \in \OK$ extends $\{ \alpha_1,\ldots,\alpha_{n-1} \}$
to a $n$-dimensional lattice; therefore we have $\|\alpha_k\|_\infty\|\alpha_{\pi(k)}\|_\infty \geq \|\alpha_k \alpha_{\pi(k)}\|_\infty \geq \lambda_n^{\infty}(\OK)$. Taking products over all $k \in \{2,\ldots,n-1\}$ we obtain
\[ \prod_{k = 2}^{n-1} \|\alpha_k\|_\infty^2 =  \prod_{k = 2}^{n-1} \|\alpha_k\|_\infty\|\alpha_{\pi(k)}\|_\infty \geq \big(\lambda_n^{\infty}(\OK) \big)^{n-2}. \]
Multiplying above equation by $\|\alpha_1\|_\infty^2 = 1$ and $\|\alpha_n\|_\infty^2 = \lambda_n^{\infty}(\OK)^2$,
and using Minkowski's second inequality%
\footnote{Note that Minkowski's second inequality as stated in~\cite[Chap. VIII, Thm. 5]{cassels2012introduction} only states that $\prod_{k = 1}^n \lambda_k^{\infty}(\Lambda) \cdot \vol(B) \leq 2^n \cdot \covol(\Lambda)$, where $B = \{ x \in K_\R\,|\, \|x\|_\infty \leq 1\}$ ($B$ is a ball of dimension $n$, living in a real vector space of dimension $2n$ since our lattice is not full rank). To compute the volume of $B$, observe that $B$ consists in $n_\R$ orthogonal copies of $B_\R = \{x \in \C\,|\, x \in \R, |x| \leq 1 \}$ and $n_\C$ orthogonal copies of $B_\C = \{(x, \bar{x}) \in \C^2\,|\, |x| \leq 1 \}$. One can check that $\vol(B_\R) = 2$ and $\vol(B_\C) = 2\pi$, leading to $\vol(B) = 2^{r_\R+r_\C} \cdot \pi^{r_\C}$. Minkowski's second theorem then implies that $\prod_{k = 1}^n \lambda_k^{\infty}(\Lambda) \leq \covol(\Lambda)$ as desired.}%
~\cite[Chap. VIII, Thm. 5]{cassels2012introduction} $\prod_{k = 1}^n \lambda_k^{\infty}(\Lambda) \leq \covol(\Lambda)$, we obtain
\[ |\dcrk| \geq \prod_{k = 1}^{n} \|\alpha_k\|_\infty^2 \geq \big(\lambda_n^{\infty}(\OK) \big)^{n}. \]
It remains to prove that $C= ([ \alpha_k \alpha_\ell]_{n})$ is non-degenerate. Suppose it is not, and there
exists $d_\ell$ for $\ell \in \{2,\ldots,n-1\}$ (not all zero) such that
\[ \Big [ \sum_{\ell = 2}^{n-1} d_\ell \alpha_k \alpha_\ell \Big]_{n} = \sum_{\ell = 2}^{n-1} d_\ell [ \alpha_k \alpha_\ell]_{n} = 0 \mbox{ for all } k \in \{2, \ldots,n-1\} \]
Writing $\beta = \sum_{\ell = 2}^{n-1}d_\ell \alpha_\ell$, this means that, for all $k \in \{1,\ldots,n-1\}$, $\alpha_k \beta$ lies in the 
span of $(\alpha_1,\ldots,\alpha_{n-1})$. In other words, $L = \Q\alpha_1 + \ldots + \Q\alpha_{n-1}$ is
$\Q(\beta)$-invariant, i.e., a $\Q(\beta)$-vector (strict) subspace of $K$. That is, $\dim_{\Q(\beta)}(L) \leq \dim_{\Q(\beta)}(K) - 1$.
But then
\begin{align*} n-1 &= \dim_\Q(L) = \dim_{\Q(\beta)}(L)\cdot [\Q(\beta):\Q] \\ 
&\leq (\dim_{\Q(\beta)}(K) - 1) \cdot [\Q(\beta):\Q] = n - [\Q(\beta):\Q], \end{align*}
yielding $[\Q(\beta):\Q] = 1$, i.e., $\beta \in \Q$, which is impossible by the fact that 
$\beta = \sum_{\ell = 2}^{n-1}d_\ell \alpha_\ell$ is assumed to be non-zero and has no $\alpha_1 = 1$ part.

We conclude that $C$ is non-degenerate, which finishes the proof.
\end{proof}

%% file: appendix/C-A4-f_upperbound.tex
\section{Upper bound on a defining polynomial of a number field} \label{A:boundf}
\noindent The following two lemmas 
show together that for any number field $K$ there exists a 
polynomial $f \in \Z[x]$ such that $K = \Q[x]/(f(x))$ and  
so that $\size(f) \leq O( \log^2|\dcrk|)$. That is, each field has a `small' defining polynomial. 
The first lemma proves the existence of a primitive element whose embeddings are relatively small.
\begin{lemma}[{Bounded primitive element, adapted from~\cite[p.59]{samuel2013algebraic}}] \label{lemma:pet}
Let $K$ be a number field of degree $n$. There exists an element $\alpha \in \OK$ such that $K = \Q(\alpha)$ and $\prod_{\sigma: K \hookrightarrow \C} \max(1,|\sigma(\alpha)|) \leq 2^{2n} \cdot |\dcrk|$.
\end{lemma}

\begin{proof}
	A proof of this statement can be found in~\cite[p.59]{samuel2013algebraic}, in the proof of Theorem~3, Section~4.3. There, it is proven that any number field $K$ admits a primitive element $\alpha \in \OK$ that lies in some compact set $B$. This set $B$ is such that, if $K$ has at least one real embedding (say $\sigma_1$), then any $x \in B$ satisfies $|\sigma_1(x)| \leq 2^n |\dcrk|^{1/2}$ and $|\sigma_i(x)| \leq 1/2$ for $i \geq 2$. And if $K$ has no real embedding, then any $x \in B$ satisfies $|\sigma_1(x)|=|\sigma_2(x)| \leq 2^{n} |\dcrk|^{1/2}$ (assume $\sigma_2 = \overline{\sigma_1}$), and $|\sigma_i(x)| \leq 1/2$ for $i \geq 3$. In both cases, any $x \in B$ satisfies $\prod_{\sigma: K \hookrightarrow \C} \max(1,|\sigma(x)|) \leq 2^{2n} \cdot |\dcrk|$, so this is in particular the case of $\alpha$.
\end{proof}

\begin{lemma} \label{lemma:smallf} Let $K$ be a number field of degree $n$. Then there exists a monic integral irreducible polynomial $g = \sum_{i =0}^n g_i x^i \in \Z[x]$ such that $K \simeq \Q[x]/g(x)$ with
\[ \max_i \log |g_i|  =  O( \log |\dcrk| )\]
and hence 
\[  \size(g) := \sum_i \log |g_i| =  O(\log^2 |\dcrk|).  \]
\end{lemma}

\begin{proof} 
	Let $\alpha \in \OK$ be as given by Lemma~\ref{lemma:pet}, and let $g = \sum_{i = 0}^n g_i x^i$ be the minimal polynomial of $\alpha$. Since $\alpha$ is a primitive element of $K$, then $K \simeq \Q[x]/g(x)$. Moreover, because $\alpha \in \OK$, then $g$ has integer coefficients. 
	Finally, we have 
\begin{align*} 
	|g_j| &=  \Big|\sum_{ \substack{J \subseteq \{1,\ldots,n\} \\ |J| = n-j} } \prod_{i \in J} \sigma_i(\alpha)\Big| \leq \sum_{ \substack{J \subseteq \{1,\ldots,n\}} } \prod_{i \in J} |\sigma_i(\alpha)| = \prod_{j = 1}^n (1+|\sigma_j(\alpha)|) \\
	& \leq 2^n \prod_{j = 1}^n  \max(1,|\sigma_j(\alpha)|) \leq 2^n \cdot 2^{2n} \cdot |\dcrk| = 2^{3n} \cdot |\dcrk|.
\end{align*}
Hence, $\max_i \log |g_i|  = O(n + \log |\dcrk|)$. Using the fact that $n = O(\log |\dcrk|)$ concludes the proof.
\end{proof}

%% file: main.bbl
\begin{thebibliography}{10}

\bibitem{ALNS20}
D.~Aggarwal, J.~Li, P.~Q. Nguyen, and N.~{Stephens-Davidowitz}.
\newblock Slide reduction, revisited - filling the gaps in {SVP} approximation.
\newblock In {\em {CRYPTO}}, 2020.

\bibitem{apostol1998introduction}
T.~Apostol.
\newblock {\em Introduction to Analytic Number Theory}.
\newblock Undergraduate Texts in Mathematics. Springer New York, 1998.

\bibitem{atiyah69}
M.~F. Atiyah and I.~G. MacDonald.
\newblock {\em Introduction to commutative algebra.}
\newblock Addison-Wesley-Longman, 1969.

\bibitem{Bach90}
E.~Bach.
\newblock Explicit bounds for primality testing and related problems.
\newblock {\em Mathematics of Computation}, 55(191):355--380, 1990.

\bibitem{Bach95}
E.~Bach.
\newblock Improved approximations for {Euler} products.
\newblock {\em Number Theory: Fourth Conference of the Canadian Number Theory
  Association, July 2-8, 1994, Dalhousie University, Halifax, Nova Scotia,
  Canada}, 15, 1995.

\bibitem{Banaszczyk1993}
W.~Banaszczyk.
\newblock New bounds in some transference theorems in the geometry of numbers.
\newblock {\em Mathematische Annalen}, 296(4):625--636, 1993.

\bibitem{bayer2006upper}
E.~{Bayer Fluckiger}.
\newblock Upper bounds for euclidean minima of algebraic number fields.
\newblock {\em Journal of Number Theory}, 121(2):305--323, 2006.

\bibitem{NLA17}
L.~Beilina, E.~Karchevskii, and M.~Karchevskii.
\newblock {\em Numerical Linear Algebra: Theory and Applications}.
\newblock Springer, 09 2017.

\bibitem{bhargava2020}
M.~Bhargava, A.~Shankar, T.~Taniguchi, F.~Thorne, J.~Tsimerman, and Y.~Zhao.
\newblock Bounds on 2-torsion in class groups of number fields and integral
  points on elliptic curves.
\newblock {\em Journal of the American Mathematical Society}, 33(4):1087--1099,
  Oct. 2020.

\bibitem{Biasse14}
J.-F. Biasse.
\newblock Subexponential time relations in the class group of large degree
  number fields.
\newblock {\em Adv. Math. Commun.}, 8(4):407--425, 2014.

\bibitem{ANTS:BiasseFiecker14}
J.-F. Biasse and C.~Fieker.
\newblock Subexponential class group and unit group computation in large degree
  number fields.
\newblock {\em LMS Journal of Computation and Mathematics}, 17:385--403, 1
  2014.

\bibitem{BiasseSong14}
J.-F. Biasse and F.~Song.
\newblock A polynomial time quantum algorithm for computing class groups and
  solving the principal ideal problem in arbitrary degree number fields.
\newblock In {\em SODA}, 2016.

\bibitem{KoenThesis}
K.~{\noopsort{Boer}}{de Boer}.
\newblock {\em Random Walks on Arakelov Class Groups}.
\newblock PhD thesis, Mathematical Institute (MI) , Faculty of Science , Leiden
  University, 2022.

\bibitem{BDF20}
K.~{\noopsort{Boer}}{de Boer}, L.~Ducas, and S.~Fehr.
\newblock On the quantum complexity of the continuous hidden subgroup problem.
\newblock In {\em Annual International Conference on the Theory and
  Applications of Cryptographic Techniques}, pages 341--370. Springer, 2020.

\bibitem{dBDPW}
K.~{\noopsort{Boer}}{de Boer}, L.~Ducas, A.~Pellet-Mary, and B.~Wesolowski.
\newblock Random self-reducibility of ideal-svp via arakelov random walks.
\newblock In D.~Micciancio and T.~Ristenpart, editors, {\em Advances in
  Cryptology -- CRYPTO 2020}, pages 243--273, Cham, 2020. Springer
  International Publishing.

\bibitem{de2017calculating}
K.~{\noopsort{Boer}}{de Boer} and C.~Pagano.
\newblock Calculating the power residue symbol and ibeta.
\newblock In {\em ISSAC}, volume~68, pages 923--934, 2017.

\bibitem{buchmann1988subexponential}
J.~Buchmann.
\newblock A subexponential algorithm for the determination of class groups and
  regulators of algebraic number fields.
\newblock {\em S\'eminaire de th\'eorie des nombres, Paris}, 1989:28--41, 1988.

\bibitem{buchmann96}
J.~Buchmann and V.~Kessler.
\newblock Computing a reduced lattice basis from a generating system.
\newblock {\em Unpublished Manuscript}, 08 1996.

\bibitem{buchmann87}
J.~Buchmann and M.~Pohst.
\newblock Computing a lattice basis from a system of generating vectors.
\newblock In {\em Proceedings of the European Conference on Computer Algebra},
  EUROCAL '87, pages 54--63, London, UK, UK, 1989. Springer-Verlag.

\bibitem{buchmannlenstra_roi}
J.~A. Buchmann and H.~W. Lenstra.
\newblock Approximating rings of integers in number fields.
\newblock {\em Journal de Théorie des Nombres de Bordeaux}, 6(2):221--260,
  1994.

\bibitem{cassels2012introduction}
J.~Cassels.
\newblock {\em An Introduction to the Geometry of Numbers}.
\newblock Classics in Mathematics. Springer Berlin Heidelberg, 2012.

\bibitem{Cohen1993}
H.~Cohen.
\newblock {\em A course in computational algebraic number theory}, volume~8.
\newblock Springer-Verlag Berlin, 1993.

\bibitem{cohen1999advanced}
H.~Cohen.
\newblock {\em Advanced Topics in Computational Number Theory}.
\newblock Graduate Texts in Mathematics. Springer New York, 1999.

\bibitem{cohen1991polynomial}
H.~Cohen and F.~Diaz Y~Diaz.
\newblock A polynomial reduction algorithm.
\newblock {\em Journal de th{\'e}orie des nombres de Bordeaux}, 3(2):351--360,
  1991.

\bibitem{HDM97}
H.~Cohen, F.~Diaz Y~Diaz, and M.~Olivier.
\newblock Subexponential algorithms for class group and unit computations.
\newblock {\em Journal of Symbolic Computation}, 24(3-4):433--441, 1997.

\bibitem{cohen2008computational}
H.~Cohen and P.~Stevenhagen.
\newblock Computational class field theory, 2008.

\bibitem{conte80}
S.~D. Conte and C.~W. De~Boor.
\newblock {\em Elementary Numerical Analysis: an Algorithmic Approach}.
\newblock International Series in Pure and Applied Mathematics. McGraw-Hill,
  New York, NY, third edition, 1980.

\bibitem{Cover2006}
T.~M. Cover and J.~A. Thomas.
\newblock {\em Elements of Information Theory 2nd Edition (Wiley Series in
  Telecommunications and Signal Processing)}.
\newblock Wiley-Interscience, 7 2006.

\bibitem{CDW21}
R.~Cramer, L.~Ducas, and B.~Wesolowski.
\newblock Mildly short vectors in cyclotomic ideal lattices in quantum
  polynomial time.
\newblock {\em {J ACM}}, 2021.

\bibitem{eisentrager2014quantum}
K.~Eisentr{\"a}ger, S.~Hallgren, A.~Kitaev, and F.~Song.
\newblock A quantum algorithm for computing the unit group of an arbitrary
  degree number field.
\newblock In {\em STOC}, pages 293--302. ACM, 2014.

\bibitem{FORD_2002}
K.~Ford.
\newblock Vinogradov's integral and bounds for the riemann zeta function.
\newblock {\em Proceedings of the London Mathematical Society},
  85(3):565–633, 2002.

\bibitem{Friedman1989AnalyticFF}
E.~Friedman.
\newblock Analytic formulas for the regulator of a number field.
\newblock {\em Inventiones mathematicae}, 98:599--622, 1989.

\bibitem{gama2008finding}
N.~Gama and P.~Q. Nguyen.
\newblock Finding short lattice vectors within mordell's inequality.
\newblock In {\em Proceedings of the fortieth annual ACM symposium on Theory of
  computing}, pages 207--216. ACM, 2008.

\bibitem{Gantmacher59}
F.~R. Gantmacher.
\newblock {\em The Theory of Matrices, Vol. {I}, {II}}.
\newblock Chelsea Publishing Comp., New York, 1959.

\bibitem{gelin2016reducing}
A.~G{\'e}lin and A.~Joux.
\newblock Reducing number field defining polynomials: an application to class
  group computations.
\newblock {\em LMS Journal of Computation and Mathematics}, 19(A):315--331,
  2016.

\bibitem{DBLP:conf/stoc/GentryPV08}
C.~Gentry, C.~Peikert, and V.~Vaikuntanathan.
\newblock Trapdoors for hard lattices and new cryptographic constructions.
\newblock In {\em STOC}, pages 197--206, 2008.

\bibitem{Greni__2015}
L.~Grenié and G.~Molteni.
\newblock Explicit versions of the prime ideal theorem for {Dedekind} zeta
  functions under {GRH}.
\newblock {\em Mathematics of Computation}, 85(298):889–906, Oct 2015.

\bibitem{HM89}
J.~L. Hafner and K.~S. McCurley.
\newblock A rigorous subexponential algorithm for computation of class groups.
\newblock {\em Journal of the American mathematical society}, 2(4):837--850,
  1989.

\bibitem{HPS11}
G.~Hanrot, X.~Pujol, and D.~Stehl{\'e}.
\newblock Analyzing blockwise lattice algorithms using dynamical systems.
\newblock In {\em {CRYPTO}}, 2011.

\bibitem{HS07_Improved}
G.~Hanrot and D.~Stehl{\'e}.
\newblock Improved analysis of {Kannan’s} shortest lattice vector algorithm.
\newblock In {\em Annual international cryptology conference}, pages 170--186.
  Springer, 2007.

\bibitem{havivregev2014}
I.~Haviv and O.~Regev.
\newblock On the lattice isomorphism problem.
\newblock In {\em Proceedings of the twenty-fifth annual ACM-SIAM symposium on
  Discrete algorithms}, pages 391--404. SIAM, 2014.

\bibitem{horn2012matrix}
R.~Horn and C.~Johnson.
\newblock {\em Matrix Analysis}.
\newblock Cambridge University Press, 2012.

\bibitem{ipsenrehman08}
I.~C.~F. Ipsen and R.~Rehman.
\newblock Perturbation bounds for determinants and characteristic polynomials.
\newblock {\em SIAM Journal on Matrix Analysis and Applications},
  30(2):762--776, 2008.

\bibitem{iwaniec2004analytic}
H.~Iwaniec and E.~Kowalski.
\newblock {\em Analytic Number Theory}.
\newblock American Mathematical Society, 2004.

\bibitem{KaltofenVillard}
E.~Kaltofen and G.~Villard.
\newblock On the complexity of computing determinants.
\newblock {\em Computational complexity}, 13(3-4):91--130, 2005.

\bibitem{kessler}
V.~Kessler.
\newblock On the minimum of the unit lattice.
\newblock {\em Séminaire de Théorie des Nombres de Bordeaux}, 3(2):377--380,
  1991.

\bibitem{Krause90}
U.~Krause.
\newblock Absch{\"a}tzungen f{\"u}r die funktion $\omega$k (x, y) in
  algebraischen zahlk{\"o}rpern.
\newblock {\em Manuscripta mathematica}, 69(1):319--331, 1990.

\bibitem{LO77}
J.~Lagarias and A.~Odlyzko.
\newblock Effective versions of the {C}hebotarev density theorem.
\newblock In {\em Algebraic number fields: $L$-functions and Galois properties
  (Proc. Sympos., Univ. Durham, Durham, 1975)}, pages 409--464. Academic Press,
  London, 1977.

\bibitem{lagarias90:_korkin_zolot_bases_and_succes}
J.~C. Lagarias, H.~W. {Lenstra Jr.}, and C.-P. Schnorr.
\newblock {Korkin-Zolotarev} bases and successive minima of a lattice and its
  reciprocal lattice.
\newblock {\em Combinatorica}, 10(4):333--348, 1990.

\bibitem{lang1994algebraic}
S.~Lang.
\newblock {\em Algebraic Number Theory}.
\newblock Graduate Texts in Mathematics. Springer, 1994.

\bibitem{lenstra1993development}
A.~K. Lenstra, H.~W. {Lenstra Jr.}, et~al.
\newblock {\em The development of the number field sieve}, volume 1554.
\newblock Springer Science \& Business Media, 1993.

\bibitem{lenstra82:_factor}
A.~K. Lenstra, H.~W. {Lenstra Jr.}, and L.~Lov\'{a}sz.
\newblock Factoring polynomials with rational coefficients.
\newblock {\em Mathematische Annalen}, 261(4):515--534, December 1982.

\bibitem{louboutin00}
S.~Louboutin.
\newblock Explicit bounds for residues of {Dedekind} zeta functions, values of
  {L}-functions at s=1, and relative class numbers.
\newblock {\em Journal of Number Theory}, 2000.

\bibitem{MGbook}
D.~Micciancio and S.~Goldwasser.
\newblock {\em Complexity of Lattice Problems: a cryptographic perspective},
  volume 671 of {\em The Kluwer International Series in Engineering and
  Computer Science}.
\newblock Kluwer Academic Publishers, Boston, Massachusetts, 2002.

\bibitem{MicciancioRegev2007}
D.~Micciancio and O.~Regev.
\newblock Worst-case to average-case reductions based on gaussian measures.
\newblock {\em SIAM J. Comput.}, 37(1):267--302, Apr. 2007.

\bibitem{Min1}
H.~Minkowski.
\newblock {\em Gesammelte Abhandlungen}.
\newblock Chelsea, New York, 1967.

\bibitem{miyake2006modular}
T.~Miyake and Y.~Maeda.
\newblock {\em Modular Forms}.
\newblock Springer Monographs in Mathematics. Springer Berlin Heidelberg, 2006.

\bibitem{neukirch2013algebraic}
J.~Neukirch and N.~Schappacher.
\newblock {\em Algebraic Number Theory}.
\newblock Grundlehren der mathematischen Wissenschaften. Springer Berlin
  Heidelberg, 2013.

\bibitem{nguyen2009lll}
P.~Q. Nguyen and D.~Stehl{\'e}.
\newblock An {LLL} algorithm with quadratic complexity.
\newblock {\em SIAM Journal on Computing}, 39(3):874--903, 2009.

\bibitem{LLLAlgorithm2009}
P.~Q. Nguyen and B.~Vall\'ee.
\newblock {\em The LLL Algorithm: Survey and Applications}.
\newblock Springer Publishing Company, Incorporated, 1st edition, 2009.

\bibitem{overholt2014course}
M.~Overholt.
\newblock {\em A Course in Analytic Number Theory}.
\newblock Graduate Studies in Mathematics. American Mathematical Society, 2014.

\bibitem{Pardo:996837}
L.~Pardo.
\newblock {\em {Statistical Inference Based on Divergence Measures}}.
\newblock CRC Press, Abingdon, 2005.

\bibitem{PeikertRosen06}
C.~Peikert and A.~Rosen.
\newblock Lattices that admit logarithmic worst-case to average-case connection
  factors.
\newblock Cryptology {ePrint} Archive, Paper 2006/444, 2006.

\bibitem{PHS19}
A.~{Pellet-Mary}, G.~Hanrot, and D.~Stehl{\'e}.
\newblock Approx-{SVP} in ideal lattices with pre-processing.
\newblock In {\em Eurocrypt}, pages 685--716. Springer, 2019.

\bibitem{PP21}
M.~Plan{\c{c}}on and T.~Prest.
\newblock Exact lattice sampling from non-{Gaussian} distributions.
\newblock In J.~A. Garay, editor, {\em Public-Key Cryptography -- PKC 2021},
  pages 573--595, Cham, 2021. Springer International Publishing.

\bibitem{Pomerance1987}
C.~Pomerance.
\newblock Fast, rigorous factorization and discrete logarithm algorithms.
\newblock In D.~S. Johnson, T.~Nishizeki, A.~Nozaki, and H.~S. Wilf, editors,
  {\em Discrete Algorithms and Complexity}, pages 119--143. Academic Press,
  1987.

\bibitem{simplexvolume}
S.~Rabinowitz.
\newblock The volume of an n-simplex with many equal edges.
\newblock {\em Missouri Journal of Mathematical Sciences}, 1, 01 1989.

\bibitem{rosser1962approximate}
J.~B. Rosser and L.~Schoenfeld.
\newblock Approximate formulas for some functions of prime numbers.
\newblock {\em Illinois Journal of Mathematics}, 6(1):64--94, 1962.

\bibitem{samuel2013algebraic}
P.~Samuel.
\newblock {\em Algebraic Theory of Numbers}.
\newblock Hermann, Paris, 1970.

\bibitem{DBLP:journals/tcs/Schnorr87}
C.-P. Schnorr.
\newblock A hierarchy of polynomial time lattice basis reduction algorithms.
\newblock {\em Theor. Comput. Sci.}, 53:201--224, 1987.

\bibitem{Schoof2008computing}
R.~Schoof.
\newblock Computing {Arakelov} class groups.
\newblock In {\em Algorithmic Number Theory: Lattices, Number Fields, Curves
  and Cryptography}, pages 447--495. Cambridge University Press, 2008.

\bibitem{Sey87}
M.~Seysen.
\newblock A probabilistic factorization algorithm with quadratic forms of
  negative discriminant.
\newblock {\em Mathematics of computation}, 48(178):757--780, 1987.

\bibitem{Shoup1995}
V.~Shoup.
\newblock A new polynomial factorization algorithm and its implementation.
\newblock {\em Journal of Symbolic Computation}, 20(4):363 -- 397, 1995.

\bibitem{storjohann}
A.~Storjohann and G.~Labahn.
\newblock Asymptotically fast computation of {Hermite} normal forms of integer
  matrices.
\newblock In {\em Proceedings of the 1996 International Symposium on Symbolic
  and Algebraic Computation}, ISSAC '96, page 259–266, New York, NY, USA,
  1996. Association for Computing Machinery.

\bibitem{sunley1972class}
J.~E. Sunley.
\newblock On the class numbers of totally imaginary quadratic extensions of
  totally real fields.
\newblock {\em Bulletin of the American Mathematical Society}, 78(1):74--76,
  1972.

\bibitem{Trudgian2016}
T.~Trudgian.
\newblock Updating the error term in the prime number theorem.
\newblock {\em The Ramanujan Journal}, 39(2):225--234, Feb 2016.

\bibitem{von1999modern}
J.~von~zur Gathen and J.~Gerhard.
\newblock {\em Modern Computer Algebra}.
\newblock Cambridge University Press, 1999.

\bibitem{VONZURGATHEN20013}
J.~von~zur Gathen and D.~Panario.
\newblock Factoring polynomials over finite fields: A survey.
\newblock {\em Journal of Symbolic Computation}, 31(1):3 -- 17, 2001.

\bibitem{wesolowski_phd}
B.~P. Wesolowski.
\newblock {\em Arithmetic and geometric structures in cryptography}.
\newblock PhD thesis, \'Ecole Polytechnique F\'ed\'erale de Lausanne, 11 2018.

\end{thebibliography}
